\documentclass[a4paper,12pt,oneside,reqno]{amsart}
\usepackage[utf8]{inputenc}
\usepackage[english]{babel}
\usepackage[dvips]{graphicx}
\usepackage{amssymb}
\usepackage{amsmath}
\usepackage{amsthm}
\usepackage{url}
\usepackage[foot]{amsaddr}

\usepackage[hmargin=3cm,vmargin=2.5cm]{geometry}
\newcommand{\sectionspace}{\par\hspace{1cm}}

\def\[#1\]{\begin{align*}#1\end{align*}}
\def\be#1\ee{\begin{align}#1\end{align}}
\def\bea#1\eea{\begin{align}#1\end{align}}
\def\ben#1\een{\begin{align*}#1\end{align*}}

\newcommand{\diff}{d}

\newcommand{\n}[1]{\left\Vert #1\right\Vert}
\newcommand{\la}{\left\langle}
\newcommand{\ra}{\right\rangle}
\newcommand{\R}{\mathbb{R}}

\newcommand{\N}{\mathbb{N}}

\newcommand{\mc}[1]{\mathcal{#1}}
\newcommand{\s}{\mc}
\newcommand{\p}[2]{\frac{\partial #1}{\partial #2}}
\newcommand{\q}[2]{\frac{\partial^2 #1}{\partial #2^2}}
\newcommand{\pa}[1]{{{\partial}\over{\partial #1}}}
\newcommand{\dif}[1]{\frac{\diff}{\diff #1}}
\def\ol#1{\overline{#1}}
\newcommand{\longto}{\mathop{\longrightarrow}}
\newcommand{\SO}{\mathrm{SO}}

\newcommand{\qmatrix}[1]{ \left( \begin{matrix} #1 \end{matrix} \right) }

\newcommand{\domega}{\diff\omega}
\newcommand{\dE}{\diff E}
\newcommand{\dsigma}{\diff\sigma}

\newcommand{\esssup}{\mathop{\mathrm{ess\hspace{1mm}sup}}}
\newcommand{\essinf}{\mathop{\mathrm{ess\hspace{1mm}inf}}}

\def\[#1\]{\begin{align*}#1\end{align*}}
\def\be#1\ee{\begin{align}#1\end{align}}
\def\bea#1\eea{\begin{align}#1\end{align}}

\newtheoremstyle{theorem}{0.5cm}{0.5cm}%
   {}
   {}
   {\bfseries}
   {}
   {2ex}
   {\thmname{#1}\thmnumber{ #2}\thmnote{ #3}}
\theoremstyle{theorem}
\newtheorem{theorem}{Theorem}[section]

\newtheorem{proposition}[theorem]{Proposition}
\newtheorem{example}[theorem]{Example}
\newtheorem{corollary}[theorem]{Corollary}
\newtheorem{remark}[theorem]{Remark}
\newtheorem{definition}[theorem]{Definition}
\newtheorem{lemma}[theorem]{Lemma}

\newtheoremstyle{assumption}{0.5cm}{0.5cm}%
   {\sffamily}
   {}
   {\bfseries}
   {}
   {2ex}
   {\thmname{#1}\thmnote{ #3}.}
\theoremstyle{assumption}
\newtheorem{assumption}{Assumption}

\begin{document}

\title[On existence of $L^2$-solutions for coupled CSDA]
{On existence of $L^2$-solutions of Coupled Boltzmann Continuous Slowing Down transport equation system}



\author{J. Tervo$^1$}
\address{$^1$University of Eastern Finland, Department of Applied Physics, Kuopio, Finland}

\author{P. Kokkonen$^2$}
\address{$^2$Varian Medical Systems Finland Oy, Helsinki, Finland}
\email{$^2$pvkokkon@gmail.com}

\author{M. Frank$^3$}
\address{$^3$RWTH Aachen University, MATHCCES, Schinkelstrasse 2, 52062, Germany}
\email{$^3$frank@mathcces.rwth-aachen.de}

\author{M. Herty$^4$}
\address{$^4$RWTH Aachen University, IGPM, Templergraben 5, 52062, Germany}
\email{$^4$herty@igpm.rwth-aachen.de}

\begin{abstract}
The paper considers a coupled system of linear Boltzmann transport equations (BTE),
and its Continuous Slowing Down Approximation (CSDA).
This system can be used to model the relevant transport of particles used e.g. in dose calculation in radiation therapy.
The evolution of charged particles (e.g. electrons and positrons) are in practice often modelled 
using the CSDA version of BTE because of the so-called forward peakedness of scattering events contributing to the 
particle fluencies (or particle densities),
which causes severe problems in numerical methods.
We shall find, after the preliminary treatments, that for some interactions CSDA-type modelling is actually
necessary due to hyper-singularities in the differential cross-sections of certain interactions,
that is, first or second order partial derivatives with respect to energy and angle
must be included into the transport part of charged particles.
The existence and uniqueness of (weak) solutions is shown, under sufficient criteria and
in appropriate $L^2$-based spaces, for a single (particle) CSDA-equation by using three
techniques, the Lions-Lax-Milgram Theorem (variational approach), the theory of $m$-dissipative operators
and the theory evolution operators (semigroup approach).
The due a priori estimates are derived and the positivity of solutions are retrieved.
In addition, we prove the corresponding results and estimates for the system of coupled transport equations.  The related existence results are given for the adjoint problem as well.
We also give some computational points (e.g. certain explicit formulas),
and we outline a related inverse problem 
at the end of the paper.
\end{abstract}

\maketitle


\sectionspace
\section{Introduction}\label{intro}

The \emph{Boltzmann transport equation} (BTE) models changes of the number density of particles in phase space (position, velocity direction, energy).
In this paper the species of particles include photons, electrons and positrons and the explored analysis of transport equations is mainly intended for dose calculation in radiation treatment planning.
However, various other kinds of transport phenomena can be modelled by equations of similar type including
in e.g. transport of particles in optical tomography (\cite{anikonov02}, \cite{arridge}), in cosmic radiation (\cite{wilson}) and in solid state physics (\cite{madelung78}).
For general theory of linear BTE with relevant boundary conditions we refer to \cite{dautraylionsv6} and \cite{agoshkov}. See also \cite{case}, \cite{cercignani}, \cite{duderstadt}, \cite{pomraning} where the subject is considered from a physical point of view.
For some recent issues (including certain inverse problems) related to linear BTE can be found in \cite{mokhtarkharroubi},
and  general non-linear aspects e.g. in \cite{ukai}, \cite{bellamo}.
A thorough mathematical survey (mathematical and physical foundations, results, problems)
of non-linear collision theory of particle transport is given in \cite{villani}. This survey is mainly
intended to collision processes in dilute gases and plasmas but analogous results and problems arise in other fields of particle 
physics. Finally, for topics related to Monte--Carlo methods in the context of BTE, both from theoretical and practical point of 
view, we refer to \cite{lapeyre}, \cite{seco} and \cite{spanier08}.

Dose calculation is of crucial importance in radiation therapy. Relevant dose calculation models 
require the (approximate) solution of a coupled system of (linear) transport equations
for fluencies (number densities in the phase space) for all considered particles.
This is a difficult problem, at least from computational point of view, due to the different particle species and their dependence on a high--dimensional phase space.
For that reason traditional dose 
calculation algorithms have applied some closed-form formulas which have their origins in analytical 
solutions, or Monte--Carlo derived solutions, of simplified problems. These algorithms however contain often empirically derived corrections 
to take more accurately into account the underlying particle physics (\cite{mayles07}, \cite{seco}).
Certain "factors" which account for e.g. the spatial inhomogeneities must be included to improve 
the accuracy of the final solution. These approaches lead to methods that are fast enough but have
typically a limited accuracy. Commonly used models are based on the so-called \emph{pencil beams}, or \emph{point kernels}, see \cite{asadzadeh}, \cite{borgers}, \cite{larsen}, \cite{mayles07}, \cite{tillikainen08}, \cite{ulmer} for more details.
A notable exception to these approximate (deterministic) methods is the Acuros code \cite{vassiliev}, which is based on a discretization of the BTE.

In radiation therapy BTE describes the evolution of  radiative particles due to scattering and absorption in tissue.
The dose delivery methods can be roughly divided into two categories.
In \emph{external therapy} the sources (below denoted by $g$) of high 
energy particles (usually photons, electrons or protons) are on the patches of patient's surface.
In \emph{internal therapy}, on the other hand, the sources
(below denoted by $f$)
are inside the patient close to the cancerous tissue.
In the energy range, say up to 25 MeV, relevant for photon and electron therapy,
the three species of particles whose simultaneous evolution should be taken into account in a realistic transport model, are photons, electrons and positrons.
In this setting, the potential creation of (or contamination by)
other heavy particles
will not be taken into account since their contribution to the dose is negligible (see \cite{seco}). 

The transport of relevant particles in tissue (in an appropriate energy range)
can be {\it formally} modelled by the following linear \emph{coupled system of three BTEs}
\be\label{intro1}
\omega\cdot\nabla_x\psi_j(x,\omega,E)+\Sigma_j(x,\omega,E)\psi_j(x,\omega,E)-(K_j\psi)(x,\omega,E)=f_j(x,\omega,E),
\ee
for $(x,\omega,E)\in G\times S\times I$ and
for $j=1,2,3$, combined with an \emph{inflow boundary condition}
(for the definition of $\Gamma_-$, )
\be\label{intro2}
{\psi_j}_{|\Gamma_-}=g_j,\quad j=1,2,3,
\ee
where for $j=1,2,3$,
\be\label{intro3}
(K_j\psi)(x,\omega,E)=\sum_{k=1}^3\int_{S\times I}\sigma_{kj}(x,\omega',\omega,E',E)\psi_k(x,\omega',E') \domega' \dE'.
\ee
This formal description of is however somewhat confusing because of hyper-singular nature of $K$. We shall explain this problem in detail   in this paper.
Here $G\subset\R^3$ is the spatial domain, $S=S_2\subset \R^3$ is the unit sphere
and $I=[0,E_{\rm m}]$ is the energy interval.
The set $\Gamma_-$ is "the inflow boundary part of $\partial G$" (see section \ref{fs}).
Above $(\omega',E')$ refers to the direction and energy of incoming particles and $(\omega,E)$ to the direction and energy of leaving particles.
For a derivation of linear BTE, see e.g. \cite{agoshkov}, \cite{allaire12}, \cite{duderstadt}, \cite{stacey01}.
The first term on the left in (\ref{intro1}) is called a \emph{convection (or advection) operator}, the second term is a \emph{(total) scattering operator}
and the third one is a \emph{collision operator}.
Notice that the (total) scattering operator
\[
\Sigma=\Sigma_{\rm t}=\Sigma_{\rm a}+\Sigma_{\rm s}
\]
(we drop the index $j$ here to simplify the notation)
contains contribution from both the absorption (term $\Sigma_{\rm a}$) and the scattering (term $\Sigma_{\rm s}$),
see \cite[Sec. 9.1]{stacey01}.
On the right in (\ref{intro1}), the functions $f_j$ represent (internal) sources and $g_j$ in (\ref{intro2}) are (inflow) boundary sources.
The system is coupled through the integral operators $K_j$ (unless, of course, $\sigma_{kj}=0$ for $j\not= k$). The solution $\psi=(\psi_1,\psi_2,\psi_3)$ of the problem (\ref{intro1})-(\ref{intro2}) 
is a vector-valued function whose components describe the radiation fluxes of photons, electrons and positrons, respectively.
Roughly speaking, the flux $\psi(x,\omega,E)$ is the flux of energy through a surface located at $x$ and normal to the direction $\omega$.
The particle number density $N$, which is another usual unknown in kinetic 
theory, is related to $\psi$ by $\psi = \n{v} N$,
where $\n{v}$ is the particle speed (\cite{stacey01}),
which is often relativistic.

The equation (\ref{intro1}) is a steady state counterpart of the dynamical equation
\be\label{intro4}
{1\over{\n{v_j}}}{\p {\psi_j}t}+\omega\cdot\nabla_x\psi_j+\Sigma_j\psi_j-K_j\psi=f_j,\quad j=1,2,3,
\ee
where $v_j$ is the velocity of the $j$-th particle type.
In radiation therapy related applications,
it is sufficient to consider the steady state equations because the
flux $\psi$ reaches the steady state nearly instantly (\cite{borgers99}).
The existence of solutions for the problem (\ref{intro1}), (\ref{intro2}),
as well as for the time-dependent problem  (\ref{intro4}), (\ref{intro2}) (with an appropriate initial condition) in $L^1$-based spaces has been studied in \cite{tervo17-up}
(the results of which remaining valid, after slight modifications, for any $1\leq p<\infty$).
In \cite{tervo17-up} it is assumed that the 
collision operator $K$ satisfies a so-called Schur criterion (for boundedness) which is not valid for all species of particle interactions.

The differential cross-sections
may have singularities, or even hyper-singularities, and in these cases the integral $\int_{S\times I}$ appearing in the collision term must be interpreted as the {\it Hadamard finite part integral}.
In section \ref{coll} we shall present some details of real, physical collision operators.
In particular, we find that certain differential cross sections may contain hyper-singularities like ${1\over{(E'-E)^j}}dE',\ j=1,2$.
These kind of singularities lead to extra pseudo-differential differential like
(or approximately partial differential) terms in the transport equation.
The analysis reveals the exact form of certain charged particle transport operators, and serves as a basis for better approximations
and error analysis.
These considerations would require further knowledge about the regularity properties of solutions,
which, to our understanding, remains an open question (cf. section \ref{err}) .
Likewise the existence analysis of (non-CSDA) exact (coupled) system of transport equations remains open.
In this work, we demonstrate how the, here considered, CSDA-approximation follows from "hyper-singular analysis" and how the CSDA equations can be derived. 
When the collision terms containing hyper-singularities are removed,
the remaining operators (so called restricted collision operators) are of the co called Schur form (\cite{halmos}, p. 22)
and hence they are, for example bounded operators.
Note that in the expression of $K_j$ the integration is performed only with respect
to $\omega'$ and $E'$, while $x$-variable is kept fixed.
Therefore, the restricted collision operators are only a so-called \emph{partial (weakly singular) integral operator} (cf. \cite{appell}).
These facts imply that the familiar properties of (singular) integral operators (compactness, for example) are not valid even for restricted collision operators.

Due to the above mentioned hyper-singularities,
the transport of new primary electrons (and positrons) is \emph{forward-peaked}.
This implies that the grid in numerical computations needs to be very tight in order to achieve reliable results.
Traditionally, to overcome the computational complexity,
one has applied to the evolution of electrons and positrons
the so-called \emph{continuous slowing down approximation} (CSDA).
The CSDA equations are vastly applied e.g. in various cosmic radiation problems (\cite{rockell}, \cite{wilson}).
The analysis given in section \ref{coll} shows that the exact transport model inherently contains pseudo-differential like terms. In some cases these terms can be successfully approximated by 
first order (or  even higher order) partial differential terms, and so the traditional use of CSDA-approximation can be justified.

The CSDA-method replaces Eqs. (\ref{intro1}) for $\psi_j$, $j=2,3$ (electrons and positrons)
by the following equations (\cite{wilson}, \cite{larsen}, \cite{frank10})
\be\label{intro5}
-{\p {(S_{j,r}\psi_j)}E}+\omega\cdot\nabla_x\psi_j+\Sigma_{j,r}\psi_j-K_{j,r}\psi=f_j,\quad j=2,3,
\ee
 where
\bea
(K_{2,r}\psi)(x,\omega,E)={}&\int_{S\times I}\sigma_{12}(x,\omega',\omega,E',E)\psi_1(x,\omega',E') \domega' \dE'\nonumber\\
&
+\int_{S\times I}\sigma_{22,r}(x,\omega',\omega,E',E)\psi_2(x,\omega',E') \domega' \dE'\nonumber\\
&
+\int_{S\times I}\sigma_{32}(x,\omega',\omega,E',E)\psi_3(x,\omega',E') \domega' \dE'\nonumber
\eea
and
\bea
(K_{3,r}\psi)(x,\omega,E)={}&\int_{S\times I}\sigma_{13}(x,\omega',\omega,E',E)\psi_1(x,\omega',E') \domega' \dE'\nonumber\\
&
+\int_{S\times I}\sigma_{23}(x,\omega',\omega,E',E)\psi_2(x,\omega',E') \domega' \dE'\nonumber\\
&
+\int_{S\times I}\sigma_{33,r}(x,\omega',\omega,E',E)\psi_3(x,\omega',E') \domega' \dE'.\nonumber
\eea
Here, for $j=2,3$, the functions $\Sigma_{j,r}=\Sigma_{j,r}(x,\omega,E)$ 
are the \emph{restricted total cross-sections}, 
$S_{j,r}=S_{j,r}(x,E)$ are the \emph{restricted stopping powers}, and
$\sigma_{jj,r}(x,\omega',\omega,E',E)$ are the \emph{restricted differential cross-sections},
which do not include soft inelastic interactions (\cite{boman}).
Besides the inflow boundary condition (\ref{intro2}),
one needs to impose on the solutions $\psi_j$, $j=2,3$,
of \eqref{intro5} an energy-boundary condition.
We make the reasonable assumption that at some (high enough, but finite) \emph{cut-off energy} $E_{\rm m}>0$
the fluxes $\psi_2$, $\psi_3$ vanish, i.e.
we take the energy-boundary condition to be
\[
\psi_2(x,\omega, E_{\rm m})=\psi_3(x,\omega, E_{\rm m})=0,
\]
(below, we often call this \emph{an initial condition} as well).
One could also demand the cut-off energy $E_{\rm m}$ to be infinite, in which case the
corresponding energy boundary condition would become,
\[
\lim_{E\to\infty}\psi_2(x,\omega, E)=\lim_{E\to\infty}\psi_3(x,\omega, E)=0.
\]
However, we shall restrict our discussion to the case where $E_{\rm m}$ is finite.
The requirement that such energy initial condition will be satisfied by $\psi_2,\psi_3$,
makes the overall problem mathematically well-posed, that is,
under relevant (physical) assumptions the problem has a unique (weak) solution.

We give here a short heuristic (traditional) derivation of CSDA
based on the presentation given in \cite[pp. 14-17]{wilson},
while a more transparent justification of it will be given in section \ref{coll}.
Firstly, we decompose for $k=2,3$ the differential cross-section into two parts
\be\label{w-decomp}
\sigma_{kk}(x,\omega',\omega,E',E)=
\sigma_{kk}^{\rm at}(x,\omega',\omega,E',E)+\sigma_{kk}^{\rm nu}(x,\omega',\omega,E',E), \ k=2,3
\ee
where "${\rm at}$" refers to interactions with atomic electrons, and "${\rm nu}$" refers to nuclear interactions.
The term $\sigma_{kk}^{\rm at}(x,\omega',\omega,E',E)$ contains the "problematic" features that
are to be approximated by the continuous slowing down model.
Secondly, we assume formally 
\[
\sigma_{kk}^{\rm at}(x,\omega',\omega,E',E)
=
\sum_{n=0}^N\sigma^{\rm at}_{k,n}(x,E+\epsilon_n)\tilde\delta_{\omega}(\omega')\delta_{E+\epsilon_n}(E'),
\]
where $\tilde{\delta}_{\omega}$ is the Dirac measure on the 2-dimensional unit sphere $S=S_2$ concentrated at $\omega$,
and $\delta_{E+\epsilon_n}$ is the Dirac measure on $\R$ concentrated at $E+\epsilon_n$.
This approximation assumes small changes in energy but not changes in angular direction.
Applying these concepts and Taylor's formula up to first--order, we find that
\bea
&\int_{S\times I}\sigma_{kk}^{\rm at}(x,\omega',\omega,E',E)\psi_k(x,\omega',E') \domega' \dE'
=
\sum_{n=0}^N\sigma^{\rm at}_{k,n}(x,E+\epsilon_n)\psi_k(x,\omega,E+\epsilon_n) \nonumber\\
\approx {}&
\sum_{n=0}^N\sigma^{\rm at}_{k,n}(x,E)\psi_k(x,\omega,E)
+\sum_{n=0}^N{\p {(\sigma^{\rm at}_{k,n}\psi_k)}E}(x,\omega,E)\epsilon_n. \label{f6}
\eea
The restricted total atomic cross-sections and the restricted stopping powers are
\[
\Sigma^{\rm at}_{k,r}(x,E):=\sum_{n=0}^N\sigma^{\rm at}_{k,n}(x,E),
\quad
S_{k,r}(x,E):=\sum_{n=0}^N\sigma^{\rm at}_{k,n}(x,E)\epsilon_n.
\]
Let $\sigma_{kk,r}:=\sigma_{kk}^{\rm nu}$.
Then by (\ref{w-decomp} and (\ref{f6}) we obtain for $k=2,3$,
\bea\label{intro9}
&
\int_{S\times I}\sigma_{kk}(x,\omega',\omega,E',E)\psi_k(x,\omega',E') \domega' \dE'\nonumber\\
={}&
\int_{S\times I}\sigma_{kk}^{\rm at}(x,\omega',\omega,E',E)\psi_k(x,\omega',E') \domega' \dE'
+
\int_{S\times I}\sigma_{kk,r}(x,\omega',\omega,E',E)\psi_k(x,\omega',E') \domega' \dE'\nonumber\\
\approx {}&
\Sigma_{k,r}^{\rm at}(x,E)\psi_k(x,\omega,E)
+{\p {(S_{k,r}\psi_k)}E}(x,\omega,E)
+
\int_{S\times I}\sigma_{kk,r}(x,\omega',\omega,E',E)\psi_k(x,\omega',E') \domega' \dE'
.
\eea
Finally, writing $\Sigma_{j,r}:=\Sigma_{j}-\Sigma_{j,r}^{\rm at}$
and substituting (\ref{intro9}) into (\ref{intro1}),
we obtain approximately the equations (\ref{intro5}) for $j=2,3$.

We remark that
the so-called CSDA-\emph{Fokker-Plank} equation  is closely related to the CSDA-Boltzmann transport equation.
The hierarchy of the Boltzmann equation, its CSDA approximation and the Fokker-Planck approximation is detailed in Appendix A of \cite{Larsen97}.
Typically, the Fokker-Planck approximation is considered not valid for electrons in tissues \cite{Larsen97}.
In addition, in \cite{pom92} a formal asymptotic analysis is performed that sheds light onto the validity of these 
approximations. For example, it was shown that the Fokker-Planck approximation is not valid for the often used Henyey-Greenstein kernel.

The \emph{dose} $D(x)=(D\psi)(x)$ is calculated from the solution of the problem
\begin{gather}
\omega\cdot\nabla_x\psi_1+\Sigma_1\psi_1-K_1\psi=f_1, \label{intro10a}\\
-{\p {(S_{j,r}\psi_j)}E}+\omega\cdot\nabla_x\psi_j+\Sigma_{j,r}\psi_j-K_{j,r}\psi=f_j,\quad j=2,3,
\label{intro10} \\
{\psi_j}_{|\Gamma_-}=g_j,\quad j=1,2,3, \label{intro11} \\[2mm]
\psi_j(\cdot,\cdot,E_{\rm m})=0,\quad j=2,3, \label{intro12}
\end{gather}
by 
\be\label{intro13}
D(x)=\sum_{j=1}^3\int_{S\times I}\varsigma_j(x,E)\psi_j(x,\omega,E) \domega \dE,
\ee
where $\varsigma_j(x,E)$ are {\it stopping powers}, which in general can be different from the restricted stopping powers $S_{j,r}$. The dose calculation is a {\it forward problem}. 
The determination of the external particle flux $g=(g_1,g_2,g_3)$ and/or the distribution of 
internal source $f=(f_1,f_2,f_3)$ is called \emph{inverse radiation treatment planning problem} 
(IRTP) which is an \emph{inverse problem}. It always requires a dose calculation model. We refer to 
\cite{schepard}, \cite{tervo17-up}, \cite{webb} and references therein for some details concerning the 
IRTP-problem. In \cite{frank10} the IRTP-problem has been studied in the context of CSDA-equation for a 
single particle (when the stopping power is independent of $x$). 
See also \cite{boman} where related  spatially $3$-dimensional numerical simulations (real case 
simulations applying finite element methods, FEM) have been explored.

This paper contains several novel contributions to the study of particle transport in tissues.
In the beginning section \ref{pre} we discuss the preliminaries including details
(many of which are reproductions of known results) of the so-called escape-time mapping and inflow trace theory.
These tools are essential in the treated analysis.

In section \ref{coll} we discuss the M\o ller  and Bremsstrahlung interaction.  These interactions model  electron's (and positron's) inelastic collisions. 
From mathematical point of view these  interactions are quite similar; both of them are governed by partial hyper-singular integral operators.

In section \ref{r-psio-1} we deduce a CSDA-type approximation (for the M\o ller scattering) retaining its pseudo-differential nature.
The approximation  is founded on the fact that the energy interval $[E,E_m]$ is decomposed into a union $[E,E_m]=[E,c(E)]\cup [c(E),E_m]$. Here $c(E)$ is so called {\it cut-off energy of the primary electrons}. Usually $c(E)=2E$ but this is not necessary. In "primary electrons' interval" $[E,c(E)]$ we assume that $\mu(E',E)\approx 1$ where $\mu(E',E)$ is the cosine of the scattering angle for primary electrons.  This approximation suggests that 
\be\label{gam-app}
\psi(x,\gamma(E',E,\omega)(s),E')\approx \psi(x,\omega,E') 
\ee
where $\gamma(E',E,\omega)$ is the path on $S$ where the cosine of scattering angle of incoming and leaving particles is $\mu(E',E)$  (see section \ref{fhsio}). That leads after some technicalities to the 
approximative collision operator which is a pseudo-differential operator whose principal term is of the form $\hat\sigma_2(x,E,E){\p {(P_1(D)\psi)}E}$ where
\[
(P_1(D)\psi)(x,\omega,E)={\rm p.f.}\int_E^\infty{1\over{E'-E}}\psi(x,\omega,E')dE'.
\]
$P_1(D)$
is a translation invariant pseudo-differential operator with symbol of the form $p_1(\xi)=-\ln(|\xi|)+\alpha_0+{\rm i}\beta_0{\rm sign}(\xi)$ (where $\xi$ is the transform variable with respect to energy).
When we further  approximate the pseudo-differential terms by partial differential terms we get a CSDA-type approximation of the transport equation (section \ref{a-psi}).
The analysis gives raise to error estimates and in section \ref{err} we sketch some principles of the error analysis.

Section \ref{rco} deals  with   {\it restricted collision operators} $K_r$ which we assume to be of the form
\bea\label{intro5-restr}
&
(K_r\psi)(x,\omega,E)=
\int_{I'}\hat{\sigma}^1(x,E',E)\psi(x,\omega,E') dE'
\nonumber\\
&
+
\int_{S'}{\sigma}^2(x,\omega',\omega,E)\psi(x,\omega',E) d\omega'
+
\int_{I'}\hat{\sigma}^3(x,E',E)\int_{0}^{2\pi}\psi(x,\gamma(E',E,\omega)(s),E')ds dE'
.
\eea
We show  that $K_r$ is bounded in $L^2(G\times S\times I)$. The result
is based on the { Schur criterion for the boundedness of (partial) integral operators}. In addition, we prove that $\Sigma-K_r$ is coercive (accretive) in $L^2(G\times S\times I)$ under relevant assumptions.

After  the above mentioned modelling aspects we consider the existence and uniqueness of solutions for a single (particle) CSDA equation 
\begin{gather}
-{\p {(S_{0}\psi)}E}+\omega\cdot\nabla_x\psi+\Sigma\psi-K\psi=f, \label{intro14} \\
\psi_{|\Gamma_-}=g, \label{intro15} \\[2mm]
\psi(\cdot, \cdot, E_{\rm m})=0, \label{intro16}
\end{gather}
in $L^2(G\times S\times I)$-based spaces.
We somewhat extend the results of \cite{tervo16-up} (section 3.3), \cite{tervo17}.
In section \ref{esols} we give a variational formulation of the problem 
(\ref{intro14}), (\ref{intro15}), (\ref{intro16}) in appropriate spaces. We give sufficient 
conditions under which the corresponding bilinear and linear forms
obey the assumptions of the 
so-called {\it Lions-Lax-Milgram Theorem},
which then provides an existence result for solution of the equation (\ref{intro14}).
Under a certain additional assumption on traces, this solution is unique and,
additionally, it satisfies  the boundary and initial conditions  (\ref{intro15}), (\ref{intro16}). 
In Section \ref{m-d} one applies the classical {\it theory of (formally) dissipative operators}
(e.g. by \cite{friedrich58}, \cite{lax}, \cite{rauch85}),
and it is via this approach that we achieve the most satisfactory existence and uniqueness result for the problem considered.
In section \ref{evcsd} we show the existence and uniqueness results for the solution
of the problem (\ref{intro14}), (\ref{intro15}), (\ref{intro16})
by using a yet another technique,
based on the theory of evolution equations and ($m$-)dissipative operators,
but for a restricted class of collision operators $K$. 
We remark that the single CSDA equation is (symmetric) hyperbolic in nature,
and so the hyperbolic theory (e.g. \cite{rauch12})
is applicable at least in the case where $K=0$ and $G=\R^3$. For $G\not=\R^3$ the existence and regularity of solutions is more 
subtle because of the inflow boundary condition. We also remark that the coupled system
(\ref{intro10a}), (\ref{intro10}) is not hyperbolic.

Section \ref{possol} deals with the non-negativity of solutions. The proofs therein are based on the Trotter's formula, the  obtained a priori estimates and (the verified) maximum principles for some transport problems (which do not require extra smoothness of solutions).

In section \ref{cosyst} we extend the corresponding existence and uniqueness results for the
coupled transport system (\ref{intro10a})-(\ref{intro12}), which has not been studied in the 
literature, in $L^2(G\times S\times I)^3$-based spaces. We also show certain a priori estimates 
(needed e.g. in section \ref{irtpre}) which in particular show  (under specific assumptions) that the solution depends continuously on the data.
Analogous results for the {\it adjoint problem} are formulated in section \ref{adjoint}.
Section \ref{comp} considers certain computational aspects. The emphasis is on how 
to calculate numerical solution (for the forward problem), in principle, without inversion of (huge) matrices.
Finally, in the last section \ref{irtpre} we outline a related IRTP-problem but its thorough study remains open for future work.

We make the following final remark. Concerning for the M\o ller scattering the outline of the error analysis in section \ref{err} suggests that (in the treated sense) the solution of the approximative problem does not necessarily converge (without further information) to the solution of the exact problem as $\mu(E',E)\to 1$.
Instead of (\ref{gam-app}) we must use the second order Taylor's polynomials for the angular approximation. When considering partial integro-differential approximations, this leads to the transport equation of the form 
(even the higher order  partial derivatives might be reasonable to include  for accuracy and stability reasons);
for some details see Remark \ref{con-mol} below
\be\label{tr-op-aa}
a(x,E){\p {\psi}E}
+\sum_{|\alpha|\leq 2}b_\alpha(x,\omega,E)\partial_\omega^\alpha\psi
+\omega\cdot\nabla\psi+\Sigma(x,\omega,E)\psi-K_r\psi=f.
\ee
Nevertheless, the approximation (\ref{gam-app}) is sufficient to guarantee the convergence (and so the modelling using (\ref{intro14}) is correct) at least for collision operators of the form
\[
&
(K\psi)(x,\omega,E)=
{\rm p.f.}\int_E^\infty\hat\sigma_{2}(x,E',E){1\over{(E'-E)^{1+\kappa_2}}}
\int_{0}^{2\pi}\psi(x,\gamma(E',E,\omega)(s),E')dsdE'
\nonumber\\
&
+
{\rm p.f.}\int_E^\infty\hat\sigma_{1}(x,E',E){1\over{(E'-E)^{\kappa_1}}}
\int_{0}^{2\pi}\psi(x,\gamma(E',E,\omega)(s),E')dsdE'
+(K_0\psi)(x,\omega,E)
\]
where $0\leq\kappa_j<{1\over 2}$.
Similarly, for the Bremsstrahlung the approximation (\ref{gam-app}) seems to be sufficient guaranteeing the modelling by (\ref{intro14}).
The analysis of equations (\ref{tr-op-aa}) require further study. In principle, the techniques utilised in this paper apply to them but  we omit  these treatments.

\sectionspace
\section{Preliminaries}\label{pre}

\subsection{Notations, Assumptions and Introduction of Relevant Function Spaces}\label{fs}

We assume that $G$ is an open bounded connected set in $\R^3$
such that $\ol{G}$ is a $C^1$-manifold with boundary (as a submanifold of $\R^3$; cf. \cite{lee}).
In particular, it follows from this definition that $G$ lies on one side of its boundary.

The unit outward (with respect to $G$) pointing normal on $\partial G$ is denoted by $\nu$,
and the surface measure (induced by the Lebesgue measure) on $\partial G$ is written as $\sigma$.
We let $S=S_2$ be the unit sphere in $\R^3$ equipped with the
usual rotationally invariant surface measure $\mu_{S}$.
Here and in what follows we typically refer to
the measures of interest to us, namely Lebesgue measure $\mc{L}^3$ in $\R^3$,
and the above (surface) measures $\sigma$ and $\mu_{S}$
simply by $dx$, $\dsigma$ and $\domega$
in the sense that for all sets $A$ measurable
with respect to relevant one of them,
\[
\mc{L}^3(A)=\int_A \diff x,
\quad
\sigma(A)=\int_A \dsigma,
\quad
\mu_{S}(A)=\int_A \domega.
\]

Furthermore, let $I=[0,E_{\rm m}]$  where
$0<E_{\rm m}<\infty$. We could replace $I$ by
$I=[E_0,E_{\rm m}]$ or $I=[E_0,\infty[$ where
$E_0\geq 0$ but we omit this generalization here.
We shall denote by $I^\circ$ the interior $]0,E_m[$ of $I$.
The interval $I$ in equipped with the 1-dimensional Lebesgue measure $\mc{L}^1$,
which we typically write as $dE$ in the sense that $\mc{L}^1(A)=\int_A dE$.

\begin{definition}
For $(x,\omega)\in G\times S$ the \emph{escape time (in the direction $\omega$)} $t(x,\omega)=t_-(x,\omega)$ is defined by 
\be
t(x,\omega)&=\inf\{s>0\ |\ x-s\omega\not\in G\}\\ \nonumber
&=\sup\{T>0\ |\ x-s\omega\in G\ {\rm for\ all}\ 0<s<T\},
\ee
for $(x,\omega)\in G\times S$.
\end{definition}

The escape time function $t(\cdot,\cdot)$ 
is known to be lower semicontinuous
in general, and continuous if $G$ is convex, see e.g. \cite{tervo17-up}.

We define
\[
\Gamma':=(\partial G)\times S,\quad
\Gamma:=\Gamma'\times I,
\]
and
\begin{alignat*}{3}
\Gamma_0' :={}&\{(y,\omega)\in \Gamma'\ |\ \omega\cdot\nu(y)=0\},
\quad &&\Gamma_0:=\Gamma'_0\times I, \\[2mm]
\Gamma_{-}':={}&\{(y,\omega)\in \Gamma'\ |\ \omega\cdot\nu(y)<0\}, && \Gamma_{-}:=\Gamma_{-}'\times I,\\[2mm]
\Gamma_{+}':={}&\{(y,\omega)\in \Gamma'\ |\ \omega\cdot\nu(y)>0\}, &&\Gamma_{+}:=\Gamma_{+}'\times I.
\end{alignat*}
Let $\mu_\Gamma = \sigma\otimes \mu_S\otimes \mc{L}^1$,
written typically as $\diff \sigma\diff \omega\diff E$
in the same sense as discussed above.
It follows that $\mu_{\Gamma}(\Gamma_0)=0$ 
and
\[
\Gamma=\Gamma_0\cup \Gamma_-\cup \Gamma_+.
\]
Let
\be\label{N0}
N_0:=\{(x,\omega,E)\in G\times S\times I\ |\ \omega\cdot\nu\big(x-t(x,\omega)\omega\big)=0\},
\ee
and
\be\label{D}
D:=(G\times S\times I)\setminus N_0.
\ee
Recall that $N_{0}$ has a measure zero in $ G\times S\times I$ (\cite[Theorem 3.8]{tervo17-up}). 

\begin{definition}
Define \emph{escape-time mappings $\tau_{\pm}(y,\omega)$ from boundary to boundary in the direction $\omega$}
as follows
\begin{alignat}{3}
&
\tau_-(y,\omega):=
\inf\{s>0\ |\ y+s\omega\not\in G\},\quad && (y,\omega)\in \Gamma'_-, \\
&
\tau_+(y,\omega):=
\inf\{s>0\ |\ y-s\omega\not\in G\},\quad && (y,\omega)\in \Gamma'_+.
\end{alignat}
\end{definition}
Note that if $(y,\omega)\in\Gamma'_-$,
then writing $y_+:=y+\tau_-(y,\omega)\omega\in\Gamma_+$,
we have that $(y_+,\omega)\in \Gamma'_+$ and
\[
\tau_-(y,\omega)=\tau_+(y_+,\omega).
\]

\begin{remark}\label{re:S_Radon}
For some purposes of this paper,
we could have chosen, instead of the particular measure $\mu_S$,
any positive Radon measure $\rho$ on (Borel sets of) the unit sphere $S$, 
as long as relevant additional assumptions are supposed, for example that $\Gamma_0$ has measure zero with respect to $\sigma\otimes \rho\otimes \mc{L}^1$,
i.e.
$\int_{\Gamma} \chi_{\Gamma_0}(y,\omega,E)d\sigma(y)d\rho(\omega)dE=0$.
Note that it follows from this assumption that $N_0$ has measure zero as well.

Especially, the main existence and uniqueness results
for the solutions in $L^2(G\times S\times I,dx d\rho dE)$-spaces of the CSDA Boltzmann transport problem,
as given in Sections \ref{single-eq} and \ref{cosyst}
(Theorems \ref{csdath3}, \ref{coth2-d}, \ref{evoth1}, \ref{cosystth2}, \ref{m-d-j-co1}, \ref{coupthev}, along with their corollaries),
as well as the non-negativity result of Subsection \ref{possol},
remain essentially true even with this more general choice of measure on $S$.
\end{remark}

Below the maps $t:G\times S\to [0,\infty]$ and $\tau_{\pm}:\partial G\times S\to [0,\infty]$,
will often be interpret as maps $t:G\times S\times I\to [0,\infty]$ and $\tau_{\pm}:\Gamma\to [0,\infty]$,
by dropping off the energy variable, e.g. $t(x,\omega,E)=t(x,\omega)$.

\begin{lemma}\label{le:esccont}
Let $(y_0,\omega_0)\in \partial G\times S$ be such that
there exists a bounded open cone 
$C:=\{x\in\R^3\ |\ \Big|{x\over{\n{x}}}-{a\over{\n{a}}}\Big|<r, 0<\n{x}<r\}\subset\R^3,\ a\not=0$ containing $-\omega_0$
such that for some $\lambda_0>0$ one has
$y_0+\lambda_0 C\subset \R^3\backslash G$.
Then we have the following:

(i) The map $t(\cdot,\cdot)$ is continuous at any point $(x_0,\omega_0)\in G\times S$
such that $y_0=x_0-t(x_0,\omega_0)\omega_0$.

(ii) Letting $y_+ = y_0 + \tau_-(y_0,\omega_0)\omega_0$,
one has the following limits
\[
& \lim_{(x,\omega)\to (y_0,\omega_0)} t(x,\omega)=0, \\
& \lim_{(x,\omega)\to (y_+,\omega_0)} t(x,\omega)=\tau_-(y_0,\omega_0),
\]
where $(x,\omega)\in G\times S$ when taking the limits.

\end{lemma}

\begin{proof} 
By using charts of a $C^1$-manifold with boundary and the assumption, there exists a bounded open cone $C'$
containing $-\omega_0$, a $\lambda_0'>0$
and a neighbourhood $U$ of $y_0$ in $\R^3$ such that
\[
y+\lambda_0' C'\subset \R^3\backslash G,\quad \forall y\in U\cap\partial G.
\]

(i) First, let us show that $t(\cdot,\cdot)$ is lower-semicontinous on $G\times S$.
Indeed, if $(x,\omega)\in G\times S$ and $0<s<t(x,\omega)$,
we have $x-[0,s]\omega:=\{x-t\omega\ |\ t\in [0,s]\}\subset G$. Thus if $(x_n,\omega_n)$ is a sequence in $G\times S$
converging to $(x,\omega)$,
one has $x_n-[0,s]\omega_n\subset G$
for all large enough $n$, which implies $t(x_n,\omega_n)\geq s$,
and thus $\liminf_{n\to\infty} t(x_n,\omega_n)\geq s$.
Letting $s\to t(x,\omega)$ from the left, we
conclude that $\liminf_{n\to\infty} t(x_n,\omega_n)\geq t(x,\omega)$,
which gives the lower semi-continuity we were to show.

It remains to show that $t(\cdot,\cdot)$ is upper semi-continuous at $(x_0,\omega_0)$
which then implies that $t(\cdot,\cdot)$ is continuous at $(x_0,\omega_0)$
since we already known that $t(\cdot,\cdot)$ is lower-semicontinous at $(x_0,\omega_0)$.

Indeed, let $(x_n,\omega_n)\to (x_0,\omega_0)$ when $n\to\infty$.
Fix $\lambda\in ]0,\lambda_0]$. Then
\[
y_0-\lambda \omega_0\in y_0+\lambda C'\subset y_0+\lambda_0 C'\subset \R^3\backslash G,
\]
from which it follows that for all large enough $n$ (which we assume from now on),
\[
x_n-(t(x_0,\omega_0)+\lambda)\omega_n\in y_0+\lambda_0 C'\subset\R^3\backslash G.
\]
Therefore, $t(x_n,\omega_n)\leq t(x_0,\omega_0)+\lambda$,
and hence
\[
\limsup_{n\to\infty} t(x_n,\omega_n)\leq t(x_0,\omega_0)+\lambda,
\]
for all $\lambda>0$, which gives us the desired upper semi-continuity at $(x_0,\omega_0)$,
\[
\limsup_{n\to\infty} t(x_n,\omega_n)\leq t(x_0,\omega_0).
\]

(ii) Arguments are analogous to the ones employed in case (i).
Consider first the case 
$(x_n,\omega_n)\to (y_0,\omega_0)$ where $(x_n,\omega_n)\in G\times S$ .
Then $y_0-\lambda_0'\omega_0 \in y_0+\lambda_0' C'\subset\R^3\backslash G$
implies that for all $n$ large enough,
one has $x_n-\lambda \omega_n\in y_0+\lambda_0' C'\subset\R^3\backslash G$
for all $0<\lambda\leq \lambda_0'$,
hence $t(x_n,\omega_n)\leq \lambda$,
from which $\limsup_{n\to\infty} t(x_n,\omega_n)\leq \lambda$,
and finally $\limsup_{n\to\infty} t(x_n,\omega_n)=0$.

We then suppose that $(x_n,\omega_n)\to (y_+,\omega_0)$.
Since $y_0=y_+ - \tau_-(y_0,\omega_0)\omega_0$
and $y_0-\lambda_0'\omega_0\in y_0+\lambda_0'C'\subset\R^3\backslash G$,
one has for all $n$ large enough and all $0<\lambda\leq \lambda_0'$,
that 
\[
x_n- (\tau_-(y_0,\omega_0)+\lambda)\omega_n \in y_0+\lambda_0'C'\subset\R^3\backslash G,
\]
whence $t(x_n,\omega_n)\leq \tau_-(y_0,\omega_0)+\lambda$,
which gives the upper limit
\[
\limsup_{n\to\infty} t(x_n,\omega_n)\leq \tau_-(y_0,\omega_0).
\]

In order to obtain the corresponding lower limit, which then shows the existence and the correct value of the limit we were set out to show,
notice that 
if $0<\sigma<\tau_-(y_0,\omega_0)$,
then 
$y_0+\sigma\omega_0\in G$, hence
for all $n$ large enough 
$x_n-(\tau_-(y_0,\omega_0)-\sigma)\omega_n
\in G$ (since $x_n-(\tau_-(y_0,\omega_0)-\sigma)\omega_n\to
x_0-(\tau_-(y_0,\omega_0)-\sigma)\omega_0=y_0+\sigma\omega_0$) ,
which  implies that $t(x_n,\omega_n)\geq \tau_+(y_+,\omega_0)-\sigma$,
and thus $\liminf_{n\to\infty} t(x_n,\omega_n)\geq  \tau_-(y_0,\omega_0)-\sigma$.
Letting $\sigma\to 0+$ allows us to conclude that
\[
\liminf_{n\to\infty} t(x_n,\omega_n)\geq \tau_-(y_0,\omega_0).
\]
\end{proof}

\begin{lemma}\label{le:esccont:1}
Let $(x_0,\omega_0,E_0)\in (G\times S\times I)\cup \Gamma_+\cup \Gamma_-$.
Then
\[
\lim_{(x,\omega,E)\to (x_0,\omega_0,E_0)} t(x,\omega)
=\begin{cases}
t(x_0,\omega_0),\ & \textrm{if}\ (x_0,\omega_0,E_0)\in G\times S\times I \\
& \textrm{and}\ \big(x_0-t(x_0,\omega_0)\omega_0,\omega_0,E_0\big)\in \Gamma_-, \\
\tau_+(x_0,\omega_0),\ & \textrm{if}\ (x_0,\omega_0,E_0)\in \Gamma_+ \\
& \textrm{and}\ (x_0-\tau_+(x_0,\omega_0)\omega_0,\omega_0,E_0)\in \Gamma_-,\\
0,\ & \textrm{if}\ (x_0,\omega_0,E_0)\in \Gamma_-.
\end{cases}
\]
where $(x,\omega,E)\in G\times S\times I$ when taking the limits.
\end{lemma}

\begin{proof} 
 
Define
\[
z_0=\begin{cases}
x_0-t(x_0,\omega_0)\omega_0,\ & \textrm{if}\ (x_0,\omega_0,E_0)\in G\times S\times I \\
x_0-\tau_+(x_0,\omega_0)\omega_0,\ & \textrm{if}\ (x_0,\omega_0,E_0)\in \Gamma_+ \\
x_0,\ & \textrm{if}\ (x_0,\omega_0,E_0)\in \Gamma_-.
\end{cases}
\]
In each of the limits concerned,
we have thus assumed that $(z_0,\omega_0,E_0)\in\Gamma_-$,
which implies,
using the fact that $\ol{G}$ is $C^1$-manifold with boundary
and $\omega_0\cdot\nu(z_0)<0$,
that there is an open finite cone $C\subset\R^3$ and $\lambda_0>0$
such that $z_0+\lambda_0 C\subset\R^3\backslash G$
and $-\omega_0\in C$.
The claim then follows from Lemma \ref{le:esccont}.
Note in particular that when
considering the limit $\lim t(x,\omega)=\tau_+(x_0,\omega_0)=\tau_-(y_0,\omega_0)$
one takes in the lemma
$y_0=x_0-\tau_+(x_0,\omega_0)\omega_0$ and (hence) $y_+=x_0$.

\end{proof}

\begin{proposition}\label{prop-ex}
Define 
\[
(\Gamma_+)_c:=\{(y,\omega,E)\in \Gamma_+\ |\ (y-\tau_+(y,\omega)\omega,\omega,E)\in \Gamma_-\}.
\]
Then $\Gamma_+\backslash (\Gamma_+)_c$ has zero-measure in $\Gamma$ and
there is a continuous extension $\ol{t}$ of 
$t:D\to \R$ onto $\mc{D}_-:=D\cup \Gamma_-\cup (\Gamma_+)_c$
given by
\[
\ol{t}(x,\omega)=\begin{cases}
t(x,\omega), & \textrm{if}\ (x,\omega,E)\in D, \\
\tau_+(x,\omega), & \textrm{if}\ (x,\omega,E)\in (\Gamma_+)_c, \\
0, & \textrm{if}\ (x,\omega,E)\in \Gamma_-.
\end{cases}
\]

\end{proposition}

\begin{proof}
The result is a straightforward consequence of the above Lemmas.
\end{proof}

\begin{example}
Note that in the case where $G$ is convex 
and its boundary is $C^1$-regular the above extension $\ol{t}$ is continuous on $\ol G\times S$. For example,
for the ball $G=B(0,r)$ we have $\tau_+(x,\omega)=2|\la x,\omega\ra|$  and $t(x,\omega)=\la x,\omega\ra+\sqrt{\la x,\omega\ra^2+r^2-\n{x}^2}$ (\cite[Example 3.1]{tervo17-up}).
We find that for $(y,\omega)\in \Gamma'=\partial G\times S$
\be\label{et}
\lim_{(x,\omega)\to (y,\omega),(x,\omega)\in G\times S}t(x,\omega)=\la y,\omega\ra+|\la y,\omega\ra|.
\ee
For $(y,\omega)\in \Gamma'_+$ the projection $\la y,\omega\ra$ is non-negative (and then the limit (\ref{et}) is $\tau_+(y,\omega)$) and
for $(y,\omega)\in \Gamma'_-$ it is non-positive
and then the limit (\ref{et}) is $0$). Hence $\ol t:\ol G\times S\to\R$  is continuous.
\end{example}

Due to the Proposition \ref{prop-ex} we can (almost everywhere) uniquely set $t(y,\omega)=0$ for $(y,\omega,E)\in\Gamma_{-}$. 
For $(y,\omega,E)\in \Gamma_+$ we set $t(y,\omega):=\tau_+(y,\omega)$. 
For further (unexplained) notations we refer to \cite{tervo17-up}.
 
In the sequel we denote for $k\in\N_0$
\[
C^k(\ol G\times S\times I):=\{\psi\in C^k(G\times S\times I^\circ)\ |\ \psi\ {\rm has\ continuous\ partial\ derivatives\ on}\ \ol G\times S\times I\}
\]
and
\[
D^k(\ol G\times S\times I):=\{\psi\in C^k(G\times S\times I^\circ)\ |\ \psi=f_{|G\times S\times I^\circ},\ f\in C_0^k(\R^n\times S\times\R)\}.
\]
It is well known that these spaces are equal, i.e. $C^k(\ol G\times S\times I)=D^k(\ol G\times S\times I)$, for a given $k$,
if the boundary $\partial G$ of $G$ is of class $C^k$ (see \cite[Part 1, Lemma 5.2]{friedman}).
Thus, in particular, since our standing assumption in this paper is that $\ol{G}$ is (at least) of class $C^1$, we have
\[
C^1(\ol G\times S\times I)=D^1(\ol G\times S\times I).
\]

Define the (Sobolev) space $W^2(G\times S\times I)$  by
\bea\label{fseq1}
W^2(G\times S\times I)
=\{\psi\in L^2(G\times S\times I)\ |\  \omega\cdot\nabla_x \psi\in L^2(G\times S\times I) \}
\eea
and its subspace $W^2_1(G\times S\times I)$ by
\bea\label{fseq2}
W^2_1(G\times S\times I)
=\big\{\psi\in W^2(G\times S\times I)\ \big|\ \p{\psi}{E}\in L^2(G\times S\times I)\big\}.
\eea
Here $\omega\cdot\nabla_x\psi$ and ${\p {\psi}E}$ are understood in the distributional sense.
In what follows, $\omega\cdot\nabla_x\psi$ will stand for (the distribution)
$\Omega\cdot\nabla_x \psi$, where $\Omega:G\times S\times I\to\R^3$; $\Omega(x,\omega,E)=\omega$.
The spaces $W^2(G\times S\times I)$, $W^2_1(G\times S\times I)$ are equipped  with the inner products, respectively
\be\label{fs4}
\la {\psi},v\ra_{W^2(G\times S\times I)}=\la {\psi},v\ra_{L^2(G\times S\times I)}+
\la\omega\cdot\nabla_x\psi,\omega\cdot\nabla_x v\ra_{L^2(G\times S\times I)}
\ee
and
\be\label{fs5}
\la\psi,v\ra_{W^2_1(G\times S\times I)}=\la {\psi},v\ra_{L^2(G\times S\times I)}+
\la\omega\cdot\nabla_x\psi,\omega\cdot\nabla_x v\ra_{L^2(G\times S\times I)}
+\la{\p {\psi}E},{\p {v}{E}}\ra_{L^2(G\times S\times I)}
\ee
Then $W^2(G\times S\times I)$ and $W^2_1(G\times S\times I)$ are Hilbert spaces.
 
We have (cf.  \cite{friedrichs}, \cite{bardos}. The proof can also be  shown by the similar considerations as in \cite[pp. 11-19]{friedman})

\begin{theorem}\label{denseth}
The space $C^1(\ol G\times S\times I)$
is a dense subspace of $W^2(G\times S\times I)$ and of $W^2_1(G\times S\times I)$. 
\end{theorem}

For $\Gamma_-=\{(y,\omega,E)\in (\partial G)\times S\times I)\ |\ \omega\cdot\nu(y)<0\}$ we   define the space of $L^2$-functions with respect to the measure
$|\omega\cdot\nu|\ d\sigma d\omega dE$ which is denoted by $T^2(\Gamma_-)$
that is, $T^2(\Gamma_-)=L^2(\Gamma_-,|\omega\cdot\nu|\ d\sigma d\omega dE)$.
$T^2(\Gamma_-)$   is a Hilbert space and its inner product is (in  this paper all functions are real-valued)
\be\label{fs8}
\la {h_1},h_2\ra_{T^2(\Gamma_-)}=\int_{\Gamma_-}h_1(y,\omega,E)h_2(y,\omega,E) |\omega\cdot\nu|\ d\sigma d\omega dE.
\ee
The space $T^2(\Gamma_+)$ (and its inner product) of $L^2$-functions on
$\Gamma_+=\{(y,\omega,E)\in (\partial G)\times S\times I)\ |\ \omega\cdot\nu(y)>0\}$
with respect to the measure
$|\omega\cdot\nu|\ d\sigma d\omega dE$ is similarly defined.
We denote by $T^2(\Gamma)$ the space of $L^2$-functions with respect to the measure
$|\omega\cdot\nu|\ d\sigma d\omega dE $ that is,
$
T^2(\Gamma)=L^2(\Gamma,|\omega\cdot\nu|\ d\sigma d\omega dE).
$
The inner product in $T^2(\Gamma)$ is
\be\label{fs10}
\la {h_1},h_2\ra_{T^2(\Gamma)}= \int_{\Gamma}h_1(y,\omega,E)h_2(y,\omega,E)|\omega\cdot\nu|\ d\sigma d\omega dE.
\ee

One can show (\cite[pp. 230-231]{dautraylionsv6}, \cite{tervo17-up}) 
that for any compact set $K\subset \Gamma_-$ 
\be\label{ttha}
\int_K|\psi(y,\omega,E)|^2|\omega\cdot\nu|\ d\sigma d\omega dE\leq C_K\n{\psi}_{W^2(G\times S\times I)}^2 \quad\forall \psi\in C^1(\ol G\times S\times I).
\ee
Hence  any element $\psi\in W^2(G\times S\times I)$ has well defined trace $\psi_{|\Gamma_-}$ in $L^2_{\rm loc}(\Gamma_-,|\omega\cdot\nu|\ d\sigma d\omega dE)$ defined by
\be\label{tthb}
\psi_{|K}:=\lim_{j\to\infty}\ {\psi_j}_{|K}\ {\rm for\ any\ compact\ subset\ } K\subset\Gamma_-,
\ee
where $\{\psi_j\}\subset C^1(\ol G\times S\times I)$ is a sequence such that $\lim_{j\to\infty}\n{\psi_j-\psi}_{W^2(G\times S\times I)} =0$. In addition the trace mapping
$\gamma_-:W^2(G\times S\times I)\to \ L^2_{\rm loc}(\Gamma_-,|\omega\cdot\nu|\ d\sigma d\omega dE)$ such that
$\gamma_-(\psi)=\psi_{|\Gamma_-}$ is continuous.
Similarly one has a continuous trace mapping $\gamma_+:W^2(G\times S\times I)\to \ L^2_{\rm loc}(\Gamma_+,|\omega\cdot\nu|\ d\sigma d\omega dE)$ and so we can define 
(a.e. unique) the trace  $\gamma(\psi)$ on $\Gamma$ for $\psi\in W^2(G\times S\times I)$.

By the Sobolev Embedding Theorem for $\psi\in C^1(\ol G\times S\times I)$ (cf. \cite[p. 22]{friedman}, or \cite[p. 220]{treves}; or see \eqref{trpr2} in the proof of Theorem \ref{tth} below)
\be\label{sobim}
\n{\psi(\cdot,\cdot,E)}_{L^2(G\times S)}\leq C\Big(\n{\psi}_{L^2(G\times S\times I)}+\n{{\p {\psi}E}}_{L^2(G\times S\times I)}\Big),\quad \forall E\in I,
\ee
and then the traces $\psi(\cdot,\cdot,0),\ \psi(\cdot,\cdot,E_{\rm m})\ \in L^2(G\times S)$ are well-defined for any $\psi\in W_1^2(G\times S\times I)$.

The trace $\gamma(\psi),\ \psi\in W^2(G\times S\times I)$ is not necessarily in the space $T^2(\Gamma)$.  Hence we define the spaces
\be\label{fs12}
\tilde W^2(G\times S\times I)=\{\psi\in W^2(G\times S\times I)\ |\ \gamma(\psi)\in T^2(\Gamma) \}.
\ee
which is equipped with the inner product
\be\label{fs14}
\la {\psi},v\ra_{\tilde W^2(G\times S\times I)}=\la {\psi},v\ra_{W^2(G\times S\times I)}+ \la {\gamma(\psi)},\gamma(v)\ra_{T^2(\Gamma)}.
\ee
The space $\tilde W^2(G\times S\times I)$ is a Hilbert space (cf. \cite{tervo17-up}).
Moreover, we define subspaces $\tilde{W}_{\pm,0}^2(G\times S\times I)$ of it by
\be 
\tilde{W}_{\pm,0}^2(G\times S\times I)=\{\psi\in \tilde W^2(G\times S\times I)\ |\ \gamma_{\pm}(\psi)=0\}.
\ee

For $v\in \tilde W^2( G\times S\times I)$ and $\psi\in\tilde W^2(G\times S\times I)$ it holds the Green's formula 
\begin{align}\label{green}
\int_{G\times S\times I}(\omega\cdot \nabla_x \psi)v\ dxd\omega dE
+\int_{G\times S\times I}(\omega\cdot \nabla_x v)\psi\ dxd\omega dE=
\int_{\partial G\times S\times I}(\omega\cdot \nu) v\ \psi\ d\sigma d\omega dE,
\end{align}
which is obtained by Stokes Theorem  for $v,\psi\in C^1(\ol G\times S\times I)$ and then by the limiting considerations for general $v\in \tilde W^2(G\times S\times I)$ and $\psi\in\tilde W^2(G\times S\times I)$.

\begin{remark}\label{re:W2_0_trace}
By Green's formula
\[
\n{\gamma_{\pm}(\psi)}_{T^2(\Gamma_{\pm})}\leq \n{\psi}_{\tilde{W}^2_{\mp}(G\times S\times I)}
\]
and thus $\gamma_{\pm} :\tilde{W}^2_{\mp,0}(G\times S\times I)\to T^2(\Gamma_{\pm})$
is bounded. 
\end{remark}

Occasionally we work in the (energy independent) spaces $L^2(G\times S)$.
The corresponding Hilbert spaces $W^2(G\times S)$, $T^2(\Gamma'_{\pm})$, $T^2(\Gamma')$
and $\tilde W^2(G\times S)$ are similarly defined, where
\[
\Gamma':={}&(\partial G)\times S, \\[2mm]
\Gamma'_{0}:={}&\{(y,\omega)\in (\partial G)\times S\ |\ \omega\cdot\nu(y)=0\}, \\[2mm]
\Gamma'_{-}:={}&\{(y,\omega)\in (\partial G)\times S\ |\ \omega\cdot\nu(y)<0\}, \\[2mm]
\Gamma'_{+}:={}&\{(y,\omega)\in (\partial G)\times S\ |\ \omega\cdot\nu(y)>0\}.
\]
In addition, the trace $\gamma'(\psi):=\psi_{|\Gamma'}$ for $\psi\in \tilde{W}^2(G\times S)$ is defined as $\gamma(\psi)$ above.

In the context of CSDA-equations we need the following additional Hilbert spaces. 
Let $H$ be the completion of $C^1(\ol G\times S\times I)$ with respect to the inner product
\be
\la \psi,v\ra_H:=\la \psi,v\ra_{L^2(G\times S\times I)}+
\la\gamma(\psi),\gamma(v)\ra_{T^2(\Gamma)}
\ee
The elements of $H$ are  of the form $\tilde\psi=(\psi,q)\in 
L^2(G\times S\times I)\times T^2(\Gamma)$. Actually, they are exactly 
elements of the closure of the graph of trace operator $\gamma:C^1(\ol G\times S\times I)\to 
C^1(\partial G\times S\times I)$ in $L^2(G\times S\times I)\times T^2(\Gamma)$. 
The inner product in $H$ is
\begin{align}\label{spaceH}
\la\tilde\psi,\tilde \psi'\ra_{H}=\la\psi, \psi'\ra_{L^2(G\times S\times I)}
+\la q,q'\ra_{T^2(\Gamma)},
\end{align}
for $\tilde\psi=(\psi,q)$, $\tilde\psi'=(\psi',q')\in H$.

Furthermore, let $H_1$ be the completion of $C^1(\ol G\times S\times I)$ with respect to the inner product 
\begin{align} \label{spaceH1}
\la \psi,v\ra_{H_1}:={}&\la \psi,v\ra_{L^2(G\times S\times I)}+
\la\gamma(\psi),\gamma(v)\ra_{T^2(\Gamma)}\nonumber \\
&
+
\la\psi(\cdot,\cdot,0),v(\cdot,\cdot,0)\ra_{L^2(G\times S)}
+
\la\psi(\cdot,\cdot,E_m),v(\cdot,\cdot,E_m)\ra_{L^2(G\times S)}.
\end{align}
The elements of $H_1$ are of the form $\tilde\psi=(\psi,q,p_0,p_m)\in 
L^2(G\times S\times I)\times T^2(\Gamma)\times L^2(G\times S)^2$. More precisely, the elements of $H_1$ are
the elements of  the closure
(in $L^2(G\times S\times I)\times T^2(\Gamma)\times L^2(G\times S)^2$)
of the graph of the (trace) operator $C^1(\ol G\times S\times I)\to C^1(\partial G\times S\times I)\times C^1(\ol G\times S)^2$ defined by
$\psi\mapsto (\gamma(\psi),\psi(\cdot,\cdot,0),\psi(\cdot,\cdot,E_m))$.
The inner product in $H_1$ is
\[
\la\tilde\psi,\tilde \psi'\ra_{H_1}={}&\la\psi, \psi'\ra_{L^2(G\times S\times I)}
+\la q,q'\ra_{T^2(\Gamma)}\nonumber\\
&
+
\la p_0,p_0'\ra_{L^2(G\times S)}+\la p_{\rm m},p_{\rm m}'\ra_{L^2(G\times S)},
\]
for $\tilde\psi=(\psi,q,p_0,p_{\rm m})$, $\tilde\psi=(\psi',q',p_0',p_{\rm m}')\in H_1$.

Finally, let
$H_2$ be the completion of $C^1(\ol G\times S\times I)$ with respect to the inner product
\bea\label{inph2}
\la \psi,v\ra_{H_2}
:={}&\la \psi,v\ra_{\tilde W^2(G\times S\times I)}+\la{\p {\psi}E},{\p {v}E} \ra_{L^2(G\times S\times I)}\nonumber\\
={}&\la \psi,v\ra_{L^2(G\times S\times I)}+
\la\omega\cdot\nabla_x \psi,\omega\cdot\nabla_x v\ra_{L^2(G\times S\times I)}\nonumber\\
&+
\la\gamma(\psi),\gamma(v)\ra_{T^2(\Gamma)}
+
\la{\p {\psi}E},{\p {v}E} \ra_{L^2(G\times S\times I)}.
\eea
Obviously $H_2\subset \tilde W^2(G\times S\times I)\cap W_1^2(G\times S\times I)$
and the inner product in $H_2$ is given by (\ref{inph2}).

\subsection{Some Details on (Inflow) Trace Theory}\label{trath}

We still bring up a refinement of the above explained trace theory. 
Let us define
\[T^2_{\tau_{\pm}}(\Gamma_{\pm}):=L^2(\Gamma_{\pm},\tau_{\pm}(\cdot,\cdot)|\omega\cdot\nu| d\sigma d\omega dE),
\]
equipped with the inner product
\be\label{fs8a}
\la {h_1},h_2\ra_{T^2_{\tau_{\pm}}(\Gamma_{\pm})}=\int_{\Gamma_{\pm}}h_1(y,\omega,E)h_2(y,\omega,E)\tau_{\pm}(y,\omega) |\omega\cdot\nu|\ d\sigma d\omega dE.
\ee 
 
Suppose that $g\in C(\Gamma_-)$ such that ${\p g{\tilde y_i}}\in C(\Gamma_-), i=1,2$,
where ${\partial\over{\partial \tilde y_i}}$ denotes any local basis of the
tangent space of $\partial G$
(and ${\p g{\tilde y_i}}\in C(\Gamma_-)$ is to be understood in a local sense),
and $\Sigma\in C(\ol G\times S\times I)$ such that
${\p {\Sigma}{x_j}}\in C(\ol G\times S\times I),\ j=1,2,3$.
Then  the unique (classical) solution of the homogeneous \emph{convection-scattering equation}
(recall the definition of $D$ in \eqref{D})
\be\label{trath1}
\omega\cdot \nabla_x\psi+\Sigma\psi=0\quad {\rm on}\ D,
\ee
satisfying the inhomogeneous inflow boundary condition
\be\label{trath2}
\psi(y,\omega,E)=g(y,\omega,E)\quad \forall (y,\omega,E)\in \Gamma_-, 
\ee
is given by (\cite[Theorem 3.13]{tervo17-up})
\be \label{trath3}
\psi(x,\omega,E)=e^{-\int_0^{t(x,\omega)}\Sigma(x-s\omega,\omega,E)ds} g(x-t(x,\omega)\omega,\omega,E).
\ee 
We denote by $L_-g$ ($=\psi$) the solution of the problem (\ref{trath1})-(\ref{trath2}) that is,
$L_-g$ is given by the right hand side of (\ref{trath3}).
Note that $L_-g$ is not generally even continuous (for non-convex $G$) even if $g$ happens to be smooth. 
We show, however,  that the formula (\ref{trath3}) gives a \emph{weak solution}
of the problem (\ref{trath1})--(\ref{trath2}) that is,
\begin{alignat}{3}
&
\la \psi,-\omega\cdot \nabla_xv+\Sigma v\ra_{L^2(G\times S\times I)}=0
\quad && {\rm for\ all}\ v\in C_0^\infty(G\times S\times I^\circ),\nonumber\\[2mm]
&
\psi(y,\omega,E)=g(y,\omega,E)\quad && {\rm for\ a.e.} \ (y,\omega,E)\in\Gamma_-,
\end{alignat}
for any $g\in T_{\tau_-}^2(\Gamma_{-})$, and that $L_-g\in  W^2(G\times S\times I)$.

Following form of Fubini's
theorem will be useful later on (cf. \cite[Lemma 2.1]{choulli}).

\begin{proposition}[(Fubini)]\label{pr:fubini}
Let $\psi\in L^1(G\times S\times I)$. Then
\[
\int_{G\times S\times I} \psi(x,\omega,E) dx d\omega dE
=
\int_{\Gamma_-}\int_0^{\tau_-(y,\omega)} \psi(y+t\omega,\omega,E) |\omega\cdot\nu(y)| dt d\sigma(y) d\omega dE.
\]
\end{proposition}

\begin{proof}
This result can be proven by the same arguments
leading to Eq. \eqref{trpr5} in the proof of Theorem \ref{tth} below,
followed by a computation resulting in Eq. \eqref{trpr5-1} in Remark \ref{changevar}. We omit further details here.
See also \cite[Lemma 2.1]{choulli}.
\end{proof}

We record the following Lemma for later use.
\begin{lemma}\label{le:lift}
Assume that $\Sigma\in L^\infty(G\times S\times I)$ and that
$\Sigma\geq 0$.
Then for any $g\in T_{\tau_-}^2(\Gamma_{-})$ the function $L_-g$ defined by (\ref{trath3}) (is measurable and)
belongs to $L^2(G\times S\times I)$, and
\be\label{trpr9}
\n{L_-g}_{L^2(G\times S\times I)}\leq \n{g}_{T^2_{\tau_-}(\Gamma_-)}.
\ee
Moreover, equality holds here if $\Sigma=0$.
\end{lemma} 

\begin{proof}
Write $L_{\Sigma,-}$ for the lift-operator defined in \eqref{trath3} for a given $\Sigma\geq 0$, i.e. ${\psi=L_{\Sigma,-} g}$.
Then
\[
\n{L_{0,-} g}_{L^2(G\times S\times I)}^2
={}&\int_{G\times S\times I} g(x-t(x,\omega)\omega,\omega,E)^2 dx d\omega dE \\
={}&\int_{\Gamma_-} g(y,\omega,E)^2 \tau_-(y,\omega)|\omega\cdot\nu(y)|d\sigma(y) d\omega dE \\
={}&\n{g}_{T^2_{\tau_-}(\Gamma_-)}^2,
\]
where in the second step we applied Proposition \ref{pr:fubini},
and noticed that \linebreak ${t(y+s\omega,\omega)=s}$ whenever $(y,\omega,E)\in \Gamma_-$.
Therefore,
\[
\n{L_{\Sigma,-} g}_{L^2(G\times S\times I)}^2
=&{}\int_{G\times S\times I} \big(e^{-\int_0^{t(x,\omega)}\Sigma(x-s\omega,\omega,E)ds}g(x-t(x,\omega)\omega,\omega,E)\big)^2 dx d\omega dE \\
\leq &{}\int_{G\times S\times I} g(x-t(x,\omega)\omega,\omega,E)^2 dx d\omega dE \\
=&{}\n{L_{0,-} g}_{L^2(G\times S\times I)}^2
=\n{g}_{T^2_{\tau_-}(\Gamma_-)}^2.
\]
\end{proof}

Since $t(y,\omega)=0$ a.e. in $\Gamma_-$ we see that $\gamma_-(L_-g)=g.$
When we verify (Lemma \ref{trathle1}) that
$\omega\cdot\nabla_x(L_-g)+\Sigma(L_-g)=0$ (weakly) in $G\times S\times I$ we can conclude that
$L_-g\in W^2(G\times S\times I)$. 

\begin{lemma}\label{trathle1}
Assume that $\Sigma\in L^\infty(G\times S\times I)$ and that
$\Sigma\geq 0$.
Let $L_{-}:T^2_{\tau_{-}}(\Gamma_{-})\to L^2(G\times S\times I)$
be defined by
\[
(L_{-} g)(x,\omega,E)=
e^{-\int_0^{t(x,\omega)}\Sigma(x-s\omega,\omega,E)ds} 
g(x- t(x,\omega)\omega,\omega,E).
\]
Then in the weak sense on $G\times S\times I$,
\[
\omega\cdot \nabla_x (L_- g)+\Sigma (L_- g)=0.
\]
\end{lemma}

\begin{proof}
Given $g\in T^2_{\tau_-}(\Gamma_-)$,
choosing a sequence $g_n$ in $C^1_0(\Gamma_-)$
that converges to $g$ in $T^2_{\tau_-}(\Gamma_-)$
(the proof of the existence of this kind sequence is quite standard and is omitted),
we have  by the continuity of $L_-$ (see (\ref{trpr9})), that $L_-g_n\to L_- g$
in $L^2(G\times S\times I)$,
hence in $\mc{D}'(G\times S\times I^\circ)$, where $I^\circ:=]0,E_{\rm m}[$,
from which we deduce that
$\omega\cdot\nabla_x (L_-g_n)+\Sigma(L_-g_n)\to \omega\cdot\nabla_x (L_- g)+\Sigma (L_-g)$.
in $\mc{D}'(G\times S\times I^\circ)$.
This shows that we may assume $g\in C^1_0(\Gamma_-)$.

Let $\varphi\in C_0^\infty (G\times S\times I^\circ)$.
Then by Fubini's Theorem
\[
&-(\omega\cdot \nabla_x (L_- g))(\varphi)
=(L_- g)(\omega\cdot \nabla_x\varphi) \\
={}&\int_{G\times S\times I} e^{-\int_0^{t(x,\omega)}\Sigma(x-s\omega,\omega,E)ds} 
g(x-t(x,\omega)\omega,\omega,E)(\omega\cdot \nabla_x \varphi)(x,\omega,E)\diff x \diff \omega \diff E \\
={}&\int_{S\times I}\int_{G_{\omega}}\int_{J_{y,\omega}} e^{-\int_0^{t(y+\tau\omega,\omega)}\Sigma(y+(\tau-s)\omega,\omega,E)ds} 
g(y+\tau\omega-t(y+\tau\omega,\omega)\omega,\omega,E) \\
& \cdot\dif{\tau}\varphi(y+\tau\omega,\omega,E)\diff s\diff y\diff \omega\diff E \\
={}&\int_{S\times I}\int_{G_{\omega}}\int_{J_{y,\omega}} e^{-\int_{\tau-t(y+\tau\omega,\omega)}^\tau\Sigma(y+s\omega,\omega,E)ds} 
g(y+\tau\omega-t(y+\tau\omega,\omega)\omega,\omega,E) \\
& \cdot\dif{\tau}\varphi(y+\tau\omega,\omega,E)\diff \tau\diff y\diff \omega\diff E
\]
Here $G_\omega$ is the orthogonal projection of $G$ along $\omega$,
\[
G_\omega=\{x-(x\cdot\omega)\omega\ |\ x\in G\},
\]
which is an $(n-1)$-dimensional open submanifold of $\R^n$.
Moreover, $J_{y,\omega}$ is the intersection of $G$ and the straight line $y+\R\omega$,
\[
J_{y,\omega}=\{\tau\in\R\ |\ y+\tau\omega\in G\}.
\]
Notice that $J_{y,\omega}$ is open subset of $\R$,
hence disjoint union of countably many open intervals $J_{y,\omega}^i=]a_{y,\omega}^i,b_{y,\omega}^i[$, $i\in\N$
(we take $J_{y,\omega}^i=\emptyset$ for all big enough $i$ if there is only finitely many non-empty ones)
\[
J_{y,\omega}=\coprod_{i\in\N} J^i_{y,\omega}.
\]
Noticing also that for all $\tau\in J^i_{y,\omega}$,
\[
t(y+\tau\omega,\omega)=\tau-a^i_{y,\omega}
\]
we have
\[
&-(\omega\cdot \nabla_x (L_- g))(\varphi) \\
={}&\int_{S\times I}\int_{G_{\omega}}\sum_i \int_{J^i_{y,\omega}} e^{-\int_{a^i_{y,\omega}}^\tau\Sigma(y+s\omega,\omega,E)ds}g(y+a^i_{y,\omega}\omega,\omega,E) \dif{\tau}\varphi(y+\tau\omega,\omega,E)\diff \tau\diff y\diff \omega\diff E \\
={}&\int_{S\times I}\int_{G_{\omega}}\sum_i g(y+a^i_{y,\omega}\omega,\omega,E)\Big( e^{-\int_{a^i_{y,\omega}}^\tau\Sigma(y+s\omega,\omega,E)ds}\varphi(y+\tau\omega,\omega,E) \Big|_{\tau=a^i_{y,\omega}}^{\tau=b^i_{y,\omega}} \Big)\diff y\diff \omega\diff E \\
&+\int_{S\times I}\int_{G_{\omega}}\sum_i \int_{J^i_{y,\omega}} \Sigma(y+\tau\omega,\omega,E)e^{-\int_{a^i_{y,\omega}}^\tau\Sigma(y+s\omega,\omega,E)ds}  \\
&\cdot g(y+a^i_{y,\omega}\omega,\omega,E)\varphi(y+\tau\omega,\omega,E)\diff y\diff \omega\diff E \\
={}&0 + \int_{G\times S\times I} \Sigma(x,\omega,E)e^{-\int_0^{t(x,\omega)}\Sigma(x-s\omega,\omega,E)ds} \\
&\cdot g(x-t(x,\omega)\omega,\omega,E)\varphi(x,\omega,E)\diff x \diff \omega \diff E \\
={}&\big(\Sigma (L_-g)\big)(\varphi),
\]
where at the 3rd equality we used the fact that $\varphi$ vanishes on the boundary of $G$
and $y+a^i_{y,\omega}\omega\in\partial G$ and $y+b^i_{y,\omega}\omega\in\partial G$. This completes the proof.

\end{proof}

\begin{remark}
Note that if we assume that there exists $c>0$
such that $\Sigma\geq c$ on $G\times S\times I$,
then in the previous lemma one can take $G$ to
be unbounded as well.
\par
Analogously to Lemma \ref{trathle1} for any $g\in T^2_{\tau_+}(\Gamma_+)$
the weak solution of the problem 
\begin{align*}
\omega\cdot \nabla_x\psi+\Sigma\psi&=0,\\ 
\psi(y,\omega,E)&=g(y,\omega,E)\quad {\rm for}\ (y,\omega,E)\in \Gamma_+\ 
\end{align*}
is given by (note that $(y,\omega)\in\Gamma_-$ if and only if
$(y,-\omega)\in\Gamma_+$) 
\[
(L_+g)(x,\omega,E):=\psi(x,\omega,E)=e^{-\int_0^{t(x,-\omega)}\Sigma(x-s\omega,\omega,E)ds} g(x+t(x,-\omega)\omega,\omega,E).
\]
\end{remark}

For later use (section \ref{comp}) we also  treat the inhomogeneous convection-scattering equation with the homogeneous boundary data.
Suppose that
$f\in C(\ol G\times S\times I)$ such that ${\p f{x_j}}\in C(\ol G\times S\times I)$,
and let $\Sigma\in C(\ol G\times S\times I)$ such that 
${\p {\Sigma}{x_j}}\in C(\ol G\times S\times I),\ j=1,2,3$  .
Then  the unique (classical) solution of the equation
\be\label{trath7}
\omega\cdot \nabla_x\psi+\Sigma\psi=f\quad {\rm on}\ D,
\ee
satisfying the homogeneous inflow boundary condition
\be\label{trath8}
\psi(y,\omega,E)=0\quad {\rm for}\ (y,\omega,E)\in\Gamma_-,
\ee
is given by (cf. \cite{dautraylionsv6}, or \cite[Section 3.3]{tervo17-up})
\be \label{trath9}
\psi(x,\omega,E)=\int_0^{t(x,\omega)}e^{-\int_0^t\Sigma(x-s\omega,\omega,E)ds} f(x-t\omega,\omega,E) dt.
\ee

We need the next lemma.
\begin{lemma}
Assume that $G$ is bounded, $d:={\rm diag}(G)$,
$\Sigma\in L^\infty(G\times S\times I)$ and that
$\Sigma\geq 0$.
Then for any $f\in L^2(G\times S\times I)$ the formula (\ref{trath9})
defines $L^2(G\times S\times I)$-function $\psi=S_{\Sigma} f$, and 
\be\label{trath12}
\n{S_{\Sigma} f}_{L^2(G\times S\times I)}\leq d\n{f}_{L^2(G\times S\times I)}.
\ee
\end{lemma}

\begin{proof}
We have
\[
&\n{S_{\Sigma} f}_{L^2(G\times S\times I)}^2
=\int_{G\times S\times I} \Big(\int_0^{t(x,\omega)}e^{-\int_0^t\Sigma(x-s\omega,\omega,E)ds} f(x-t\omega,\omega,E) dt\Big)^2 dx d\omega dE \\
={}&\int_{\Gamma_-} \int_0^{\tau_-(y,\omega)} \Big(\int_0^{t(y+r\omega,\omega)}e^{-\int_0^t\Sigma(y+(r-s)\omega,\omega,E)ds} f(y+(r-t)\omega,\omega,E) dt\Big)^2 dr |\omega\cdot\nu(y)|d\sigma(y) d\omega dE \\
={}&\int_{\Gamma_-} \int_0^{\tau_-(y,\omega)} \Big(\int_0^{r}e^{-\int_0^t\Sigma(y+(r-s)\omega,\omega,E)ds} f(y+(r-t)\omega,\omega,E) dt\Big)^2 dr |\omega\cdot\nu(y)|d\sigma(y) d\omega dE \\
\leq &\int_{\Gamma_-} \int_0^{\tau_-(y,\omega)} r\int_0^{r}|f(y+(r-t)\omega,\omega,E)|^2 dt dr |\omega\cdot\nu(y)|d\sigma(y) d\omega dE \\
={}&\int_{\Gamma_-} \int_0^{\tau_-(y,\omega)} r\int_0^{r}|f(y+t\omega,\omega,E)|^2 dt dr |\omega\cdot\nu(y)|d\sigma(y) d\omega dE \\
\leq &\int_{\Gamma_-} \tau_-(y,\omega)^2\int_0^{\tau_-(y,\omega)}|f(y+t\omega,\omega,E)|^2 dt |\omega\cdot\nu(y)|d\sigma(y) d\omega dE,
\]
where we again in the second step applied the change of variables in integration explained in the proof of Theorem \ref{tth} below and
in the fourth step applied the Cauchy-Schwarz inequality.
Since $\tau_-(y,\omega)\leq d$,
this gives
\[
\n{S_{\Sigma} f}_{L^2(G\times S\times I)}^2
\leq {}& d^2\int_{\Gamma_-} \int_0^{\tau_-(y,\omega)}|f(y+t\omega,\omega,E)|^2 dt|\omega\cdot\nu(y)|d\sigma(y) d\omega dE \\
={}&d^2\n{f}_{L^2(G\times S\times I)}^2.
\]
\end{proof}

More generally, we have the following.

\begin{lemma}\label{trathle2}
Assume that $\Sigma\in L^\infty(G\times S\times I)$ and that
$\Sigma\geq 0$.
Then $\psi$ defined by \eqref{trath9} satisfies
weakly in $G\times S\times I$,
\be\label{trath11}
\omega\cdot \nabla_x \psi+\Sigma\psi=f,
\ee
and the inflow boundary condition (\ref{trath8}) is valid.
\end{lemma}

\begin{proof}
Due to (\ref{trath12})  it suffices to show (\ref{trath11}) only for $f\in C^1(\ol G\times S\times I)$.
Using the notations from the  proof of Lemma \ref{trathle1},
for $\varphi\in C_0^\infty(G\times S\times I^\circ)$
we get by the Fubini's Theorem
\[
&-(\omega\cdot \nabla_x \psi)(\varphi)
=\psi(\omega\cdot\nabla_x\varphi) \\
={}&\int_{G\times S\times I} \psi(x,\omega,E) (\omega\cdot \nabla_x \varphi)(x,\omega,E)\diff x \diff \omega \diff E \\
={}&\int_{S\times I}\int_{G_{\omega}}\sum_i \int_{J^i_{y,\omega}} \psi(y+\tau\omega,\omega,E) \dif{\tau}\varphi(y+\tau\omega,\omega,E)\diff \tau\diff y\diff \omega\diff E \\
={}&\int_{S\times I}\int_{G_{\omega}}\sum_i \int_{J^i_{y,\omega}} \int_0^{\tau-a^i_{y,\omega}}e^{-\int_0^t\Sigma(y+(\tau-s)\omega,\omega,E)ds} f(y+(\tau-t)\omega,\omega,E)\\
&\cdot\dif{\tau}\varphi(y+\tau\omega,\omega,E) dt\diff \tau\diff y\diff \omega\diff E \\
={}&\int_{S\times I}\int_{G_{\omega}}\sum_i \int_{J^i_{y,\omega}} \int_{a^i_{y,\omega}}^\tau e^{-\int_0^{\tau-t}\Sigma(y+(\tau-s)\omega,\omega,E)ds} f(y+t\omega,\omega,E)\\
&\cdot\dif{\tau}\varphi(y+\tau\omega,\omega,E) dt\diff \tau\diff y\diff \omega\diff E \\
={}&\int_{S\times I}\int_{G_{\omega}}\sum_i \int_{J^i_{y,\omega}} f(y+t\omega,\omega,E) \int_t^{b^i_{y,\omega}} e^{-\int_{t}^\tau\Sigma(y+s\omega,\omega,E)ds} \\
&\cdot\dif{\tau}\varphi(y+\tau\omega,\omega,E) \diff \tau\diff t\diff y\diff \omega\diff E.
\]
In the last step, we changed the order of integration $dtd\tau\to d\tau dt$,
in which the domain of integration
\[
\{(\tau,t)\ |\ \tau\in J^i_{y,\omega}=]a^i_{y,\omega},b^i_{y,\omega}[,\ t\in ]a^i_{y,\omega},\tau[\}
\]
changes into
\[
\{(t,\tau)\ |\ t\in J^i_{y,\omega}=]a^i_{y,\omega},b^i_{y,\omega}[,\ \tau\in ]t,b^i_{y,\omega}[\}
\]
as usual.

Observing that
\[
&\int_t^{b^i_{y,\omega}} e^{-\int_{t}^\tau\Sigma(y+s\omega,\omega,E)\diff s}\cdot\dif{\tau}\varphi(y+\tau\omega,\omega,E) \diff \tau \\
={}&\Big(e^{-\int_{t}^\tau\Sigma(y+s\omega,\omega,E)\diff s}
\varphi(y+\tau\omega,\omega,E) \Big|_{\tau=t}^{\tau=b^i_{y,\omega}} \Big) \\
&+\int_t^{b^i_{y,\omega}} \Sigma(y+\tau\omega,\omega,E)e^{-\int_{t}^\tau\Sigma(y+s\omega,\omega,E)\diff s}
\varphi(y+\tau\omega,\omega,E) \diff \tau \\
={}&-\varphi(y+t\omega,\omega,E)+\int_t^{b^i_{y,\omega}} \Sigma(y+\tau\omega,\omega,E)e^{-\int_{t}^\tau\Sigma(y+s\omega,\omega,E)\diff s}
\varphi(y+\tau\omega,\omega,E) \diff \tau
,
\]
we obtain
\[
&-(\omega\cdot \nabla_x \psi)(\varphi) \\
={}&\int_{S\times I}\int_{G_{\omega}}\sum_i \int_{J^i_{y,\omega}} f(y+t\omega,\omega,E)\Big(-\varphi(y+t\omega,\omega,E) \\
&+ \int_{t}^{b^i_{y,\omega}} \Sigma(y+\tau\omega,\omega,E)e^{-\int_{t}^\tau\Sigma(y+s\omega,\omega,E)\diff s}
\varphi(y+\tau\omega,\omega,E) \diff \tau\Big)\diff t\diff y\diff \omega\diff E \\
={}&
-\int_{G\times S\times I}f(x,\omega,E)\varphi(x,\omega,E)\diff x\diff \omega\diff E \\
&+\int_{G\times S\times I} \Sigma(x,\omega,E)\int_0^{t(x,\omega)} e^{-\int_0^t\Sigma(x-s\omega,\omega,E)\diff s}f(x-t\omega,\omega,E)\diff t
\varphi(x,\omega,E)\diff x\diff \omega\diff E \\
={}&\int_{G\times S\times I} \big(-f(x,\omega,E)+\Sigma(x,\omega,E)\psi(x,\omega,E)\big)\varphi(x,\omega,E)\diff x\diff\omega\diff E
\]
which is what we set out to prove.

\end{proof}

Choosing especially $\Sigma=0, \ f=1$ in Lemma \ref{trathle2} we find that in the weak sense $\omega\cdot\nabla_x t=1$ in $G\times S\times I$.

We are now ready to prove the inflow trace theorem.
Since $\n{\omega}=1$ and since the domain $G$ is bounded we have
$\tau_{\pm}(y,\omega)\leq d$.
Hence from \cite[p. 252]{dautraylionsv6},  \cite{cessenat85} or \cite{choulli} (where the result is considered for a more general $G$) we obtain the following  theorem. For completeness we  give its detailed proof.

\begin{theorem}\label{tth}
The trace mappings 
\[
\gamma_{\pm}:W^2(G\times S\times I)\to T^2_{\tau_{\pm}}(\Gamma_{\pm})
\]
are (well-defined) bounded surjective operators with bounded right inverses $L_{\pm}:T^2_{\tau_{\pm}}(\Gamma_{\pm})\to
W^2(G\times S\times I)$ that is, $\gamma_{\pm}\circ L_{\pm}=I$ (the identity). The operators $L_{\pm}$ are called {\it lifts}.
\end{theorem}

Below we use abbreviation $L:=L_-$ if no confusion is possible.

\begin{proof} 
A. For the first instance we recall an elementary estimate for (smooth) functions defined on an interval $[0,T]\subset\R$.  
Let $f\in C^1([0,T])$. Then for any $s\in [0,T]$
\[
f(0)=-\int_0^sf'(t)dt +f(s)
\]
which implies (by the Cauchy-Schwarz's inequality) that 
\be\label{trpr1}
|f(0)|^2\leq 
2\Big(s\int_0^s|f'(t)|^2dt+|f(s)|^2\Big)
\leq 2\Big(T\n{f'}^2_{L^2([0,T])}+|f(s)|^2\Big).
\ee
Hence integrating over $[0,T]$ we obtain the estimate
\be\label{trpr2}
T|f(0)|^2\leq 
2\Big(T^2\n{f'}^2_{L^2([0,T])}+\n{f}_{L^2([0,T])}^2\Big).
\ee 

B. We shall only deal with  the trace $\gamma_-$ because the treatment of the trace $\gamma_+$ is analogous. At first we show the boundedness of $\gamma_-$. Since $C^1(\ol G\times S\times I)$ is  dense in $W^2(G\times S\times I)$ (by Theorem \ref{denseth}) it suffices to prove that
there exists a constant $C>0$ such that
\be\label{trpr3}
\n{\gamma_-(\psi)}_{T^2_{\tau_-}(\Gamma_-)}\leq C\n{\psi}_{W^2(G\times S\times I)}\quad \forall \psi\in C^1(\ol G\times S\times I).
\ee

We apply the change of variables given in \cite[proof of Lemma 5.8]{tervo17-up} (see also \cite[Prop. 2.1]{choulli}). 
Assume for simplicity 
that $\partial G$ has a $C^1$-parametrization
(which is almost global), say $h:V\to\partial G\setminus \Gamma_1 $ where $\Gamma_1$ has zero surface measure. Generally we have a finite number of parametrized patches that cover $\partial G$.
Applying {\it for each fixed} $\omega$ the change of variables (in $x$-variable) $x=h(v)+t\omega=:H(v,t)$,
we find that the Jacobian  $J_H$ of $H$ is
\[
J_H(v,t)=\omega\cdot (\partial_1h\times\partial_2h)(v)=
\omega\cdot\nu(h(v))
\n{(\partial_1h\times\partial_2h)(v)},
\]
since $\nu(h(v))={{(\partial_1h\times\partial_2h)(v)}\over{\n{(\partial_1h\times\partial_2h)(v)}}}$.
We notice that $J_H(v,t)$ depends only on $v$ (and $\omega$), but not on $t$, and hence we write it as $J_H(v)$.
Moreover,
almost everywhere $H(W)=G$, where
$W:=\{(v,t)\ |\ v\in V_-,\ 0<t<\tau_-(h(v),\omega)\}$
and
$V_-:=\{v\in V\ |\ \omega\cdot\nu(h(v))< 0\}$ (which depend on $\omega$). 
Hence for any $\psi\in C^1(\ol G\times S\times I)$
\bea\label{trpr5}
&
\int_{G\times S\times I}|\psi(x,\omega,E)|^2 dx d\omega dE
=
\int_{S\times I}\Big(\int_G|\psi(x,\omega,E)|^2 dx \Big)d\omega dE\nonumber\\
={}&
\int_{S\times I}\int_W|\psi(H(v,t),\omega,E)|^2 |J_H(v,t)|dv dt d\omega dE\nonumber\\
={}&
\int_{S\times I}\int_{V_-}\int_0^{\tau_-(h(v),\omega)}|\psi(h(v)+t\omega,\omega,E)|^2 |J_H(v,t)|dtdv d\omega dE.
\eea

For a fixed $(v,\omega)\in V_-$ and $E\in I$ we apply (\ref{trpr2}) to the $C^1$-mapping $f:[0,\tau_-(h(v),\omega)]\to\R$ defined by
\[
f(t):=\psi(h(v)+\omega t,\omega,E).
\]
Noting that $f'(t)=(\omega\cdot\nabla_x\psi)(h(v)+\omega t,\omega,E)$
we obtain the estimate
\bea\label{trpr6}
\tau_-(h(v),\omega)|\psi(h(v),\omega,E)|^2
\leq {}&
2\Big(\tau_-(h(v),\omega)^2
\int_0^{\tau_-(h(v),\omega)}|(\omega\cdot\nabla_x\psi)(h(v)+\omega t,\omega,E)|^2 dt\nonumber\\
&
+
\int_0^{\tau_-(h(v),\omega)}|\psi(h(v)+\omega t,\omega,E)|^2 dt
\Big).
\eea

Utilizing these preliminaries and the fact that $\tau_-(h(v),\omega)\leq d$ we get  
\bea\label{trpr4}
&
\n{\gamma_-(\psi)}_{ T^2_{\tau_-}(\Gamma_-)}^2
=
\int_{\Gamma_-}|\psi(y,\omega,E)|^2\tau_-(y,\omega)|\omega\cdot\nu(y)| d\sigma  d\omega dE
\nonumber\\
={}&
\int_{S\times I}\int_{V_-}|\psi(h(v),\omega,E)|^2
\tau_-(h(v),\omega)|\omega\cdot\nu(h(v))|\ \n{(\partial_1h\times\partial_2h)(v)} dv  d\omega dE
\nonumber\\
={}&
\int_{S\times I}\int_{V_-}|\psi(h(v),\omega,E)|^2
\tau_-(h(v),\omega)|J_H(v)| dv  d\omega dE
\nonumber\\
\leq &
\int_{S\times I}\int_{V_-}
2\Big(d^2
\int_0^{\tau_-(h(v),\omega)}|(\omega\cdot\nabla_x\psi)(h(v)+\omega t,\omega,E)|^2 |J_H(v)|dt\nonumber\\
&\hspace{2cm}
+\int_0^{\tau_-(h(v),\omega)}|\psi(h(v)+\omega t,\omega,E)|^2|J_H(v)| dt
\Big)
\nonumber\\
={}&
2\Big(d^2\n{\omega\cdot\nabla_x\psi}^2_{L^2(G\times S\times I)}
+
\n{\psi}^2_{L^2(G\times S\times I)}\Big)
\eea
where we in the last step applied (\ref{trpr5}) to $\omega\cdot\nabla_x\psi$ and to $\psi$. This completes the boundedness claim of $\gamma_-$.

C. 
Next we prove the existence of the right inverses. Again we consider only the case  of $\gamma_-$. We choose $\Sigma=0$ in Lemma \ref{trathle1}
and define 
the right inverse $L_-:T_{\tau_-}^2(\Gamma_-)\to W^2(G\times S\times I)$ by
\be\label{liftb}
L_-g:=g(x-t(x,\omega)\omega,\omega,E).
\ee
Then by (\ref{trpr9}) $L_-:
T^2_{\tau_{-}}(\Gamma_-)\to
W^2(G\times S\times I)$ is a well-defined  bounded linear operator, which is even isometric embedding (which follows since $\Sigma=0$).
This completes the proof.
\end{proof}

\begin{remark}\label{changevar}
From the proof of the previous lemma we get the following useful formulas (when the integrals exist)
\be 
\int_{\Gamma_-}g(y,\omega,E)d\sigma d\omega dE
=
\int_{S\times I}\int_{V_-}g(h(v),\omega,E) \n{(\partial_1h\times\partial_2h)(v)}dv  d\omega dE,
\ee
where $V_-=\{v\in V\ |\ \omega\cdot\nu(h(v))< 0\}$ (which depends on $\omega$),
and (see (\ref{trpr5}))
\bea\label{trpr5-1}
&
\int_{G\times S\times I}|\psi(x,\omega,E)|^2 dx d\omega dE\nonumber\\
={}&
\int_{S\times I}\int_{V_-}\int_0^{\tau_-(h(v),\omega)}|\psi(h(v)+t\omega,\omega,E)|^2 |\omega\cdot\nu(h(v))|\n{(\partial_1h\times\partial_2h)(v)}dt dv d\omega dE\nonumber\\
={}&
\int_{\Gamma_-}\int_0^{\tau_-(y,\omega)}|\psi(y+t\omega,\omega,E)|^2 |\omega\cdot\nu(y)| dt d\sigma(y) d\omega dE.
\eea
\end{remark}

\begin{remark}\label{wl}
A. For any compact set $K\subset\Gamma_-$ we have $\tau_-(y,\omega)\geq c_K>0$ for all $(y,\omega,E)\in K$ and so the estimate (\ref{ttha}) follows from Theorem \ref{tth}.

B. 
The formula (\ref{liftb}) gives the lift $L_-g$ explicitly which is useful e.g. in numerical computations. Note that the lift is not unique (for example, we are able to define $L_-=L_{\Sigma,-}$ for any $\Sigma$ given above).

Analogously to $L_-$ the (isometric) lift $L_+:T_{\tau_+}^2(\Gamma_+)\to W^2(G\times S\times I)$  can be chosen to be
\[
(L_+g)(x,\omega,E):= g(x+t(x,-\omega)\omega,\omega,E).
\]

C.
We have for any $w\in L^2(G\times S\times I)$ and $g\in T^2(\Gamma_-)$
(as in (\ref{trpr5})
\bea
&
\la L_-g,w\ra_{L^2(G\times S\times I)}
=
\int_{G\times S\times I}g(x-t(x,\omega)\omega,\omega,E) w(x,\omega,E) dx d\omega dE\nonumber\\
&
=
\int_{S\times I}\int_{V_-}\int_0^{\tau_-(h(v),\omega)}g(h(v)+s\omega-t(h(v)+s\omega,\omega)\omega,\omega,E) w(h(v)+s\omega,\omega,E)|J_H(v,s)|ds dv  d\omega dE\nonumber\\
&
=
\int_{S\times I}\int_{V_-}g(h(v),\omega,E) (L_-^*w)(h(v),\omega,E)
\tau_-(h(v),\omega)|\omega\cdot\nu(h(v))|\n{\partial_1h\times\partial_2h} dv  d\omega dE\nonumber\\
&
=\la g,L_-^*w\ra_{T_{\tau_-}^2(\Gamma_-)},
\eea
where we noticed that $t(h(v)+s\omega,\omega)=s$ and defined
\be\label{liftad}
(L_-^*w)(y,\omega,E):={1\over{\tau_-(y,\omega)}}\int_0^{\tau_-(y,\omega)}w(y+s\omega,\omega,E)ds.
\ee
Hence $L_-^*$ is the adjoint of the  operator $L_-:T_{\tau_-}^2(\Gamma_-)\to L^2(G\times S\times I)$.
Applying the preceding computations we find that 
\[
\n{L_-^*w}_{T_{\tau_-}^2(\Gamma_-)}\leq \n{w}_{L^2(G\times S\times I)}.
\]

D.  
Interpreting $\gamma_-$ as a densely defined linear operator
$L^2(G\times S\times I)\to T^2_{\tau_-}(\Gamma_-)$ we recall that the adjoint $\gamma_-^*:T^2_{\tau_-}(\Gamma_-)\to L^2(G\times S\times I)$ exists and it is defined by
\[
\la\gamma_-(\psi),g\ra_{T_{\tau_-}^2(\Gamma_-)}
=\la\psi,\gamma_-^*(g)\ra_{L^2(G\times S\times I)}
\ {\rm for}\ \psi\in W^2(G\times S\times I),\ g\in D(\gamma_-^*). 
\]

The knowledge of adjoints is significant in the theory of existence of solutions. For example, the Fredholm alternative theorem  is essential in the cases where the unique solution does not exist. To demonstrate that, consider  (for simplicity) the problem (\ref{intro1}), 
(\ref{intro2}). Let  $T=(T_1,T_2,T_3)$ be the transport operator related to the problem  that is,
$T_j\psi:=\omega\cdot\nabla_x\psi_j+\Sigma_j\psi_j-K_j\psi_j$.
Then the transport problem reads
\be 
A\psi:=\qmatrix{T\cr \gamma_-\cr}\psi=\qmatrix{f\cr g\cr}
\ee
where $A$ is interpreted as a densely defined closed operator
$L^2(G\times S\times I)^3\to L^2(G\times S\times I)^3\times T_{\tau_-}^2(\Gamma_-)^3$ with $D(A):=W^2(G\times S\times I)^3$.
We find that the adjoint $A^*$ of $A$ is a densely defined operator 
$L^2(G\times S\times I)^3\times T_{\tau_-}^2(\Gamma_-)^3\to L^2(G\times S\times I)^3$ given by
\be 
A^*\qmatrix{u\cr v\cr}:=\qmatrix{T^*& \gamma_-^*\cr}
\qmatrix{u\cr v\cr}\ {\rm for}\ \qmatrix{u\cr v\cr}\in D(A^*) 
\ee
where $T^*=(T_1^*,T_2^*,T_3^*)$ with $T_j^*u:=-\omega\cdot\nabla_x u_j+\Sigma_j u_j-K_j^*u$ (the operator $K_j^*$ is given in section \ref{adjoint}).
Supposing that $A$ is a Fredholm operator we have $R(A)=N(A^*)^\perp$. Hence a necessary and sufficient condition for the existence of solutions of
$A\psi=\qmatrix{f\cr g\cr}$ is the orthogonality criterion 
\be 
\la (f,g),(u,v)\ra_{ L^2(G\times S\times I)^3\times 
T_{\tau_-}^2(\Gamma_-)^3}=0
\ee 
for all $(u,v)\in N(A^*)$ that is, for all  $(u,v)\in D(A^*)\subset
L^2(G\times S\times I)^3\times 
T_{\tau_-}^2(\Gamma_-)^3
$ which obey
\[
T^*u+\gamma_-^*v=0.
\]

Under the due assumptions the transport operator $T$ is (after an appropriate change of variables) a pseudo-differential operator (actually only the term $K$ needs careful treatment).
Hence the application of the pseudo-differential boundary value operator calculus  is possible.
We omit here the details of such approaches.

E. For $(y,\omega,E)\in\Gamma_+$ it is reasonable to set 
(cf. Proposition \ref{prop-ex})
\[
(L_-g)(y,\omega,E):=g(y-\tau_+(y,\omega)\omega,\omega,E),\quad g\in T^2_{\tau_-}(\Gamma_-).
\]
Let $\tau:\Gamma\to\R$ be defined by
 $\tau|_{\Gamma_-}=\tau_-$, $\tau|_{\Gamma_+}=\tau_+$ and $\tau|_{\Gamma_0}=0$. The space $T^2_\tau(\Gamma)$ is similarly defined as the spaces $T^2_{\tau_{\pm}}(\Gamma)$. 
From \cite[p. 253]{dautraylionsv6} (or \cite{cessenat85}) it follows that for $g\in T^2_{\tau}(\Gamma)$ 
there exists an element $\tilde\psi\in \tilde W^2(G\times S\times I)$ such that
\[
\gamma_-(\tilde\psi)=g_{|\Gamma_-}=:g_-\quad {\rm and}\quad 
\gamma_+(\tilde\psi)=g_{|\Gamma_+}=:g_+
\]
if and only if 
\[
(g-L_-(g_-))_{|\Gamma_+}\in L^2(\Gamma_+,\tau_+^{-1}(y,\omega)|\omega\cdot\nu(y)|d\sigma d\omega dE).
\]

G. It can also be shown that $\gamma_-:\tilde W_{-,0}^2(G\times S\times I)\to L^2(\Gamma_-,\tau_{-}^{-1}(y,\omega)|\omega\cdot\nu| d\sigma dE)$ is bounded (note again that for bounded $G$ we have $\tau_-(x,\omega)\leq d$) and that it has a bounded right inverse
$L_{-,0}: 
L^2(\Gamma_-,\tau_-^{-1}(y,\omega)|\omega\cdot\nu| d\sigma dE)\to \tilde W_{-,0}^2(G\times S\times I)$
(\cite{cessenat84} or \cite[p.252]{dautraylionsv6}). 
Similar result hold for $\gamma_+$.
\end{remark}

We still consider the following special case of the trace theory for an exterior of a {\it convex} bounded domain $G$.
Let $G_e$ be the complement (the exterior of $G$) $G_e:=\R^3\setminus \ol G$. Then $\partial G_e=\partial G$. Denote (as above for $G$)
\[
\Gamma_{e,+}:={}&\{(y,\omega,E)\in \partial G_e\times S\times I\ |\ \omega\cdot\nu_e(y)>0\}, \\[2mm]
\Gamma_{e,-}:={}&\{(y,\omega,E)\in \partial G_e\times S\times I\ |\ \omega\cdot\nu_e(y)<0\}, \\[2mm]
\Gamma_e:={}&\Gamma_{e,+}\cup\Gamma_{e,-}
\]
where $\nu_e$ is the unit outward pointing normal vector on $\partial G_e$.

We find that $\nu_e=-\nu$ and then $\Gamma_{e,\pm}=\Gamma_{\mp}$ and $\gamma_{e,\pm}(\psi):=\psi_{|\Gamma_{e,\pm}}=\gamma_{\mp}(\psi)$.
Furthermore, let 
$t_e(x,\omega)$ be the escape time mapping for the domain $G_e$ and let
for an element $(y,\omega,E)\in \Gamma_{e,-}$ (as above)
$\tau_{e,-}(y,\omega)=\inf\{s>0\ |\ y+s\omega\not\in G_e  \}.$
We observe that for the convex set $G$ actually $\tau_{e,-}(y,\omega)=\infty$ for all $(y,\omega,E)\in \Gamma_{e,-}$.

\begin{theorem}\label{tthcon}
The trace mappings 
\[
\gamma_{e,\pm}:W^2(G_e\times S\times I)\to T^2(\Gamma_{e,\pm})
\]
are (well-defined) bounded surjective operators with bounded right inverses (lifts) $L_{e,\pm}:  
T^2(\Gamma_{e,\pm})\to W^2(G_e\times S\times I)$.
\end{theorem}

\begin{proof} 
Again it needs only to consider  the trace operator $\gamma_-$.
The proof runs similarly to the proof of Theorem \ref{tth} with following changes. Instead of estimate (\ref{trpr2}) we utilize the inequality
\begin{multline}
|f(0)|^2=\Big|\int_0^\infty {d\over{dt}}(f(t)^2)dt\Big|
=\Big|\int_0^\infty 2f'(t)f(t)dt\Big|\\
\leq
2\Big(\int_0^\infty|f'(t)|^2dt\Big)^{1/2}
\Big(\int_0^\infty|f(t)|^2dt\Big)^{1/2}
\leq
\int_0^\infty|f'(t)|^2dt
+\int_0^\infty|f(t)|^2dt, \label{trpr13}
\end{multline}
which is valid for all $f\in C_0^1([0,\infty[)$ (note that $f(t)=0$ for sufficiently large $t$).

In addition, one applies the change of variables $H$ we have $H(W)=G_e$ where  $W:=\{(v,t)\ |\ v\in V_-,\ 0<t<\infty\}$ and $V_-:=\{v\in V\ |\ \omega\cdot\nu_e(h(v))<0\}$.

Let $\lambda>0$. For any $g_{e}\in C(\Gamma_{e,-})$ such that ${\p {g_{e}}{\tilde y_i}}\in C(\Gamma_{e,-})$
the (classical) solution $\Psi$ of the problem 
\be\label{exle1}
\omega\cdot\nabla_x\Psi+\lambda\Psi={}&0\quad {\rm on}\ D_e, \\
\Psi_{|\Gamma_{e,-}}={}&g_{e}, \nonumber
\ee
where $D_e$ is the set (\ref{D}) corresponding to $G_e\times S\times I$,
is given explicitly by (cf. (\ref{trath3}))
\be\label{exle2}
\Psi(x,\omega,E)=\begin{cases}
e^{-\lambda t_e(x,\omega)}g_{e}(x-t_e(x,\omega)\omega,\omega,E),\ &{\rm when}\ t_e(x,\omega)\ {\rm is\ finite}\\ 0,\ &{\rm otherwise}\end{cases}.
\ee
Similarly to the proof of Lemma \ref{trathle1} we find that (\ref{exle1}) holds weakly in $G_e\times S\times I$.
For $g_{e}\in T^2(\Gamma_{e,-})$  we define the lift  explicitly by
\be\label{trpr14}
L_{e,-}g_{e}:=\Psi.
\ee
Then as in the proof of Lemma 5.8 given in \cite{tervo17-up} we get that 
$L_{e,-}g_{e}\in \tilde W^2(G_e\times S\times I)$ and
(note that $\n{\omega\cdot\nabla_x\Psi}_{L^2(G_e\times S\times I)}=
\lambda\n{\Psi}_{L^2(G_e\times S\times I)}$)
\bea\label{trpr15}
\n{L_{e,-}g_{e}}_{W^2(G_e\times S\times I)}
={}&\n{\Psi}_{W^2(G_e\times S\times I)}=
\sqrt{1+\lambda^2}\n{\Psi}_{L^2(G_e\times S\times I)}
\nonumber\\
={}&
\sqrt{\frac{1+\lambda^2}{2\lambda}}\n{g_{e}}_{T^2(\Gamma_{e,-})},
\eea
since (cf. Remark \ref{changevar})
\[
\n{\Psi}^2_{L^2(G_e\times S\times I)}
={}&\int_{\Gamma_{e,-}} \int_0^{\tau_{e,-}(x,\omega)} \big(e^{-\lambda s}g_{e}(y,\omega,E)\big)^2 |\omega\cdot\nu(y)| ds d\sigma(y)d\omega dE \\
={}&\frac{1}{2\lambda}\int_{\Gamma_{e,-}} g_{e}(y,\omega,E)^2 |\omega\cdot\nu(y)| d\sigma(y)d\omega dE \\
={}&\frac{1}{2\lambda}\n{g_e}_{T^2(\Gamma_{e,-})}^2,
\]
where $\tau_{e,-}$ is $\tau_-$ for the domain $G_e\times S\times I$.
We omit further details.
\end{proof}

\begin{remark}\label{wla}
A. By (\ref{exle2}) one can show that
\[
L_{e,-}(g_{e})_{|\Gamma_{e,+}}=0,\quad \forall g_e\in T^2(\Gamma_{e,-}).
\]

B. The lift $L_{e,+}:T^2(\Gamma_{e,+})\to W^2(G_e\times S\times I)$ is given by 
\[
(L_{e,+}g_e)(x,\omega,E)=\begin{cases}e^{-\lambda t_e(x,-\omega)}g_{e}(x+t_e(x,-\omega)\omega,\omega,E),\ &{\rm when}\ t_e(x,-\omega)\ {\rm is\ finite},\\ 0,\ &{\rm otherwise}.\end{cases}
\]
We find that
\[
L_{e,+}(g_{e})_{|\Gamma_{e,-}}
=0,\quad \forall g_e\in T^2(\Gamma_{e,+}).
\]

C.
Let $g_e\in T^2(\Gamma_e)$ and let $g_{e,\pm}:={g_e}_{|\Gamma_{e,\pm}}$.
Then for
\[
\tilde{\Psi}:=L_{e,-}(g_{e,-})+L_{e,+}(g_{e,+})\in W^2(G_e\times S\times I),
\]
we find that $\tilde{\Psi}\in \tilde{W}^2(G_e\times S\times I)$ and
\[
\tilde\Psi_{|\Gamma_{e,\pm}}=g_{e,\pm}.
\]
(See Remark \ref{wl}, Part E.)
\end{remark}

As a corollary
we show the following extension result. In the proof we  explicitly
construct the {\it extension of} $\psi$ (cf. \cite[p. 415, proof of Lemma 2]{dautraylionsv6}). 

\begin{corollary}\label{exle}
Suppose that $G\subset\R^3$  is as above and that it is convex. Then for any $\psi\in \tilde W^2(G\times S\times I)$ there exists an extension ${\s E}\psi\in W^2(\R^3\times S\times I)$ of $\psi$ that is, ${\s E}\psi_{|G\times S\times I}=\psi$. In addition, the linear operator ${\s E}:\tilde W^2(G\times S\times I)\to W^2(\R^3\times S\times I)$ is bounded.
\end{corollary}

\begin{proof}
Suppose that $\psi\in\tilde W^2(G\times S\times I)$.
Denote $g:=\psi_{|\Gamma}$, which belongs to $T^2(\Gamma)=T^2(\Gamma_{e})$.
Let $\tilde\Psi\in \tilde W^2(G_e\times S\times I)$ given in Remark \ref{wla}, Part C.
Define ${\s E}\psi$  by
\be\label{cg5}
{\s E}\psi:=
\begin{cases}
\psi\ & {\rm on}\ G\times S\times I\\ 
\gamma_{\pm}(\psi)=\gamma_{e,\mp}(\tilde\Psi)\ & {\rm on}\ \Gamma_\pm\\ 
\tilde\Psi\ & {\rm on}\ G_e\times S\times I
\end{cases}
\ee
Then ${\s E}\psi$ is in $W^2(\R^3\times S\times I)$. 
This follows from the Green's formula (\ref{green}) since for all $v\in C_0^\infty(\R^3\times S\times I^\circ)$
\bea
&\int_{\R^3\times S\times I}({\s E}\psi)\ (\omega\cdot\nabla_x v) dx d\omega dE \\
={}&
\int_{G\times S\times I}({\s E}\psi)\ (\omega\cdot\nabla_x v) dx d\omega dE\nonumber
+
\int_{G_e\times S\times I}({\s E}\psi)\ (\omega\cdot\nabla_x v) dx d\omega dE\nonumber\\
={}&
-\int_{G\times S\times I} (\omega\cdot\nabla_x \psi) v dx d\omega dE
+
\int_{\partial G\times S\times I}\gamma(\psi)\ \gamma(v) (\omega\cdot \nu)d\sigma d\omega dE\nonumber\\
&
-
\int_{G_e\times S\times I} (\omega\cdot\nabla_x \tilde\Psi) v dx d\omega dE
+
\int_{\partial G_e\times S\times I}\gamma_e(\tilde\Psi)\ \gamma_e(v) (\omega\cdot \nu_e)d\sigma d\omega dE\nonumber\\
={}&
-\int_{G\times S\times I} (\omega\cdot\nabla_x \psi) v dx d\omega dE
-\int_{G_e\times S\times I} (\omega\cdot\nabla_x \tilde\Psi) v dx d\omega dE
\eea
where we used the facts that $\partial G=\partial G_e$ and $\nu_e=-\nu$ and so $\gamma_{\pm}(\psi)=:g_{\pm}=g_{e,\mp}:=\gamma_{e,\mp}(\tilde\Psi)$. Hence $\omega\cdot\nabla_x ({\s E}\psi)\in L^2(\R^3\times S\times I)$, as desired.

Finally, we find that by (\ref{trpr15}) (recall Remark \ref{wla}, Part C.)  
\bea
&\n{{\s E}\psi}_{W^2(\R^3\times S\times I)}
=\n{\psi}_{W^2(G\times S\times I)}+
\n{\tilde\Psi}_{W^2(G_e\times S\times I)}\nonumber\\
\leq &
\n{\psi}_{W^2(G\times S\times I)}+\n{L_{e,-}(g_{e,-})}_{W^2(G_e\times S\times I)}+
\n{L_{e,+}(g_{e,+})}_{W^2(G_e\times S\times I)}\nonumber\\
={}&
\n{\psi}_{W^2(G\times S\times I)}+
\sqrt{\frac{1+\lambda^2}{2\lambda}}\Big(\n{\gamma_+(\psi)}_{T^2(\Gamma_{e,-})}+
\n{\gamma_-(\psi)}_{T^2(\Gamma_{e,+})}\Big)
\eea
which implies the boundedness of ${\s E}$.
This completes the proof.
\end{proof}

Let $\tilde {\bf W}^2(G\times S\times I)$ be the completion of $C^1(\ol G\times S\times I)$ with respect to $\n{\cdot}_{\tilde { W}^2(G\times S\times I)}$-norm. 
For a convex set $G\subset\R^3$ have the following density result.

\begin{corollary}\label{convg}
Suppose that $G\subset\R^3$  is as above and that it is \emph{convex}. Then
\be\label{cg1}
\tilde{{\bf W}}^2(G\times S\times I)=\tilde{W}^2(G\times S\times I),
\ee
and
\be\label{cg2}
H_2=
\tilde{W}^2(G\times S\times I)\cap W_1^2(G\times S\times I).
\ee
(Definitions of the spaces $W_1^2$, $\tilde{W}^2$ and $H_2$
were given in \eqref{fseq2}, \eqref{fs12} and \eqref{inph2}, respectively.)
\end{corollary}

\begin{proof}
At first, we deal with the claim (\ref{cg1}). The inclusion $"\subset"$ is clear and then it suffices to prove only the opposite inclusion. 

Because $G$ is convex  we have by Theorem \ref{tthcon}
\be\label{trth}
\n{\gamma_{e}(\Psi)}_{T^2(\Gamma_{e})}\leq C\n{\Psi}_{W^2(G_e\times S\times I)}\quad \forall \Psi\in 
W^2(G_e\times S\times I).
\ee
Let $\psi\in\tilde W^2(G\times S\times I)$ and let 
${\s E} \psi\in 
 W^2(\R^3\times S\times I)
$ 
be its extension provided by Lemma \ref{exle}.
Since $C^1_0(\R^3\times S\times \R)$ is dense in $W^2(\R^3\times S\times I)$ there exists a sequence $\{\Psi_n\}\subset C^1_0(\R^3\times S\times \R)$ such that $\n{\Psi_n-{\s E}\psi}_{W^2(\R^3\times S\times I)}\to 0$ for $n\to\infty$. Let $\psi_n:={\Psi_n}_{|G\times S\times I}$. 
We have
\[
\n{\gamma(\psi_n-\psi)}_{T^2(\Gamma)}=
\n{\gamma_e(\Psi_n-{\s E}\psi)}_{T^2(\Gamma_e)}.
\]
Hence we get by (\ref{trth})
\bea\label{cg6}
\n{\psi_n-\psi}_{\tilde W^2(G\times S\times I)}^2
={}&
\n{\psi_n-\psi}_{ W^2(G\times S\times I)}^2+
\n{\gamma(\psi_n-\psi)}_{T^2(\Gamma)}^2\nonumber\\
\leq &
\n{\Psi_n-{\s E}\psi}_{W^2(G\times S\times I)}^2
+
C^2\n{\Psi_n-{\s E}\psi}_{W^2(G_e\times S\times I)}^2\nonumber\\
\leq &
\n{\Psi_n-{\s E}\psi}_{W^2(\R^3\times S\times I)}^2
+
C^2\n{\Psi_n-{\s E}\psi}_{W^2(\R^3\times S\times I)}^2,
\eea
which implies that $\psi\in \tilde {\bf W}^2(G\times S\times I)$. This completes the proof of (\ref{cg1}). 

Since $C^1_0(\R^3\times S\times I)$ is dense in $W_1^2(\R^3\times S\times I)$
the proof of (\ref{cg2}) is quite similar and so the proof is complete.  
\end{proof}

Finally, we notice that 
the following  continuous inclusions are valid
\begin{gather*}
\tilde W^2(G\times S\times I)\subset H; \\
\psi\mapsto (\psi,\gamma(\psi)),
\end{gather*}
and
\begin{gather*}
\tilde W^2(G\times S\times I)\cap W_1^2(G\times S\times I)\subset H_1;\\
\psi\mapsto (\psi,\gamma(\psi),\psi(\cdot,\cdot,0),\psi(\cdot,\cdot,E_m)).
\end{gather*}

\begin{remark}\label{fre1}
Let $\rho_1,\ \rho_2:I\to \R$ be positive (weight) functions in $L^\infty(I)$. We can define a linear space (more generally instead of $W^2(G\times S\times I)$) by
\bea\label{f16}
&
W^2_{\rho_1,\rho_2}(G\times S\times I)
=\{\psi\in L^2(G\times S\times I)\ |& & \rho_1 \omega\cdot\nabla_x \psi\in L^2(G\times S\times I), & \\
& & & \rho_2{\p {\psi}E}\in L^2(G\times S\times I)\},
\eea
which can be equipped with the inner product
\bea\label{f17}
\la\psi,v\ra_{W^2_{\rho_1,\rho_2}(G\times S\times I)}
={}&\la {\psi},v\ra_{L^2(G\times S\times I)}+
\la\rho_1\omega\cdot\nabla_x\psi,\rho_1\omega\cdot\nabla_x v\ra_{L^2(G\times S\times I)}\nonumber\\
&+\la\rho_2{\p {\psi}E},\rho_2{\p {v}{E}}\ra_{L^2(G\times S\times I)},
\eea
rendering $W^2_{\rho_1,\rho_2}(G\times S\times I)$ to a Hilbert space. 
Similar weighted spaces can be defined generalizing other spaces above.
These spaces are needed e.g. in the context of time-dependent transport equations (where $\rho_1=\rho_2=\sqrt{E}$).
\end{remark}

\sectionspace
\section{ On Hyper-singular Collision Operators Related to Dose Calculation}\label{coll}

The differential cross-sections
may have singularities, or even hyper-singularities, which would lead to extra  pseudo-differential-like terms in the transport equation.
Instead of explaining systematically the underlying theory, the following slightly informal description suffices for the purposes of this work.

In the case where $\sigma(x,\omega',\omega,E',E)$ has hyper-singularities (like  M\o ller and Bremsstrahlung differential cross 
sections analysed below) the integral $\int_{S'}\int_{I'}$ occurring in the collision operator must be understood in the 
sense of {\it Cauchy principal value} ${\rm p.v.}\int_{S'}\int_{I'}$ or more generally in the sense of {\it Hadamard finite part integral} ${\rm p.f.}\int_{S'}\int_{I'}$ (\cite[Sec. 3.2]{hsiao},  \cite[pp. 104-105]{schwarz}, \cite{estrada}, sections 1.5 and 1.6).

Consider the following partial hyper-singular integral operator,
\[
(K\psi)(x,\omega,E)={\rm p.f.}\int_{I'}\int_{S'}\sigma(x,\omega',\omega,E',E)\psi(x,\omega',E')d\omega' dE'.
\]
The simplest case is where  $\sigma=\sigma_0(x,\omega',\omega,E',E)$ 
is a measurable non-negative function $G\times S\times S\times (I\times I\setminus D)\to\R$,
where $D=\{(E,E)\ |\ E\in I\}$ is the diagonal of $I\times I$,
obeying for $E\neq E'$ the estimates
\[
&
\esssup_{(x,\omega)}\int_{S'}\sigma_0(x,\omega',\omega,E',E)d\omega'\leq  {C\over{|E-E'|^\kappa}}, \\
&
\esssup_{(x,\omega)}\int_{S'}\sigma_0(x,\omega,\omega',E,E')d\omega'\leq  {C\over{|E-E'|^\kappa}},
\]
where $\kappa<1$, meaning that $\sigma_0(x,\omega',\omega,E',E)$ may have a so-called \emph{weak singularity} with respect to energy.
The corresponding collision operator 
\be 
(K\psi)(x,\omega,E)
=
\int_{I'}\int_{S'}\sigma_0(x,\omega',\omega,E',E)\psi(x,\omega',E')d\omega' dE',
\ee
is the usual partial Schur  integral operator that is, $\sigma_0(x,\omega',\omega,E',E)$ satisfies the Schur criterion for the boundedness 
and so $K$ is a  bounded operator $L^2(G\times S\times I)\to L^2(G\times S\times I)$ (see section \ref{rco}).  

Nevertheless, the  collision operator $K$ is not generally of the above form.
The $(E',E)$-dependence in differential cross section 
$\sigma(x,\omega'\omega,E',E)$ may contain hyper-singularities of higher order,
${1\over{(E'-E)}^m}$, for $m=1,2$. 
Moreover, the $(\omega',\omega,E',E)$-dependence in differential
cross-sections may contain Dirac's $\delta$-distributions (on $\R$). 
More precisely, in  $\sigma(x,\omega'\omega,E',E)$ there may occur terms
like $\delta(\omega\cdot\omega'-\mu(E',E))$ or $\delta(E-E')$ which require special treatment.
Below we shall consider in more detail the M\o ller scattering and Bremsstrahlung.

\subsection{Some Tools from Analysis}\label{taylor-S}

We recall the following standard concepts from analysis which we  frequently 
need.
The Taylor's expansion (of order $r\in\N_0$) for sufficiently smooth functions
$f:U\to\R$ on an open set $U\subset \R^N$,
\be\label{taylor}
f(x)=\sum_{|\alpha|\leq r}{1\over{\alpha !}}{{\partial^\alpha f}\over{\partial x^\alpha}}(x_0)(x-x_0)^\alpha + \sum_{|\alpha|=r+1}
R_\alpha(x)(x-x_0)^\alpha
\ee
where the residual term (one of its variant forms) is
\[
R_\alpha(x):={{|\alpha|}\over{\alpha !}}\int_0^1(1-t)^{|\alpha|-1}
{{\partial^\alpha f}\over{\partial x^\alpha}}(x_0+t(x-x_0))dt.
\]
Recall also 
the definitions of Hadamard finite part integrals for discontinuous functions $f:[a,b]\to\R$ by \cite{martin-rizzo}, pp. 5 and 32, formulas (14) and (32) therein or \cite[p. 104]{schwarz}.
Applying these definitions (for a fixed $x$) to the function $F_x(t):=\chi_{[x,b]}(t)f(t)$,
where $\chi_{[x,b]}(t)$ is the characteristic of the interval $[x,b]$, we have
\be \label{def-h1}
{\rm p.f.}\int_a^b{{F_x(t)}\over{t-x}}dt
={\rm p.f.}\int_x^{b}{{f(t)}\over{t-x}}dt
=\lim_{\epsilon\to 0}\Big(\int_{x+\epsilon}^{b}{{f(t)}\over{t-x}}dt
+f(x^+)\ln(\epsilon)\Big)
\ee
and
\be \label{def-h2}
{\rm p.f.}\int_a^b{{F_x(t)}\over{(t-x)^2}}dt
=&{\rm p.f.}\int_x^{b}{{f(t)}\over{(t-x)^2}}dt \nonumber \\
=&\lim_{\epsilon\to 0}\Big(\int_{x+\epsilon}^{b}{{f(t)}\over{(t-x)^2}}dt
+f'(x^+)\ln(\epsilon)-{1\over\epsilon}f(x^+)\Big).
\ee
These formulas give
\be\label{hada1}
{\rm p.f.}\int_x^{b}{1\over{t-x}}dt
=\ln(b-x),
\ee
\be\label{hada2}
{\rm p.f.}\int_x^{b}{1\over{(t-x)^2}}dt
=-{1\over{b-x}}.
\ee
Note that ${\rm p.f.}\int_x^{b}{{f(t)}\over{t-x}}dt$ is well-defined (at least) for all $f\in C^\alpha([a,b]),\ \alpha>0$ and (cf. \cite{chan})
\be\label{c-0-a}
{\rm p.f.}\int_x^{b}{{f(t)}\over{t-x}}dt
=\int_{x}^{b}{{f(t)-f(x)}\over{t-x}}dt+
f(x)\ln({b-x}) .
\ee

We need additionally the Taylor's expansion  for a sufficiently smooth function $f:S\to\R$.
Since $S$ is a manifold the expansion requires some explanation. The detailed presentation of the subject is outside of this paper and so we give only some essential technicalities. For complete formulations see e.g. \cite{carmo} or more concisely \cite{mukherjee}., pp. 185-187.

The second order Taylor's expansion of a function $f:S\to\R$ which is $C^3$ around $\omega\in S$ is of the form
\be\label{tay-ex}
f(\omega')=f(\omega)+\la(\nabla_\omega f)(\omega),v\ra +
({\rm Hes}_\omega f)(\omega)(v,v)+(R_\omega f)(\eta)(v,v,v),\ v\in T_\omega(S)
\ee
where $\nabla_\omega$ is the gradient on $S$, ${\rm Hes}_\omega f$ is the co-variant Hessian 2-tensor on $S$ and the residue $R_\omega f$  is a co-variant 3-tensor on $S$. In addition, there exists a constant $C\geq 0$ such that
\be\label{tay-app}
\n{v}\leq C\n{\omega'-\omega}.
\ee
Leaving the residue $R_\omega f$ away we get approximations for $f(\omega')$ near $\omega$.

The basic principle in deriving (\ref{tay-ex}) is to apply an appropriate pull-back $H_\omega:V\to U_\omega$
where $U_\omega\subset S$ and $V\subset T_\omega(S)$ open neighbourhoods such that $\omega\in U_\omega,\ 0\in V$. One assumes that $H_\omega$
is a sufficiently smooth diffeomorphism and so the (smooth) inverse mapping $H_\omega^{-1}:U_\omega\to V$ exists. Let $\omega'\in U_\omega$ and let
\[
H_\omega^{-1}(\omega')=v=\xi_1\Omega_1+\xi_2\Omega_2
\]
where
$\Omega_1=\Omega_1(\omega),\ \Omega_2=\Omega_2(\omega)$ are the local tangent vectors of $S$ at $\omega\in S$,
\[
\Omega_1={1\over{\sqrt{\omega_1^2+\omega_2^2}}}(-\omega_1,\omega_2,0),
\]
\[
\Omega_2=\big({{\omega_1\omega_3}\over{\sqrt{\omega_1^2+\omega_2^2}}},
{{\omega_2\omega_3}\over{\sqrt{\omega_1^2+\omega_2^2}}}.
-\sqrt{1-\omega_3^2}\big).
\]
One often chooses $H_\omega$ to be the (Riemannian geometry's) exponential mapping $H_\omega=\exp_\omega$. The pull-back obeys (locally)
\be\label{e4}
\n{H_\omega^{-1}(\omega')}=\n{v}\leq C\n{\omega'-\omega}.
\ee

The tangent space $T_\omega(S)$ can be isomorphically (and isometrically) identified with $\R^2$ by
\begin{align}\label{eq:J_iso}
v=\xi_1\ol\Omega_1+\xi_2\ol\Omega_2\sim_J (\xi_1,\xi_2)=:\xi,
\end{align}
where the isomorphism is $J$ that is, $\xi=J(v)$ and let $J(V)=V'\subset\R^2$.
Using this identification we find that the mapping $f\circ H_\omega:V'\to \R$ is well-defined (and as smooth as $f$).
Hence we can write the Taylor's expansion near $0\in\R^2$ 
\bea\label{t-ex}
&
(f\circ H_\omega)(v)=(f\circ H_\omega)(0)
+\sum_{j=1}^2\partial_j(f\circ H_\omega)(0)\xi_j
+{1\over 2}\sum_{i=1}^2\sum_{j=1}^2\partial_i\partial_j(f\circ H_\omega)(0)\xi_i\xi_j
\nonumber\\
&
+\sum_{|\alpha|=3}{{|\alpha|}\over{\alpha !}}\int_0^1(1-t)^2\partial_\xi^\alpha
((f\circ H_\omega)(t\xi)dt \cdot \xi^\alpha. 
\eea
By (\ref{e4})
\be\label{e7}
\n{\xi}=\n{v}\leq C\n{\omega'-\omega}.
\ee
Finally,  it can be shown that 
\be\label{e8}
\sum_{j=1}^2\partial_j(f\circ H_\omega)(0)\xi_j
=\la(\nabla_\omega f)(\omega),v\ra 
\ee
and 
\be\label{e9}
\sum_{i=1}^2\sum_{j=1}^2\partial_i\partial_j(f\circ H_\omega)(0)\xi_i\xi_j
=({\rm Hes}_\omega f)(\omega)(v,v)
\ee
and so (\ref{tay-ex}) can be seen.

\begin{remark}
The gradient $\nabla_\omega f$ on sphere $S$ can be shortly  depicted as follows (for $n=3$). 
Suppose that $f$ is defined and smooth in a neighbourhood of $S\subset\R^3$. Then
\be\label{gradS}
\nabla_\omega f=\la\nabla f,\Omega_1\ra\Omega_1+
\la\nabla f,\Omega_2\ra\Omega_2
\ee
where $\nabla f$ is the gradient of $f$ in the ambient space $\R^3$.
 
\end{remark}

The first and second order Taylor's expansions of $\psi$ with respect to $\omega$ around $\omega$ are
\[
\psi(x,\omega',E')\approx \psi(x,\omega,E')
+\la(\nabla_{\omega}\psi)(x,\omega,E'),v\ra 
\]
and 
\[
\psi(x,\omega',E')\approx \psi(x,\omega,E')
+\la(\nabla_{\omega}\psi)(x,\omega,E'),v\ra 
+({\rm Hes}_\omega \psi)(x,\omega,E')(v,v)
\]
where $v\in T_\omega(S)$ and satisfies by (\ref{e7})
\[
\n{v}\leq C\n{\omega'-\omega}.
\]

\subsection{M\o ller Scattering}\label{sec:moller}

As can be verified from Example \ref{ex:moller} below,
cross section $\sigma$ for the 
M\o ller interaction is of the form
\begin{multline*}
\sigma(x,\omega',\omega,E',E)
=\chi(E',E)\Big(
{1\over{(E'-E)^2}}\sigma_2(x,\omega',\omega,E',E)\\
+{1\over{E'-E}}\sigma_1(x,\omega',\omega,E',E)+\sigma_0(x,\omega',\omega,E',E)\Big)
\end{multline*}
where 
\[
\chi(E',E):=\chi_{\R_+}(E-E_0)\chi_{\R_+}(E_m-E)\chi_{\R_+}(E'-E).
\]
Here each of $\sigma_j(x,\omega',\omega,E',E)$, $j=0,1,2$ may contain the above mentioned $\delta$-distributions,
and hence they are not necessarily measurable functions on $G\times S\times S\times I\times I$.
Denote for $j=0,1,2$,
\[
(\ol {\s K}_j\psi)(x,\omega,E',E):={}&\int_{S'}\sigma_j(x,\omega',\omega,E',E)\psi(x,\omega',E') d\omega', \\[2mm]
(\widehat {\s K}_j\psi)(x,\omega,E',E):={}&\chi(E',E)(\ol {\s K}_j\psi)(x,\omega,E',E).
\]
Here the integral $\int_{S}$ is originally interpreted as a distribution. 
However, we shall find   that $\ol{\s K}_j$ is of the form 
\bea\label{eq:ol_s_K_22_j}
(\ol {\s K}_{j}\psi)(x,\omega,E',E)
=
{}&
\hat\sigma_{j}(x,E',E)
\int_{S'}
\delta(\omega'\cdot\omega - \mu(E',E))\psi(x,\omega',E')d\omega'
\nonumber\\
={}&\hat{\sigma}_{j}(x,E',E)\int_{0}^{2\pi}\psi(x,\gamma(E',E,\omega)(s),E')ds,
\eea
where
\begin{align}\label{eq:mu}
\mu(E',E):=\sqrt{{{E(E'+2)}\over{E'(E+2)}}}
\end{align}
and
$\gamma=\gamma(E',E,\omega):[0,2\pi]\to S$
is a parametrization of the curve
\[
\Gamma(E',E,\omega)=\{\omega'\in S\ |\ \omega'\cdot\omega-\mu(E',E)=0\}
\]
with constant speed
\[
\n{\gamma'(s)}=\sqrt{1-\mu(E',E)^2},\quad s\in [0,2\pi].
\]

For example,  we can choose
\be
\gamma(E',E,\omega)(s)=R(\omega)\big(\sqrt{1-\mu^2}\cos(s),\sqrt{1-\mu^2}\sin(s),\mu\big),\quad s\in [0,2\pi],
\ee
where $\mu=\mu(E',E)$, and $R(\omega)$ is any rotation (unitary) matrix which maps the vector $e_3=(0,0,1)$ into $\omega$.

The formula (\ref{eq:ol_s_K_22_j}) can be seen as follows.
Let $\eta_\epsilon\subset C_0^\infty(\R)$ be such that
$\lim_{\epsilon\to 0} \eta_\epsilon\to\delta$ in $H^{-1}(\R)$.
Then for $\psi\in C_0^\infty(G\times S\times I^\circ)$,
we have (by definition)
\begin{multline}\label{path-1}
\int_{S'}
\delta(\omega'\cdot\omega-\mu(E',E))\psi(x,\omega',E')d\omega'
=\lim_{\epsilon\to 0} 
\int_{S'}
\eta_\epsilon(\omega'\cdot\omega-\mu(E',E))\psi(x,\omega',E')d\omega'.
\end{multline}
For each $t\in [-1,1]$, let
\[
\gamma_t(s):=R(\omega)(\sqrt{1-t^2}\cos(s),\sqrt{1-t^2}\sin(s),t),\quad s\in [0,2\pi],
\]
and let $\Gamma_t$ be the curve (Jordan loop) corresponding to $\gamma_t$,
i.e. $\Gamma_t=\gamma_t([0,2\pi])$.
Then the (differential) surface measure $d\mu_S(\omega)=d\omega$ on $S$ can be disintegrated
into the family $\frac{1}{\sqrt{1-t^2}}d\ell_t\otimes dt$, with $t\in ]-1,1[$,
where $d\ell_t (s)=\n{\gamma_t'(s)}ds$ is the (differential) path length measure (on $\Gamma_t$) along $\gamma_t$
(note that the differential path length of the circle $S_1(0,1)$ is $\frac{1}{\sqrt{1-t^2}} dt$ when we use the parametrization $\alpha(t)=(t,\sqrt{1-t^2}),\ t\in ]-1,1[$ for $S(0,1)$). Noting that $\gamma_t(s)\cdot\omega=t$
we obtain using Fubini's Theorem (in disintegration sense),
\bea\label{path-2}
& 
\hspace{5mm} \int_{S'}
\eta_\epsilon(\omega'\cdot\omega-\mu(E',E))\psi(x,\omega',E')d\omega' \nonumber\\
&=
\int_{-1}^1\int_{\Gamma_t}
\eta_\epsilon(\gamma_t\cdot\omega-\mu(E',E))\psi(x,\gamma_t,E')\frac{1}{\sqrt{1-t^2}}d\ell_t\otimes dt
\nonumber\\
&=
\int_{-1}^1\int_{0}^{2\pi}
\eta_\epsilon(t-\mu(E',E))\psi(x,\gamma_t(s),E')\frac{1}{\sqrt{1-t^2}}\n{\gamma_t'(s)} dsdt
\nonumber\\
&=
\int_{-1}^1\int_{0}^{2\pi}
\eta_\epsilon(t-\mu(E',E))\psi(x,\gamma_t(s),E') dsdt
\nonumber\\
& \mathop{\longto}_{\epsilon\to 0}
\int_{-1}^1\int_{0}^{2\pi}
\delta(t-\mu(E',E))\psi(x,\gamma_t(s),E') dsdt
\nonumber\\
&=
\int_{0}^{2\pi} \psi(x,\gamma_{\mu(E',E)}(s),E')ds.
\eea
Hence combining (\ref{path-1}) and (\ref{path-2}) we get
\begin{align}\label{eq:path-3}
\int_{S'}
\delta(\omega'\cdot\omega-\mu(E',E))\psi(x,\omega',E')d\omega'
=
\int_{0}^{2\pi}
\psi(x,\gamma_{\mu(E',E)}(s),E')ds,
\end{align}
which implies the claim, since
$\gamma_{\mu(E',E)}(s)=\gamma(E',E,\omega)(s)$.

Combining the above treatments one can show that $K$  is of the
form 
\begin{multline}
({K}\psi)(x,\omega,E)
=
{\s H}_2\big((\ol {\s K}_2\psi)(x,\omega,\cdot,E)\big)(E) 
\\
+
{\s H}_1\big((\ol {\s K}_1\psi)(x,\omega,\cdot,E)\big)(E)
+\int_{I}(\widehat {\s K}_0\psi)(x,\omega,E',E) dE',
\label{co-bb}
\end{multline}
where ${\s H}_m$, $m=1,2$, are the {\it Hadamard finite part operators} with respect to $E'$-variable defined by
\[
({\s H}_m u)(E):={\rm p.f.}\int_{E}^{E_m}{1\over{(E'-E)^m}}u(E')dE'.
\]
The expression \eqref{co-bb} is the {\it hyper-singular integral form} of $K$.

Moreover, one can verify that \eqref{co-bb} can be equivalently given in the "\emph{pseudo-differential-like form}" by
\bea\label{co-cc}
({K}\psi)(x,\omega,E)
={}&
{\partial\over{\partial E}}\Big(
{\s H}_1\big((\ol{\s K}_2\psi)(x,\omega,\cdot,E)\big)(E)\Big)
-
{\s H}_1\big(({\p {(\ol{\s K}_2\psi)}E}(x,\omega,\cdot,E)\big)(E)
\nonumber\\
{}&
+{\partial\over{\partial E'}}\Big(
(\ol{\s K}_2\psi)(x,\omega,E',E)\Big)_{|E'=E}\nonumber\\
{}&
+
{\s H}_1\big((\ol{\s K}_1\psi)(x,\omega,\cdot,E)\big)(E)
+\int_{I}(\widehat{\s K}_0\psi)(x,\omega,E',E) dE'
\eea
where only ${\s H}_1$ appears. This formulation reveals the mathematical nature of charged particles' collision operators. 
We neglect the details but recall that the derivation of (\ref{co-cc}) founded on the use of the following lemma (which we here somewhat supplement).

\begin{lemma}\label{hadale}
Suppose that $f\in C^2([a,b]\times [a,b])$.  Then for $x\in [a,b]$
\be\label{ch-id}
{d\over{dx}}\Big({\rm p.f.}\int_x^{b}{{f(x,t)}\over{t-x}}dt\Big)
={\rm p.f.}\int_x^{b}{{f(x,t)}\over{(t-x)^2}}dt+ 
{\rm p.f.}\int_x^{b}{{{\p f{x}}(x,t)}\over{t-x}}dt
-{\p f{t}}(x,x).
\ee
\end{lemma}

\begin{proof}
In virtue of the Taylor's formula and (\ref{hada1})
\bea\label{ch-id1}
I(x):={}&{\rm p.f.}\int_x^{b}{{f(x,t)}\over{t-x}}dt\nonumber\\
={}&
{\rm p.f.}\int_x^{b}{{f(x,x)+\big(\int_0^1{\p f{t}}(x,x+s(t-x))ds\big)(t-x)}\over{t-x}}dt\nonumber\\
={}&f(x,x)\ln(b-x)+\int_x^{b}\int_0^1{\p f{t}}(x,x+s(t-x))dsdt.
\eea
Hence
\bea\label{ch-id2}
&
I'(x)={\p f{x}}(x,x)\ln(b-x)
+{\p f{t}}(x,x)\ln(b-x)
-f(x,x){1\over{b-x}}-{\p f{t}}(x,x)\nonumber\\
&
+\int_x^{b}\int_0^1{{\partial^2 f}\over{\partial x\partial t}}(x,x+s(t-x))dsdt
+
\int_x^{b}\int_0^1{{\partial^2 f}\over{\partial t^2}}(x,x+s(t-x))(1-s)dsdt
\eea

On the other hand by the Taylor's formula and by (\ref{hada1}), (\ref{hada2}) 
\bea\label{ch-id3}
&{\rm p.f.}\int_x^{b}{{f(x,t)}\over{(t-x)^2}}dt\nonumber\\
={}&
{\rm p.f.}\int_x^{b}{1\over{(t-x)^2}}\Big[f(x,x)+{\p f{t}}(x,x)(t-x)+\big(\int_0^1(1-s){{\partial^2f}\over{\partial t^2}}(x,x+s(t-x))ds\big)(t-x)^2\Big]dt\nonumber\\
={}&-f(x,x){1\over{b-x}}+{\p f{t}}(x,x)\ln(b-x)+\int_x^{b}\int_0^1(1-s){{\partial^2f}\over{\partial t^2}}(x+s(t-x))dsdt\nonumber\\
={}&I'(x)-{\rm p.f.}\int_x^{b}{{{\p f{x}}(x,t)}\over{t-x}}dt +{\p f{t}}(x,x)
\eea
where we noticed that by the Taylor's formula
\[
{\p f{x}}(x,x)\ln(b-x)+
\int_x^{b}\int_0^1{{\partial^2f}\over{\partial x\partial t}}(x,x+s(t-x))dsdt
=
{\rm p.f.}\int_x^{b}{{{\p f{x}}(x,t)}\over{t-x}}dt.
\]
This completes the proof.
\end{proof}

\begin{remark}\label{psiDO}
We give the following observation which explains the above word "pseudo-differential-like".
The operators of the form
\be\label{r-1}
(Pu)(x,E):={}&
{\rm p.f.}\int_{E_0}^{E_m}{{\sigma_0(x,E,E')}\over{E'-E}}u(x,E,E')dE',\quad u\in C_0^\infty(G\times I^\circ\times I^\circ),
\ee
can be treated as in \cite[Chapter 7]{hsiao}. Note that in (\ref{r-1}) the integration is over the whole interval $[E_0,E_m]$.
Under relevant assumptions on $\sigma_0$, the operators (\ref{r-1}) can be shown to be pseudo-differential operators.
In particular, we recall that the {\it partial Hilbert transform} 
\[
(Hu)(x,E):={}&
{\rm p.f.}\int_{E_0}^{E_m}{{u(x,E,E')}\over{E-E'}}dE'
={\rm p.v.}\int_{E_0}^{E_m}{{u(x,E,E')}\over{E-E'}}dE'
\]
is a pseudo-differential operator with symbol $-{\rm i}\ {\rm sign}(\xi)$ where $\xi$ is the transform variable with respect to $E'$-variable.

Note that the Hadamard finite part operator ${\s H}_1$ introduced above  is of the form (we omit the $\omega$-variable)
\be\label{r-1-a}
({\s H}_1u)(x,E):={}&
{\rm p.f.}\int_{E}^{E_m}{{\sigma_0(x,E,E')}\over{E'-E}}u(x,E,E')dE',\quad u\in C_0^\infty(G\times I^\circ\times I^\circ).
\ee
The problematic feature in the expression of $({\s H}_1u)(x,E)$ is that the integration is over $[E,E_m]$.
Similar observations concern the operator ${\s H}_2$. 
In consistence with \cite{hsiao},
formal computations suggest that the expected pseudo-differential symbols of ${\s H}_j$, $j=1,2$, are
\be 
p_1(x,E,\xi)=
{\rm p.f.}\int_E^{\infty}{{\sigma_0(x,E,E')}\over{E'-E}} e^{{\rm i}(E'-E)\xi}dE'
=
{\rm p.f.}\int_0^{\infty}{{\sigma_0(x,E,E+z)}\over{z}} e^{{\rm i} z\xi}dz,
\ee
and
\be 
p_2(x,E,\xi)=
{\rm p.f.}\int_0^{\infty}{{\sigma_0(x,E,E+z)}\over{z^2}} e^{{\rm i} z\xi}dz,
\ee
respectively.
Careful analysis of these operators (whether e.g. they truly are  pseudo-differential operators) remains to our knowledge open.
\end{remark}

As a conclusion we see that some 
interactions produce the first-order partial derivatives
with respect to energy $E$ combined with the Hadamard part operator (which is
a pseudo-differential-like operator).
These problematic interactions
are the primary electron-electron, primary positron-positron collisions and Bremsstrahlung (see \eqref{coll-2} and Example \ref{ex:brems}).
The analysis of the above pseudo-differential-type terms
require further study which we omit here.  However, we mention that 
the pseudo-differential-like parts can be approximated 
by partial differential operators
(cf. section
\ref{a-psi}) to obtain pure partial differential operator approximations for collision operators.

\begin{remark}
Bremsstrahlung interaction  produces hyper-singularities to the collision operator as well.
It seems to us that the corresponding cross sections are somewhat obscured in literature but the following general level description can be given.
In certain cases, and as argued in Example \ref{ex:brems} below, the corresponding 
differential cross section for electrons is of the form
\begin{multline}
\sigma(x,\omega',\omega,E',E)
=
\chi(E',E)\Big({{\ln(E'-E)}\over{E'-E}}\sigma_3(x,\omega',\omega,E',E)\\
+{1\over{E'-E}}\sigma_2(x,\omega',\omega,E',E)+
\ln(E'-E)\sigma_1(x,\omega',\omega,E',E)
+\sigma_0(x,\omega',\omega,E',E)\Big)
\label{coll-2}
\end{multline}
where $\chi(E',E)$ is a product of characteristic functions.
Here each of $\sigma_j(x,\omega',\omega,E',E)$, $j=0,1,2,3$ is a measurable non-negative function on $G\times S^2\times I^2$. Actually they are of the more refined form
\be\label{mo-1}
\sigma_j(x,\omega',\omega,E',E)=\hat\sigma_j(x,E',E)\eta_j(\omega'\cdot\omega,E')
\ee
where  $\hat\sigma_j\in L^\infty(G\times I^2),\ \eta_j\in L^\infty([-1,1]\times I)$.
The analysis of related hyper-singular collision operators goes analogously to the M\ oller scattering. However, we remark that the singularities are weaker 
and the most singular operator is of the form 
\[
(\widetilde{\s H}_2u)(E):={\rm p.f.}\int_E^{E_m}{{\ln(E'-E)}\over{E'-E}}u(E') dE'.
\]
This hyper-singular operator is emerging from the cross section 
${{\ln(E'-E)}\over{E'-E}}\sigma_3(x,\omega',\omega,E',E)$ occurring in (\ref{coll-2}).
We omit further discussion of this case.

\end{remark}

\begin{remark}\label{notes-3}
 
At least in the existence and uniqueness analysis of solutions, it might be more fruitful to use the pseudo-differential-like expressions (like (\ref{co-cc})) of the exact transport equation.  
Nevertheless, the numerical methods might apply directly the hyper-singular partial integral equation $T\psi=f$. For instance, the Galerkin (discontinuous) finite element methods (FEM) are able to treat hyper-singular partial integral terms.  
These techniques are well-known e.g. in the field of
boundary element methods (BEM) where the hyper-singular
integral kernels are emerging from single and double layer potentials.
We remark that carefully chosen (special) numerical integration schemes, and the choice of basis functions are needed in 
computing element matrices for hyper-singular integral operators.
The applicability of these methods for the problem considered remains an open question.

Besides of the existence theory treated in this paper it is important to understand regularity of solutions of the considered transport equations in the {\it mixed-norm (anisotropic) Sobolev-Slobodevskij spaces} $H^s(G\times S\times I^\circ),\ s=(s_1,s_2,s_3)$.
Regularity is needed e.g. in approximation analysis and, in particular, in numerical analysis (e.g. FEM).
In the existence and regularity analysis the above derived pseudo-differential form-like expressions of collision operators might also be useful. Regularity results remain open as well.

\end{remark}

\begin{remark}\label{coll-n=2}
The example 2.28 in our earlier version \cite{tervo16-up} is somewhat 
erroneous since the partial derivatives ${\p {\xi_{\pm}}{E'}}(E,E,\omega)$ do 
not exist for the given $\mu$. We give its correction as follows.

When the spatial dimension $n=2$, we have (with $S=S_1$ the unit circle on $\R^2$)
\[
R(\omega)=\qmatrix{-\omega_2 & \omega_1\\ \omega_1 & \omega_2},
\]
\[
\ \gamma(E',E,\omega)(\pm 1)=R(\omega)
\qmatrix{\pm\sqrt{1-\mu(E',E)^2}\cr \mu(E',E)\cr}
=\mu(E',E)\omega\pm \sqrt{1-\mu(E',E)^2}\omega^\perp
\]
and so
\bea\label{kj-b}
(\ol {\s K}_{j}\psi)(x,\omega,E',E) 
={}&
\hat{\sigma}_{j}(x,E',E)\big(\psi(x,\mu(E',E)\omega+\sqrt{1-\mu(E',E)^2}\omega^\perp,E')\nonumber\\
&
+
\psi(x,\mu(E',E)\omega-\sqrt{1-\mu(E',E)^2}\omega^\perp,E')\big),
\eea
where $\omega^\perp:=(-\omega_2,\omega_1)$ (the tangent vector of the unit circle $S$ at $\omega$).

Denote
\[
\xi_{\pm}=
\xi_{\pm}(E',E,\omega):=\mu(E',E)\omega\pm\sqrt{1-\mu(E',E)^2}\omega^\perp.
\]
Then
\bea\label{n2k-2}
{\partial\over{\partial E'}}\Big((\ol{\s K}_{2}\psi)(x,\omega,E',E)\Big)
={}&
{\partial\over{\partial E'}}
\Big(\hat\sigma_2(x,E',E)\psi(x,\xi_+(E',E,\omega),E')\Big) \nonumber\\
&+
{\partial\over{\partial E'}}
\Big(\hat\sigma_2(x,E',E)\psi(x,\xi_-(E',E,\omega),E')\Big)
\eea
and so for $E'\not=E$
\bea\label{n2k-4}
&
{\partial\over{\partial E'}}\Big((\ol{\s K}_{2}\psi)(x,\omega,E',E)\Big) =
{\p {\hat\sigma_2}{E'}}(x,E',E)\psi(x,\xi_+,E')
\nonumber\\
{}&+
\hat\sigma_2(x,E',E)(\nabla_\omega\psi)(x,\xi_+,E')\cdot
{\p {\xi_+}{E'}}(E',E,\omega)
+
\hat\sigma_2(x,E',E){\p {\psi}{E}}(x,\xi_+,E')
\nonumber\\
{}&+
{\p {\hat\sigma_2}{E'}}(x,E',E)\psi(x,\xi_-,E')
\nonumber\\
{}&+
\hat\sigma_2(x,E',E)(\nabla_\omega\psi)(x,\xi_-,E')\cdot
{\p {\xi_-}{E'}}(E',E,\omega)
+
\hat\sigma_2(x,E',E){\p {\psi}{E}}(x,\xi_-,E')
.
\eea

It can be shown that the limit below exists,
\begin{multline*}
\lim_{E'\to E}\Big(
\hat\sigma_2(x,E',E)(\nabla_\omega\psi)(x,\xi_+,E')\cdot
{\p {\xi_+}{E'}}(E',E,\omega)
\nonumber\\
\shoveright{
+
\hat\sigma_2(x,E',E)(\nabla_\omega\psi)(x,\xi_-,E')\cdot
{\p {\xi_-}{E'}}(E',E,\omega)\Big)
} \\
\shoveleft{
=
2\, \hat{\sigma}_{2}(x,E,E)R(\omega)\qmatrix{0\cr (\partial_{E'}\mu)(E,E)\cr}\cdot(\nabla_{\omega}\psi)(x,\omega,E)
}
 \\
+
\hat{\sigma}_{2}(x,E,E)
\sum_{|\alpha|=2}a_{\alpha}(E,\omega)(\partial_{\omega}^\alpha\psi)(x,\omega,E),
\end{multline*} 
where 
for $\alpha=(i,j)$
\begin{multline*}
\sum_{|\alpha|=2}a_\alpha (\omega,E)
(\partial_{\omega}^\alpha\psi)(x,\omega,E')\nonumber\\
:={1\over 2}\sum_{i=1}^2\sum_{j=1}^2\partial_i\partial_j (\psi\circ  H_\omega)(x,0,E')\lim_{E'\to E}\int_0^{2\pi}\xi_i((E',E,\omega,s)\zeta_j(E',E,\omega,s) ds,
\end{multline*}
where $\xi(E',E,\omega,s):=J(v(E',E,\omega,s))$, $\zeta(E',E,\omega,s):=J({\p {\gamma}{E'}}(E',E,\omega)(s))$,
and $J$ is the isomorphism defined in \eqref{eq:J_iso}.

Hence we have
\bea\label{aa}
&
{\partial\over{\partial E'}}\Big((\ol{\s K}_{2}\psi)(x,\omega,E',E)  \Big)_{|E'=E}
=
2\hat\sigma_2(x,E,E){\p {\psi}{E}}(x,\omega,E)+
2{\p {\hat\sigma_2}{E'}}(x,E,E)\psi(x,\omega,E)\nonumber\\
&
+
2\hat\sigma_2(x,E,E)R(\omega)\qmatrix{0\cr (\partial_{E'}\mu)(E,E)\cr}\cdot (\nabla_\omega\psi)(x,\omega,E)\nonumber\\
&
+
\hat{\sigma}_{2}(x,E,E)
\sum_{|\alpha|=2}a_{\alpha}(E,\omega)(\partial_{\omega}^\alpha\psi)(x,\omega,E)
.
\eea

Moreover, for $E'\not=E$
\bea\label{n2k-5}
&
{\p {(\ol{\s K}_{2}\psi)}E}(x,\omega,E',E)
\nonumber\\
={}&
{\partial\over{\partial E}}
\Big(\hat\sigma_2(x,E',E)\psi(x,\xi_+(E',E,\omega),E')
+\hat\sigma_2(x,E',E)\psi(x,\xi_-(E',E,\omega),E')
\Big)
\nonumber\\
={}&
{\p {\hat\sigma_2}{E}}(x,E',E)\psi(x,\xi_+(E',E,\omega),E') \nonumber\\
{}&
+\hat\sigma_2(x,E',E)(\nabla_\omega\psi)(x,\xi_+(E',E,\omega),E')\cdot {\p {\xi_+}E}(E',E) \nonumber\\
{}&
+{\p {\hat\sigma_2}{E}}(x,E',E)\psi(x,\xi_-(E',E,\omega),E') \nonumber\\
{}&
+\hat\sigma_2(x,E',E)(\nabla_\omega\psi)(x,\xi_-(E',E,\omega),E') 
\cdot {\p {\xi_-}E}(E',E).
\eea

We omit further analysis here but emphasise that it is instructive to keep 
in mind this kind of simplified situation when analysing the exact transport equations.  
\end{remark}

The corresponding collision operators can be analogously computed in the general spatial dimension $n$ (which we omit here). 
The {M\o ller} collision operator produces first order partial differential terms with respect to $E$ along with a Hadamard finite part operator and second order partial derivatives with respect to $\omega$.
The exact form of {M\o ller} collision operator
allows for accessing relevant approximation schemes for which the error analysis 
can be carried out.

\sectionspace
\section{Approximative Pseudo-Differential Collision Operators}\label{r-psio-1}

The expression (\ref{co-bb}) (or (\ref{co-cc})) gives
the exact collision operator  for M\o ller scattering,
and can be opened up further.
The M\o ller collision operator  contains  partial derivatives with respect to energy
and angle. 
The resulting operator is quite complex and so it is reasonable to seek
simpler approximations of it. In the sequel we give one potential scheme for an approximation.

The M\o ller collision operator given by (\ref{co-bb}) is written explicitly for $E\geq E_0$ by
\bea\label{mol-a}
&
(K\psi)(x,\omega,E)=
{\rm p.f.}\int_E^{E_m}\hat\sigma_{2}(x,E',E){1\over{(E'-E)^2}}
\int_{0}^{2\pi}\psi(x,\gamma(E',E,\omega)(s),E')dsdE'
\nonumber\\
&
+
{\rm p.f.}\int_E^\infty\hat\sigma_{1}(x,E',E){1\over{E'-E}}
\int_{0}^{2\pi}\psi(x,\gamma(E',E,\omega)(s),E')dsdE'
+(K_0\psi)(x,\omega,E)
\eea
where
\bea\label{psio-0b}
&
(K_0\psi)(x,\omega,E):=\int_{I'}(\hat{\s K}_0\psi)(x,\omega,E',E)dE'
\nonumber\\
&
=
\int_{I'}\int_0^{2\pi}\chi(E',E)\hat\sigma_0(x,E',E)\psi(x,\gamma(E',E,\omega)(s),E') ds dE'.
\eea

The following analysis applies in quite straightforward manner
to more general collision operators of the form
\bea\label{mol-b}
&
(K\psi)(x,\omega,E)=
{\rm p.f.}\int_E^{E_m}{1\over{(E'-E)^{1+\kappa_2}}}
\int_{S'}\sigma_2(x,\omega',\omega,E',E)\psi(x,\omega',E')d\omega'dE'
\nonumber\\
&
+
{\rm p.f.}\int_E^{E_m}{1\over{(E'-E)^{\kappa_1}}}
\int_{S'}\sigma_1(x,\omega',\omega,E',E)\psi(x,\omega',E')d\omega'dE'
+(K_0\psi)(x,\omega,E)
\eea
where $0\leq\kappa_j\leq 1$ and where $K_0$ is a Schur partial integral operator.
Above $\sigma_j$ can be chosen in several ways, for example
($j=0,1,2$)
\begin{enumerate}
\item 
either
$\sigma_j$ are measurable functions, e.g. as in Bremsstrahlung,
\item
or, e.g. as in M\o ller and Compton scattering,
\[
\int_{S'}\sigma_j(x,\omega',\omega,E',E)\psi(x,\omega',E')d\omega'
=\int_0^{2\pi}\hat\sigma_j(x,E',E)\psi(x,\gamma(E',E,\omega)(s),E') ds 
\]
\item
or, when a term like $\delta(\omega'\cdot\omega-1)$ is involved in $\sigma_j$,
\[
\int_{S'}\sigma_j(x,\omega',\omega,E',E)\psi(x,\omega',E')d\omega'
=\ol\sigma_j(x,\omega,E',E)\psi(x,\omega,E').
\]
\end{enumerate}
To keep the considerations concise we carry out the computations only when $\kappa_2=\kappa_1=1$
and for the case (2).
In more general cases one ought, for example, to separate cases $\kappa_j<1$ and $\kappa_j=1$.
The case $\kappa_2=\kappa_1=1$ also reveals
some of the difficulties more transparently,
and it provides an indication of the tools needed for other cases.

We shall show in section \ref{rco} that $K_0$ given by (\ref{psio-0b}) is an ordinary partial Schur integral operator and so it suffices only to approximate terms 
${\s H}_j\big((\ol{\s K}_j\psi)(x,\omega,\cdot,E)\big)(E),\ j=1,2$.
The corresponding collision operators are for $E\geq E_0$
\be\label{psio-0a}
(K_j\psi)(x,\omega,E):=
{\rm p.f.}\int_E^{E_m}\hat\sigma_{j}(x,E',E){1\over{(E'-E)^j}}
\int_{0}^{2\pi}\psi(x,\gamma(E',E,\omega)(s),E')dsdE'.
\ee
In this section we shall approximate these terms 
in such the way that the pseudo-differential nature of them and the related symbols can be recognized.
We assume that a {\it cut-off energy} has been specified at $E'=c(E)$,
where $c(E)$ is a continuous function obeying  $c(E)\geq \alpha E$ for some $\alpha>1$.
For example, $c(E)=2E$ that is $E={1\over 2}E'$.
We shall not here impose exact assumptions on $\hat\sigma_j(x,E',E)$ and $\psi$ but we proceed formally.

We decompose the integration in (\ref{psio-0a}) as follows (for  $j=1,2$)
\bea\label{psio-0c}
&
{\rm p.f.}\int_E^{E_m}\hat\sigma_{j}(x,E',E){1\over{(E'-E)^j}}
\int_{0}^{2\pi}\psi(x,\gamma(E',E,\omega)(s),E')dsdE'
\nonumber\\
={}&
{\rm p.f.}\int_E^{c(E)}\hat\sigma_{j}(x,E',E){1\over{(E'-E)^j}}
\int_{0}^{2\pi}\psi(x,\gamma(E',E,\omega)(s),E')dsdE'\nonumber\\
&
+
\int_{c(E)}^{E_m}\hat\sigma_{j}(x,E',E){1\over{(E'-E)^j}}
\int_{0}^{2\pi}\psi(x,\gamma(E',E,\omega)(s),E')dsdE',
\eea
where we note that the last integral is the ordinary Lebesgue integral.
Let
\bea
(K_{j,1}\psi)
&
(x,\omega,E)\nonumber\\
:= {}&
{\rm p.f.}\int_E^{c(E)}\hat\sigma_{j}(x,E',E){1\over{(E'-E)^j}}
\int_{0}^{2\pi}\psi(x,\gamma(E',E,\omega)(s),E')dsdE', \label{def-K_{j}-a} \\[2mm] 
(K_{j,0}\psi)
&
(x,\omega,E)\nonumber\\
:= {}&
\int_{c(E)}^{E_m}\hat\sigma_{j}(x,E',E){1\over{(E'-E)^j}}
\int_{0}^{2\pi}\psi(x,\gamma(E',E,\omega)(s),E')dsdE'. \label{def-K_{j}-b}
\eea
Then
\be\label{de}
(K_j\psi)(x,\omega,E)=
(K_{j,1}\psi)(x,\omega,E)+(K_{j,0}\psi)(x,\omega,E)
.
\ee
Since
\bea\label{psio-0d}
&
(K_{j,0}\psi)(x,\omega,E)\nonumber\\
&
=
\int_{I'}\chi_{\R_+}(E'-c(E))\hat\sigma_{j}(x,E',E){1\over{(E'-E)^j}}
\int_{0}^{2\pi}\psi(x,\gamma(E',E,\omega)(s),E')dsdE'
\eea
we shall find  in section \ref{rco} below that under appropriate assumptions $K_{j,0}$ are partial Schur integral operators. 
Hence it suffices to consider only the operators 
$K_{j,1}$.

\subsection{Approximative First-order Hadamard Singular Integral Operator}\label{fhsio}

At first we
consider the partial hyper-singular integral operator $K_{1,1}$ given by (\ref{def-K_{j}-a}) that is,
\[
(K_{1,1}\psi)(x,\omega,E)=
{\rm p.f.}\int_E^{c(E)}\hat\sigma_{1}(x,E',E){1\over{E'-E}}
\int_{0}^{2\pi}\psi(x,\gamma(E',E,\omega)(s),E')dsdE'.
\]
By the Taylor's formula we can expand
\be\label{psi-3}
\hat\sigma_1(x,E',E)=\hat\sigma_1(x,E,E)+\int_0^1{\p {\hat\sigma_1}{E'}}(x,E+t(E'-E),E)dt\big)(E'-E),
\ee
which gives
\bea\label{psi-4}
&
( K_{1,1}\psi)(x,\omega,E)
=
{\rm p.f.}\int_{E}^{c(E)}\int_0^{2\pi}{{\hat\sigma_1(x,E,E)}\over{E'-E}}\psi(x,\gamma(E',E,\omega)(s),E') dsdE'
\nonumber\\
&
+
\int_{E}^{c(E)}\int_0^{2\pi}\int_0^1{\p {\hat\sigma_1}{E'}}(x,E+t(E'-E),E)dt\big)
\psi(x,\gamma(E',E,\omega)(s),E') dsdE'\nonumber\\[2mm]
={}&
( K_{1,1,1}\psi)(x,\omega,E)+
( K_{1,1,0}\psi)(x,\omega,E)
\eea
where
\[
&
( K_{1,1,1}\psi)(x,\omega,E):=2\pi\ 
{\rm p.f.}\int_{E}^{c(E)}{{\hat\sigma_1(x,E,E)}\over{E'-E}}\int_0^{2\pi}\psi(x,\gamma(E',E,\omega)(s),E') dsdE'\nonumber\\
={}&
\hat\sigma_1(x,E,E)
{\rm p.f.}\int_{E}^{c(E)}{{1}\over{E'-E}}\int_0^{2\pi}\psi(x,\gamma(E',E,\omega)(s),E') dsdE'
\]
and
\[
( K_{1,1,0}\psi)(x,\omega,E):=
\int_{E}^{c(E)}\int_0^{2\pi} \int_0^1{\p {\hat\sigma_1}{E'}}(x,E+t(E'-E),E)dt\big)
\psi(x,\gamma(E',E,\omega)(s),E') dsdE'
.
\]
Since 
\begin{multline*}
( K_{1,1,0}\psi)(x,\omega,E)
:=
\int_{I'}\int_0^{2\pi}\chi_{\R_+}(E'-E)\chi_{\R_+}(c(E)-E')\big(\int_0^1{\p {\hat\sigma_1}{E'}}(x,E+t(E'-E),E)dt\big) \\
\cdot\psi(x,\gamma(E',E,\omega)(s),E') dsdE'
\end{multline*}
we shall find in section \ref{rco} that under relevant assumptions 
$ K_{1,1,0}$ is the partial Schur  integral operator
\be 
( K_{1,1,0}\psi)(x,\omega,E):=
\int_{I'}\int_0^{2\pi}\hat\sigma_{1,0}(x,E',E)
\cdot\psi(x,\gamma(E',E,\omega)(s),E') dsdE'
\ee
where 
\[
\hat\sigma_{1,0}(x,E',E):=
\chi_{\R_+}(E'-E)\chi_{\R_+}(c(E)-E')\big(\int_0^1{\p {\hat\sigma_1}{E'}}(x,E+t(E'-E),E)dt\big).
\]

{\it We assume that $\mu(E',E)$, the cosine of the scattering angle, obeys (formally)} 
\be \label{cos-ap}
\mu(E',E)\approx 1\ 
{\rm on\ the\ interval}\ E'\in [E,c(E)].
\ee
Then, formally again,
\[
\gamma(E',E,\omega)(s)\approx R(\omega)(0,0,1)=\omega
\]
which implies that
\be\label{ess-app}
\psi(x,\gamma(E',E,\omega)(s),E')\approx \psi(x,\omega,E') 
\ee
on the interval $E'\in [E,c(E)]$.
From mathematical point of view the essential approximation is (\ref{ess-app}).
As for the error analysis  the approximation (\ref{ess-app}) might be problematic (cf. section \ref{err}).  

Using (\ref{ess-app}) we obtain
\[
(K_{1,1,1}\psi)(x,\omega,E)
\approx
(\widetilde K_{1,1,1}\psi)(x,\omega,E)
:=2\pi\, \hat\sigma_1(x,E,E)\, (P\psi)(x,\omega,E),
\]
where
\begin{align}\label{eq:op_P}
(P\psi)(x,\omega,E):=
{\rm p.f.}\int_{E}^{c(E)}{{1}\over{E'-E}}\psi(x,\omega,E') dE'.
\end{align}
Let $\mc F_p$ denote the {\it partial Fourier transform} with respect to $E$ variable and let $\xi$ be the corresponding transform variable that is, 
\[
(\mc F_p\psi)(x,\omega,\xi)=(2\pi)^{-1/2}\int_{\R}\psi(x,\omega,E)e^{-{\rm i}\xi E} dE.
\]

We have the following characterization of operator $P$

\begin{theorem}\label{psi-th1}
The operator
$P $ is a  pseudo-differential operator (cf. \cite{abels11})
\be\label{pd}
(P(E,D)\psi)(x,\omega,E)=(2\pi)^{-1/2}\int_{\R}p(E,\xi)(\mc F_p\psi)(x,\omega,\xi)e^{{\rm i}\xi E}d\xi,
\ee
for all $\psi\in C_0^\infty(G\times S\times I^\circ)$,
and its symbol is
\[
p(E,\xi):={\rm p.f.}\int_{0}^{c(E)-E}{{1}\over{z}}
e^{{\rm i}z\xi}dz.
\]
\end{theorem}

\begin{proof}

A. At first we proceed formally.
Let $\psi\in C_0^\infty(G\times S\times I^\circ)$.  
Then we  have by the inverse Fourier formula
\bea\label{psi-6}
&
(P\psi)(x,\omega,E):=(2\pi)^{-1/2}{\rm p.f.}\int_{E}^{c(E)}{{1}\over{E'-E}}\Big(\int_{\R}(\mc F_p\psi)(x,\omega,\xi)e^{{\rm i}E'\xi}d\xi\Big)dE'\nonumber\\
&
=(2\pi)^{-1/2}\int_{\R}\Big({\rm p.f.}\int_{E}^{c(E)}{{1}\over{E'-E}}
e^{{\rm i}E'\xi}dE'\Big)
(\mc F_p\psi)(x,\omega,\xi)d\xi\nonumber\\
&
=(2\pi)^{-1/2}\int_{\R}\Big({\rm p.f.}\int_{E}^{c(E)}{{1}\over{E'-E}}
e^{{\rm i}(E'-E)\xi}dE'\Big)
(\mc F_p\psi)(x,\omega,\xi)e^{{\rm i}E\xi}d\xi\nonumber\\
&
=
(2\pi)^{-1/2}\int_{\R}p(E,\xi)
(\mc F_p\psi)(x,\omega,\xi)e^{{\rm i}E\xi}d\xi
\eea
where
\be 
p(E,\xi):=
{\rm p.f.}\int_{E}^{c(E)}{{1}\over{E'-E}}
e^{{\rm i}(E'-E)\xi}dE'.
\ee
The change of order of integration, in the second step above, will be
justified in the last section C. of this proof.

B.
We find that (recall (\ref{def-h1}))
\bea\label{psi-6a}
p(E,\xi)
&
=
{\rm p.f.}\int_{E}^{c(E)}{{1}\over{E'-E}}
e^{{\rm i}(E'-E)\xi}dE'\nonumber\\
&
=\lim_{\epsilon \to 0^+}\Big(\int_{E+\epsilon}^{c(E)}{{1}\over{E'-E}}
e^{{\rm i}(E'-E)\xi}dE'+\ln(\epsilon)\cdot 1\Big)\nonumber\\
&
=
\lim_{\epsilon\to 0^+}\Big(\int_{\epsilon}^{c(E)-E}{1\over{z}}
e^{{\rm i}z\xi}dz+\ln(\epsilon)\cdot 1\Big)
=
{\rm p.f.}\int_{0}^{c(E)-E}{{1}\over{z}}
e^{{\rm i}z\xi}dz
\eea
where we in the third step applied the change of variables $E'-E=z$.

C. Finally, we provide here a justification for the second step in \eqref{psi-6}.
By the Taylor's formula (\ref{taylor}) and by (\ref{hada1})
\bea\label{psi-9-a}
&
\int_{\R}\Big({\rm p.f.}\int_{E}^{c(E)}{1\over{E'-E}}e^{{\rm i} E'\xi}dE'\Big)(\mc F_p\psi)(x,\omega,\xi)d\xi\nonumber\\
={}&
\int_{\R}\Big({\rm p.f.}\int_{E}^{c(E)}{1\over{E'-E}}\big(e^{{\rm i} E\xi}
+\int_0^1{\rm i}\xi (E'-E) e^{{\rm i}(E+s(E'-E))\xi}ds\big)
dE'\Big)(\mc F_p\psi)(x,\omega,\xi)d\xi\nonumber\\
={}&
\int_{\R}\Big({\rm p.f.}\int_{E}^{c(E)}{1\over{E'-E}}dE'\Big)e^{{\rm i} E\xi}(\mc F_p\psi)(x,\omega,\xi)d\xi\nonumber\\
&
+\int_{\R}\Big(\int_{E}^{c(E)}\int_0^1{\rm i}\xi e^{{\rm i}(E+s(E'-E))\xi}ds
dE'\Big)(\mc F_p\psi)(x,\omega,\xi)d\xi\nonumber\\
={}&
(2\pi)^{1/2}\ln(c(E)-E)\, \psi(x,\omega,E)
+
(2\pi)^{1/2}
\int_E^{c(E)}\int_0^1{\p {\psi}E}(x,\omega,E+s(E'-E)) ds dE'
\eea
where in the third step Fubini's theorem, the formula
\[
\mc F_p\big({\p {\psi}E}\big)(x,\omega,\xi)={\rm i}\xi (\mc F_p\psi)(x,\omega,\xi)
\]
and the Fourier inversion theorem were applied to the second term
on the right of the equality number two.

On the other hand, by the Taylor's formula and by (\ref{hada1}),
\bea\label{psi-10-a}
&
 {\rm p.f.}\int_{E}^{c(E)}{1\over{E'-E}}\Big(\int_{\R}(\mc F_p\psi)(x,\omega,\xi)e^{{\rm i} E'\xi}d\xi\Big)dE'\nonumber\\
&
=(2\pi)^{1/2}
{\rm p.f.}\int_{E}^{c(E)}{1\over{E'-E}}\psi(x,\omega,E')dE'
\nonumber\\
&
=(2\pi)^{1/2}
{\rm p.f.}\int_{E}^{c(E)}{1\over{E'-E}}\Big(\psi(x,\omega,E)+(E'-E)\int_0^1{\p {\psi}E}(E+s(E'-E))ds \Big)dE'\nonumber\\
&
=
(2\pi)^{1/2}
{\rm p.f.}\int_{E}^{c(E)}{1\over{E'-E}}\psi(x,\omega,E)dE'
+(2\pi)^{1/2}
\int_{E}^{c(E)}\int_0^1{\p {\psi}E}(E+s(E'-E))ds dE'\nonumber\\
&
=
(2\pi)^{1/2}\ln(c(E)-E)\, \psi(x,\omega,E)+
(2\pi)^{1/2}
\int_{E}^{c(E)}\int_0^1{\p {\psi}E}(E+s(E'-E))ds dE'
.
\eea
We thus see that the left hand sides of \eqref{psi-9-a} and \eqref{psi-10-a} are equal,
an observation which renders the second equality in \eqref{psi-6} legitimate,
and therefore completes this proof.
\end{proof}

Consider the symbol $p(E,\xi)$ for $E>0$. We have
\be\label{m-1}
p(E,\xi)={\rm p.f.}\int_{0}^{c(E)-E}{{1}\over{z}}
e^{{\rm i}z\xi}dz
=
{\rm p.f.}\int_{0}^{\infty}{{1}\over{z}}
e^{{\rm i}z\xi}dz
-\int_{c(E)-E}^{\infty}{{1}\over{z}}
e^{{\rm i}z\xi}dz=p_1(\xi)-p_2(E,\xi)
\ee
where
\[
p_1(\xi):={\rm p.f.}\int_{0}^{\infty}{{1}\over{z}}
e^{{\rm i}z\xi}dz,\quad
p_2(E,\xi):=\int_{c(E)-E}^{\infty}{{1}\over{z}}
e^{{\rm i}z\xi}dz.
\]
Here the integral ${\rm p.f.}\int_{0}^{\infty}{{1}\over{z}}
e^{{\rm i}z\xi}dz$ is interpreted as
\[
{\rm p.f.}\int_{0}^{\infty}{{1}\over{z}}
e^{{\rm i}z\xi}dz
=
{\rm p.f.}\int_{0}^{1}{{1}\over{z}}
e^{{\rm i}z\xi}dz
+\int_{1}^{\infty}{{1}\over{z}}
e^{{\rm i}z\xi}dz
\] 
where the last integral is the ordinary improper (oscillatory) integral. The integral $\int_{c(E)-E}^{\infty}{{1}\over{z}}
e^{{\rm i}z\xi}dz$ is also the ordinary improper integral (recall that \linebreak ${c(E)-E\geq (\alpha-1)E>0}$).
Hence the operator $P(E,D)$ is the sum 
\be\label{m-2}
P(E,D)=P_1(D)-P_2(E,D)
\ee
where $P_1(D)$ and $P_2(E,D)$ are the pseudo-differential operators with symbols $p_1(\xi)$ and $p_2(E,\xi)$, respectively. Note that the pseudo-differential operators $P_1(D)$ and $P_2(E,D)$ are    for $\psi\in C_0^\infty(G\times S\times I^\circ)$
\be\label{eq:op_P_1}
(P_1(D)\psi)(x,\omega,E)=
{\rm p.f.}\int_{E}^{\infty}{{1}\over{E'-E}}\psi(x,\omega,E') dE',
\ee
\be\label{eq:op_P_2} 
(P_2(E,D)\psi)(x,\omega,E)=
\int_{c(E)}^{\infty}{{1}\over{E'-E}}\psi(x,\omega,E') dE'.
\ee
The symbol $p_1(\xi)$ can be given explicitly.

\begin{lemma}
The symbol $p_1(\xi)$ of $P_1(D)$ is
\[
p_1(\xi)=-\ln(|\xi|)+\alpha_0+{\rm i}\beta_0{\rm sign}(\xi),
\quad \xi\not=0
\]
where 
\be\label{const}
\alpha_0:= \int_0^1{{1-\cos(t)}\over t}dt+\int_1^\infty{{\cos(t)}\over t} dt,
\quad
\beta_0:=\sqrt{{{\pi}\over 2}}.
\ee
\end{lemma}

\begin{proof}
Note that for $\psi\in C_0^\infty(G\times S\times I^\circ)$
the integral
\be\label{pd-a}
\int_{\R}\big(-\ln(|\xi|)+\alpha_0+{\rm i}\beta_0{\rm sign}(\xi) \big)(\mc F_p\psi)(x,\omega,\xi)e^{{\rm i}\xi E}d\xi
\ee
convergences since (recalling that $\int \ln(\xi)d\xi=\xi\ln(\xi)-\xi$ for $\xi>0$)
\[
\int_{-1}^1|\ln(|\xi|)|d\xi=-2\int_{0}^1\ln(\xi)d\xi=2
\]
and $\ln(|\xi|)\leq C|\xi|$ for $|\xi|\geq 1$.

We have
\be\label{psi-8} 
p_1(\xi)={\rm p.f.}\int_{0}^\infty{{1}\over{z}}
e^{{\rm i}z\xi}dz=
{\rm p.f.}\int_{0}^1{{1}\over{z}}
e^{{\rm i}z\xi}dz
+\int_{1}^\infty{{1}\over{z}}
e^{{\rm i}z\xi}dz
\ee
and by formula (\ref{c-0-a})
\be\label{psi-9}
{\rm p.f.}\int_{0}^1{{1}\over{z}}
e^{{\rm i}z\xi}dz
=\int_{0}^1{{1}\over{z}}
(e^{{\rm i}z\xi}-1)dz+0.
\ee
As a result of Calderon-Zygmund formula (\cite{hsiao}, p. 358 or \cite{mikhlin}, pp. 245-249) for $n\in \N$ 
\be \label{c-z}
\int_0^1 {{(e^{-{\rm i}r \omega\cdot \eta}-1)}\over r}dr+
\int_1^{\infty}{{e^{-{\rm i}r\omega\cdot\eta}}\over r}dr
=-\ln|\omega\cdot\eta|-{\rm i}\sqrt{{{\pi}\over 2}}{{\omega\cdot\eta}\over{|\omega\cdot\eta|}}+\alpha_0
\ee
where $\omega,\eta$ are unit vectors in $\R^n$, and $\alpha_0$ is as given in (\ref{const}).
Hence by  (\ref{psi-8}), (\ref{psi-9}) and (\ref{c-z}) with $n=1$ and $\omega=-1,\ \eta=\xi$ the symbol $p_1(x)=-\ln(|\xi|)+\alpha_0+{\rm i}\sqrt{{{\pi}\over{2}}}{\rm sign}(\xi)$ is as claimed.

\end{proof}

Let $E>0$. Performing the change of variables $z'=z-(c(E)-E)$
the symbol $p_2(E,\xi)$ can be written as
\[
p_2(E,\xi)=
\int_{c(E)-E}^{\infty}{{1}\over{z}}
e^{{\rm i}z\xi}dz=
e^{{\rm i}(c(E)-E)\xi}\int_{0}^{\infty}{{1}\over{z+c(E)-E}}
e^{{\rm i}z\xi}dz.
\]
The symbol $p(E,\xi)$ is
\be\label{m-4}
p(E,\xi)=-\ln(|\xi|)+\alpha_0+{\rm i}\beta_0{\rm sign}(\xi)- p_2(E,\xi).
\ee

As a conclusion we get
\be\label{K111}
(K_{1,1}\psi)(x,\omega,E)=(K_{1,1,1}\psi)(x,\omega,E)+(K_{1,1,0}\psi)(x,\omega,E)
\ee
where $K_{1,1,0}$ is a Schur partial integral operator and where
\bea\label{H1}
&
(K_{1,1,1}\psi)(x,\omega,E)
\approx
(\widetilde K_{1,1,1}\psi)(x,\omega,E)
\nonumber\\
&
=
2\pi\ \hat\sigma_1(x,E,E)\ {\rm p.f.}\int_{E}^{c(E)}{1\over{E'-E}}\psi(x,\omega,E') dE'\nonumber\\
&
=
2\pi\ \hat\sigma_1(x,E,E)(P(E,D)\psi)(x,\omega,E)\nonumber\\
&
=
2\pi\ \hat\sigma_1(x,E,E)(P_1(D)\psi)(x,\omega,E)
-2\pi\ \hat\sigma_1(x,E,E)(P_2(E,D)\psi)(x,\omega,E)
\eea
where  $P_1(D)$ and $P_2(E,D)$ are the  pseudo-differential operators with the above given symbols $p_1(\xi)$ and $p_2(E,\xi)$.
We remark that the (principal) operator $P_1(D)$ is
a {\it translation invariant}  pseudo-differential operator
(that is, its symbol depends only on the transform variable) and $P_2(E,D)$ is an $L^2(G\times S\times I)$-bounded pseudo-differential operator.

\subsection{Approximative Second-order Hadamard Singular Integral Operator}\label{shsio}

Consider the second-order partial hyper-singular integral operator $K_{2,1}$ given by (recall (\ref{def-K_{j}-a}))
\bea\label{psi-13}
&
(K_{2,1}\psi)(x,\omega,E)
\nonumber\\
&
=
{\rm p.f.}\int_{E}^{c(E)}\int_0^{2\pi}\hat\sigma_2(x,E',E){1\over{(E'-E)^2}}\psi(x,\gamma(E',E,\omega)(s),E') ds dE'
\eea
where $\gamma(E',E,\omega)(s)$ is as above.

Using the Taylor's formula (\ref{taylor}) we expand
\bea\label{psi-15}
&
\hat\sigma_2(x,E',E)=\hat\sigma_2(x,E,E)+
{\p {\hat\sigma_2}E}(x,E,E)(E'-E)\nonumber\\
&
+
{2\over{2!}}\big(\int_0^1(1-t){{\partial^2 \hat\sigma_2}\over{\partial E'^2}}(x,E+t(E'-E),E)dt\big)(E'-E)^2
\eea
which gives
\bea\label{psi-16}
&
( K_{2,1}\psi)(x,\omega,E)
=
{\rm p.f.}\int_{E}^{c(E)}\int_0^{2\pi}{{\hat\sigma_2(x,E,E)}\over{(E'-E)^2}}\psi(x,\gamma(E',E),\omega)(s),E') dsdE'\nonumber\\
&
+
{\rm p.f.}\int_{E}^{c(E)}\int_0^{2\pi}{{{\p {\hat\sigma_2}{E'}}(x,E,E)}\over{(E'-E)}}\psi(x,\gamma(E',E),\omega)(s),E') dsdE'
\nonumber\\
&
+
\int_{E}^{c(E)}\int_0^{2\pi}\big(\int_0^1(1-t){\q {\hat\sigma_2}{E'}}(x,E+t(E'-E),E)dt\big)
\psi(x,\gamma(E',E,\omega)(s),E') dsdE'.
\eea

Denote
\begin{multline*}
( K_{2,1,0}\psi)(x,\omega,E) \\
:=
\int_{E}^{c(E)}\int_0^{2\pi}
\big(\int_0^1(1-t){\q {\hat\sigma_2}{E'}}(x,E+t(E'-E),E)dt\big)
\psi(x,\gamma(E',E,\omega)(s),E') dsdE',
\end{multline*}
\[
( K_{2,1,1}\psi)(x,\omega,E):= {\p {\hat\sigma_2}{E'}}(x,E,E)
{\rm p.f.}\int_{E}^{c(E)}{1\over{E'-E}}\int_0^{2\pi}\psi(x,\gamma(E',E,\omega)(s),E') dsdE'
\]
and
\[
(K_{2,1,2}\psi)(x,\omega,E):= \hat\sigma_2(x,E,E)
{\rm p.f.}\int_{E}^{c(E)}{{1}\over{(E'-E)^2}}\int_0^{2\pi}\psi(x,\gamma(E',E,\omega)(s),E') dsdE'.
\]

Since 
\bea
&
( K_{2,1,0}\psi)(x,\omega,E)\nonumber\\
&
= 
\int_{I'}\int_0^{2\pi}\chi_{\R_+}(E'-E)\chi_{\R_+}(c(E)-E')\nonumber\\
&
\cdot
\big(\int_0^1(1-t){\q {\hat\sigma_2}{E'}}(x,E+t(E'-E),E)dt\big)
\psi(x,\gamma(E',E,\omega)(s),E') dsdE'
\eea
we shall find in section \ref{rco} that under relevant assumptions $ K_{2,1,0}$ is a partial Schur  integral operator
\be 
( K_{2,1,0}\psi)(x,\omega,E):=
\int_{I'}\int_0^{2\pi}\hat\sigma_{2,0}(x,E',E)
\psi(x,\gamma(E',E,\omega)(s),E') dsdE'
\ee
where 
\begin{multline*}
\hat\sigma_{2,0}(x,E',E):=
\chi_{\R_+}(E'-E)\chi_{\R_+}(c(E)-E')
\\
\cdot
\big(\int_0^1(1-t){\q {\hat\sigma_2}{E'}}(x,E+t(E'-E),E)dt\big).
\end{multline*}

Again {\it we assume that} $\mu(E',E)\approx 1$ on the interval $[E,c(E)]$. Then $\gamma(E',E,\omega)(s)\approx \omega$ and so
by the previous section, 
we get formally an approximation  $\widetilde K_{2,1,1}$ for
the operator $K_{2,1,1}$,
\be \label{66}
 K_{2,1,1}\psi\approx
\widetilde K_{2,1,1}\psi:=
2\pi\ {\p {\hat\sigma_2}{E'}}(x,E,E)P(E,D)\psi
\ee
where
$P(E,D)$ is the pseudo-differential operator given in \eqref{eq:op_P}
(see also Theorem \ref{psi-th1}).

Therefore our analysis reduces to treating the operator $ K_{2,1,2}$.
Using the approximation $\gamma(E',E,\omega)(s)\approx \omega$ we get
\begin{multline}
( K_{2,1,2}\psi)(x,\omega,E)\approx
(\widetilde K_{2,1,2}\psi)(x,\omega,E) \\
:=2\pi\ 
\hat\sigma_2(x,E,E)\ {\rm p.f.}\int_{E}^{c(E)}{1\over{(E'-E)^2}}\psi(x,\omega,E') dE'. \label{psi-14}
\end{multline}
Consider the operator
\be
(Q\psi)(x,\omega,E):={\rm p.f.}\int_{E}^{c(E)}{1\over{(E'-E)^2}}\psi(x,\omega,E') dE'.
\ee 
We decompose (again  the integral ${\rm p.f.}\int_0^\infty$ is interpreted as
${\rm p.f.}\int_0^\infty={\rm p.f.}\int_0^1+\int_1^\infty$)
\bea\label{m-5}
&
{\rm p.f.}\int_{E}^{c(E)}{1\over{(E'-E)^2}}\psi(x,\omega,E') dE'
\nonumber\\
={}&
{\rm p.f.}\int_{E}^{\infty}{1\over{(E'-E)^2}}\psi(x,\omega,E') dE'
-
\int_{c(E)}^{\infty}{1\over{(E'-E)^2}}\psi(x,\omega,E') dE'\nonumber\\
={}&
(Q_1\psi)(x,\omega,E)-(Q_2\psi)(x,\omega,E)
\eea
where
\[
(Q_1\psi)(x,\omega,E):={}& {\rm p.f.}\int_{E}^{\infty}{1\over{(E'-E)^2}}\psi(x,\omega,E') dE', \\
(Q_2\psi)(x,\omega,E):={}& \int_{c(E)}^{\infty}{1\over{(E'-E)^2}}\psi(x,\omega,E') dE'.
\]
By Lemma \ref{hadale},
\bea
&
(Q_1\psi)(x,\omega,E) 
={\rm p.f.}\int_{E}^{\infty}{1\over{(E'-E)^2}}\psi(x,\omega,E') dE'
\nonumber\\
={}&
{\partial\over{\partial E}}\Big({\rm p.f.}\int_{E}^{\infty}{1\over{(E'-E)}}\psi(x,\omega,E') dE'\Big)
+{\p \psi{E}}(x,\omega,E)\nonumber\\
={}&
{\partial\over{\partial E}}\Big((P_1(D)\psi)(x,\omega,E) \Big)
+{\p \psi{E}}(x,\omega,E)
\eea
where $P_1(D)$ is as given in \eqref{eq:op_P_1}.
Hence $Q_1$ is a (translation invariant) pseudo-differential operator
$Q_1(D)={\partial\over{\partial E}}\circ P_1(D)+{\partial\over{\partial E}}$ with the symbol 
$q_1(\xi):=-{\rm i}\xi\ p_1(\xi)-{\rm i}\xi$. 
Furthermore,
\bea
&
(Q_2\psi)(x,\omega,E)=
\int_{c(E)}^{\infty}{1\over{(E'-E)^2}}\psi(x,\omega,E') dE'
\nonumber\\
={}&
(2\pi)^{-1/2}\int_{\R}q_2(E,\xi)(\mc F_p\psi)(x,\omega,\xi)e^{{\rm i}\xi E}d\xi=(Q_2(E,D)\psi)(x,\omega,E)
\eea
where
\be\label{m-6}
q_2(E,\xi)=e^{{\rm i}(c(E)-E)\xi}\int_0^\infty {1\over{(z+c(E)-E)^2}}e^{{\rm i}z\xi}dz.
\ee
That is why $Q=Q_1-Q_2$ is a pseudo-differential operator $Q(E,D)=Q_1(D)-Q_2(E,D)$ with the symbol
\be\label{psi-18}
q(E,\xi)=-{\rm i}\xi\ p_1(\xi)-{\rm i}\xi-q_2(E,\xi).
\ee

As a conclusion we get
\bea\label{H2}
&
(K_{2,1}\psi)(x,\omega,E)\approx  (\widetilde K_{2,1}\psi)(x,\omega,E)
\nonumber\\
={}&
(\widetilde K_{2,1,2}\psi)(x,\omega,E)+ (\widetilde K_{2,1,1}\psi)(x,\omega,E)+ ( K_{2,1,0}\psi)(x,\omega,E)\nonumber\\
={}&
2\pi\ \hat\sigma_2(x,E,E)(Q(E,D)\psi)(x,\omega,E)
+2\pi\ {\p {\hat\sigma_2}{E}}(x,E,E)(P(E,D)\psi)(x,\omega,E) \nonumber\\
&
+( K_{2,1,0}\psi)(x,\omega,E)\nonumber\\
={}&
2\pi 
\ \hat\sigma_2(x,E,E)({\partial\over{\partial E}}\circ P_1(D)\psi)(x,\omega,E)\nonumber\\
&
+2\pi\ \hat\sigma_2(x,E,E){\p {\psi}E}(x,\omega,E)
+2\pi\ {\p {\hat\sigma_2}{E'}}(x,E,E)(P_1(D)\psi)(x,\omega,E)
\nonumber\\[2mm]
&
-2\pi\ \hat\sigma_2(x,E,E)(Q_2(E,D)\psi)(x,\omega,E)\nonumber\\
&
-2\pi\ {\p {\hat\sigma_2}{E'}}(x,E,E)(P_2(E,D)\psi)(x,\omega,E)
+( K_{2,1,0}\psi)(x,\omega,E)
\eea
where we used that $P(E,D)\psi=P_1(D)\psi-P_2(E,D)\psi$.
Operator $K_{2,1,0}$ is a partial Schur integral operator. 
The pseudo-differential operators $P_2(E,D)$ and $Q_2(E,D)$ have zero-order.

\subsection{Joint Approximative Collision and Transport Operators}\label{a-psi}

\subsubsection{Approximative Pseudo-differential Operators}

Putting together (\ref{de}), (\ref{H1}) and (\ref{H2}) we get the following approximation for the collision operator governing M\o ller interaction
\bea\label{a-psi4-a}
&
K\psi
=K_{2,1}\psi+K_{1,1}\psi+K_{2,0}\psi+K_{1,0}\psi+K_0\psi\nonumber\\
\approx{}&
\widetilde K\psi:= 2\pi\ \hat\sigma_2(x,E,E)Q(E,D)\psi\nonumber\\
&
+ 
2\pi\ \big(\hat\sigma_1(x,E,E)+{\p {\hat\sigma_2}E}(x,E,E)\big)P(E,D)\psi
+K_r\psi
\eea
where $K_r$ is a Schur partial integral operator,
or more explicitly
\bea\label{a-psi4}
\widetilde K\psi={}&
2\pi\ 
\hat\sigma_2(x,E,E){\p {(P_1(D)\psi)}{E}}+2\pi\ 
\hat\sigma_2(x,E,E){\p \psi{E}}\nonumber\\
&
+2\pi\ \big({\p {\hat\sigma_2}{E'}}(x,E,E)+\hat\sigma_1(x,E,E)\big)P_1(D)\psi
-2\pi\ \hat\sigma_2(x,E,E)Q_2(E,D)\psi\nonumber\\
&
-2\pi\
\big({\p {\hat\sigma_2}{E'}}(x,E,E)+\hat\sigma_1(x,E,E)\big)P_2(E,D)\psi
+K_r\psi
\eea
The  operator $K_r$ is of the form 
\be\label{a-psi5}
(K_r\psi)(x,\omega,E)=
\int_{I'}\hat{\sigma}^3(x,E',E)\int_{0}^{2\pi}\psi(x,\gamma(E',E,\omega)(s),E')ds dE'
.
\ee
We call the operator $K_r$ {\it restricted collision operator}.

The resulting approximative transport operator is
\bea\label{a-psi4-r}
T\psi\approx \widetilde T\psi:={}&
-2\pi\ 
\hat\sigma_2(x,E,E){\p {(P_1(D)\psi)}{E}}-2\pi\ 
\hat\sigma_2(x,E,E){\p \psi{E}}\nonumber\\
&
-2\pi\ \big({\p {\hat\sigma_2}{E'}}(x,E,E)+\hat\sigma_1(x,E,E)\big)P_1(D)\psi
-2\pi\ \hat\sigma_2(x,E,E)Q_2(E,D)\psi\nonumber\\
&
-2\pi\
\big({\p {\hat\sigma_2}{E'}}(x,E,E)+\hat\sigma_1(x,E,E)\big)P_2(E,D)\psi
+\omega\cdot\nabla_x\psi+\Sigma\psi-K_r\psi.
\eea

\subsubsection{Approximative Partial Differential Operators}

The above pseudo-differential operators $P(E,D)$ and $Q(E,D)$ can  be formally further approximated by partial differential operators. By Taylor's formula
for $E'\approx E$ we have the approximations
\[
\psi(x,\omega,E')=\psi(x,\omega,E)+\int_0^1{\p {\psi}E}(x,\omega,E+t(E'-E))dt (E'-E)\approx \psi(x,\omega,E)
\]
and
\[
\psi(x,\omega,E')={}&
\psi(x,\omega,E)+{\p {\psi}E}(x,\omega,E)(E'-E)\nonumber\\
&
+{1\over 2}\int_0^1(1-t){\q {\psi}E}(x,\omega,E+t(E'-E))dt (E'-E)^2\\
\approx{}& \psi(x,\omega,E)+{\p {\psi}E}(x,\omega,E)(E'-E)
.
\]
Hence 
\be\label{P-app}
(P(E,D)\psi)(x,\omega,E)\approx {\rm p.f.}\int_E^{c(E)}{1\over{E'-E}}\psi(x,\omega,E) dE'=
\ln(c(E)-E)\psi(x,\omega,E),
\ee
\bea\label{Q-app}
(Q(E,D)\psi)(x,\omega,E)
\approx {}&
{\rm p.f.}\int_E^{c(E)}{1\over{(E'-E)^2}}\Big(\psi(x,\omega,E)+{\p {\psi}E}(x,\omega,E)(E'-E) \Big) dE'\nonumber\\
={}&
-{1\over{c(E)-E}}\psi(x,\omega,E)+\ln(c(E)-E){\p {\psi}E}(x,\omega,E).
\eea

Using (\ref{a-psi4-a}), (\ref{P-app}) and (\ref{Q-app}) we obtain formally a partial differential approximation for the transport operator
\bea\label{T-app}
&
T\psi\approx \widetilde T\psi:=
2\pi\ \hat\sigma_2(x,E,E)
\ln(c(E)-E){\p {\psi}E}\nonumber\\
&
+2\pi\ \Big(\hat\sigma_1(x,E,E)\ln(c(E)-E)
+{\p {\hat\sigma_2}{E'}}(x,E,E)\ln(c(E)-E) -  \hat\sigma_2(x,E,E) {1\over{c(E)-E}}\Big)\psi\nonumber\\
&
+\omega\cdot\nabla_x\psi+\Sigma\psi
-K_r\psi.
\eea
In the context of M\o ller scattering
the use of these approximations requires more careful
study but in some cases they can be justified; see e.g. Example \ref{con-brem} below.
Actually, the next section suggests that better angular approximation (instead of (\ref{ess-app}) is needed for the above treated M\o ller scattering.
This leads to transport equations which contain second order  partial derivatives.

\subsection{On Error Estimates}\label{err}

We leave  the systematic error analysis of the above approximations. Nevertheless, we notice that  the analysis can be founded on the following 
principles.   Again
we assume without further mention that $\psi$ and cross sections are regular enough.
In the following paragraphs A.-C. we bring up some basic principles for error analysis.

A.
Consider the collision operator (\ref{mol-b}) given by
\bea\label{mol-b-1-a}
&
(K\psi)(x,\omega,E)=
{\rm p.f.}\int_E^{E_m}{1\over{(E'-E)^{1+\kappa_2}}}
\int_{S'}\sigma_2(x,\omega',\omega,E',E)\psi(x,\omega',E')d\omega'dE'
\nonumber\\
&
+
{\rm p.f.}\int_E^{E_m}{1\over{(E'-E)^{\kappa_1}}}
\int_{S'}\sigma_1(x,\omega',\omega,E',E)\psi(x,\omega',E')d\omega'dE'
+(K_0\psi)(x,\omega,E)\nonumber\\
=:{}&
(K_2\psi)(x,\omega,E)+(K_1\psi)(x,\omega,E)+(K_0\psi)(x,\omega,E)
\eea
where $K_0$ is a Schur partial integral operator.

We decompose the the finite part integrals as follows
\bea\label{mol-b-2-a}
&
(K_2\psi)(x,\omega,E)=
{\rm p.f.}\int_E^{c(E)}{1\over{(E'-E)^{1+\kappa_2}}}
\int_{S'}\sigma_2(x,\omega',\omega,E',E)\psi(x,\omega',E')d\omega'dE'
\nonumber\\
&
+
{\rm p.f.}\int_{c(E)}^{E_m}{1\over{(E'-E)^{1+\kappa_2}}}
\int_{S'}\sigma_2(x,\omega',\omega,E',E)\psi(x,\omega',E')d\omega'dE'\nonumber\\
=:{}&
(K_{2,1}\psi)(x,\omega,E)+(K_{2,0}\psi)(x,\omega,E)
\eea
and similarly
\bea\label{mol-b-2-b}
&
(K_1\psi)(x,\omega,E)=
{\rm p.f.}\int_E^{c(E)}{1\over{(E'-E)^{\kappa_1}}}
\int_{S'}\sigma_1(x,\omega',\omega,E',E)\psi(x,\omega',E')d\omega'dE'
\nonumber\\
&
+
{\rm p.f.}\int_{c(E)}^{E_m}{1\over{(E'-E)^{\kappa_1}}}
\int_{S'}\sigma_1(x,\omega',\omega,E',E)\psi(x,\omega',E')d\omega'dE'\nonumber\\
=:{}&
(K_{1,1}\psi)(x,\omega,E)+(K_{1,0}\psi)(x,\omega,E).
\eea

As above we see that
the operators $K_{2,0}$ and $K_{1,0}$ are Schur partial integral operators and they need no approximation. 
The problematic "forward peaked operators" $K_{2,1}$ and $K_{1,1}$ are approximated in an appropriate way,
for example as in sections \ref{fhsio} and \ref{shsio} above or as in Example \ref{con-brem} below. Let the approximations of $K_{2,1}$ and $K_{1,1}$ be $\widetilde K_{2,1}$ and $\widetilde K_{1,1}$,
respectively.

B.
Let $\psi$ be the solution of the exact problem
\be\label{pr-a}
T\psi:=-K_2\psi-
K_1\psi
+\omega\cdot\nabla_x\psi+\Sigma\psi-K_0\psi=f,\quad
\psi_{|\Gamma_-}=g
\ee
and let $\tilde\psi$
be the solution of the approximative problem with the same data
\be\label{pr-b}
\widetilde T\tilde\psi:=-\widetilde K_{2,1}\tilde\psi
-
\widetilde K_{1,1}\tilde \psi
+\omega\cdot\nabla_x\tilde \psi+\Sigma\tilde\psi-K_r\tilde\psi=f,\quad
\tilde\psi_{|\Gamma_-}=g
\ee
where $K_r:=K_0+K_{2,0}+K_{1,0}$.
Denote the error by $\delta\psi:=\psi-\tilde\psi$. Then
\be\label{e-3}
T\psi- \widetilde T\tilde\psi=0,\quad (\psi-\tilde\psi)_{|\Gamma_-}=0
\ee
which implies that
\be\label{e-4}
\widetilde T(\delta\psi)=( \widetilde T-T)\psi,\quad (\delta\psi)_{|\Gamma_-}=0.
\ee
The apriori estimates such as \eqref{csda40aaa-dd} below imply that
\be\label{e-9}
c'\n{\delta\psi}_{L^2(G\times S\times I)}\leq \n{( \widetilde T-T)\psi}_{L^2(G\times S\times I)}.
\ee

C.
According to Part A.,
\be\label{e-5}
( \widetilde T-T)\psi= -\widetilde K_{1,1}\psi+K_{1,1}\psi
 -\widetilde K_{2,1}\psi+K_{2,1}\psi 
\ee
and then by (\ref{e-9})
\be\label{e-5-a}
c'\n{\delta\psi}_{L^2(G\times S\times I)}\leq 
\n{K_{1,1}\psi- \widetilde K_{1,1}\psi}_{L^2(G\times S\times I)} +
\n{K_{2,1}\psi-\widetilde K_{2,1}\psi}_{L^2(G\times S\times I)}.
\ee
The basic idea is to show that for $j=1,2$,
\be\label{Kj-conv-a}
\n{K_{j,1}\psi-\widetilde K_{j,1}\psi}_{L^2(G\times S\times I)}\to 0
\ee
in some sense. By \eqref{e-5-a} this ensures that $\delta\psi\to 0$ in $L^2$,
i.e. solution to \eqref{pr-b} in an $L^2$-approximation of the solution to \eqref{pr-a}.

The following example illustrates these techniques.

\begin{example}\label{con-brem}
Consider the collision operator (\ref{mol-b}) with $\kappa_1=1,\ \kappa_2<1$ for which the item (1) of section \ref{r-psio-1} is valid that is,
\bea\label{mol-b-1}
&
(K\psi)(x,\omega,E)=
{\rm p.f.}\int_E^{E_m}{1\over{(E'-E)^{1+\kappa_2}}}
\int_{S'}\sigma_2(x,\omega',\omega,E',E)\psi(x,\omega',E')d\omega'dE'
\nonumber\\
&
+
{\rm p.f.}\int_E^{E_m}{1\over{E'-E}}
\int_{S'}\sigma_1(x,\omega',\omega,E',E)\psi(x,\omega',E')d\omega'dE'
+(K_0\psi)(x,\omega,E)\nonumber\\
=:{}&
(K_2\psi)(x,\omega,E)+(K_1\psi)(x,\omega,E)+(K_0\psi)(x,\omega,E)
\eea
where $K_0$ is a Schur partial integral operator
and where
$\sigma_j:G\times S^2\times I^2\to\R$ are measurable non-negative functions. We assume that $\psi$ and  $\sigma_j$ "regular enough", without specifying this assumption in detail.

Consider the operator $K_2$. We decompose the finite part integral as follows
\bea\label{mol-b-2}
&
(K_2\psi)(x,\omega,E)=
{\rm p.f.}\int_E^{c(E)}{1\over{(E'-E)^{1+\kappa_2}}}
\int_{S'}\sigma_2(x,\omega',\omega,E',E)\psi(x,\omega',E')d\omega'dE'
\nonumber\\
&
+
{\rm p.f.}\int_{c(E)}^{E_m}{1\over{(E'-E)^{1+\kappa_2}}}
\int_{S'}\sigma_2(x,\omega',\omega,E',E)\psi(x,\omega',E')d\omega'dE'\nonumber\\
=:{}&
(K_{2,1}\psi)(x,\omega,E)+(K_{2,0}\psi)(x,\omega,E).
\eea

As above, we find that $K_{2,0}$ is a Schur partial integral operator and it needs no approximation. For $K_{2,1}$ we assume that 
\be\label{mol-b-5}
|\omega'-\omega|\leq C|E'-E|,
\ee
when $E'\in [E,c(E)]$. Note that energy and scattering angle are in correlation and so (\ref{mol-b-5}) is in certain cases reasonable (see Remark \ref{con-mol} below). Assuming  that $E'\approx E$ we have by (\ref{mol-b-5}) $\omega'\approx\omega$ and so
\be\label{mol-b-3}
\psi(x,\omega',E')\approx \psi(x,\omega,E').
\ee
This suggests
\bea\label{mol-b-4}
&
(K_{2,1}\psi)(x,\omega,E)\approx
{\rm p.f.}\int_E^{c(E)}{1\over{(E'-E)^{1+\kappa_2}}}\Big(
\int_{S'}\sigma_2(x,\omega',\omega,E',E)d\omega'\Big)\psi(x,\omega,E')dE'
\nonumber\\
={}&
{\rm p.f.}\int_E^{c(E)}{1\over{(E'-E)^{1+\kappa_2}}}\ol \sigma_2(x,\omega,E',E)\psi(x,\omega,E')dE'=:(\widetilde K_{2,1}\psi)(x,\omega,E)
\eea
where
\[
\ol\sigma_2(x,\omega,E',E):=\int_{S'}\sigma_2(x,\omega',\omega,E',E)d\omega'.
\]

Using Taylor's formula, we further approximate
\bea\label{mol-b-6}
&
(\widetilde K_{2,1}\psi)(x,\omega,E)=
{\rm p.f.}\int_E^{c(E)}{1\over{(E'-E)^{1+\kappa_2}}}\ol\sigma_2(x,\omega,E',E)\psi(x,\omega,E')dE'
\nonumber\\
={}&
{\rm p.f.}\int_E^{c(E)}{1\over{(E'-E)^{1+\kappa_2}}}\ol\sigma_2(x,\omega,E',E)\Big(\psi(x,\omega,E)+{\p {\psi}E}(x,\omega,E)(E'-E)\nonumber\\
&
+\int_0^1(1-t){\q {\psi}E}(x,\omega,E+t(E'-E)) dt(E'-E)^2\Big)dE'
\nonumber\\
={}&
a_2(x,\omega,E)\psi(x,\omega,E)+b_2(x,\omega,E){\p {\psi}E}(x,\omega,E)
+(R\psi)(x,\omega,E) \\
\approx {}&
a_2(x,\omega,E)\psi(x,\omega,E)+b_2(x,\omega,E){\p {\psi}E}(x,\omega,E)
\eea
where
\[
a_2(x,\omega,E):={}&
{\rm p.f.}\int_E^{c(E)}{1\over{(E'-E)^{1+\kappa_2}}}\ol\sigma_2(x,\omega,E',E)dE', \\
b_2(x,\omega,E):={}&
{\rm p.f.}\int_E^{c(E)}{1\over{(E'-E)^{\kappa_2}}}\ol\sigma_2(x,\omega,E',E)dE'.
\]

Similarly we get for the operator $K_1$
\be\label{}
( K_{1,1}\psi)(x,\omega,E)\approx
(\widetilde K_{1,1}\psi)(x,\omega,E)
=a_1(x,\omega,E)\psi(x,\omega,E)+b_1(x,\omega,E){\p {\psi}E}(x,\omega,E)
\ee
where
\[
a_1(x,\omega,E):={\rm p.f.}\int_E^{c(E)}{1\over{(E'-E)^{\kappa_1}}}\ol\sigma_1(x,\omega,E',E)dE',
\]
\[
b_1(x,\omega,E):={\rm p.f.}\int_E^{c(E)}(E'-E)^{1-\kappa_2}\ol\sigma_1(x,\omega,E',E)dE'.
\]

For example, with the choice $c(E)=\alpha E$, $\alpha>1$,
one can verify that for $\kappa_2\leq 1$ the residual
\begin{multline}
(R\psi)(x,\omega,E)  \\
=\int_E^{c(E)}(E'-E)^{1-\kappa_2}\ol\sigma_2(x,\omega,E',E)
\Big(\int_0^1(1-t){\q {\psi}E}(x,\omega,E+t(E'-E)) dt\Big)dE'
\label{1-22}
\end{multline}
convergences to zero in $L^2(G\times S\times I)$ when $\alpha\to 0$,
and similarly for the residual corresponding to the operator $K_{1,1}$. 
Hence the convergences in \eqref{Kj-conv-a} do occur in this case.
\end{example}

\begin{remark}\label{con-mol}
For the M\o ller scattering (section \ref{sec:moller}) the approximation (\ref{mol-b-3}) or in other words
\be\label{app-1}
\psi(x,\gamma(E',E,\omega)(s),E')\approx \psi(x,\omega,E')
\ee
seems insufficient.
In addition, (\ref{mol-b-5}) is not valid but we have
the following.
Using as before a choice of $R(\omega)\in \SO(3)$ s.t. $R(\omega)e_3=\omega$,
and letting
\[
\gamma_0(E',E)(s):=\big(\sqrt{1-\mu^2}\cos(s),\sqrt{1-\mu^2}\sin(s),\mu\big)=R(\omega)^{-1}\gamma(E',E,\omega)(s),
\]
where $\mu=\mu(E',E)$ is given in \eqref{eq:mu},
one obtains
\bea\label{ex-21}
&
\n{\gamma(E',E,\omega)(s)-\omega}^2
=\n{R(\omega)(\gamma_0(E',E)(s)-e_3)}^2
=
\n{\gamma_0(E',E)(s)-e_3}^2  \nonumber\\[2mm]
={}&
2(1-\mu(E',E))
=2{{1-\mu(E',E)^2}\over{1+\mu(E',E)}}
=
4{{|E'-E|}\over{(1+\mu(E',E))E'(E+2)}}\nonumber\\[2mm]
\leq {}& 
4{{|E'-E|}\over{E(E+2)}}
\leq 
4{{|E'-E|}\over{E_0(E_0+2)}}=C_1|E'-E|,
\eea
where $C_1:=4{{1}\over{E_0(E_0+2)}}$,
and we used that $E'\geq E\geq E_0$.
A convergent approximation  requires instead of  (\ref{app-1}) the usage of higher order Taylor's expansions as outlined below.

We can improve the approximation (\ref{app-1}) by using  the Taylor's  first and second order polynomials with respect to angle. Recall from section \ref{taylor-S} that
the first and second order Taylor's expansions of $\psi$ with respect to $\omega$ around $\omega$ are
\[
\psi(x,\gamma(E',E,\omega)(s),E')\approx \psi(x,\omega,E')
+\la(\nabla_{\omega}\psi)(x,\omega,E'),v\ra 
\]
and 
\be\label{ess-app-2}
\psi(x,\gamma(E',E,\omega)(s),E')\approx \psi(x,\omega,E')
+\la(\nabla_{\omega}\psi)(x,\omega,E'),v\ra 
+(({\rm Hes}_\omega \psi)(x,\omega,E'))(v,v)
\ee
where $v=v(E',E,\omega,s)\in T_\omega(S)$ and satisfies
\[
\n{v}\leq C\n{\gamma(E',E,\omega)(s)-\omega}.
\]

For instance, inserting the approximation (\ref{ess-app-2})
we get the following pseudo-differential approximation for the collision operator
\bea\label{es-10}
&
(\widetilde K\psi)(x,\omega,E)=(\widetilde K_{2,1}\psi)(x,\omega,E)+(\widetilde K_{1,1}\psi)(x,\omega,E)+( K_{r}\psi)(x,\omega,E)\nonumber\\
={}&
\sum_{|\beta|=2}
{\rm p.f.}\int_{E}^{c(E)}\hat\sigma_1(x,E',E)a_\beta(\omega,E',E)
{1\over{E'-E}}(\partial_{\omega}^\beta\psi)(x,\omega,E')  dE'
\nonumber\\
&
+
\sum_{|\alpha|=1} {\rm p.f.}\int_{E}^{c(E)}\hat\sigma_1(x,E',E)b_\alpha(\omega,E',E)
{1\over{E'-E}}(\partial_{\omega}^\alpha\psi)(x,\omega,E') dE'\nonumber\\ 
&
+
2\pi\ {\rm p.f.}\int_{E}^{c(E)}\hat\sigma_1(x,E',E){1\over{E'-E}}\psi(x,\omega,E') dE'
\nonumber\\
&
+
\sum_{|\beta|=2} {\rm p.f.}\int_{E}^{c(E)}\hat\sigma_2(x,E',E)a_\beta(\omega,E',E)
{1\over{(E'-E)^2}}(\partial_{\omega}^\alpha\psi)(x,\omega,E') dE'
\nonumber\\
&
+
\sum_{|\alpha|=1} {\rm p.f.}\int_{E}^{c(E)}\hat\sigma_2(x,E',E)b_\alpha(\omega,E',E)
{1\over{(E'-E)^2}}(\partial_{\omega}^\alpha\psi)(x,\omega,E') dE'\nonumber\\ 
&
+
2\pi\ {\rm p.f.}\int_{E}^{c(E)}\hat\sigma_2(x,E',E){1\over{(E'-E)^2}}\psi(x,\omega,E') dE'\nonumber\\
&
+
( K_{r}\psi)(x,\omega,E)
\eea
where 
\begin{multline}
\sum_{|\beta|=2}a_\beta(\omega,E',E)
(\partial_{\omega}^\beta\psi)(x,\omega,E')\nonumber\\
:={}
{1\over 2}\sum_{i=1}^2\sum_{j=1}^2\partial_i\partial_j (\psi\circ  H_\omega)(x,0,E')\int_0^{2\pi}\xi_i(E',E,\omega,s)\xi_j(E',E,\omega,s) ds,\nonumber
\end{multline}
\[
\sum_{|\alpha|=1}b_\alpha(\omega,E',E)
(\partial_{\omega}^\alpha\psi)(x,\omega,E')
:=\sum_{j=1}^2\partial_j (\psi\circ  H_\omega)(x,0,E')\int_0^{2\pi}\xi_j(E',E,\omega,s) ds.
\]
Above $\xi(E',E,\omega,s):=J(v(E',E,\omega,s))$,
and $J$ is defined in \eqref{eq:J_iso}.

Combining the approximations given in section \ref{a-psi} we obtain an approximation of the collision operator which contains only partial differential and Schur partial integral operators. 
Concerning M\o ller scattering, a preliminary error analysis  suggests that  
the modelling it using CSDA equations may be somewhat problematic in the sense that
the solution of the approximative problem does not necessarily converge (without further information) to the solution of the exact problem.  
Here the convergence is understood in the sense that the cut-off energy for 
electrons is $c(E)=\alpha E$ and we let $\alpha\to 1^+$
(which implies that $\mu(E',E)\to 1$ for $E\in I$ and $E'\in [E,c(E)]$).
Instead of the simple approximation \eqref{app-1} on $[E,c(E)]$,
 we must use the second order Taylor's polynomials for the angular approximation.
When considering partial differential approximations the resulting approximative transport operator is a partial integro-differential operator of the form
\bea\label{tr-op-gene}
&
T\psi=a(x,E){\p {\psi}E}
+\sum_{|\alpha|\leq 2}b_\alpha(x,\omega,E)\partial_\omega^\alpha\psi
+\omega\cdot\nabla\psi+\Sigma(x,\omega,E)\psi-K_r\psi
\eea
where $\partial_\omega^\alpha:=\partial_\omega^{(\alpha_1,\alpha_2)}$. As usual, here 
\[
\partial_\omega^{(\alpha_1,\alpha_2)}:=\partial_{\tilde\omega_1}^{\alpha_1}\partial_{\tilde\omega_2}^{\alpha_2},
\]
where $\partial_{\tilde\omega_j}:={\partial\over{\partial\tilde\omega_j}},\ j=1,2$
in a given local coordinate system of $S$,
with $b_\alpha(x,\omega,E)$ depending on those coordinates.
It might also be reasonable to use the second order Taylor's expansions with respect to energy variable, leading to the second order partial integro-differential operator
\begin{multline}
T\psi=
a_1(x,E){\q {\psi}E}+a_2(x,E){\p {\psi}E}
+\sum_{|\alpha|\leq 2}b_\alpha(x,\omega,E)\partial^\alpha_\omega\psi \\
+
\omega\cdot\nabla\psi+\Sigma(x,\omega,E)\psi-K_r\psi. \label{tr-op}
\end{multline}
Note that in the above operator the term $\sum_{|\alpha|\leq 2}b_\alpha(x,\omega,E)\partial_\omega^\alpha\psi$ corresponds to the second order 
partial differential term with respect to angular variables appearing in 
Fokker-Plank equation (see \cite{morel81}).
In the case where $I=[E_0,E_m]$ relevant inflow boundary and initial conditions for the operator (\ref{tr-op}) are
\begin{gather}
\psi_{|\Gamma_-}=g, \label{b-cond} \\
\psi(\cdot,\cdot,E_m)={\p {\psi}E}(\cdot,\cdot,E_m)=0. \label{in-cond}
\end{gather}

We finally remark that  for less singular interactions (such as Bremsstrahlung-like collisions) the approximation (\ref{mol-b-3}) is sufficient to guarantee the  convergences (\ref{Kj-conv-a})
and the resulting transport operator can successfully approximated by the operator of the form
\be\label{tr-op-brem}
T\psi= a(x,E){\p {\psi}E}
+\omega\cdot\nabla\psi+\Sigma(x,\omega,E)\psi-K_r\psi.
\ee
In addition, the approximation (\ref{mol-b-4}) is sufficient to guarantee the convergences (\ref{Kj-conv-a}) for certain values of parameters $\kappa_j$  when the collision operator is of the form (\ref{mol-b}). The resulting approximative transport operator is of the form (\ref{tr-op-brem}).
We omit all details of these cases.
 
\end{remark}

The above approximations are reasonable only if $\psi$ is regular enough, say $\psi$ belongs to the { mixed-norm Sobolev-Slobodevskij space}
$ H^{(s_1,s_2,s_3)}(G\times S\times I^\circ)$  for an appropriate index $s=(s_1,s_2,s_3)$
Related regularity analysis of solutions of BTE is beyond the scope of this paper.

\begin{remark}
Instead of the (single particle) operator
\[
T\psi:=\omega\cdot\nabla_x\psi+\Sigma\psi-K\psi
\]
defined for $\psi\in L^2(G\times S\times I)$,
one could have postulated that the linear Boltzmann transport operator 
modelling the radiation transport have one of the two alternative
forms $T_s$ or $T_s'$
described below.

To fix the ideas, we assume in this remark that
$K$ has the (classical) form given in \eqref{eq:inelastic_K_no_constraints},
namely
\[
(K\psi)(x,\omega,E)=\int_{S\times I}\sigma(x,\omega',\omega,E',E)\psi(x,\omega',E')d\omega' dE',
\]
and that $\sigma$ satisfies the conditions given in Theorem \ref{bound-K1-a}.
After appropriate minor changes,
the discussion below applies also to the case
where collision operator $K$ has the form \eqref{mollerk}
or \eqref{elastick}
with $\hat{\sigma}$ or $\sigma$ satisfying the conditions
set in Theorems \ref{bound-K3-a} or \ref{el-k-b},
respectively.

Write
\[
\Sigma_s(x,\omega,E):={}& \int_{S\times I}\sigma(x,\omega,\omega',E,E')d\omega' dE', \\
\Sigma_s'(x,\omega,E):={}& \int_{S\times I}\sigma(x,\omega',\omega,E',E)d\omega' dE',
\]
so that $\Sigma_a:=\Sigma-\Sigma_s$ is the \emph{absorption} coefficient,
and define analogously \linebreak ${\Sigma_a':=\Sigma-\Sigma_s'}$.

In \cite{morel81} transport operator  $\Gamma_{\rm B}$
(see Eq. (3a) in \cite{morel81}) analysed
takes the form
\[
T_s\psi:=\omega\cdot\nabla_x\psi+\Sigma_a\psi-K_s\psi
\]
where, 
\[
(K_s\psi)(x,\omega,E)
:={}&
(K\psi-\Sigma_s\psi)(x,\omega,E) \\
={}&
\int_{S\times I}\big(\sigma(x,\omega',\omega,E',E)\psi(x,\omega',E')-\sigma(x,\omega,\omega',E,E')\psi(x,\omega,E)\big)d\omega' dE'
\]
On the other hand, in \cite{maire06}
the relevant (stationary version of) transport operator (see $\mc{T}$ at the beginning of section 2.2.1 in \cite{maire06})
would be 
\[
T_s'\psi:=\omega\cdot\nabla_x\psi+\Sigma'_a\psi-K_s'\psi,
\]
where
\[
(K_s'\psi)(x,\omega,E)
:={}&
(K\psi-\Sigma'_s\psi)(x,\omega,E) \\
={}&
\int_{S\times I}\sigma(x,\omega',\omega,E',E)\big(\psi(x,\omega',E')-\psi(x,\omega,E)\big)d\omega' dE'
\]
Furthermore, the (stationary version of) the transport
operator in Eq. (36) of \cite{bal09} happens to be of both types $T_s$ and $T_s'$ (since $\Sigma_s=\Sigma'$ in that case).

It should be pointed out that collision operator $K_s$ satisfies
\[
\int_{S\times I} (K_s\psi)(x,\omega,E) d\omega dE=0,
\]
if, for example, $\sigma$ is non-negative and satisfies \eqref{ass5-a} (with $\sigma^1=\sigma)$, and \linebreak ${\psi\in L^2(G\times S\times I)}$.
It is evident
that under the stated conditions on $\sigma$,
the operators $T$, $T_s$ and $T_s'$  have the same domains of definition.

\end{remark}

\sectionspace
\section{On Restricted Collision Operators}\label{rco}

In this section we everywhere assume that 
\be
\Sigma\in L^\infty(G\times S\times I),\quad \Sigma\geq 0.
\ee
We shall find that the restricted collision $K_r$ operator is a sum of partial Schur integral operators. For the basic theory of Schur integral operators we refer to \cite{halmos}.  Our aim is to show boundedness of $K_r$ and accretivity (coercitivity) of $\Sigma-K_r$ in $L^2(G\times S\times I)$. We begin with the consideration of the restricted collision operators related to the 
M\o ller and Compton scattering.

\subsection{Restricted Collision Operator Related to Inelastic Scattering Processes Obeying Kinematic Constraints} \label{rco-m}

When a scattering process is inelastic and all the particles
participating to it
are bound by kinematic constraints (conservation of energy-momentum),
then in some instances the corresponding scattering operator $K$ takes the form
\be\label{mollerk}
(K\psi)(x,\omega,E)=
\int_{I'}\int_0^{2\pi}\hat\sigma(x,E',E)\psi(x,\gamma(E',E,\omega)(s),E') ds dE'.
\ee

As we have already mentioned above the restricted collision operator for M\o ller phenomenon is of the form \eqref{mollerk} (once we absorb $\chi(E',E)$ into $\hat\sigma(x,E',E)$).
Also, the collision operator for the Compton-Klein-Nishina scattering is also of this form, see Example \ref{compton} below.

\begin{remark}\label{re:kinematics}
Respecting kinematic constraints in the scattering process
typically result in a $\delta$-function factor
in a cross section
such as $\delta(\omega'\cdot\omega-\mu(E',E))$
in M\o ller scattering in Eq. \eqref{eq:ol_s_K_22_j}
(similarly for Compton-Klein-Nishina scattering).

Indeed, in M\o ller scattering (normalizing the electron mass and the speed of light to $=1$),
the conservation of energy-momentum (in laboratory system) reads
in $\R^{4}$
\begin{align}\label{eq:m_m_kinematics}
(E'+1,{p'}\omega')+(1,0)=(E+1,p\omega)+(E_s+1,p_s\omega_s)
\end{align}
where triples $(E',{p'},\omega')$, $(E,p,\omega)$
and $(E_s,p_s,\omega_s)$
are (kinetic energy, momentum magnitude, momentum direction (=unit vector))
for incoming, primary outgoing and secondary outgoing electrons
participating into the scattering process
(the incoming secondary electron is at rest,
reflected by the term $(1,0)$ on the left hand side).
Secondary electron is selected by the condition $E_s\leq E$.
Moreover, by the relativistic energy-momentum relation
$(E'+1)^2={p'}^2+1$ and similarly for $(E,p)$ and $(E_s,p_s)$.

It easily follows from \eqref{eq:m_m_kinematics}
that
\[
E_s=E'-E,
\quad
\omega'\cdot\omega=\frac{{p'}^2+p^2-p_s^2}{2p'p}.
\]
Substituting expressions for $p'$, $p$, $p_s$
in terms of $E'$, $E$, $E_s$, respectively,
as given by energy-momentum relation,
and finally substituting $E_s=E'-E$
results in
\[
\omega'\cdot\omega=\frac{E'(E'+2)+E(E+2)-(E'-E)(E'-E+2)}{2\sqrt{E(E+2)E'(E'+2)}}.
\]
Straightforward simplification of the right and side
shows that it equals $\mu(E',E)$ of Eq. \eqref{eq:mu}.

Note that these computations are, of course, 
valid for any massive particle,
(not only for electrons)
colliding with an identical one.

It is the fulfilment of the above equation
$\omega'\cdot\omega=\mu(E',E)$ that is ensured by 
the composite delta function $\delta(\omega'\cdot\omega-\mu(E',E))$
associated with the M\o ller and Compton scattering cross section (see \eqref{eq:ol_s_K_22_j} and Example \ref{compton} below),
eventually leading to collision operator of the form \eqref{mollerk} (see section \ref{sec:moller}).

Similar remark holds true e.g. for Compton-Klein-Nishina cross section, although $\mu(E',E)$ is different since
the relevant kinematic equation is
different from \eqref{eq:m_m_kinematics}
due to the fact that in that case,
one particle is massless (photon) while the other is massive (electron).
\end{remark}

\begin{example}\label{ex:moller}
{\it Electron-electron scattering - M\o ller}.
We denote the corresponding differential cross section by $\sigma_{22}(x,\omega',\omega,E',E)$. It has a decomposition  (\cite{duclous}, \cite{lorence}, \cite{boman}, \cite{hensel})
\be\label{i-e-e1} 
\sigma_{22}(x,\omega',\omega,E',E)
=\sigma^p_{22}(x,\omega',\omega,E',E)+\sigma^s_{22}(x,\omega',\omega,E',E).
\ee
where $\sigma^p_{22}(x,\omega',\omega,E',E)$ is corresponding to the (new) primary electrons
and  $\sigma^s_{22}(x,\omega',\omega,E',E)$ is corresponding to the secondary electrons.
In this scattering process the spins have been averaged out,
and the two electrons completely lose their identity.
Therefore, categorizing the electrons as "primary" and "secondary" is simply done
by assigning the electron exiting the scattering event with the highest energy to be the primary one.
The scattering cross section for primary electron $\sigma^p_{22}(x,\omega',\omega,E',E)$ has an expression
\begin{multline*}
\sigma^p_{22}(x,\omega',\omega,E',E)
=
\sigma_{0}(x){{(E'+1)^2}\over{E'(E'+2)}}\Big({1\over{E^2}}
+{1\over{(E'-E)^2}}+{1\over{(E'+1)^2}} \\
-{{2E'+1}\over{(E'+1)^2E(E'-E)}}\Big)\chi_{22,p}(E',E)\delta(\omega'\cdot\omega-\mu_{22,p}(E',E)),
\end{multline*}
where $\sigma_{0}(x)$ depends on the background material, and
\[
& \mu_{22,p}(E',E):=\sqrt{{E(E'+2)}\over{E'(E+2)}} \\
& \chi_{22,p}(E',E):=\chi_{\R_+}(E-E_0)\chi_{\R_+}(E-{E'\over 2})\chi_{\R_+}(E'-E),
\]
while the cross section for the secondary electron $\sigma^s_{22}(x,\omega',\omega,E',E)$ is
\begin{multline*}
\sigma^s_{22}(x,\omega',\omega,E',E)
=
\sigma_0(x){{(E'+1)^2}\over{E'(E'+2)}}\Big({1\over{E^2}}
+{1\over{(E'-E)^2}}+{1\over{(E'+1)^2}}\\
-{{2E'+1}\over{(E'+1)^2E(E'-E)}}\Big)\chi_{22,s}(E',E)\delta(\omega'\cdot\omega-\mu_{22,s}(E',E)),
\end{multline*}
where 
\[
& \mu_{22,s}(E',E):=\mu_{22,p}(E',E'-E), \\
& \chi_{22,s}(E',E):=\chi_{\R_+}({{E'}\over 2}-E)\chi_{\R_+}(E-E_0).
\]
Since $\sigma^s_{22}(x,\omega',\omega,E',E)=0$ for $E'\leq 2E$ the 
singularities at $E'=E$ do not cause any problems for the secondary electrons.
\end{example}

\begin{example}\label{compton}
{\it Photon-electron scattering - Compton-Klein-Nishina}.
This scattering process describes the collision of a photon with (free) electron,
and the corresponding photon$\to$ photon (i.e. $1\to 1$) scattering cross-section is given by
(for its derivation, see \cite[Chap. III \S 16]{heitler47}, \cite[Sec. 8.7]{weinberg}; see also \cite[Section VII]{lorence})
\bea\label{eq:compton}
\sigma_{11}(x,\omega',\omega,E',E)
={}&
\hat{\sigma}_{11}(x,E',E)\chi_{11}(E,E')
\delta(\omega'\cdot\omega-\mu_{11}(E',E)),
\eea
where
\[
\hat{\sigma}_{11}(x,E',E):={}&\sigma_0(x)\Big({{1}\over{E'}}\Big)^2\Big({{E'}\over{E}}+{{E}\over{E'}}-1+\mu_{11}(E',E)^2\Big)
\nonumber\\
\chi_{11}(E',E):={}&\chi_{\R_+}(E-E_0)\chi_{\R_+}\big(E-\frac{E'}{1+2E'}\big)\chi_{\R_+}(E'-E) \\
\mu_{11}(E',E):={}&1+{{1}\over{E'}}-{{1}\over{E}}.
\]
Here $(\omega', E')$ and $(\omega, E)$ are, respectively, the (direction, energy) of the incident and the scattered (outgoing) photons.
If the scattering angle is written as $\theta_{11}$, then
$\omega\cdot\omega'=\cos(\theta_{11})=\mu_{11}(E',E)$,
and this condition is enforced by the delta-distribution term in $\sigma_{11}$.

We point out that if one defines (for a presentation more or less in this way, see \cite[Appendix A.1.]{hensel}),
\[
\ol{\sigma}_{11}(x,\omega',\omega,E',E)
=&\sigma_0(x)\ol{\sigma}_{11}'(E',E)\delta(E-E'\hspace{0.5mm}\ol{P}(\omega,\omega',E')),
\]
where
\[
& P(E',E):=\frac{1}{1+E'(1-\mu_{11}(E',E))}=\frac{E}{E'}, \\
& \ol{P}(\omega,\omega',E'):=\frac{1}{1+E'(1-\omega'\cdot\omega)}, \\
& \ol{\sigma}_{11}'(E',E)
:=P(E',E)^2\Big(P(E',E)+\frac{1}{P(E',E)}-1+\mu_{11}(E',E)^2\Big),
\]
Then $\sigma_0(x)\ol{\sigma}_{11}'(E',E)=E^2\hat{\sigma}_{11}(x,E',E)$,
using which one can further show that the collision operator, say $\ol{K}_{11}$, corresponding to $\ol{\sigma}_{11}$
is equal to the collision operator $K_{11}$ defined by $\sigma_{11}$, i.e. $K_{11}=\ol{K}_{11}$.
\end{example}

\begin{lemma}\label{le-m:0}
Let $-1<t<1$, and for any $\omega\in S$ let
\[
\Gamma_{t}^\omega=\{\omega'\in S\ |\ \omega\cdot\omega'=t\}.
\]
Then for any $f\in L^1(S)$ one has
\begin{align}\label{eq:le-m:0}
\int_S \int_{\Gamma_t^\omega} f(\omega')d\ell(\omega') d\omega
=2\pi\sqrt{1-t^2}\int_S f(\omega)d\omega,
\end{align}
where $\int_{\Gamma_t^\omega} f(\omega')d\ell(\omega')$ is the path integral of $f$ along $\Gamma_t^\omega$
(it is well-defined for almost every $\omega\in S$).
\end{lemma}

Notice that $\ell(\Gamma^\omega_t)=2\pi\sqrt{1-t^2}$, where
$\ell(\Gamma^\omega_t)$ is the measure (length) of $\Gamma^\omega_t$.

\begin{proof}
We give here a direct elementary proof of this claim.
For an alternative proof, which relies on properties of the Haar measure on $\SO(3)$, see Appendix \ref{ap:le-m:0}.

Let $-1<t_0<t<1$.
Then since for every $\omega\in S$,
the usual measure $d\mu_S(\omega)=d\omega$ on $S$
disintegrates into a family of measures
$(1-s^2)^{-1/2} d\ell_s\otimes ds$, $-1<s<1$,
where $\ell_s$ is the path length measure along the curve
$\Gamma^\omega_s$. Hence
we have
\[
{}&
\int_{t_0}^{t} (1-s^2)^{-1/2}\int_S \int_{\Gamma_s^\omega} f(\omega')d\ell(\omega')d\omega ds \\
= {}&
\int_S \int_{t_0}^{t} \int_{\Gamma_s^\omega} f(\omega')(1-s^2)^{-1/2} d\ell(\omega')ds d\omega \\
= {}&
\int_S \int_{S_{\omega,[t_0,t]}} f(\omega')d\omega' d\omega,
\]
where
\[
S_{\omega,[t_0,t]}
=\{\omega'\in S\ |\ \omega'\cdot\omega\in [t_0,t]\}.
\]
Noticing that
\[
{}& \{(\omega,\omega')\in S\times S'\ |\  \omega'\in S_{\omega,[t_0,t]}\} \\
= {}&
\{(\omega,\omega')\in S\times S'\ |\ \omega'\cdot\omega\in [t_0,t]\} \\
= {}&
\{(\omega',\omega)\in S'\times S\ |\  \omega\in S_{\omega',[t_0,t]}\},
\]
it follows from Fubini's theorem that we can exchange
integrals w.r.t $\omega$ and $\omega'$ as follows
\[
\int_S \int_{S_{\omega,[t_0,t]}} f(\omega')d\omega' d\omega
=
\int_{S'} \int_{S_{\omega',[t_0,t]}} f(\omega')d\omega d\omega'.
\]
Thus
\begin{multline*}
\int_{t_0}^{t} (1-s^2)^{-1/2}\int_S \int_{\Gamma_s^\omega} f(\omega')d\ell(\omega')d\omega ds \\
=
\int_{S'} \int_{S_{\omega',[t_0,t]}} f(\omega')d\omega d\omega'
=
\int_{S'} \mu_S(S_{\omega',[t_0,t]}) f(\omega') d\omega'.
\end{multline*}

Next we notice that
$\mu_S(S_{\omega',[t_0,t]})$ is independent of $\omega'\in S$,
and show that its value is nothing more than (the surface area of the spherical zone) $2\pi (t-t_0)$.
Indeed,
\begin{multline*}
\mu_S(S_{\omega',[t_0,t]})
=
\int_{t_0}^{t} \int_{\Gamma_t^{\omega'}} (1-t^2)^{-1/2}d\ell(\omega') dt \\
=
\int_{t_0}^{t} (1-t^2)^{-1/2}\, 2\pi (1-t^2)^{1/2} dt
=2\pi (t-t_0).
\end{multline*}

We have therefore shown that
\[
\int_{t_0}^{t} (1-s^2)^{-1/2}\int_S \int_{\Gamma_s^\omega} f(\omega')d\ell(\omega')d\omega ds
=
2\pi (t-t_0) \int_{S'} f(\omega') d\omega'.
\]
Finally, taking derivative w.r.t $t$ we arrive at the
stated formula
\[
(1-t^2)^{-1/2}\int_S \int_{\Gamma_t^\omega} f(\omega')d\ell(\omega')d\omega
=
2\pi \int_S f(\omega') d\omega'.
\]
\end{proof}

Our next lemma will be an immediate corollary of the lemma just stated.
As in section as in section \ref{sec:moller},
let $\gamma=\gamma(E',E,\omega):[0,2\pi]\to S$
be a parametrization of the curve
\[
\Gamma(E',E,\omega)=\{\omega'\in S\ |\ \omega'\cdot\omega-\mu(E',E)=0\}.
\]
with constant speed $\n{\gamma'(s)}=\sqrt{1-\mu(E',E)^2}$.

\begin{lemma}\label{le-m}
For $\psi\in L^1(G\times S\times I)$ and a.e. $(x,E,E')\in G\times I\times I$ we have
\[
\int_{S}\int_0^{2\pi} \psi(x,\gamma(E',E,\omega)(s),E') ds d\omega
=
2\pi \int_{S} \psi(x,\omega,E') d\omega.
\]
\end{lemma}

\begin{proof}
For a fixed $(x,E',E)$,
we write $f(\omega''):=\psi(x,\omega'',E')$,
which belongs in $L^1(S)$ for a.e. choice of $(x,E',E)$.
Write also $\gamma=\gamma(E',E,\omega)$.
Then,
\[
{}&
\int_0^{2\pi} \psi(x,\gamma(E',E,\omega)(s),E') ds
=\int_0^{2\pi} f(\gamma(s)) ds \\
={}&\frac{1}{\sqrt{1-\mu(E',E)^2}}\int_0^{2\pi} f(\gamma(s))\n{\gamma'(s)} ds
=\frac{1}{\sqrt{1-\mu(E',E)^2}}\int_{\Gamma^\omega_{\mu(E',E)}} f(\omega')d\ell(\omega'),
\]
Integrating over $S$ and using Lemma \ref{le-m:0} 
we thus obtain
\[
\int_S \int_0^{2\pi} \psi(x,\gamma(E',E,\omega)(s),E') ds d\omega
={}&
\frac{1}{\sqrt{1-\mu(E',E)^2}}\int_S \int_{\Gamma^\omega_{\mu(E',E)}} f(\omega')d\ell(\omega') d\omega \\
={}&
2\pi\int_S f(\omega)d\omega=2\pi\int_S \psi(x,\omega,E') d\omega,
\]
which is what we wanted to show.
\end{proof}

Next we shall verify the boundedness of $K$.

\begin{theorem}\label{bound-K3-a}
Suppose that $\hat\sigma:G\times I^2\to\R$ is a non-negative measurable function such that for a.e. $(x,\omega,E)\in G\times S\times I$,
\bea\label{k3-c}
&
\int_{I'}\hat\sigma(x,E',E) dE'\leq M_1
\nonumber\\
&
\int_{I'}\hat\sigma(x,E,E') dE'\leq M_2,
\eea
for some constants $0\leq M_1,M_2<\infty$.
Then $K: L^2(G\times S\times I)\to L^2(G\times S\times I)$ is bounded and
\be
\n{K}\leq 2\pi\sqrt{M_1M_2}. 
\ee
\end{theorem}

\begin{proof}

We have by the Cauchy-Schwarz inequality (at second step)
\bea\label{k3-c-2}
&
|(K\psi)(x,\omega,E)|
\nonumber\\
\leq {}&
\int_{I'}\int_0^{2\pi}|\hat\sigma(x,E',E)|^{1/2+1/2}|\psi(x,\gamma(E',E,\omega)(s),E')| ds dE' 
\nonumber\\
\leq {}&
\Big[\int_{I'}\int_0^{2\pi}|\hat\sigma(x,E',E)|ds dE'\Big]^{1/2}
\nonumber\\
{}&
\cdot
\Big[\int_{I'}\int_0^{2\pi}\hat\sigma(x,E',E)|\psi(x,\gamma(E',E,\omega)(s),E')|^2   ds dE' \Big]^{1/2}
\nonumber\\
\leq {}&
(2\pi M_1)^{1/2}
\Big[\int_{I'}\int_0^{2\pi}\hat\sigma(x,E',E)|\psi(x,\gamma(E',E,\omega)(s),E')|^2   ds dE' \Big]^{1/2}
\eea
and then by Lemma \ref{le-m}
\bea\label{bound-c-3}
&
\int_G\int_{S}\int_I |(K\psi)(x,\omega,E)|^2 dE d\omega dx
\nonumber\\
\leq {}&
(2\pi)^2 M_1
\int_G\int_{I}\int_{I'}\int_{S''}|\hat\sigma(x,E',E)||\psi(x,\omega'',E')|^2
d\omega'' dE' dE dx 
\nonumber\\
\leq {}&
(2\pi)^2M_1 M_2 \int_G\int_{I'}\int_{S''}|\psi(x,\omega'',E')|^2
d\omega'' dE' dx 
\eea
which implies the assertion.
\end{proof}

An accretivity result for $\Sigma-K$ reads as follows.

\begin{theorem}\label{accre-K3-a}
Suppose that $\hat\sigma:G\times I^2\to\R$ is a non-negative measurable
function and $c\in\R$ is a constant such that
for a.e. $(x,\omega,E)\in G\times S\times I$,
\bea\label{k3-n3-disas}
&
\Sigma(x,\omega,E)-2\pi \int_{I'}\hat\sigma(x,E',E) dE'\geq c
\nonumber\\
&
\Sigma(x,\omega,E)-2\pi\int_{I'}\hat\sigma(x,E,E') dE'\geq c.
\eea
Then for all $\psi\in L^2(G\times S\times I)$
\be
\la (\Sigma-K)\psi,\psi\ra_{L^2(G\times S\times I)}\geq c\n{\psi}_{L^2(G\times S\times I)}^2.
\ee
\end{theorem}

\begin{proof}

We have
\be\label{k3-n3-1}
\la K\psi,\psi\ra_{L^2(G\times S\times I)}
=\int_G\Big(\int_{S}\int_I(K\psi)(x,\omega,E)\psi(x,\omega,E)  d\omega dE\Big) dx
\ee
where by the Cauchy-Schwarz's inequality (2nd step) and by Lemma
\ref{le-m} (3rd step)
\bea\label{k3-n3-2}
&
\int_{S}\int_I(K\psi)(x,\omega,E)\psi(x,\omega,E) dE d\omega
\nonumber\\
\leq {}& 
\int_{S}\int_I\int_{I'}\int_0^{2\pi}|\hat\sigma(x,E',E)|^{1/2+1/2}|\psi(x,\gamma(E',E,\omega)(s),E')|\ |\psi(x,\omega,E)| ds dE' dE d\omega
\nonumber\\
\leq {}&
\Big(\int_{S}\int_I\int_{I'}\int_0^{2\pi}\hat\sigma(x,E',E)|\psi(x,\gamma(E',E,\omega)(s),E')|^2   ds dE' dE d\omega \Big)^{1/2}
\nonumber\\
&
\cdot
\Big(\int_{S}\int_I\int_{I'}\int_0^{2\pi}\hat\sigma(x,E',E)|\psi(x,\omega,E)|^2   ds dE' dE d\omega \Big)^{1/2}
\nonumber\\
= {}&
\Big(2\pi\int_I\int_{I'}\int_{S''}\hat\sigma(x,E',E)|\psi(x,\omega'',E')|^2     d\omega'' dE' dE\Big)^{1/2}
\nonumber\\
&
\cdot
\Big(2\pi\int_{S}\int_I\int_{I'}\hat\sigma(x,E',E)|\psi(x,\omega,E)|^2 dE' dE d\omega\Big)^{1/2}
\nonumber\\
= {}&
\Big(\int_{I'}\int_{S''}\Big(2\pi\int_I \hat\sigma(x,E',E) dE\Big)|\psi(x,\omega'',E')|^2 d\omega'' dE' \Big)^{1/2}
\nonumber\\
&
\cdot
\Big(\int_{S}\int_I \Big(2\pi\int_{I'}\hat\sigma(x,E',E)dE'\Big) |\psi(x,\omega,E)|^2 dE d\omega\Big)^{1/2}
\nonumber\\
\leq {}&
\int_{S}\int_{I} (\Sigma(x,\omega,E)-c)|\psi(x,\omega,E)|^2 dE d\omega
\eea
where
in the last step we applied the assumptions (\ref{k3-n3-disas}).
Hence by integrating (\ref{k3-n3-2}) over $G$ we obtain
\[
\la (\Sigma-K)\psi,\psi\ra_{L^2(G\times S\times I)}\geq c\n{\psi}_{L^2(G\times S\times I)}^2
\]
which completes the proof. 

\end{proof}

\subsection{Collision Operator Related to Elastic Scattering}\label{el-k}

In elastic scattering $K$ is typically of the form
\be\label{elastick}
(K\psi)(x,\omega,E)=
\int_{S'}\sigma(x,\omega',\omega,E)\psi(x,\omega',E) d\omega'.
\ee

One example of elastic scattering processes is \emph{screened Rutherford scattering} 
of electrons whose cross section is (\cite{boman}, p. 34 )
\be \label{e-el}
\sigma(x,\omega',\omega,E)=
\sigma_0(x){{(E+1)^2}\over{E^2(E+2)^2}}{1\over{(1-\omega'\cdot\omega+2\eta(x,E))^2}}
\ee
where $\sigma_0\in L^\infty(G)$ and $\eta\in L^\infty(G\times I)$ for which $\eta \geq \eta_0>1$.

Another relevant example is elastic \emph{Rayleigh scattering}
of photons whose cross section is
\be \label{eq:rayleigh}
\sigma(x,\omega',\omega,E)
=\sigma_0(x)(1+(\omega'\cdot\omega)^2).
\ee

For further examples of elastic scattering processes, e.g. Mott scattering with Molière screening (for electrons),
and their cross sections, see \cite{lorence}.

We show that under relevant criteria $K$ is bounded.

\begin{theorem}\label{el-k-b}
Suppose that $\sigma:G\times S^2\times I\to\R$ is a non-negative measurable function such that for a.e. $(x,\omega,E)\in G\times S\times I$
\bea\label{bound-b}
&
\int_{S'}\sigma(x,\omega',\omega,E)d\omega' \leq M_1
\nonumber\\
&
\int_{S'}\sigma(x,\omega,\omega',E)d\omega'\leq M_2,
\eea
for some constants $0\leq M_1,M_2<\infty$.
Then $K: L^2(G\times S\times I)\to L^2(G\times S\times I)$ is bounded and
\be
\n{K}\leq \sqrt{M_1M_2}. 
\ee
\end{theorem}

\begin{proof}
We have by the Cauchy-Schwarz's inequality
\[
&
|(K\psi)(x,\omega,E)|
\nonumber\\
\leq {}&
\int_{S'}|\sigma(x,\omega',\omega,E)|^{1/2+1/2}|\psi(x,\omega',E)| d\omega' 
\nonumber\\
\leq {}& 
\Big(\int_{S'}|\sigma(x,\omega',\omega,E)|
d\omega' \Big)^{1/2}
\Big(\int_{S'}|\sigma(x,\omega',\omega,E)||\psi(x,\omega',E)|^2
d\omega' \Big)^{1/2}\nonumber\\
\leq {}&
M_1^{1/2} \Big(\int_{S'}|\sigma(x,\omega',\omega,E)||\psi(x,\omega',E)|^2
d\omega' \Big)^{1/2}
\]
and so 
\[
&
\int_G\int_{S}\int_I|(K\psi)(x,\omega,E)|^2 dx d\omega dE
\nonumber\\
\leq {}&
M_1
\int_G\int_{S}\int_I\int_{S'}|\sigma(x,\omega',\omega,E)||\psi(x,\omega',E)|^2
d\omega' dE d\omega dx
\nonumber\\
\leq {}&
M_1 M_2 \int_G\int_{I}\int_{S'}|\psi(x,\omega',E)|^2
d\omega' dE dx 
\]
which implies the assertion.

\end{proof}

Next result gives a natural condition ensuring accretivity of $\Sigma-K$.

\begin{theorem}\label{accre-K2-a}
Suppose that $\sigma:G\times S^2\times I\to\R$ is a non-negative measurable function and $c\in\R$ is a constant such that 
for a.e. $(x,\omega,E)\in G\times S\times I$
\bea\label{k3-n3-disas-a}
&
\Sigma(x,\omega,E)-\int_{S'}\sigma(x,\omega',\omega,E)d\omega'\geq c
\nonumber\\
&
\Sigma(x,\omega,E)-\int_{S'}\sigma(x,\omega,\omega',E)d\omega'\geq c 
\eea
Then for all $\psi\in L^2(G\times S\times I)$
\be
\la (\Sigma-K)\psi,\psi\ra_{L^2(G\times S\times I)}\geq c\n{\psi}_{L^2(G\times S\times I)}^2.
\ee
\end{theorem}

\begin{proof}

We have
\be\label{k3-n3-1-a}
\la K\psi,\psi\ra_{L^2(G\times S\times I)}
=\int_G\int_I\Big(\int_{S}(K\psi)(x,\omega,E)\psi(x,\omega,E)  d\omega \Big)dE dx.
\ee
By the Cauchy-Schwarz inequality (at second step)
\bea\label{k3-n3-2-a}
&
\int_{S}(K\psi)(x,\omega,E)\psi(x,\omega,E) d\omega
\nonumber\\
\leq {}&
\int_{S}\int_{S'}|\sigma(x,\omega',\omega,E)|^{1/2+1/2}|\psi(x,\omega',E)|\ |\psi(x,\omega,E)| d\omega' d\omega
\nonumber\\
\leq {}&
\Big(\int_{S}\int_{S'}|\sigma(x,\omega',\omega,E)||\psi(x,\omega',E)|^2
d\omega' d\omega\Big)^{1/2}
\nonumber\\
&
\cdot 
\Big(\int_{S}\int_{S'}|\sigma(x,\omega',\omega,E)||\psi(x,\omega,E)|^2
d\omega' d\omega\Big)^{1/2}
\nonumber\\
\leq {}& 
\Big(\int_{S'}(\Sigma(x,\omega',E)-c)|\psi(x,\omega',E)|^2
d\omega' \Big)^{1/2}
\nonumber\\
&
\cdot 
\Big(\int_{S}(\Sigma(x,\omega,E)-c)|\psi(x,\omega,E)|^2
 d\omega\Big)^{1/2}
\nonumber\\
= {}&
\int_{S}(\Sigma(x,\omega,E)-c)|\psi(x,\omega,E)|^2  d\omega
\eea
where
in the last step we applied the assumptions (\ref{k3-n3-disas-a}).
Hence by integrating (\ref{k3-n3-2-a}) over $G\times I$ we obtain
\[
\la (\Sigma-K)\psi,\psi\ra_{L^2(G\times S\times I)}\geq c\n{\psi}_{L^2(G\times S\times I)}^2
\]
which completes the proof. 

\end{proof}

\subsection{Restricted Collision Operator Related to Inelastic Scattering Processes without Kinematic Constraints}\label{subseq:el_2}

When a scattering process is inelastic,
but kinematic constraints (conservation of energy-momentum)
are not (fully) respected,
then the corresponding scattering operator $K$ may well take the form
\be\label{eq:inelastic_K_no_constraints}
(K\psi)(x,\omega,E)=
\int_{S'}\int_{I'}\sigma(x,\omega',\omega,E',E)\psi(x,\omega',E') d\omega'dE'.
\ee

One example of this type is the electron-nucleus Bremsstrahlung process (see \cite{boman,lorence,haug04}).
As explained in \cite[sec. 3.2]{haug04}, \cite[Chap. III, \S 17]{heitler47},
the outgoing electron and emitted photon
of this process are considered to have independent momenta
because of the relatively large mass of the nucleus with respect to that of the electron.
In this respect, the kinematic constrains are not (fully) imposed
when deriving the
electron-nucleus Bremsstrahlung cross section
such as that given in \cite[Eq. (3.87)]{haug04}, \cite[Eq. (13), p. 164]{heitler47}.

\begin{example}\label{ex:brems}
The cross section for electron-electron electron Bremsstrahlung scattering process
is an approximation of the form
\begin{align}\label{eq:brems}
\sigma(x,\omega',\omega,E',E)
=\hat{\sigma}(x)\sigma_k(E',E)\eta(\omega'\cdot\omega,E'),
\end{align}
where $(E',\omega')$ and $(E,\omega)$
are the (kinetic energy, direction)
of the incoming and outgoing electrons, respectively,
associated to this scattering process.
Moreover, $v(E)$ is the velocity of the outgoing electron,
which is given in terms of its (relativistic) kinetic energy $E$
as $v(E)^2=1-(E+1)^{-2}$.
We also assume that the angular factor
$\eta$ belongs to $L^\infty([-1,1]\times I)$
(see \eqref{mo-1}).

In Born approximation without screening, the so-called (3BN)
the factor $\sigma_k$ in the above cross section is given
by (\cite[Chap. III, \S 17, Eq. (16)]{heitler47}, \cite[Eq. (III.3)]{lorence})
\begin{multline} \label{eq:3BN}
\sigma_k(E',E_k)
=\frac{1}{E_k}\frac{p}{p'}\Big[\frac{4}{3}-2(E'+1)(E+1)\frac{p^2+{p'}^2}{p^2{p'}^2} \\
+\frac{\epsilon(E')(E+1)}{{p'}^3}+\frac{\epsilon(E)(E'+1)}{p^3}-\frac{\epsilon(E')\epsilon(E)}{p'p} \\
+L(E',E)\Big(\frac{8}{3}\frac{(E'+1)(E+1)}{p'p}+\frac{E_k^2}{{p'}^3 p^3}({(E'+1)}^2(E+1)^2+{p'}^2p^2) \\
+\frac{E_k}{2p'p}\big(\frac{(E'+1)(E+1)+{p'}^2}{{p'}^3}\epsilon(E')
-\frac{(E'+1)(E+1)+p^2}{p^3}\epsilon(E)+\frac{2E_k(E'+1)(E+1)}{{p'}^2p^2}\big)\Big)\Big],
\end{multline}
where the photon energy $E_k=E_k(E',E)$ is
given by the (approximate) conservation of energy
$E_k=E'-E$ in this process, see \cite[Chap. 3]{haug04}, \cite[Chap. III, \S 17]{heitler47},
$p'=p'(E')$, $p=p(E)$ are the norms of 3-momenta of the
incoming and outgoing electrons, respectively,
so that ${p'}^2=(E'+1)^2-1=E'(E'+2)$
and $p^2=E(E+2)$.
The remaining quantities are
\[
\epsilon(\ol{E}):=2\ln\big(\ol{E}+1+\ol{p}\big)
=2\ln\big(\ol{E}+1+\sqrt{\ol{E}(\ol{E}+2)}\big)
\]
where $\ol{p}^2:=\ol{E}(\ol{E}+2)$, $\ol{p}\geq 0$,
and
\[
L(E',E):=2\ln\Big(\frac{(E'+1)(E+1)+p'p-1}{E'-E}\Big).
\]

In the expression for $\epsilon(\ol{E})$,
the term inside the logarithm $\ol{E}+1+\sqrt{\ol{E}(\ol{E}+2)}$
is clearly bounded away from zero for $\ol{E}\geq 0$.
Similarly, the magnitudes of momenta $p=p(E)$ and $p'=p'(E')$
are also bounded away from zero for $E,E'\in I=[E_0,E_m]$
(recall that $E_0>0$ by assumption).
Finally, the numerator inside the logarithm in the expression for $L(E',E)$
is also bounded away from zero for $E',E\in I$ since
\[
(E'+1)(E+1)+p'p-1=E'E+E+E'+\sqrt{E'(E'+2)E(E+2)}.
\]

That having been discussed,
one sees after substituting $E_k=E'-E$
into \eqref{eq:3BN} 
that with such $\sigma_k$,
the cross section \eqref{eq:brems}
can be written into the form given in \eqref{coll-2}.
This provides physics based justification for the terms in \eqref{coll-2}.
\end{example}

A criterion for the boundedness of $K$ is analogous to the above collision operators.

\begin{theorem}\label{bound-K1-a}
Suppose that $\sigma:G\times S^2\times I^2\to\R$ is a non-negative measurable function such that for a.e. $(x,\omega,E)\in G\times S\times I$
\bea\label{bound-a}
&
\int_{S'\times I'}\sigma(x,\omega',\omega,E',E)d\omega' dE'\leq M_1
\nonumber\\
&
\int_{S'\times I'}\sigma(x,\omega,\omega',E,E')d\omega'dE'\leq M_2,
\eea
for some constants $0\leq M_1,M_2<\infty$.
Then $K: L^2(G\times S\times I)\to L^2(G\times S\times I)$ is bounded and
\be
\n{K}\leq \sqrt{M_1M_2}. 
\ee
\end{theorem}

\begin{proof}
We have
\[
&
|(K\psi)(x,\omega,E)|
\nonumber\\
\leq {}&
\int_{S'}\int_{I'}|\sigma(x,\omega',\omega,E',E)|^{1/2+1/2}|\psi(x,\omega',E')|  dE' d\omega'
\nonumber\\
\leq {}& 
\Big(\int_{S'}\int_{I'}|\sigma(x,\omega',\omega,E',E)|
dE' d\omega'\Big)^{1/2}
\nonumber\\
&
\cdot 
\Big(\int_{S'}\int_{I'}|\sigma(x,\omega',\omega,E',E)||\psi(x,\omega',E')|^2
dE' d\omega'\Big)^{1/2}\nonumber\\
\leq {}& 
M_1^{1/2} \Big(\int_{S'}\int_{I'}|\sigma(x,\omega',\omega,E',E)||\psi(x,\omega',E')|^2
dE' d\omega'\Big)^{1/2}
\]
and so 
\[
&
\int_G\int_{S}\int_I|(K\psi)(x,\omega,E)|^2 dE d\omega dx
\nonumber\\
\leq {}&
M_1
\int_G\int_{S}\int_I\int_{S'}\int_{I'}|\sigma(x,\omega',\omega,E',E)||\psi(x,\omega',E')|^2
dE' d\omega' dE d\omega dx
\nonumber\\
\leq {}& 
M_1 M_2 \int_G\int_{S'}\int_{I'}|\psi(x,\omega',E')|^2
dE' d\omega' dx 
\]
which implies the assertion.

\end{proof}

Moreover, a criterion for the accretivity of $\Sigma-K$ goes as above.

\begin{theorem}\label{accre-K1-a}
Suppose that $\sigma:G\times S^2\times I^2\to\R$ is a non-negative measurable function and $c\in\R$ is a constant such that
for a.e. $(x,\omega,E)\in G\times S\times I$
\bea\label{k3-n3-disas-c}
&
\Sigma(x,\omega,E)-\int_{S'\times I'}\sigma(x,\omega',\omega,E',E)d\omega' dE'\geq c
\nonumber\\
&
\Sigma(x,\omega,E)-\int_{S'\times I'}\sigma(x,\omega,\omega',E,E')d\omega'dE'\geq c 
\eea
Then for all $\psi\in L^2(G\times S\times I)$
\be
\la (\Sigma-K)\psi,\psi\ra_{L^2(G\times S\times I)}\geq c\n{\psi}_{L^2(G\times S\times I)}^2.
\ee
\end{theorem}

\begin{proof}

We have
\be\label{k3-n3-1-c}
\la K\psi,\psi\ra_{L^2(G\times S\times I)}
=\int_G\Big(\int_{S}\int_I(K\psi)(x,\omega,E)\psi(x,\omega,E) dE d\omega\Big) dx
\ee
where by the Cauchy-Schwarz inequality (at second step)
\bea\label{k3-n3-2-c}
&
\int_{S}\int_I(K\psi)(x,\omega,E)\psi(x,\omega,E) dE d\omega
\nonumber\\
\leq {}&
\int_{S}\int_I\int_{S'}\int_{I'}|\sigma(x,\omega',\omega,E',E)|^{1/2+1/2}|\psi(x,\omega',E')|\ |\psi(x,\omega,E)| dE' d\omega' dE d\omega
\nonumber\\
\leq {}&
\Big(\int_{S}\int_I\int_{S'}\int_{I'}|\sigma(x,\omega',\omega,E',E)||\psi(x,\omega',E')|^2
dE' d\omega'dE d\omega\Big)^{1/2}
\nonumber\\
&
\cdot 
\Big(\int_{S}\int_I\int_{S'}\int_{I'}|\sigma(x,\omega',\omega,E',E)||\psi(x,\omega,E)|^2
dE' d\omega'dE d\omega\Big)^{1/2}
\nonumber\\
\leq {}&
\Big(\int_{S'}\int_{I'}(\Sigma(x,\omega',E')-c)|\psi(x,\omega',E')|^2
dE' d\omega'\Big)^{1/2}
\nonumber\\
&
\cdot 
\Big(\int_{S}\int_I(\Sigma(x,\omega,E)-c)|\psi(x,\omega,E)|^2
 dE d\omega\Big)^{1/2}
\nonumber\\
= {}& 
\int_{S}\int_I(\Sigma(x,\omega,E)-c)|\psi(x,\omega,E)|^2  dE d\omega
\eea
where
in the last step we applied the assumptions (\ref{k3-n3-disas-c}).
Hence by integrating (\ref{k3-n3-2-c}) over $G$ we obtain
\[
\la (\Sigma-K)\psi,\psi\ra_{L^2(G\times S\times I)}\geq c\n{\psi}_{L^2(G\times S\times I)}^2
\]
which completes the proof. 

\end{proof}

\subsection{Boundedness and Accretivity of the Sum of Collision Operators} \label{col-sum}

Assume that the  collision operator is  the sum
\be\label{esols1} 
K=K^1+K^2+K^3.
\ee
Here $K^1$ is of the form (see subsection \ref{subseq:el_2})
\[
(K^1\psi)(x,\omega,E)=\int_{S'\times I'}\sigma^1(x,\omega',\omega,E',E)\psi(x,\omega',E')d\omega' dE',  
\]
where $\sigma^1:G\times S^2\times I^2\to\R$ is a non-negative measurable function such that 
\bea\label{ass5-a}
&\int_{S'\times I'}\sigma^1(x,\omega',\omega,E',E)d\omega' dE'\leq M_1,\nonumber\\
&\int_{S'\times I'}\sigma^1(x,\omega,\omega',E,E')d\omega' dE'\leq M_2,
\eea
for a.e. $(x,\omega,E)\in G\times S\times I$.

The operator
$K^2$ is of the form (see subsection \ref{el-k})
\[
(K^2\psi)(x,\omega,E)=\int_{ S'}\sigma^2(x,\omega',\omega,E)\psi(x,\omega',E) d\omega',  
\]
where $\sigma^2:G\times S^2\times I\to\R$ is a non-negative  measurable function  such that
\bea\label{ass7}
&\int_{S'}\sigma^2(x,\omega',\omega,E)d\omega'\leq M_1,\nonumber\\
&\int_{S'}\sigma^2(x,\omega,\omega',E) d\omega'\leq M_2,
\eea
for a.e. $(x,\omega,E)\in G\times S\times I$.

Finally, $K^3$ is of the form (see subsection \ref{rco-m})
\[
(K^3\psi)(x,\omega,E)
=
\int_{I'}\int_{0}^{2\pi}
\hat\sigma^3(x,E',E)
\psi(x,\gamma(E',E,\omega)(s),E')ds dE'
\]
where
$\gamma=\gamma(E',E,\omega):[0,2\pi]\to S$
be a parametrization of the curve
\[
\Gamma(E',E,\omega)=\{\omega'\in S\ |\ \omega'\cdot\omega-\mu(E',E)=0\}.
\]
with constant speed $\n{\gamma'(s)}=\sqrt{1-\mu(E',E)^2}$,
and $\mu$ is a measurable function such that $-1\leq \mu(E',E)\leq 1$,
Moreover,
$\hat{\sigma}^3:G\times I^2\to\R$ is a non-negative measurable function such that
\bea\label{ass-8}
&\int_{I'}\hat{\sigma}^3(x,E',E)dE'\leq M_1, \nonumber\\
&\int_{I'}\hat{\sigma}^3(x,E,E')dE'\leq M_2,
\eea
for a.e. $(x,E)\in G\times I$. 
Furthermore, in \eqref{ass5-a}, \eqref{ass7} and \eqref{ass-8} it is assumed that $M_1,M_2$ are finite non-negative constants.
The following result in an immediate consequence of Theorems \ref{bound-K3-a}, \ref{el-k-b} and \ref{bound-K1-a}.

\begin{theorem}\label{esol-th1}
The sum $K=K^1+K^2+K^3$  is a bounded operator 
$L^2(G\times S\times I)\to L^2(G\times S\times I)$.
\end{theorem}

In order to render the operator $\Sigma-K$ accretive,
we shall assume that 
\begin{multline}
\Sigma(x,\omega,E)-\int_{S'\times I'}\sigma^1(x,\omega,\omega',E,E') d\omega' dE'
\\ \label{ass8-aa}
-\int_{S'}\sigma^2(x,\omega,\omega',E)  d\omega'
-2\pi\int_{I'}\hat{\sigma}^3(x,E,E')dE'
\geq c,
\end{multline}
and
\begin{multline}
\Sigma(x,\omega,E)-\int_{S'\times I'}\sigma^1(x,\omega',\omega,E',E) d\omega' dE'
\\ \label{ass9-a}
-\int_{S'}\sigma^2(x,\omega',\omega,E)  d\omega'
-2\pi\int_{I'}\hat{\sigma}^3(x,E',E)dE'
\geq c,
\end{multline}
for a.e. $(x,\omega,E)\in G\times S\times I$.
The constant $c$ can be any real number.

Next result addresses accretivity of $K-\Sigma$ under the appropriate assumptions.

\begin{theorem}\label{SK-dissip}
Suppose that the assumptions  (\ref{ass8-aa}) and (\ref{ass9-a}) are valid. 
Then 
\be\label{K-coer}
\la (\Sigma-K)\psi,\psi\ra_{L^2(G\times S\times I)}\geq c\n{\psi}^2_{L^2(G\times S\times I)}\quad \forall \psi\in L^2(G\times S\times I).
\ee
\end{theorem}

\begin{proof} 
By the Cauchy-Schwarz inequality (2nd step)
\bea\label{kc-2}
&
\int_{S\times I}(K\psi)(x,\omega,E)\psi(x,\omega,E) d\omega dE
\nonumber\\
= {}&
\int_{S\times I}\int_{S'\times I'}
\sigma^1(x,\omega',\omega,E',E)\psi(x,\omega',E')\psi(x,\omega,E)d\omega' dE' d\omega dE\nonumber\\
&
+
\int_{S\times I}\int_{ S'}\sigma^2(x,\omega',\omega,E)\psi(x,\omega',E) \psi(x,\omega,E)d\omega' d\omega dE\nonumber\\
&
+
\int_{S\times I}
\int_{I'}\int_{0}^{2\pi}
\hat\sigma^3(x,E',E)
\psi(x,\gamma(E',E,\omega)(s),E')\psi(x,\omega,E)ds dE' d\omega dE\nonumber\\
\leq {}&
\Big(\int_{S\times I}\int_{S'\times I'}
\sigma^1(x,\omega',\omega,E',E)
|\psi(x,\omega',E')|^2
d\omega' dE' d\omega dE\Big)^{1/2}
\nonumber\\
&
\cdot
\Big(\int_{S\times I}\int_{S'\times I'}
\sigma^1(x,\omega',\omega,E',E)
|\psi(x,\omega,E)|^2d\omega' dE' d\omega dE\Big)^{1/2}
\nonumber\\
&
+
\Big(\int_{S\times I}\int_{ S'}\sigma^2(x,\omega',\omega,E)|\psi(x,\omega',E)|^2 d\omega' d\omega dE\Big)^{1/2}
\nonumber\\
&
\cdot
\Big(\int_{S\times I}\int_{ S'}\sigma^2(x,\omega',\omega,E)|\psi(x,\omega,E)|^2 d\omega' d\omega dE\Big)^{1/2}\nonumber\\
&
+
\Big(
\int_{S\times I}
\int_{I'}\int_{0}^{2\pi}
\hat\sigma^3(x,E',E)
|\psi(x,\gamma(E',E,\omega)(s),E')|^2ds dE' d\omega dE\Big)^{1/2}\nonumber\\
&
\cdot 
\Big(
\int_{S\times I}
\int_{I'}\int_{0}^{2\pi}
\hat\sigma^3(x,E',E)
|\psi(x,\omega,E)|^2ds dE' d\omega dE\Big)^{1/2}.
\eea
For non-negative real numbers it holds (Young's inequality for $p=2$)
\[
2\sqrt{a}\sqrt{b}\leq (a+b)
\] 
and so by (\ref{kc-2}) and Lemma \ref{le-m} (at second step)
\bea\label{kc-3}
&
2\int_{S\times I}(K\psi)(x,\omega,E)\psi(x,\omega,E) d\omega dE
\nonumber\\
\leq {}&
\int_{S\times I}\int_{S'\times I'}
\sigma^1(x,\omega',\omega,E',E)
|\psi(x,\omega',E')|^2
d\omega' dE' d\omega dE
\nonumber\\
&
+
\int_{S\times I}\int_{S'\times I'}
\sigma^1(x,\omega',\omega,E',E)
|\psi(x,\omega,E)|^2d\omega' dE' d\omega dE
\nonumber\\
&
+
\int_{S\times I}\int_{ S'}\sigma^2(x,\omega',\omega,E)|\psi(x,\omega',E)|^2 d\omega' d\omega dE
\nonumber\\
&
+
\int_{S\times I}\int_{ S'}\sigma^2(x,\omega',\omega,E)|\psi(x,\omega,E)|^2 d\omega' d\omega dE\nonumber\\
&
+
\int_{S\times I}
\int_{I'}\int_{0}^{2\pi}
\hat\sigma^3(x,E',E)
|\psi(x,\gamma(E',E,\omega)(s),E')|^2ds dE' d\omega dE\nonumber\\
&
+
\int_{S\times I}
\int_{I'}\int_{0}^{2\pi}
\hat\sigma^3(x,E',E)
|\psi(x,\omega,E)|^2ds dE' d\omega dE\nonumber\\
={}&
\int_{S\times I}\Big(\int_{S'\times I'}
\sigma^1(x,\omega',\omega,E',E)
d\omega' dE'  \nonumber\\
&
\quad
+
\int_{ S'}\sigma^2(x,\omega',\omega,E)d\omega'
+2\pi\int_{I'}\hat\sigma^3(x,E',E)dE'
\Big)|\psi(x,\omega,E)|^2  d\omega dE
\nonumber\\
&
+
\int_{S'\times I'}\Big(\int_{S\times I}
\sigma^1(x,\omega',\omega,E',E)
d\omega dE \nonumber\\
&
\quad
+
\int_{ S}\sigma^2(x,\omega',\omega,E')d\omega
+2\pi\int_{I}\hat\sigma^3(x,E',E)dE
\Big)|\psi(x,\omega',E')|^2 d\omega'  dE'
\nonumber\\
\leq {}&
\int_{S\times I}\big(\Sigma(x,\omega,E)-c\big)|\psi(x,\omega,E)|^2  d\omega dE \\
&
+
\int_{S'\times I'}\big(\Sigma(x,\omega',E')-c\big)|\psi(x,\omega',E')|^2 d\omega' dE'
\eea
where we in the last step applied the assumptions (\ref{ass8-aa}), (\ref{ass9-a}). 
Hence by integration (\ref{kc-3}) over $G$ we conclude
\be\label{kc-4}
2\int_{G\times S\times I}(K\psi)(x,\omega,E)\psi(x,\omega,E)dx d\omega dE
\leq
2
\int_{G\times S\times I}\big(\Sigma(x,\omega,E)-c\big)|\psi(x,\omega,E)|^2 dx d\omega dE
\ee
which implies the proof.

\end{proof}

\subsection{Restricted Collision Operator Related to Coupled System} \label{co-cs}

To obtain a concise presentation we use below Hausdorff measures in formulations. However, recall that Hausdorff measures are not generally $\sigma$-finite and so, for example, in this context the Fubini's theorem is not available.
The collision operator related to elastic scattering has the form
(recall (\ref{elastick}))
\be\label{co-1}
(K\psi)(x,\omega,E)=\int_{S'}\sigma(x,\omega',\omega,E)\psi(x,E,\omega') d\omega'
\ee
which can be given as
\be\label{co-2}
(K\psi)(x,\omega,E)=\int_{I'}\int_{S'}\tilde{\sigma}(x,\omega',\omega,E',E)\psi(x,E',\omega') d\omega' d\rho_I(E'),
\ee
where $\tilde{\sigma}(x,\omega',\omega,E',E):=\sigma(x,\omega',\omega,E)\chi_{\{0\}}(E'-E)$,
the function $\chi_{\{0\}}$ is the characteristic function
of the singleton $\{0\}$, and $\rho_I:=\mu_H^0$  is the 0-dimensional Hausdorff measure on $I$ (see \cite[pp. 7--10]{falconer86}).
We note that it is important to keep the order
of the iterated integrals
$\int_{I'}\int_{S'}$ fixed in \eqref{co-2},
since Hausdorff-measure $\rho_I=\mu^0_H$ is not $\sigma$-finite on $I'$.
In this setting
$\tilde{\sigma}(x,\omega',\omega,E',E)$ is a measurable function on $G\times S^2\times I^2$,
and it is clear that
\be\label{co-3}
\int_{I'}\int_{S'}\tilde{\sigma}(x,\omega',\omega,E',E)
 d\omega' d\rho_I(E')
=
\int_{S'}{\sigma}(x,\omega',\omega,E)
 d\omega'.
\ee

We retrieve a similar formulation regarding to the restricted M\o ller and Klein-Nishina scattering operator (\ref{mollerk}).
Letting $\rho_{S}:=\mu_H^1$ be the 1-dimensional Hausdorff measure on $S$ (\cite{falconer86}),
and writing 
\[
\underline{\sigma}(x,\omega',\omega,E',E)
=\frac{\hat\sigma(x,E',E)}{\sqrt{1-\mu(E',E)^2}}\chi_{\mc{M}}(\omega',\omega,E',E),
\]
where $\chi_{\mc{M}}$ is the characteristic function of the set
\[
\mc{M}:=\{(\omega',\omega,E',E)\in S^2\times I^2\ |\ \omega'\cdot\omega-\mu(E',E)=0\},
\]
then we  have 
\be\label{k-11}
(K\psi)(x,\omega,E)=\int_{I'}\int_{S'} \underline{\sigma}(x,\omega',\omega,E',E)\psi(x,\omega',E')d\rho_{S}(\omega')dE'.
\ee
Again, order of the iterated integrals $\int_{I'}\int_{S'}$ needs
to be carefully retained because $\rho_S=\mu^1_H$ is not $\sigma$-finite on $S$.

The key observation here is that in \eqref{mollerk}, the inner integral can be written as (\cite{falconer86})
\begin{multline*}
\int_{0}^{2\pi}\psi(x,\gamma(E',E,\omega)(s),E')ds
=\frac{1}{\sqrt{1-\mu(E',E)^2}}\int_{\Gamma(E',E,\omega)} \psi(x,\cdot,E')d\ell, \\
=
\frac{1}{\sqrt{1-\mu(E',E)^2}}\int_{S'} \chi_{\mc{M}}(\omega',\omega,E',E)\psi(x,\omega',E')d\rho_{S}(\omega')
\end{multline*}
where $\int_{\Gamma(E',E,\omega)} (\cdots)d\ell$ is the path integral along the curve $\Gamma(E',E,\omega)$.  Hence we get
the expression (\ref{k-11}),
and in particular
\be\label{co-4}
\int_{I'}\int_{S'} \underline{\sigma}(x,\omega',\omega,E',E)d\rho_{S}(\omega')dE'
=
2\pi 
\int_{I'}\hat\sigma(x,E',E) dE'.
\ee

Next we consider the coupled collision operator. Recall that $K=(K_1,K_2,K_3)$.
We assume that
\[
(K_{j}\psi)(x,\omega,E)=\sum_{k=1}^3\int_{I'}\int_{S'}\sigma_{kj}(x,\omega',\omega,E',E)\psi_k(x,\omega',E') d\rho^{kj}_{S}(\omega')d\rho^{kj}_I(E'),\quad j=1,2,3,
\]
where we use pairs of measures $(\rho_I^{kj},\rho_{S})=({\mc L}^1,\mu_{S})$ or
$(\rho_I^{kj},\rho_{S})=({\mc L}^1,\mu_H^1)$ or
$(\rho_I^{kj},\rho_{S})=(\mu_H^0,\mu_{S})$.
Here $\mc L^1$ ($=dE$) and $\mu_{S}$ ($=d\omega$) are the 1-dimensional Lebesgue measure on $I$ and the usual surface measure on $S$, respectively.
These notations serve only to compress the different
varieties of collision operators $K$ considered in earlier 
subsections \ref{rco-m}, \ref{el-k} and \ref{subseq:el_2}.

By similar techniques as above (cf. \cite{tervo07}) we are able to show the following boundedness and coercitivity results for the coupled system.
A boundedness criterion is given by

\begin{theorem}\label{coup-bound}
Suppose that $\sigma_{kj}:G\times S^2\times I^2\to\R$ are non-negative measurable functions such that a.e. $(x,\omega,E)\in G\times S\times I$
\bea\label{bound-coupk}
&
\sum_{k=1}^3\int_{I'}\int_{S'}\sigma_{kj}(x,\omega',\omega,E',E)
d\rho_{S}^{kj}(\omega') d\rho_I^{kj}(E')\leq M_1
\nonumber\\
&
\sum_{k=1}^3\int_{I'}\int_{S'}\sigma_{jk}(x,\omega,\omega',E,E')
d\rho_{S}^{jk}(\omega') d\rho_I^{jk}(E')
\leq M_2
\eea
Then $K: L^2(G\times S\times I)^3\to L^2(G\times S\times I)^3$ is bounded and
\be
\n{K}\leq 2\pi\sqrt{M_1M_2}. 
\ee
\end{theorem}

\begin{theorem}\label{diss-for-coupled}
Suppose that for $j=1,2,3$ and a.e. $(x,\omega,E)\in G\times S\times I$
\bea\label{k3-n3-coupdiss}
&
\Sigma_j(x,\omega,E)-\sum_{k=1}^3\int_{ I'}\int_{S'}\sigma_{kj}(x,\omega',\omega,E',E)d\rho_{S}^{kj}(\omega') d\rho_I^{kj}(E')\geq c
\nonumber\\
&
\Sigma_j(x,\omega,E)-\sum_{k=1}^3\int_{I'}\int_{S'}\sigma_{jk}(x,\omega,\omega',E,E')d\rho_S^{jk}(\omega') d\rho_I^{jk}(E')\geq c 
\eea
Then for all $\psi\in L^2(G\times S\times I)^3$
\be
\la (\Sigma-K)\psi,\psi\ra_{L^2(G\times S\times I)^3}\geq c\n{\psi}_{L^2(G\times S\times I)^3}^2.
\ee
\end{theorem}

\begin{remark}\label{dissip-re}

In virtue of Theorem \ref{diss-for-coupled} the operator
$-\Sigma+K$ satisfies the following  condition: For all $\lambda >0$ and $\psi\in L^2(G\times S\times I)^3$ one has
\be\label{kc-5}
\n{\big(\lambda I-(-\Sigma+K+cI)\big)\psi}_{L^2(G\times S\times I)^3}
\geq \lambda\n{\psi}_{L^2(G\times S\times I)^3}.
\ee
In other words, the operator $-\Sigma+K+cI:L^2(G\times S\times I)^3\to
L^2(G\times S\times I)^3$ is {\it dissipative}.
The inequality (\ref{kc-5}) implies (by substituting $\lambda+c$ for $\lambda$) that 
\be\label{kc-6}
\n{(\lambda I-(-\Sigma+K))\psi}_{L^2(G\times S\times I)^3}\geq(\lambda+c) \n{\psi}_{L^2(G\times S\times I)^3} 
\ee 
for all $\lambda >0$. In particular, $-\Sigma+K$ is dissipative.
Similar observations are naturally valid for uncoupled collision operators.
\end{remark}

\sectionspace
\section{Single Continuous Slowing Down Equation}\label{single-eq}
\subsection{Preliminaries}\label{presingle-eq}

At first we consider a {\it single CSDA transport equation} given by
\be\label{se1}
 -{\p {(S_0\psi)}E}+\omega\cdot\nabla_x\psi+\Sigma\psi
- K\psi= f\quad \textrm{on}\ G\times S\times I,
\ee
where the solution satisfies inflow boundary and initial value conditions
\begin{alignat}{3}
\psi_{|\Gamma_-}&=g\quad && \textrm{on}\ \Gamma_-, \label{se2} \\[2mm]
\psi(\cdot,\cdot,E_{\rm m})&=0\quad && \textrm{on}\ G\times S. \label{se3}
\end{alignat}
We assume that 
\be\label{ass1}
\Sigma\in L^\infty(G\times S\times I),\quad \Sigma\geq 0
\ee
a.e. on $G\times S\times I$.

We assume that the restricted collision operator is as in section \ref{col-sum} the sum $K=K^1+K^2+K^3$ where $\sigma^j,\ j=1,2,3$ obey
(\ref{ass5-a}), (\ref{ass7}), (\ref{ass-8}). Furthermoe, we assume that
\begin{multline}
\label{ass8-aa-a}
\Sigma(x,\omega,E)-\int_{S'\times I'}\sigma^1(x,\omega,\omega',E,E')e^{C(E'-E)} d\omega' dE'
\\
-\int_{S'}\sigma^2(x,\omega,\omega',E)  d\omega'
-2\pi\int_{I'}\hat{\sigma}^3(x,E,E')e^{C(E'-E)}dE'
\geq c,
\end{multline}
\begin{multline}
\label{ass9-a-b}
\Sigma(x,\omega,E)-\int_{S'\times I'}\sigma^1(x,\omega',\omega,E',E)e^{C(E-E')} d\omega' dE'
\\
-\int_{S'}\sigma^2(x,\omega',\omega,E)  d\omega'
-2\pi\int_{I'}\hat{\sigma}^3(x,E',E)e^{C(E-E')}dE'
\geq c
\end{multline}
for a.e. $(x,\omega,E)\in G\times S\times I$, and
where the constant $C\geq 0$ is specified below (see \eqref{eq:def_C}). 
Note that in some cases we will assume $c$ to be strictly positive, $c>0$.
This assumption has been relaxed for certain problems in \cite{egger14} (see also \cite[Remark 15, pp. 241-242]{dautraylionsv6}).

In what follows, we assume that the stopping power $S_0:\ol G\times I\to\R$ satisfies (at least) the following assumptions:
\be
& S_0\in L^\infty(G\times  I), \label{csda9} \\[2mm]
& {\p {S_0 }{E}}\in L^\infty(G\times  I), \label{csda9aa} \\[2mm]
& \kappa:=\inf_{(x,E)\in \ol G\times I}S_0(x,E)>0,\label{csda9a} \\[2mm]
& \nabla_x S_0\in L^\infty(G\times I). \label{csda9b}
\ee
We remark that the assumption (\ref{csda9b}) will be needed only in the context of the theory of evolution operators in section \ref{evcsd}.

We begin with a lemma.

\begin{lemma}\label{csdale0}
For all $\psi\in C^1(\ol G\times S\times I)$,
\begin{multline}\label{se4}
\la {\p {(S_0\psi)}E},\psi\ra_{L^2(G\times S\times I)}\leq q\n{\psi}^2_{L^2(G\times S\times I)} \\
+\frac{1}{2}\int_{G\times S} \big(S_0(x,E_m)\psi^2(x,\omega,E_m)-
 S_0(x,0)\psi^2(x,\omega,0)\big)dx d\omega,
\end{multline}
where
\bea\label{q}
q:={1\over 2}\esssup_{(x,E)\in G\times I}{\p {S_0}E}(x,E).
\eea
\end{lemma}

\begin{proof}
Integrating by parts, we have
\bea\label{se5}
&\la {\p {(S_0\psi)}E},\psi\ra_{L^2(G\times S\times I)}
=
\la {\p {S_0}E}\psi,\psi\ra_{L^2(G\times S\times I)}
+\la {\p {\psi}E},S_0\psi\ra_{L^2(G\times S\times I)}\nonumber\\
={}&
\la {\p {S_0}E}\psi,\psi\ra_{L^2(G\times S\times I)} -\la \psi, {\p {(S_0\psi)}E}\ra_{L^2(G\times S\times I)} \nonumber \\
&+\int_{G\times S} \big(S_0(x,E_m)\psi(x,\omega,E_m)^2 - S_0(x,0)\psi(x,\omega,0)^2\big)dx d\omega,
\eea
and therefore
\bea 
&2\la {\p {(S_0\psi)}E},\psi\ra_{L^2(G\times S\times I)} \nonumber\\
={}&
\la {\p {S_0}E}\psi,\psi\ra_{L^2(G\times S\times I)}
+\int_{G\times S} \big(S_0(x,E_m)\psi(x,\omega,E_m)^2 - S_0(x,0)\psi(x,\omega,0)^2\big) dx d\omega
\nonumber\\
\leq {}&
2q\n{\psi}^2_{L^2(G\times S\times I)}
+\int_{G\times S} \big(S_0(x,E_m)\psi(x,\omega,E_m)^2-S_0(x,0)\psi(x,\omega,0)^2\big) dx d\omega.
\eea
This finishes the proof.
\end{proof}

Note that if $E\mapsto S_0(x,E)$ is decreasing for every $x\in G$ then $q\leq 0$ (and therefore $C$ below vanishes).

Let
\begin{align}\label{eq:def_C}
C:=\frac{\max\{q,0\}}{\kappa}.
\end{align}
We make  the following change of the unknown function.
We replace $\psi$ by
\begin{align}\label{eq:exp_trick}
\phi(x,\omega,E):=e^{CE}\psi(x,\omega,E).
\end{align}
This substitution changes the equation (\ref{se1}) to 
(here and below by writing $e^{CE}$ we mean a function $(x,\omega,E)\mapsto e^{CE}$)
\be\label{csda3A}
-{\p {(S_0\phi)}E}+\omega\cdot\nabla_x\phi+C S_0\phi+\Sigma\phi
-K_C\phi=e^{CE}f,
\ee
where $K_C$ is given by
\bea\label{collc}
&
(K_C\phi)(x,\omega,E)=
\int_{S'\times I'}\sigma^1(x,\omega',\omega,E',E)\phi(x,\omega',E')e^{C(E-E')} d\omega' dE'
\nonumber\\
&
-\int_{S'}\sigma^2(x,\omega',\omega,E) \phi(x,\omega',E) d\omega'
-\int_{I'}\int_0^{2\pi}\hat{\sigma}^3(x,E',E)e^{C(E-E')}\phi(x,\gamma(E',E,\omega)(s),E')dsdE'
\eea
The inflow boundary   and the initial conditions are
\be
{\phi}_{|\Gamma_-}={}&e^{CE}g, \label{finalbc} \\[2mm]
\phi(x,\omega,E_m)={}&0, \label{finalic}
\ee
the latter (initial) condition holding for a.e. $(x,\omega)\in G\times S$.

\begin{lemma}\label{csdale1a}
Assume that the conditions (\ref{ass1}), (\ref{ass5-a}), (\ref{ass7}), (\ref{ass-8}), (\ref{ass8-aa-a}) and (\ref{ass9-a-b}) are valid. Then
\[
\Sigma-K_C:L^2(G\times S\times  I)\to L^2(G\times S\times  I)
\]
is a bounded operator and it satisfies the following accretivity condition
\be\label{se7}
\la (\Sigma-K_C)\phi,\phi\ra_{L^2(G\times S\times  I)}\geq c\n{\phi}^2_{L^2(G\times S\times  I)},
\quad \phi\in L^2(G\times S\times  I).
\ee
\end{lemma}

\begin{proof}
The proof follows from Theorems \ref{esol-th1} and \ref{SK-dissip}.
\end{proof}

\begin{remark}
In the study below, the standing assumption for the (inflow) boundary condition
is $g\in T^2(\Gamma_-)$ (single equation), or $g\in T^2(\Gamma_-)^3$ (coupled equation),
however we point out that in some parts,
one could do with the more general assumption $g\in T^2_{\tau_-}(\Gamma_-)$, or $g\in T^2_{\tau_-}(\Gamma_-)^3$.
In this paper we omit these generalizations.
\end{remark}

\subsection{Existence of Solutions for a Single Continuous Slowing Down Equation by a Variational Formulation}\label{esols}

We shall consider the existence of solutions for the problem 
(\ref{csda3A}), (\ref{finalbc}), (\ref{finalic}) by applying the so-called \emph{Lions-Lax-Milgram Theorem} 
(or generalized Lax-Milgram Theorem) which is based on the variational formulation. Hence we begin with some computations which lead one to find the related bilinear and linear forms. 
In the following we denote
\be\label{eq:b_fg}
{f}_C=e^{CE}f, \quad { g}_C=e^{CE}g .
\ee
Let $f_-$ ($f_+$) be the negative (positive) part of a function. Recall that
\be
f=f_+-f_-\quad {\rm and}\quad |f|=f_++f_-.
\ee
Applying the Green's formula (\ref{green}) and integrating by parts, we have for
and $\phi\in C^1(\ol G\times S\times I)$ that satisfies
the equation (\ref{csda3A}) and for any $v\in C^1(\ol G\times S\times I)$,
\bea\label{csda25}
&-\la{\p {(S_0\phi)}E},v\ra_{L^2(G\times S\times I)}+\la\omega\cdot\nabla_x\phi,v\ra_{L^2(G\times S\times I)}+\la CS_0\phi,v\ra_{L^2(G\times S\times I)}\nonumber\\
&+\la\Sigma\phi,v\ra_{L^2(G\times S\times I)}-\la K_C\phi,v\ra_{L^2(G\times S\times I)}
\nonumber\\[2mm]
={}&
\la\phi,S_0{\p {v}E}\ra_{L^2(G\times S\times I)}-
\int_{G\times S} S_0\phi v\Big|_{E=0}^{E=E_{\rm m}} dx d\omega  \nonumber\\ 
&-\la\phi,\omega\cdot\nabla_x v\ra_{L^2(G\times S\times I)}
+\int_{\partial G\times S\times I}(\omega\cdot\nu)\phi v d\sigma d\omega dE  \nonumber\\
&+\la \phi,CS_0v\ra_{L^2(G\times S\times I)}
+\la\phi,\Sigma^* v\ra_{L^2(G\times S\times I)}-\la \phi,K_C^*v\ra_{L^2(G\times S\times I)} \nonumber\\[2mm]
={}&\la { f}_C,v\ra_{L^2(G\times S\times I)}
\eea
where
\be 
\Sigma^*=\Sigma
\ee
and 
\be 
(K_C^*v)(x,\omega,E)=
\int_{S\times I}\sigma(x,\omega,\omega',E,E')
e^{C(E'-E)}v(x,\omega',E')d\omega' dE'.
\ee
Assuming that the inflow boundary  condition  $\phi_{|\Gamma_-}={g}_C$ and the initial condition $\phi(\cdot,\cdot,E_{\rm m})=0$ are valid,
the equation (\ref{csda25}) is equivalent to
\bea\label{csda26}
&
\la\phi,S_0{\p {v}E}\ra_{L^2(G\times S\times I)}+
\la \phi(\cdot,\cdot,0),S_0(\cdot,\cdot,0) v(\cdot,\cdot,0)\ra_{L^2(G\times S)} 
-\la\phi,\omega\cdot\nabla_x v\ra_{L^2(G\times S\times I)}\nonumber\\
&+\int_{\partial G\times S\times I}(\omega\cdot\nu)_+\phi v d\sigma d\omega dE 
+\la \phi,CS_0v\ra_{L^2(G\times S\times I)} \nonumber\\
&+\la\phi,\Sigma^* v\ra_{L^2(G\times S\times I)}-\la \phi,K_C^*v\ra_{L^2(G\times S\times I)}\nonumber\\
={}&\la { f}_C,v\ra_{L^2(G\times S\times I)}+\int_{\partial G\times S\times I}(\omega\cdot\nu)_-{ g_C} v d\sigma d\omega dE.
\eea
Clearly
\be
\int_{\partial G\times S\times I}(\omega\cdot\nu)_-{g_C} v d\sigma d\omega dE=
\la { g_C},\gamma_-(v)\ra_{T^2(\Gamma_-)}
\ee
and
\be
\int_{\partial G\times S\times I}(\omega\cdot\nu)_+\phi v d\sigma d\omega dE=
\la \gamma_+(\phi),\gamma_+(v)\ra_{T^2(\Gamma_+)}.
\ee

One thus deduces that the relevant bilinear from $B$ and linear form $F$ are
\bea\label{csda27}
B(\phi,v)={}&
\la\phi,S_0{\p {v}E}\ra_{L^2(G\times S\times I)}
-\la\phi,\omega\cdot\nabla_x v\ra_{L^2(G\times S\times I)}\nonumber\\
&+C\la \phi,S_0v\ra_{L^2(G\times S\times I)}+\la\phi,(\Sigma^*-K_C^*) v\ra_{L^2(G\times S\times I)}\nonumber\\
&+\la \gamma_+(\phi),\gamma_+(v)\ra_{T^2(\Gamma_+)}
+\la \phi(\cdot,\cdot,0),S_0(\cdot,0) v(\cdot,\cdot,0)\ra_{L^2(G\times S)},
\eea
and
\[
F(v)=\la { f}_C,v\ra_{L^2(G\times S\times I)}+\la { g}_C,\gamma_-(v)\ra_{T^2(\Gamma_-)}.
\]
The variational equation corresponding to the problem (\ref{csda3A}), (\ref{finalbc}), (\ref{finalic}) (in the classical sense)   is
\[
B(\phi,v)=F(v)\quad \forall v\in C^1(\ol G\times S\times I).
\]

We show that the bilinear form $B:C^1(\ol G\times S\times I)\times C^1(\ol G\times S\times I)\to\R$ obeys the following {\it boundedness and coercivity} conditions:

\begin{theorem}\label{csdath1}
Suppose that the assumptions  (\ref{ass1}), (\ref{ass5-a}), (\ref{ass7}), (\ref{ass-8}), (\ref{ass8-aa-a}) and (\ref{ass9-a-b}) (with $C=\frac{\max\{q,0\}}{\kappa}$ and $c>0$) and (\ref{csda9}), (\ref{csda9aa}), (\ref{csda9a})   are valid.
Then there exists a constant $M>0$  such that 
\be\label{csda29}
|B(\phi,v)|\leq M\n{\phi}_{H_1}\n{v}_{ H_2}\quad \forall\phi,\ v\in C^1(\ol G\times S\times I)
\ee
and 
\be\label{csda30}
B(\phi,\phi)\geq c'\n{\phi}_{H_1}^2\quad \forall \phi\in C^1(\ol G\times S\times I)
\ee
where we assume that 
\bea\label{cprime}
c':=\min\{{1\over 2},{{\kappa}\over 2},c\}
\eea 
is strictly positive. Note that $\kappa$ is defined in \eqref{csda9a} and $c$ in \eqref{ass8-aa-a}, \eqref{ass9-a-b}.
(Recall that the spaces $H_1$ and $H_2$ were defined in equations \eqref{spaceH1} and \eqref{inph2}, respectively.)
\end{theorem}

\begin{proof}
A. At first we show the boundedness of $B(\cdot,\cdot)$.
The assumptions \eqref{csda9}, \eqref{csda9aa} and the fact that the (energy) interval $I$ is bounded imply
by the Sobolev Embedding theorem (see \eqref{sobim}) that $S_0(\cdot,0)\in L^\infty(G)$. Hence we find that 
\bea\label{csda31}
|B(\phi,v)|\leq &
\n{\phi}_{L^2(G\times S\times I)}\n{S_0}_{L^\infty(G\times I)}\n{{\p {v}E}}_{L^2(G\times S\times I)}
\nonumber\\
&+\n{\phi}_{L^2(G\times S\times I)} \n{\omega\cdot\nabla_x v}_{L^2(G\times S\times I)}\nonumber
\\
&+C\n{S_0}_{L^\infty(G\times I)}\n{\phi}_{L^2(G\times S\times I)}\n{v}_{L^2(G\times S\times I)}\nonumber\\
&+\n{\phi}_{L^2(G\times S\times I)}\n{(\Sigma-K_C)^*}\n{v}_{L^2(G\times S\times I)} \nonumber\\
&+\n{\gamma(\phi)}_{T^2(\Gamma)} \n{\gamma(v)}_{T^2(\Gamma)} \nonumber\\
&+\n{\phi(\cdot,\cdot,0)}_{L^2(G\times S)}\n{S_0(\cdot,0)}_{L^\infty(G)}\n{v(\cdot,\cdot,0)}_{L^2(G\times S)}.
\eea
This implies the assertion \eqref{csda29} by the definition of the spaces $H_1$, $H_2$,
and after observing that Sobolev Embedding theorem (see \eqref{sobim}) implies the existence of
a constant $C'\geq 0$ such that
\[
\n{w(\cdot,\cdot,0)}_{L^2(G\times S)}\leq C'\n{w}_{H_2},\quad \forall w\in H_2.
\]

B. We verify the coercitivity (\ref{csda30}).
Integrating by parts we have
\begin{multline}\label{csda32}
\la\phi,S_0{\p {\phi}E}\ra_{L^2(G\times S\times I)}
=
-\la\phi,{\p {(S_0\phi)}E}\ra_{L^2(G\times S\times I)} \\
+\la\phi(\cdot,\cdot,E_m),S_0(\cdot,E_m)\phi(\cdot,\cdot,E_m)\ra_{L^2(G\times S)}-
\la\phi(\cdot,\cdot,0),S_0(\cdot,0)\phi(\cdot,\cdot,0)\ra_{L^2(G\times S)}
\end{multline}
Using the Green's formula (\ref{green}) we have
\be\label{csda33}
-\la \phi, \omega\cdot\nabla_x \phi\ra_{L^2(G\times S\times I)}
=\la \omega\cdot\nabla_x\phi,\phi\ra_{L^2(G\times S\times I)}
-\int_{\partial G\times S\times I}(\omega\cdot\nu)\phi^2 d\sigma d\omega dE\nonumber\\
\ee
which implies
\be\label{csda33a}
\la \omega\cdot\nabla_x\phi,\phi\ra_{L^2(G\times S\times I)}
={}&\frac{1}{2}\int_{\partial G\times S\times I}(\omega\cdot\nu)\phi^2 d\sigma d\omega dE \nonumber\\
={}&\frac{1}{2}
\int_{\partial G\times S\times I}((\omega\cdot\nu)_+-(\omega\cdot\nu)_-)\phi^2 d\sigma d\omega dE \\
={}&\frac{1}{2}\big(\n{\gamma_+(\phi)}_{T^2(\Gamma_+)}^2-\n{\gamma_-(\phi)}_{T^2(\Gamma_-)}^2\big).
\ee
By the definition \eqref{eq:def_C} of $C$ and assumption \eqref{csda9a},
we have
\[
C=\frac{\max\{q,0\}}{\kappa}\geq \frac{\max\{q,0\}}{S_0}\geq \frac{q}{S_0}
\]
a.e. and hence
\[
C\la \phi,S_0\phi\ra_{L^2(G\times S\times I)}\geq q\n{\phi}_{L^2(G\times S\times I)}^2.
\]

Taking Lemmas \ref{csdale0} and \ref{csdale1a} into account, one can thus estimate
\[
B(\phi,\phi)={}&
-\la\phi,{\p {(S_0\phi)}E}\ra_{L^2(G\times S\times I)}
+\la\phi(\cdot,\cdot,E_m),S_0(\cdot,E_m)\phi(\cdot,\cdot,E_m)\ra_{L^2(G\times S)}
\nonumber\\
&-\frac{1}{2}\big(\n{\gamma_+(\phi)}_{T^2(\Gamma_+)}^2-\n{\gamma_-(\phi)}_{T^2(\Gamma_-)}^2\big)
+\n{\gamma_+(\phi)}_{T^2(\Gamma_+)}^2 \nonumber\\
&+C\la \phi,S_0\phi\ra_{L^2(G\times S\times I)}+\la(\Sigma-K_C)\phi,\phi\ra_{L^2(G\times S\times I)}
\nonumber\\
\geq {}&
-q\n{\phi}^2_{L^2(G\times S\times I)}
-\frac{1}{2}\la\phi(\cdot,\cdot,E_m),S_0(\cdot,E_m)\phi(\cdot,\cdot,E_m)\ra_{L^2(G\times S)} \\
&+\frac{1}{2}\la\phi(\cdot,\cdot,0),S_0(\cdot,0)\phi(\cdot,\cdot,0)\ra_{L^2(G\times S)}
+\la\phi(\cdot,\cdot,E_m),S_0(\cdot,E_m)\phi(\cdot,\cdot,E_m)\ra_{L^2(G\times S)}\nonumber\\
&
+\frac{1}{2}\n{\gamma(\phi)}_{T^2(\Gamma)}^2
+q\n{\phi}_{L^2(G\times S\times I)}^2+c\n{\phi}^2_{L^2(G\times S\times I)} \\
\geq {}&
\frac{\kappa}{2}\big(\n{\phi(\cdot,\cdot,0)}_{L^2(G\times S)}^2
+\n{\phi(\cdot,\cdot,E_m)}_{L^2(G\times S)}^2\big)
+\frac{1}{2}\n{\gamma(\phi)}_{T^2(\Gamma)}^2
+c\n{\phi}^2_{L^2(G\times S\times I)}.
\]
\end{proof}

Because $C^1(\ol G\times S\times I)\times C^1(\ol G\times S\times I)$ is dense in $H_1\times  H_2$ and since \eqref{csda29} holds, the bilinear form $B(\cdot,\cdot): C^1(\ol G\times S\times I)\times C^1(\ol G\times S\times I)\to\R$ has an unique extension $\tilde B(\cdot,\cdot):H_1\times H_2\to\R$
which satisfies 
\be\label{csda36}
|\tilde B(\tilde\phi,v)|\leq M\n{\tilde\phi}_{H_1}\n{v}_{H_2}\quad \forall \tilde\phi\in H_1,\ v\in H_2
\ee
and 
\be\label{csda37}
\tilde B(v,v)\geq c'\n{v}_{H_1}^2\quad \forall v\in H_2.
\ee
We see that actually
\begin{align}\label{coex}
\tilde B(\tilde\phi,v)
={}&\la\phi,S_0{\p {v}E}\ra_{L^2(G\times S\times I)}
-\la\phi,\omega\cdot\nabla_x v\ra_{L^2(G\times S\times I)} \nonumber\\
&+C\la \phi,S_0v\ra_{L^2(G\times S\times I)}+\la\phi,(\Sigma^*-K_C^*) v\ra_{L^2(G\times S\times I)} \nonumber\\
&+\la q_{|\Gamma_+},\gamma_+(v)\ra_{T^2(\Gamma_+)}+\la p_0,S_0(\cdot,0) v(\cdot,\cdot,0)\ra_{L^2(G\times S)},
\end{align}
when $\tilde\phi=(\phi,q,p_0,p_m)\in H_1$ and $v\in H_2$.
In addition, since for $v\in C^1(\ol G\times S\times I)$ we have
$\n{\gamma_-(v)}_{T^2(\Gamma_-)}\leq \n{\gamma(v)}_{T^2(\Gamma)}$,
it follows that
\bea\label{csda39}
|F(v)|
\leq{}& |\la {f}_C,v\ra_{L^2(G\times S\times I)}|+|\la { g}_C,\gamma_-(v)\ra_{T^2(\Gamma_-)}|
\nonumber\\
\leq{}&
\n{{ f}_C}_{L^2(G\times S\times I)}\n{v}_{L^2(G\times S\times I)}+\n{{g_C}}_{T^2(\Gamma_-)}\n{\gamma(v)}_{T^2(\Gamma)},
\eea
and therefore, since $C^1(\ol{G}\times S\times I)$ is dense in $H_1$,
the linear form $F:C^1(\ol{G}\times S\times I)\to\R$ has a unique bounded extension,
which we still denote by $F$,
\begin{align}\label{Fex}
F:H_1\to\R;\quad
F(\tilde{\phi})=\la { f_C},\phi\ra_{L^2(G\times S\times I)}+\la { g_C}, q\ra_{T^2(\Gamma_-)},
\end{align}
when $\tilde\phi=(\phi,q,p_0,p_m)\in H_1$.
Recall also that the embedding $H_2\subset H_1$ is continuous.

We need the following result, so called {\it Lions-Lax-Milgram Theorem} (generalized Lax-Milgram Theorem).

\begin{theorem}\label{glm}
Let $X$ and $Y$ be Hilbert spaces, with $Y$ continuously embedded into $X$.
Assume that $B(\cdot,\cdot):X\times Y\to\R$ is a bilinear form satisfying the following properties with $M\geq 0,\ c>0$,
\be\label{csda38}
|B(u,v)|\leq M\n{u}_{X}\n{v}_{Y}\quad \forall u\in X,\ v\in Y \quad ({\rm boundedness})
\ee
and 
\be\label{csda37a}
B(v,v)\geq c\n{v}_{X}^2\quad \forall v\in Y \quad ({\rm coercivity}).
\ee
Suppose that $F:X\to\R$ is a bounded linear form. Then there exists $u\in X$ (possibly non-unique) such that
\be\label{csda38-a}
B(u,v)=F(v)\quad \forall v\in Y.
\ee
\end{theorem}

\begin{proof}
See e.g. \cite[p. 403]{treves} or \cite[p. 234]{grisvard}.
\end{proof}

Let
\[
P(x,\omega,E,D)\phi:= -{\p {(S_0\phi)}E}+\omega\cdot\nabla_x\phi
\]
The space
\begin{multline}
{\s H}_P(G\times S\times I^\circ):=\{\phi\in L^2(G\times S\times I)\ | \\
P(x,\omega,E,D)\phi\in L^2(G\times S\times I)\ {\rm in\ the \ weak\ sense}\}
\end{multline}
is a Hilbert space when equipped with the inner product (cf. section \ref{m-d})
\[
\la \phi,v\ra_{{\s H}_P(G\times S\times I^\circ)}=\la \phi,v\ra_{L^2(G\times S\times I)}+\la P(x,\omega,E,D)\phi,P(x,\omega,E,D)v\ra_{L^2(G\times S\times I)}.
\]

With this notation, the equations \eqref{se1} and \eqref{csda3A} can be written as
\begin{align*}
P(x,\omega,E,D)\psi+\Sigma\psi - K\psi = f,
\end{align*}
and
\begin{align*}
P(x,\omega,E,D)\phi+C S_0\phi+\Sigma\phi-K_C\phi=e^{CE}f,
\end{align*}
respectively.
 
In the context if Lions-Lax-Milgram Theorem we shall make us of the following assumption which we call as ${\bf TC}$:

\begin{assumption}[TC]\label{as:TC}
Let $\gamma_{\pm}(\phi)=\phi_{|\Gamma_{\pm}}$ and $\gamma_{\rm m}(\phi):=\phi(\cdot,\cdot,E_{\rm m}),\ \gamma_{0}(\phi):=\phi(\cdot,\cdot,0)$. The assumption is that the linear maps
\begin{align*}
\gamma_{\pm}:{\s H}_P(G\times S\times I^\circ) & \to L_{\rm loc}^2(\Gamma_{\pm},|\omega\cdot\nu| d\sigma d\omega dE), \\
\gamma_{\rm m}:{\s H}_P(G\times S\times I^\circ) & \to L_{\rm loc}^2(G\times S), \\
\gamma_{0}:{\s H}_P(G\times S\times I^\circ) & \to L_{\rm loc}^2(G\times S),
\end{align*}
are well-defined and continuous.
\end{assumption}

\begin{remark}
In the case where $S_0=S_0(E)$ is independent of $x$ and $S_0\in C(I)$ one can show that the assumption ${\bf TC}$ holds. The proof can be based on the techniques of the proof of Theorem \ref{tth} and Example \ref{desolex1}.
If $S_0=S_0(x,E)$ depends also on $x$, we conjecture that it suffices only to assume that $S_0$ is regular enough
for validity of ${\bf TC}$.
\end{remark}

Let
\[
P'(x,\omega,E,D)v=
S_0{\p {v}E}-\omega\cdot\nabla_x v
\]
be the formal transpose of $P(x,\omega,E,D)$.
Making the assumption ${\bf TC}$ the (extended) Green formula
\bea\label{green-ex}
&
\int_{G\times S\times I}(P(x,\omega,E,D) \phi)v\ dxd\omega dE
-\int_{G\times S\times I}(P'(x,\omega,E,D) v)\phi\ dxd\omega dE\nonumber\\
={}&
\int_{\partial G\times S\times I}(\omega\cdot \nu) v\ \phi\ d\sigma d\omega dE\nonumber\\
&
+
\int_{G\times S}\big(S_0(\cdot,0)\phi(\cdot,\cdot,0)v(\cdot,\cdot,0)-S_0(\cdot,E_{\rm m})\phi(\cdot,\cdot,E_{\rm m})v(\cdot,\cdot,E_{\rm m})\big)dx d\omega
\eea
is valid for all $\phi,\ v\in {\s H}_P(G\times S\times I^\circ)$ for which
$({\rm supp}(v))\cap \partial (G\times S\times I)$ is a compact subset of $\Gamma_-\cup \Gamma_+\cup (G\times S\times\{E_m\})\cup(G\times S\times\{0\})$.
Moreover, (\ref{green-ex}) holds for 
$\phi,\ v\in {\s H}_P(G\times S\times I^\circ)$ when $\gamma_{\pm}(\phi)\in T^2(\Gamma_{\pm})$ and $\gamma_{\rm m}(\phi),\ \gamma_0(\psi)\in L^2(G\times S)$. We omit the proof of both these claims.

We are now in position to formulate and prove the following theorem.

\begin{theorem}\label{csdath3}
Suppose that the assumptions  (\ref{ass1}), (\ref{ass5-a}), (\ref{ass7}), (\ref{ass-8}), (\ref{ass8-aa-a}) and (\ref{ass9-a-b}) (with $C=\frac{\max\{q,0\}}{\kappa}$ and $c>0$)
and (\ref{csda9}), (\ref{csda9aa}), (\ref{csda9a}) are valid.
Let ${ f_C}\in L^2(G\times S\times I)$ and ${ g_C}\in T^2(\Gamma_-)$.
Then the following assertions hold.

(i) The variational equation (see \eqref{coex}, \eqref{Fex})
\be\label{csda40a}
\tilde{B}(\tilde\phi,v)=F(v)\quad \forall v\in H_2,
\ee
has a solution $\tilde\phi=(\phi,q,p_0,p_{\rm m})\in H_1$.

Furthermore, $\phi \in {\s H}_P(G\times S\times I^\circ)$ and it is a weak (distributional) solution of the equation (\ref{csda3A}).

(ii) Suppose that additionally the assumption ${\bf TC}$ holds. Then a solution $\phi$ of the
equation \eqref{csda3A} obtained in part (i) is a solution of the problem 
\eqref{csda3A}, \eqref{finalbc}, \eqref{finalic}.

In addition, we have $q_{|\Gamma_+}=\gamma_+(\phi)$ and $p_0=\phi(\cdot,\cdot,0)$,
when $\tilde{\phi}=(\phi,q,p_0,p_m)$ is a solution in $H_1$ obtained in part (i).

(iii) Under the assumptions imposed in part (ii),
any solution $\phi\in \mc{H}_P(G\times S\times I^\circ)$
of the problem \eqref{csda3A}, \eqref{finalbc}, \eqref{finalic}
is unique and obeys the estimate
\be\label{csda40aa}
\n{\phi}_{{H_1}}\leq {1\over{c'}}
\big(\n{{ f_C}}_{L^2(G\times S\times I)}+\n{{g_C}}_{T^2(\Gamma_-)}\big),
\ee
where $c'$ is given in \eqref{cprime}.
\end{theorem}

\begin{proof}
The proof is based on "variations" and it is quite standard.

(i)
We apply Theorem \ref{glm} with $X=H_{1}$, $Y=H_{2}$,
and with $B(\cdot,\cdot)=\tilde B(\cdot,\cdot)$ and $F$ given by \eqref{coex} and \eqref{Fex}, respectively.
As mentioned above $\tilde B(\cdot,\cdot)$ satisfies (\ref{csda38}) and (\ref{csda37a}),
while $F$ is a bounded linear functional,
hence Theorem \ref{glm} guarantees the existence of
a solution $\tilde\phi=(\phi,q,p_0,p_m)\in H_1$ such that (\ref{csda40a}) holds.

We verify that  $\phi\in L^2(G\times S\times I)$ is a weak solution of the equation (\ref{csda3A}).
Let $I^\circ:=]0,E_m[$. From (\ref{csda40a}) it follows that  
\be\label{csda41}
\tilde B(\tilde\phi,v)=F(v),\quad \forall v\in C_0^\infty(G\times S\times I^\circ).
\ee
Since for $v\in C_0^\infty(G\times S\times I^\circ)$ we have $v(\cdot,\cdot,0)=v(\cdot,\cdot,E_m)=0$ and $v_{|\Gamma}=0$,
we see from (\ref{coex}) that
\bea\label{csda42}
\tilde{B}(\tilde{\phi},v)
={}&\la\phi,S_0{\p {v}E}\ra_{L^2(G\times S\times I)}
-\la \phi, \omega\cdot\nabla_x v\ra_{L^2(G\times S\times I)} \nonumber\\
&+\la CS_0\phi,v\ra_{L^2(G\times S\times I)}+\la (\Sigma-K_C)\phi,v\ra_{L^2(G\times S\times I)}\nonumber\\
={}&F(v)
=\la { f_C},v\ra_{L^2(G\times S\times I)},
\eea
for all $v\in C_0^\infty(G\times S\times I^\circ)$,
which means that \eqref{csda3A} holds in the weak sense.

Since $\phi\in L^2(G\times S\times I)$, and by the above
\[
\la \phi,P'(x,\omega,E,D)v\ra_{L^2(G\times S\times I)}=\la -(CS_0+(\Sigma-K_C))\phi + { f_C},v\ra_{L^2(G\times S\times I)},
\]
we see that $\phi\in {\s H}_P(G\times S\times I^\circ)$.

(ii)  Suppose that the assumption ${\bf TC}$ holds and that
$\tilde\phi=(\phi,q,p_0,p_m)\in H_1$ satisfies (\ref{csda40a}). Then for all $v\in H_2$
\bea\label{coexpr}
&\la\phi,S_0{\p {v}E}\ra_{L^2(G\times S\times I)}
-\la\phi,\omega\cdot\nabla_x v\ra_{L^2(G\times S\times I)} \nonumber\\
&\quad +C\la \phi,S_0v\ra_{L^2(G\times S\times I)}+\la\phi,(\Sigma^*-K_C^*) v\ra_{L^2(G\times S\times I)} \nonumber\\
&\quad +\la q,\gamma_+(v)\ra_{T^2(\Gamma_+)}+\la p_0,S_0(\cdot,0) v(\cdot,\cdot,0)\ra_{L^2(G\times S)} \nonumber\\
&
=
\tilde B(\tilde\phi,v)=
\la { f_C},v\ra_{L^2(G\times S\times I)}+\la { g_C}, \gamma_-(v)\ra_{T^2(\Gamma_-)}.
\eea

Recall that $\Gamma'=\partial G\times S$ and $\Gamma'_{-}=\{(y,\omega)\in \partial G\times S\ |\ \omega\cdot\nu(y)<0\}$.
Choose  any $\eta\in C^1_0(I^\circ)$  and $\theta\in C^1(\ol G\times S)$ such that ${\rm supp}(\theta)\cap \Gamma'$ is a compact subset of $\Gamma_-'$. Then $w(x,\omega,E):=\theta(x,\omega)\eta(E)\in C^1(\ol{G}\times S\times I)$
and $w(\cdot,\cdot,0)=w(\cdot,\cdot,E_{\rm m})=0$, $w_{|\Gamma_+}=0$.
Hence  $\la p_0,S_0(\cdot,0) w(\cdot,\cdot,0)\ra_{L^2(G\times S)}=0$ and
$\la q,\gamma_+(w)\ra_{T^2(\Gamma_+)}=0$,
and so by (\ref{coexpr}) for these $w$,
\bea\label{coexpra}
&\la\phi,S_0{\p{w}E}\ra_{L^2(G\times S\times I)}
-\la\phi,\omega\cdot\nabla_x w\ra_{L^2(G\times S\times I)} \nonumber\\
&\quad +\la \phi, CS_0w\ra_{L^2(G\times S\times I)}+\la\phi, (\Sigma^*-K_C^*) w\ra_{L^2(G\times S\times I)} \nonumber\\
&=\la { f_C},w\ra_{L^2(G\times S\times I)}+\la {g_C}, \gamma_-(w)\ra_{T^2(\Gamma_-)}.
\eea

Since the solution $\phi$ obtained in part (i) belongs to $\mc{H}_P(G\times S\times I^\circ)$,
we have by virtue of the Green's formula \eqref{green-ex} and \eqref{coexpra} that
\bea\label{coexpra-a}
&
\la { f_C},w\ra_{L^2(G\times S\times I)}+\la { g_C}, \gamma_-(w)\ra_{T^2(\Gamma_-)}
\nonumber\\
={}&
\la{-\p{(S_0\phi)}{E}}+\omega\cdot\nabla_x\phi+CS_0\phi+(\Sigma-K_C)\phi,w\ra_{L^2(G\times S\times I)}
+\la \gamma_-(\phi),\gamma_-(w) \ra_{T^2(\Gamma_-)} \nonumber \\
={}&
\la { f_C},w\ra_{L^2(G\times S\times I)}
+\la \gamma_-(\phi),\gamma_-(w) \ra_{T^2(\Gamma_-)}
\eea
and hence
\be\label{csda42d}
\la \gamma_-(\phi),\gamma_-(w) \ra_{T^2(\Gamma_-)}
=
\la { g_C},\gamma_-(w) \ra_{T^2(\Gamma_-)},
\ee 
for any $w$ of the form as chosen above. This clearly implies that
$\gamma_-(\phi)={ g_C}\in T^2(\Gamma_-)$. 

Next, choose $\tilde{\eta}\in C^1(I)$ such that $\tilde{\eta}(0)=0$ and choose $\tilde{\theta}\in C_0^1(G)$. Let
$\tilde{w}:=\tilde{\eta} \tilde{\theta}$.
Then by similar calculation as above, we see that 
the Green's formula (\ref{green-ex}) and (\ref{coexpr})
imply
\[
\la \phi(\cdot,\cdot,E_m),S_0(\cdot,E_{\rm m}) \tilde w(\cdot,\cdot,E_m)\ra_{L^2(G\times S)}=0
\]
for these $\tilde{w}$.
Consequently,
$\phi(\cdot,\cdot,E_m)=0$ a.e. in $G\times S$ (since $S_0\geq\kappa>0$), as desired.

Finally, if $v\in C^1(\ol{G}\times S\times I)$, we have
\[
& \la { f_C},v\ra_{L^2(G\times S\times I)}+\la { g_C},\gamma_-(v)\ra_{T^2(\Gamma_-)}
=F(v)=\tilde{B}(\tilde{\phi},v) \\[2mm]
={}& 
\la \phi, P'(x,\omega,E,D)v\ra_{L^2(G\times S\times I)}+\la CS_0\phi,v\ra_{L^2(G\times S\times I)}+\la (\Sigma-K_C)\phi,v\ra_{L^2(G\times S\times I)} \\
{}&+\la q_{|\Gamma_+},\gamma_+(v)\ra_{T^2(\Gamma_+)}+\la p_0,S_0(\cdot,0)v(\cdot,\cdot,0)\ra_{L^2(G\times S)} \\[2mm]
={}&\la T_C\phi,v\ra_{L^2(G\times S\times I)}
+\la q_{|\Gamma_+}-\gamma_+(\phi),\gamma_+(v)\ra_{T^2(\Gamma_+)}
+\la \gamma_-(\phi),\gamma_-(v)\ra_{T^2(\Gamma_-)} \\
{}&
+\la p_0-\phi(\cdot,\cdot,0),S_0(\cdot,0)v(\cdot,\cdot,0)\ra_{L^2(G\times S)}+\la \phi(\cdot,\cdot,E_m),S_0(\cdot,E_m)v(\cdot,\cdot,E_m)\ra_{L^2(G\times S)} \\[2mm]
={}&\la { f_C},v\ra_{L^2(G\times S\times I)}+\la q_{|\Gamma_+}-\gamma_+(\phi),\gamma_+(v)\ra_{T^2(\Gamma_+)}+\la { g_C},\gamma_-(v)\ra_{T^2(\Gamma_-)} \\
{}&+\la p_0-\phi(\cdot,\cdot,0),S_0(\cdot,0)v(\cdot,\cdot,0)\ra_{L^2(G\times S)}
\]
where on the second to last phase we wrote $T_C\phi=(P(x,\omega,E,D)+CS_0+\Sigma-K_C)\phi$ and used Green's formula \eqref{green-ex},
and on the last phase we made use of the already proven facts: $T_C\phi={f_C}$, $\gamma_-(\phi)={ g_C}$, and $\phi(\cdot,\cdot,E_m)=0$.
Thus it holds
\[
\la q_{|\Gamma_+}-\gamma_+(\phi),\gamma_+(v)\ra_{T^2(\Gamma_+)}+\la p_0-\phi(\cdot,\cdot,0),S_0(\cdot,0)v(\cdot,\cdot,0)\ra_{L^2(G\times S)}=0,
\]
for all $v\in C^1(\ol{G}\times S\times I)$, which clearly implies that
$q_{|\Gamma_+}=\gamma_+(\phi)$ and $p_0=\phi(\cdot,\cdot,0)$.

(iii)
By Part (i), $\phi_{|\Gamma_{\pm}}\in T^2(\Gamma_{\pm})$ and 
$\phi(\cdot,\cdot,0)$, $\phi(\cdot,\cdot,E_{\rm m})\in L^2(G\times S)$,
and moreover $\phi\in \mc{H}_P(G\times S\times I^\circ)$.
These properties allow us to apply the Green's formula (\ref{green-ex}),
which in combination with the fact that
\[
P'(x,\omega,E,D)\phi=-P(x,\omega,E,D)\phi-\p{S_0}{E}\phi,
\]
leads us to
\[
\la P(x,\omega,E,D)\phi,\phi\ra_{L^2(G\times S\times I)}
={}&
-\la P(x,\omega,E,D)\phi,\phi\ra_{L^2(G\times S\times I)}
-\la \p{S_0}{E}\phi,\phi\ra_{L^2(G\times S\times I)} \\
{}&
+\n{\gamma_+(\phi)}_{T^2(\Gamma_+)}^2-\n{\gamma_-(\phi)}_{T^2(\Gamma_-)}^2 \\
{}&
+\la S_0(\cdot,0)\gamma_0(\phi),\gamma_0(\phi)\ra_{L^2(G\times S)}
-\la S_0(\cdot,E_m)\gamma_m(\phi),\gamma_m(\phi)\ra_{L^2(G\times S)}.
\]
Using this equation, and performing estimations as in the proof of Theorem \ref{csdath1},
allows us to deduce the inequality
\begin{align*}
\la { f_C},\phi\ra_{L^2(G\times S\times I)}+\la { g_C}, \gamma_-(\phi)\ra_{T^2(\Gamma_-)} 
\geq 
c'\n{\phi}_{H_1}^2,
\end{align*}
from which the desired estimate \eqref{csda40aa}, and therefore uniqueness of solutions, follow.
\end{proof}

\begin{remark}\label{ttc}
Suppose that the assumption ${\bf TC}$ is valid and that $\phi\in {\s H}_P$ such that
\be\label{assttc}
\phi_{|\Gamma_-}\in T^2(\Gamma_-)
\quad {\rm and}\quad
\phi(\cdot,\cdot,E_{m})\in L^2(G\times S).
\ee
Then at least in some cases one is able to show that (cf. \cite{cessenat85})
\be\label{asscl-a}
\phi_{|\Gamma_+}\in T^2(\Gamma_+)
\quad {\rm and}\quad
\phi(\cdot,\cdot,0)\in L^2(G\times S).
\ee
This would make the assumption of part (iii) of Theorem \ref{csdath3} superfluous.
We omit further considerations of this issue here.
\end{remark}

Similarly as above we see that the variational equation corresponding to the original problem
(\ref{se1}), (\ref{se2}), (\ref{se3}) is 
\[
\tilde{B}_0(\tilde{\psi},v)=F_0(v)\quad \forall v\in H_2,
\]
where $\tilde\psi \in H_1$ and $\tilde{B}_0(\cdot,\cdot)$ is the continuous extension
onto $H_1\times H_2$ of the bilinear form $B_0(\cdot,\cdot):C^1(\ol G\times S\times I)\times C^1(\ol G\times S\times I)\to\R$ defined  by (that is, the bilinear form (\ref{csda27}) with $C=0$)
\bea\label{csda27a}
B_0(\psi,v)={}&
\la\psi,S_0{\p {v}E}\ra_{L^2(G\times S\times I)}
-\la\psi,\omega\cdot\nabla_x v\ra_{L^2(G\times S\times I)}
+\la\psi,(\Sigma^* -K^*)v\ra_{L^2(G\times S\times I)} \nonumber\\
&+\la \gamma_+(\psi), \gamma_+(v)\ra_{T^2(\Gamma_+)}
+\la \psi(\cdot,\cdot,0),S_0(\cdot,0) v(\cdot,\cdot,0)\ra_{L^2(G\times S)}.
\eea
The linear form $F_0:C^1(\ol G\times S\times I)\to\R$ is given by
\[
F_0(v)=
\la {f},v\ra_{L^2(G\times S\times I)}+
\la g, \gamma_-(v)\ra_{T^2(\Gamma_-)},
\]
and it admits a unique extension to a bounded linear form $F_0:H_1\to\R$.
Note that the bilinear form (\ref{csda27a}) is not necessarily coercive that is (\ref{csda30}) does not necessarily hold, which justifies the need for the change of unknown $\phi=e^{CE}\psi$ performed above.

We have the following immediate corollary for the existence of solutions of the original CSDA-problem.

\begin{corollary}\label{csdaco1}
Suppose that the assumptions  (\ref{ass1}), (\ref{ass5-a}), (\ref{ass7}), (\ref{ass-8}), (\ref{ass8-aa-a}) and (\ref{ass9-a-b}) (with $c>0$), (\ref{csda9}), (\ref{csda9aa}) and (\ref{csda9a}) are valid.
Let ${ f}\in L^2(G\times S\times I)$ and  ${ g}\in T^2(\Gamma_-)$.
Then the following assertions hold.

(i) The variational equation
\be\label{vareq1}
\tilde B_0(\tilde\psi,v)=F_0(v)\quad \forall v\in  H_2
\ee
has a solution $\tilde\psi=(\psi,q,p_0,p_m)\in H_1$.

Furthermore, $\psi\in\mc{H}_P(G\times S\times I)$ and it is a weak (distributional) solution of the equation (\ref{se1}).

(ii) Suppose that additionally the assumption ${\bf TC}$ holds. Then a solution $\psi$
of the equation \eqref{se1} obtained in part (i) is a solution of the problem 
\eqref{se1}, \eqref{se2}, \eqref{se3}. 

In addition, we have $q_{|\Gamma_+}=\gamma_+(\psi)$ and $p_0=\psi(\cdot,\cdot,0)$,
when $\tilde{\psi}=(\psi,q,p_0,p_m)$ is a solution in $H_1$ obtained in part (i).

(iii) Under the assumptions imposed in part (ii)
any solution $\psi\in \mc{H}_P(G\times S\times I^\circ)$
of the problem \eqref{se1}, \eqref{se2}, \eqref{se3}
is unique and obeys the estimate
\be\label{csda40aaa}
\n{\psi}_{{H_1}}\leq \frac{e^{CE_{\rm m}}}{c'}
\big(\n{{ f}}_{L^2(G\times S\times I)}+\n{{ g}}_{T^2(\Gamma_-)}\big).
\ee 
(Recall that $C$ is defined in \eqref{eq:def_C}, $c'$ in \eqref{cprime} and that $E_m$ is the cutoff energy.)
\end{corollary}

\begin{proof}
Let ${ f}\in L^2(G\times S\times I)$ and ${ g}\in T^2(\Gamma_-)$. Since $I$ is finite interval
we see that ${f_C}:=e^{CE}f\in L^2(G\times S\times I)$ and ${ g_C}:=e^{CE}g\in T^2(\Gamma_-)$,
where $C=\frac{\max\{q,0\}}{\kappa}$ (see \eqref{eq:def_C}, \eqref{eq:b_fg}).
By Theorem \ref{csdath3} the variational problem (\ref{csda40a})
has a solution $\tilde\phi=(\phi,q',p_0',p_m')\in H_1$,
from which one deduces without difficulty that $\tilde\psi=(\psi,q,p_0,p_m):=e^{-CE}\tilde\phi\in H_1$
is a solution of the variational problem (\ref{vareq1}).
Similarly, the fact that $\phi\in \mc{H}_P(G\times S\times I^\circ)$ implies that $\psi\in\mc{H}_P(G\times S\times I^\circ)$.
This can be seen by substituting $e^{-CE}v\in H_2$ instead of $v$ into (\ref{csda40a}).
Besides the estimate \eqref{csda40aaa}, all the claims are consequences of the corresponding items in Theorem \ref{csdath3}.

Recalling that ${ f_C}=e^{CE}f$,
${ g_C}=e^{CE}g$ the estimate (\ref{csda40aaa}) is obtained as follows,
\bea 
\n{\psi}_{H_1}
={}&\n{e^{-CE}\phi}_{H_1}
\leq
\n{\phi}_{H_1}\leq {{1}\over {c'}}   
\big(\n{{ f_C}}_{L^2(G\times S\times I)}+\n{{g_C}}_{T^2(\Gamma_-)}\big)
\nonumber\\
\leq{} &
 {{e^{CE_m}}\over {c'}}   
\big(\n{{ f}}_{L^2(G\times S\times I)}+\n{{ g}}_{T^2(\Gamma_-)}\big).
\eea
This completes the proof.
\end{proof}

\begin{remark}
A.
Let
\[
T\psi:=-{\p {(S_0\psi)}E}+\omega\cdot\nabla_x\psi+\Sigma\psi-K\psi
\]
and let
\[
T^*v=S_0{\p {v}E}-\omega\cdot\nabla_x v+\Sigma^*v-K^*v
\]
be the formal transpose of $T$. Furthermore,
let the assumptions 
(\ref{ass1}), (\ref{ass5-a}), (\ref{ass7}), (\ref{ass-8}), (\ref{ass8-aa-a}) and (\ref{ass9-a-b}), (\ref{csda9}), (\ref{csda9aa}) and (\ref{csda9a}) of Theorem \ref{csdath3} be valid,
and let $\tilde\psi=(\psi,q,p_0,p_m)\in H_1$ be a solution of (\ref{vareq1}) (guaranteed by Corollary \ref{csdaco1}). 
By (\ref{csda27a}), we find that for any $v\in C^1(\ol G\times S\times I)$ for which the {\it adjoint boundary conditions} $v(\cdot,\cdot,0)=0$ and $\gamma_+(v)=0$ hold,
we have
\be 
\la\psi,T^*v\ra=
\tilde B_0(\psi,v)=\la f,v\ra_{L^2(G\times S\times I)}+\la g,\gamma_-(v)\ra_{T^2(\Gamma_-)},
\ee
that is
\be 
\la\psi,T^*v\ra
-\la T\psi,v\ra_{L^2(G\times S\times I)}=\la g,\gamma_-(v)\ra_{T^2(\Gamma_-)}.
\ee
This means that $\tilde\psi\in H_1$ is a {\it weak solution of the boundary (initial) value problem} (\ref{se1}), (\ref{se2}), (\ref{se3}),
a terminology which goes back to \cite{lax}, \cite{sarason}. Also the terminology that boundary initial values are weakly valid is used.
The validity of the trace theorems and the Green formula (\ref{green-ex}) with $q=\gamma(\psi)$, $p_0=\psi(\cdot,\cdot,0)$,
$ p_m=\psi(\cdot,\cdot,E_m)$ are keys for obtaining well-defined solutions (that is, solutions for which the boundary 
(initial) values really hold). Trace theorems always demand geometrical treatments where the smoothness of the boundary $\partial G$ is essential.

B. We also remark that the following generalized Green formula (see e.g. \cite[Theorems 1 and 7]{rauch85})
\be\label{green-gene}
\la T \psi, v\ra_{L^2(G\times S\times I}
-\la T^* v,\psi\ra_{L^2(G\times S\times I}=({\s A}_\nu\psi)(v),
\quad \psi\in {\s H}_P(G\times S\times I),
\ee
for $v\in C^1(\ol G\times S\times I)$,
is valid where ${\s A}_\nu\psi$ is interpreted as an element of the dual $H^{1/2}(\partial (G\times S\times I))^*$.
The value ${\s A}_\nu\psi$ is obtained by (uniquely) extending the bilinear form (cf. (\ref{green-ex}))
\begin{multline}
({\s A}_\nu\psi)(v)
:=
\int_{\partial G\times S\times I}(\omega\cdot \nu) v\ \psi\ d\sigma d\omega dE\nonumber\\
+
\int_{G\times S}\big(S_0(\cdot,0)\psi(\cdot,\cdot,0)v(\cdot,\cdot,0)-S_0(\cdot,E_{\rm m})\psi(\cdot,\cdot,E_{\rm m})v(\cdot,\cdot,E_{\rm m})\big)dx d\omega,
\end{multline}
where $\psi,\ v\in C^1(\ol G\times S\times I)$.

\end{remark}

\begin{remark}\label{reg-re}
Define the transport operator
\[
T_C\phi:=
 -{\p {(S_0\phi)}E}+\omega\cdot\nabla_x\phi+CS_0\phi+(\Sigma-K_C)\phi.
\]
Then from the estimate (\ref{csda30}) it follows that
for all $\phi\in C^1(\ol G\times S\times I)$,
\[
c'\n{\phi}_{H_1}
\leq \n{T_C\phi}_{L^2(G\times S\times I)}+\n{\gamma_-(\phi)}_{T^2(\Gamma_-)}
+\n{S_0}_{L^\infty(G\times I)}\n{\phi(\cdot,\cdot,E_{\rm m})}_{L^2(G\times S)}.
\]
This \emph{a priori} estimate might be (as is standard)
an appropriate starting point for developing {\it regularity results of solutions} at least in the case where $G=\R^3$. 
For these kind of transport equations more regularity must be sought in mixed-norm (anisotropic) Sobolev-Slobodevskij spaces $H^{2,(s_1,s_2,s_3)}(G\times S\times I^\circ)$ where $s_j\geq 0$. 
Regularity results are needed e.g. in considerations of various approximation errors and in convergence analysis of numerical schemes.  
The regularity analysis remains open.
\end{remark}

\begin{remark}
The variational methods used above apply (after relevant modifications of the assumptions)
also to more general problems of the form 
\begin{gather}
a{\p \psi{E}}+ F\cdot\nabla_{\tilde\omega}\psi+\omega\cdot\nabla_x\psi+(\Sigma-K)\psi=f
\nonumber\\
\psi_{|\Gamma_-}=g,\quad \psi(\cdot,\cdot,E_{\rm m})=0, \label{add1}
\end{gather}
where $a=a(x,\omega,E)$, $F=(F_1(x,\omega,E),F_2(x,\omega,E),F_3(x,\omega,E))$
and $\nabla_{\tilde{\omega}}$ is the gradient operator with respect to the Riemannian metric on the sphere $S$ induced by the Euclidean metric on $\R^3$.
Notice that since $\nabla_{\tilde{\omega}}\psi$ is tangent to $S$, we may assume that $\omega\cdot F(x,\omega,E)=0$ for all $(x,\omega,E)$.
More explicitly, if $\phi\in C^1(S)$,
and if $\tilde{\phi}:\R^3\backslash\{0\}\to\R$ is given
by $\tilde{\phi}(y)=\phi(\frac{y}{\n{y}})$, then
\[
\nabla_{\tilde{\omega}}\phi =\nabla_y\tilde{\phi}-(\omega\cdot\nabla_y\tilde\phi)\omega.
\]

In this context the space $H_2$ must be replaced  with
the space ${\s H}_2$ defined as follows.
The space
${\s H}_2$ is the completion of $C^1(G\times S\times I)$ with respect to the inner product\
\[
\la \psi,v\ra_{H_2}+\la F\cdot\nabla_{\tilde\omega}\psi,F\cdot\nabla_{\tilde\omega}v\ra_{L^2(G\times S\times I)}.
\]
This observation is based on the fact that
\[
2\la F\cdot\nabla_{\tilde\omega}\psi,\psi\ra_{L^2(G\times S\times I)}=\la d\psi,\psi\ra_{L^2(G\times S\times I)},
\]
where
\[
d:=-\mathrm{div}_{\tilde{\omega}}(F),
\]
and $\mathrm{div}_{\tilde{\omega}}$ is the divergence operator on $S$ with respect to its Riemannian metric.
 
The equation (\ref{add1}) in velocity coordinates $(x,v)\in\R^3\times\R^3$ can be written as
\be\label{add2}
\tilde F\cdot\nabla_v\Psi+v\cdot\nabla_x\Psi+(\tilde\Sigma-\tilde K)\Psi=\tilde f,
\ee
which is known as the (linear) Vlasov-Boltzmann equation. 

Similar observations concern the  methods used in the next sections. Nevertheless, the $m$-dissipativity
of the operator $\psi\mapsto {1\over a}( F\cdot\nabla_{\tilde\omega}\psi+\omega\cdot\nabla_x\psi+C\psi)$, for $C$ large enough (in relevant modified spaces), requires extra analysis.

Finally notice that no boundary condition is needed in \eqref{add1} with respect to $\omega$-variable
because the unit sphere $S$ is a compact manifold without boundary.
\end{remark}

\subsection{Existence Results Based on $m$-dissipativity}\label{m-d}

In this section we apply an alternative method based on the  results of  dissipative first-order partial differential operators. 
Let
\[
P(x,\omega,E,D)\phi := -{\p {(S_0\phi)}E}+\omega\cdot\nabla_x\phi
+CS_0\phi
\]
where $S_0\in C^1(\ol G\times S\times I)$ and $C$ is constant.
Recall that the {\it formal transpose (adjoint) } of $P(x,\omega,E,D)$ is
\be
P'(x,\omega,E,D)v
:=
S_0{\p {v}{E}}-\omega\cdot\nabla_x v+CS_0v.
\ee
It should be pointed out that here the notation for $P(x,\omega,E,D)$ and $P'(x,\omega,E,D)$ differs from that used in section \ref{esols}
in that here the operators $P(x,\omega,E,D)$ and $P'(x,\omega,E,D)$ also include the term $CS_0$.

Define linear operators $P,P':L^2(G\times S\times I)\to L^2(G\times S\times I)$ with domains
of definition $D(P)$, $D(P')$ by setting
\be
D(P):=D^1(\ol G\times S\times I),\quad P\phi:=P(x,\omega,E,D)\phi,
\ee
and
\[
D(P'):=C_0^\infty(G\times S\times I^\circ),\quad P'v:=P'(x,\omega,E,D)v.
\]
Clearly both $P$ and $P'$ are densely defined.

Let $P'^*:L^2(G\times S\times I)\to L^2(G\times S\times I)$ be the adjoint operator of $P'$.
Then $\phi\in L^2(G\times S\times I)$ is said to be a \emph{weak solution} of
\be\label{d1}
P\phi=f,\quad f\in L^2(G\times S\times I)
\ee
if and only if
\[
\phi\in D(P'^*)\quad \mathrm{and}\quad P'^*\phi=f.
\] 
Since the adjoint $P'^*$ is a closed operator, the space 
\begin{multline}
{\s H}_{P}(G\times S\times I^\circ):=\{\phi\in L^2(G\times S\times I)\ | \\
P(x,\omega,E,D)\phi\in L^2(G\times S\times I)\ {\rm in\ the \ weak\ sense}\}
\end{multline}
is a Hilbert space when equipped with the inner product
\[
\la \phi,v\ra_{{\s H}_{P}(G\times S\times I^\circ)}:=\la \phi,v\ra_{L^2(G\times S\times I)}+\la P(x,\omega,E,D)\phi,P(x,\omega,E,D)v\ra_{L^2(G\times S\times I)}.
\]
Notice that
\begin{align}\label{eq:H_P_is_D-P-dot-star}
{\s H}_{P}(G\times S\times I^\circ)=D(P'^*),
\end{align}
when $D(P'^*)$ is equipped with the graph norm of $P'^*$.

We say that $\phi\in L^2(G\times S\times I)$
is a {\it strong solution of (\ref{d1}) (without boundary conditions)}  if there exists a sequence 
$\{\phi_n\}\subset D^1(\ol G\times S\times I)$ ($=D(P)$) such that
\[
\n{\phi-\phi_n}_{L^2(G\times S\times I)}+\n{P\phi_n-f}_{L^2(G\times S\times I)}\longto 0\quad {\rm when}\ n\to\infty.
\]
Let $\tilde P:L^2(G\times S\times I)\to L^2(G\times S\times I)$ be the smallest closed extension (closure) of $P$. Then one sees that $\phi$ is a strong solution (without boundary conditions) if and only if $\phi\in D(\tilde P)$ and $\tilde P\phi=f$.

\begin{remark}\label{re:P_closure}
To see that the closure $\tilde{P}$ exists, notice that if $\phi_n\in D(P)$, $f\in L^2(G\times S\times I)$
and $\phi_n\to 0$, $P\phi_n\to f$ in $L^2(G\times S\times I)$, then for every $v\in C_0^\infty(G\times S\times I^\circ)$
we have
\[
\la f,v\ra_{L^2(G\times S\times I)}={}&\lim_{n\to\infty} \la P(x,\omega,E,D)\phi_n,v\ra_{L^2(G\times S\times I)} \\
={}&\lim_{n\to\infty} \la \phi_n, P'(x,\omega,E,D)v\ra_{L^2(G\times S\times I)}=0,
\]
which implies $f=0$.
\end{remark}

One immediately sees that every strong solution of (\ref{d1}) is its weak solution.
When the boundary $\partial G$ is smooth enough the converse is also true, the result which goes back to {\it Friedrich } \cite{friedrich44} (the main theorem on page 135 together with Theorem 4.2, p. 144); see also \cite{rauch85}.

\begin{theorem}\label{density}
Let $G\subset\R^n$ be an open bounded subset, lying on one side of its boundary,
and with boundary of class $C^1$.
Then every weak solution  of the equation (\ref{d1}) is its strong solution (without boundary conditions).
In other words, $D^1(\ol G\times S\times I)$ is dense in ${\s H}_{P}( G\times S\times I^\circ)$.
\end{theorem}

\begin{proof}
See \cite{friedrich44}, or \cite[Proposition 1]{rauch85}.
\end{proof}

The claim of the above theorem  can also be equivalently stated as
\[
\tilde{P}=P'^*.
\]
 
We still formulate an existence result of strong solutions concerning homogeneous inflow boundary (initial) value problems. 
To this end, the following modified definition of the strong solution is needed.
One says that $\phi\in L^2(G\times S\times I)$
is a {\it strong solution of (\ref{d1}) with homogeneous inflow boundary conditions} if there exists a sequence 
$\{\phi_n\}\subset \tilde W^2(G\times S\times I)\cap H^1(I,L^2(G\times S))$ 
such that
\[
\n{\phi-\phi_n}_{L^2(G\times S\times I)}+\n{P(x,\omega,E,D)\phi_n-f}_{L^2(G\times S\times I)}\longto 0,\quad {\rm when}\ n\to\infty,
\]
and
\be\label{inf1}
{\phi_n}_{|\Gamma_{-}}=0,\quad \phi_n(\cdot,\cdot,E_{\rm m})=0.
\ee
Define a linear operator $P_{0}:L^2(G\times S\times I)\to L^2(G\times S\times I)$ by 
\bea
D(P_{0}):={}&\{\phi\in  \tilde W^2(G\times S\times I)\cap H^1(I,L^2(G\times S))\ |\ \phi_{|\Gamma_{-}}=0,\
\phi(\cdot,\cdot,E_{\rm m})=0\}\nonumber\\[2mm]
P_{0}\phi:={}&P(x,\omega,E,D)\phi.
\eea
Let $\tilde P_{0}:L^2(G\times S\times I)\to L^2(G\times S\times I)$ be the smallest closed extension (closure) of $P_{0}$
(the existence of which can be seen using the argument of Remark \ref{re:P_closure}).
Then one sees that $\phi$ is a strong solution with the homogeneous inflow boundary conditions  if and only if $\phi\in D(\tilde P_{0})$ and $\tilde P_{0}\phi=f$.
Moreover, one sees that every  strong solution with homogeneous inflow boundary conditions
is a weak solution of $P(x,\omega,E,D)\phi=f$, that is
\begin{align}\label{eq:P_0_in_P-dot-star}
\tilde P_{0}\subset P'^*.
\end{align}

Since  $\tilde P_{0}$ is a closed operator, the space
\be\label{eq:H_tilde-P_0_is_D-tilde-P_0}
{\s H}_{P_{0}}(G\times S\times I^\circ):=
D(\tilde P_{0})
\ee
is a Hilbert space when equipped with the inner product
\[
\la \phi,v\ra_{{\s H}_{P_0}(G\times S\times I^\circ)}:=\la \phi,v\ra_{L^2(G\times S\times I)}+\la \tilde{P}_0\phi,\tilde{P}_0v\ra_{L^2(G\times S\times I)}.
\]

\begin{remark}
When
$\phi\in  {\s H}_{P_{0}}(G\times S\times I^\circ)$ 
we say that the (homogeneous) initial and boundary conditions
\[
\phi_{|\Gamma_-}=0,\quad \phi(\cdot,\cdot,E_{\rm m})=0
\]
are valid in the strong sense.
\end{remark}

We show the $m$-dissipativity of $\tilde{P}_0$ using the theory of evolution operators presented in section \ref{evcsd} below.

\begin{theorem}\label{md-evoth}
Suppose that 
\begin{align}
{}& S_0\in C^2(I,L^\infty(G)),\label{evo16} \\[2mm]
{}& \kappa:=\inf_{(x,E)\in \ol{G}\times I}S_0(x,E)>0, \label{evo8-a} \\[2mm]
{}& \nabla_xS_0\in  L^\infty(G\times I), \label{evo9-a} \\[2mm]
{}& -{\p {S_0}E}+2CS_0\geq 0. \label{flp1-ab}
\end{align}
Then 
\be\label{inf4-a}
R(I+\tilde P_0)=L^2(G\times S\times I)
\ee
and
\be\label{inf5-a}
\la \tilde P_0\phi,\phi\ra_{L_2(G\times S\times I)}\geq 0,\quad \forall\phi\in D(\tilde P_0).
\ee 
\end{theorem}

\begin{proof}
We apply Theorem \ref{evoth1} (see below) with $K=0,\ \Sigma=CS_0+1,\ g=0$.
Let $f\in C_0^\infty(G\times S\times I^\circ)$.
Then it follows that the problem
\be 
-{\p {(S_0\phi)}E}+\omega\cdot\nabla_x\phi
+(CS_0+1)\phi=
(I + P(x,\omega,E,D))\phi=f,\quad \phi(\cdot,\cdot,E_{\rm m})=0,
\ee
has a unique solution 
$\phi\in C(I,\tilde{W}^2_{-,0}(G\times S))\cap C^1(I,L^2(G\times S))$. 
We find that
\[
\big\{\psi\in C(I,\tilde W^2_{-,0}(G\times S))\cap C^1(I,L^2(G\times S))\ \big|\ \psi(\cdot,\cdot,E_m)=0\big\}\subset D(P_0),
\]
and so for any $f\in C_0^\infty(G\times S\times I)$ the equation
$(I+ P_0)\phi=f$ has a solution. Since $C_0^\infty(G\times S\times I^\circ)$ is dense in $L^2(G\times S\times I)$ we find that
the range $R(I+ P_0)$ is dense that is,
\be\label{evo17-a}
\ol{R(I+ P_0)}=L^2(G\times S\times I).
\ee

As in the proof of Lemma \ref{csdale0} (note that here the assumptions are somewhat weaker),
we have for all $\phi\in D(P_0)$, (we write $L^2=L^2(G\times S\times I)$)
\bea\label{ineq}
&
\la P_0\phi,\phi\ra_{L^2}
=\la -{\p {(S_0\phi)}E},\phi\ra_{L^2}
+\la \omega\cdot\nabla_x\phi,\phi\ra_{L^2}+
\la CS_0\phi,\phi\ra_{L^2}
\nonumber\\
={}&
\la \Big(-\frac{1}{2}\p{S_0}{E}+CS_0\Big)\phi,\phi\ra_{L^2}
+
\frac{1}{2}\la \phi,\phi\ra_{T^2(\Gamma_+)}^2
+{1\over 2}\int_{G\times S}S_0(\cdot,0)\phi(\cdot,\cdot,0)^2 dx d\omega,
\eea
which in combination with the assumptions \eqref{evo8-a}, \eqref{flp1-ab} implies
\[
\la P_0\phi,\phi\ra_{L^2(G\times S\times I)}\geq 0,\quad \forall \phi\in D(P_0).
\]

If $\phi\in D(\tilde{P}_0)$, choose a sequence $\phi_n\in D(P_0)$ such that $\phi_n\to \phi$ and $P_0\phi_n\to \tilde{P}_0\phi$
in $L^2(G\times S\times I)$ when $n\to\infty$.
By the above inequality, we have
\[
\la \tilde{P}_0\phi,\phi\ra_{L^2(G\times S\times I)}
=\lim_n \la P_0\phi_n,\phi_n\ra_{L^2(G\times S\times I)}\geq 0,
\]
which gives \eqref{inf5-a}.

Finally, from \eqref{inf5-a}
it follows that
\be\label{evo18-ab}
\n{(I+\tilde P_0)\phi}_{L^2(G\times S\times I)}\geq \n{\phi}_{L^2(G\times S\times I)},\quad \forall \phi\in D(\tilde P_0).
\ee
and therefore, since the operator $I+\tilde P_0$ is closed the range  $R(I+\tilde P_0)$ is closed in $L^2(G\times S\times I)$.
This result,
the observation that $R(I+P_0)\subset R(I+\tilde{P}_0)$,
and \eqref{evo17-a} show that $R(I+\tilde P_0)=L^2(G\times S\times I)$.
The proof is complete.
\end{proof}

Theorem \ref{md-evoth} says that:

\begin{corollary}\label{d-cor}
The operator $-\tilde P_0:L^2(G\times S\times I)\to L^2(G\times S\times I)$  is $m$-dissipative, or equivalently $\tilde P_0$ is $m$-accretive.
\end{corollary}

\begin{proof}
See e.g. \cite[p. 340]{dautraylionsv5}.
\end{proof}

\begin{remark}\label{F-L-P-R}
The general theory of initial boundary value problems of symmetric formally dissipative first order partial differential operators can alternatively be applied to show the $m$-dissipativity of $-\tilde P_0$.
The (classical) results for positive symmetric initial boundary value problems can be found in \cite{lax}, Theorem 3.2 and discussion in section 4 therein;  \cite{friedrich58}, together with discussion in section 17 therein;  \cite{sarason} and  \cite{rauch85}. 
The spatial domain there is replaced with $G\times S$
(here the additional smooth compact manifold $S$ without boundary does not affect to the conclusions; for the construction of 
the Friedrich's mollifier on $S$ we refer to \cite{fukuoka}).
The above references are, however, valid only for problems in which the dimension of the  kernel ${\rm Ker}(A_\nu)$, given below,
is constant on $\Gamma$ (i.e. $A_\nu$ has constant multiplicity),
and thus we are not able to directly apply them.
Some results for variable 
multiplicity can be found  e.g. in \cite{rauch94}, \cite{nishitani96}, \cite{nishitani98} and \cite{takayama02} together with their references. The 
latter researches require
an additional assumption which concerns the "transition with a non-zero derivative" on the smooth $(6-2)$-dimensional 
manifold $\Gamma_0$. For simplicity we deal only with the case where $S_0=S_0(x)$ is independent of $E$ and use formulations of 
\cite{nishitani96}. More general formulations can be based on \cite{nishitani98} and related references but we omit them here.

We make the  change of variables and of the unknown function
\be\label{evo2} 
\Phi(x,\omega,E):=\phi(x,\omega,E_m-E)
\ee
and denote 
\[
\tilde f(x,\omega,E)=f(x,\omega,E_{\rm m}-E).
\]
Making the changes, we find that the problem 
$(I+\tilde P_0)\phi=f$ is equivalent to
$\phi$ satisfying the equation
\be\label{evo3}
{\p {\Phi}E}+{1\over{S_0}}\omega\cdot\nabla_x\Phi
+{1\over{ S_0}}(C S_0+1)\Phi
={1\over{ S_0}}\tilde f=:{\ol f},
\ee
on $G\times S\times I$,
along with satisfying the following inflow boundary and initial value conditions
\be\label{evo4}
\Phi_{|\Gamma_-}=0, \quad
\Phi(\cdot,\cdot,0)=0. 
\ee
Define
\[
Q(x,\omega,E,D)\Phi:={1\over{ S_0}}\omega\cdot\nabla_x\Phi
+{1\over{ S_0}}(C S_0+1)\Phi
\]
and the boundary matrix
\[
A_{\nu}(z)={1\over{ S_0}}\omega\cdot\nu(y),\quad z=(y,\omega,E)\in\Gamma.
\]

We assume that $\partial G=\{x\in\R^3\ |\ f(x)=0\}$ is a level set of class $C^\infty$, i.e. $f\in C^\infty(\R^3)$ and $\nabla_x f(x)\neq 0$ whenever $f(x)=0$,
and we also assume that $S_0\in C^\infty(G)$.
The outward unit normal vector field $\nu(y)$ on $\partial G$ is given by
\[
\nu(y)={{(\nabla_x f)(y)}\over{\n{(\nabla_x f)(y)}}}.
\]
The additional assumption is as follows. 
Suppose that there exist $C^\infty$-functions $h$ and $A:V_{y_0}\times S\to \R$ such that
\begin{align}
& A_f(x,\omega):={1\over{S_0(x)}}\omega\cdot \nabla_x f(x)=-h(x,\omega)A(x,\omega) \label{ass-on-G-1}  \\[2mm]
& (V_{y_0}\times S)\cap\Gamma_+'=\{(y,\omega)\in\partial G\times S\ |\ h(y,\omega)>0\}, \label{ass-on-G-2} \\[2mm]
& (V_{y_0}\times S)\cap\Gamma_-'=\{(y,\omega)\in\partial G\times S\ |\ h(y,\omega)<0\}, \label{ass-on-G-3} \\[2mm]
& (V_{y_0}\times S)\cap\Gamma_0'=\{(y,\omega)\in\partial G\times S\ |\ h(y,\omega)=0\}. \label{ass-on-G-4}
\end{align}
Let
\[
A_h(x,\omega):={1\over{S_0(x)}}(\omega\cdot \nabla_x h)(x,\omega).
\]
The additional assumption is
\be\label{add-ass}
A(y,\omega)\ {\rm and}\ A_h(y,\omega)\ {\rm are\ positive\ definite\ on}\   (V_{y_0}\times S)\cap \Gamma_0'.
\ee

For example, in the case of the ball $G=B(0,1)\subset\R^3$ we can choose $f(x)=1-\n{x}^2$. Then $A_f(y,\omega)=-2{1\over{S_0(y)}}(y\cdot \omega)$
and we choose $h(y,\omega)=2(y\cdot\omega)$, $A(y,\omega)={1\over{S_0(y)}}>0$. Moreover, we find that $A_h(y,\omega)=2{1\over{S_0(y)}}>0$. Hence the stated assumptions can be met for the case of the ball $G=B(0,1)$.

The following result holds in this context.
Suppose that $\partial G$ is  in the class $C^1$ and that the assumption (\ref{add-ass}) holds. Furthermore,
suppose that $S_0\in C^1(\ol G\times I)$ such that
\begin{align}
S_0>0\quad {\rm on}\ \ol G\times I. \label{flp2}
\end{align}
Then 
\be\label{inf4}
C_0^\infty (G\times S\times I^\circ)\subset R(I+P_0) 
\ee
and
\be\label{inf5}
\la \tilde P_0\phi,\phi\ra_{L_2(G\times S\times I)}\geq 0,\quad \forall\phi\in D(\tilde P_0).
\ee

Since our equation is scalar valued it is symmetric in the sense of \cite{nishitani96}.
We choose the linear subspace of \cite{nishitani96} by
\[
M(y,\omega)=\begin{cases}
\R,\quad & (y,\omega)\in \Gamma'_+\cup\Gamma'_0\\
\{0\},\quad & (y,\omega)\in\Gamma'_-
\end{cases}.
\]
Then we find that $M(y,\omega)$ is maximal positive in the sense of \cite{nishitani96}.

Due to the Theorem 5.5 of \cite{nishitani96}, for any $f\in C_0^\infty(G\times S\times I^\circ)$
the equation ${\p \Phi{E}}+Q(x,\omega,E,D)\Phi={\ol f}$ has a unique strong solution $\Phi$ satisfying the initial and boundary values.
Then $\phi(x,\omega,E):=\Phi(x,\omega,E_{\rm m}-E)$ is the solution of $(I+P_0)\phi=f$.
This completes the proof of (\ref{inf4}).
The inequality (\ref{inf5}) can be shown similarly as above and so the conclusion follows.
 
\end{remark}

\begin{remark}\label{md-gene}
The previous theorem has the following generalization which can be applied for more general transport problems.
Let
\[
P(x,\omega,E,D)\phi= -S_0{\p {\phi}E}+F_1\cdot\nabla_x\phi
+F_2\cdot\nabla_{\tilde\omega}\phi+a\phi
\]
be the first order partial differential operator with coefficients $S_0,\ F_1,\ F_2,\ a\in C^1(\ol G\times S\times I)$.
Assume that $\partial G$ is  in the class $C^2$ and that $\Gamma=\Gamma_{+}\cup\Gamma_{-}\cup\Gamma_{0}$ where $\Gamma_{0}$ has the zero surface measure.
In addition, we assume that $\Gamma_{\pm}$ are open in $\partial G\times S\times I^\circ$.
Finally,
suppose that 
\begin{align}
& {\p {S_0}E}-{\rm div}_x(F_1)-{\rm div}_{\omega}(F_2)+2a\geq 0, \label{flp1-aa} \\[2mm]
& \inf_{(x,\omega,E)\in G\times S\times I}S_0(x,\omega,E)>0, \label{flp2-aa}
\end{align}
and that
\be\label{flp3-aa}
F_1\cdot\nu<0\ {\rm on}\ {\Gamma}_{-},\quad F_1\cdot\nu>0\ {\rm on}\ {\Gamma}_{+}.
\ee
Then 
\be\label{inf4-aa}
R( I+\tilde P_0)=L^2(G\times S\times I)
\ee
and
\be\label{inf5-aa}
\la \tilde P_0\phi,\phi\ra_{L_2(G\times S\times I)}\geq 0,\quad \forall \phi\in D(\tilde P_0).
\ee
\end{remark}

\begin{remark}\label{dis-re}
In certain cases, the $m$-dissipativity of $-\tilde P_0$ -like operator can even
be proved by using explicit formulas for the solution. We only sketch the idea here.

Suppose  that $S_0=S_0(E)$ is independent of $x$ and that $a=a(x,\omega)\in C^1(\ol G\times S)$ is independent of $E$.
Let
\[
P(x,\omega,E,D)\phi=-{\p {(S_0\phi)}E}+\omega\cdot\nabla_x \phi+a\phi.
\]
Assume that ${ f}\in C_0^\infty(G\times S\times I^\circ)$.
Then by the formula (\ref{comp17:1}) below,
the solution of the problem
\be\label{S0const}
P(x,\omega,E,D)\phi={f},\quad \phi_{|\Gamma_-}=0,\quad \phi(\cdot,\cdot,E_{\rm m})=0
\ee
is given by
\be\label{comp17:1-a}
\phi(x,\omega,E) 
=
{1\over{S_0(E)}}\Big(
\int_0^{r(x,\omega,E)}
e^{-\int_0^s a(x-\tau \omega,\omega)d\tau}\tilde{f}(x-s\omega,\omega,R(E)+s) ds\Big),
\ee
where
\[
r(x,\omega,E):={}&\min\{R(E_m)-R(E),t(x,\omega)\}, \\[2mm]
R(E):={}&\int_0^E{1\over{S_0(\tau)}}d\tau, \\[2mm]
\tilde f(x,\omega,\eta):={}& S_0(R^{-1}(\eta))f(x,\omega,R^{-1}(\eta)).
\]
We find that there exists a constant $C_1>0$ such that
\be\label{inv}
\n{\phi}_{L^2(G\times S\times I)}\leq C_1\n{f}_{L^2(G\times S\times I)}.
\ee

Let ${ f}\in L^2(G\times S\times I)$ and let $\{f_n\}\subset C_0^\infty(G\times S\times I^\circ)$ be a sequence such that
$\n{f_n-f}_{L^2(G\times S\times I)}\to 0$ when $n\to\infty$. 
Define
\[
\phi_n:=
{1\over{S_0(E)}}\Big(
\int_0^{r(x,\omega,E)}
e^{-\int_0^sa(x-\tau s\omega,\omega)d\tau}\tilde{f_n}(x-s\omega,\omega,R(E)+s) ds\Big).
\]
We find that 
\[
\phi_n\in C^0(\ol G\times S\times I)\cap H^1(G\times S\times I^\circ),\quad {\phi_n}_{|\Gamma_-}=0,\quad \phi_n(\cdot,\cdot,E_{\rm m})=0,
\]
and then $\phi_n\in D( P_{0})$.
In showing that $\phi_n\in H^1(G\times S\times I^\circ)$ notice that
\[
\tilde{f}_n(x-t(x,\omega)\omega,\omega,R(E)+t(x,\omega))=0,\quad \textrm{if}\ R(E_m)-R(E)>t(x,\omega),
\]
and so ${\p t{x_j}}$, ${\p t{\tilde\omega_j}}$ do not appear in 
${\p {\phi_n}{x_j}}$, ${\p {\phi_n}{\tilde\omega_j}}$.

By (\ref{inv}) there exists $\phi\in L^2(G\times S\times I)$ such that $\n{\phi_n-\phi}_{L^2(G\times S\times I)}\to 0$.
In addition, $\n{P_0\phi_n-f}_{L^2(G\times S\times I)}=\n{f_n-f}_{L^2(G\times S\times I)}\to 0,\ n\to\infty$,
which shows that $\phi\in D(\tilde P_0)$ and $\tilde P_0\phi=f$. Hence $R(\tilde P_0)=L^2(G\times S\times I)$.
Consequently, in this special case these methods (based only on explicit solution formulas) give an alternative proof for the surjectivity of $\tilde P_{0}$ (and hence together with (\ref{inf5}) the $m$-dissipativity of 
$-\tilde P_{0}$; cf. the treatise of $A_0$ in \cite[proof of Theorem 4.7]{tervo17-up}).
\end{remark}

We return to the existence and uniqueness of solutions for the following problem.
Given ${ f_C}\in L^2(G\times S\times I)$, find $\phi\in L^2(G\times S\times I)$ such that
\begin{gather}
-{\p {(S_0\phi)}E}+\omega\cdot\nabla_x\phi+CS_0\phi +\Sigma\phi -K_C\phi ={ f_C},
\nonumber\\
{\phi}_{|\Gamma_-}=0,\quad\ \phi(\cdot,\cdot,E_{\rm m})=0. \label{co3-d}
\end{gather}
Let (for clarity, we have included here the subscript $C$ into $P$)
\bea
P_C(x,\omega,E,D)\phi
:={}&
-{\p {(S_0\phi)}E}+\omega\cdot\nabla_x\phi
+CS_0\phi\nonumber\\
={}&
-S_0{\p {\phi}E}+\omega\cdot\nabla_x\phi -{\p {S_0}E}\phi
+CS_0\phi.
\eea

We shall seek a strong solution of \eqref{co3-d}
with the homogeneous inflow boundary conditions.
Using the above notations, the problem \eqref{co3-d} is equivalent to
\[
(\tilde{P}_{C,0}+\Sigma-K_C)\phi={ f_C},
\]
where $\phi\in D(\tilde{P}_{C,0})$.
The following arguments are analogous to those used in \cite[Section 5.3]{tervo17-up}.

\begin{theorem}\label{coth2-d}
Suppose that the assumptions 
(\ref{ass1}), (\ref{ass5-a}), (\ref{ass7}), (\ref{ass-8}), (\ref{ass8-aa-a}), (\ref{ass9-a-b}), (\ref{evo16}), (\ref{evo8-a}) and (\ref{evo9-a}) are valid with $C={{\max\{q,0\}}\over{\kappa}}$ and $c>0$.
Then for every ${ f_C}\in L^2(G\times S\times I)$ the problem \eqref{co3-d}
has a unique strong solution $\phi\in  D(\tilde P_{C,0})={\s H}_{P_{C,0}}(G\times S\times I^\circ)$ with homogeneous inflow boundary conditions.
\end{theorem}

\begin{proof}
Recall that $q={1\over 2}\sup_{(x,E)\in G\times I}{\p {S_0}E}(x,E)$ (see \eqref{q}).
For $C={{\max\{q,0\}}\over{\kappa}}$ we have
\[
{\p {S_0}E}\leq 2q=2\kappa{q\over\kappa}\leq 2S_0{{\max\{q,0\}}\over\kappa}=2S_0C,
\]
that is
\be\label{flp1-d}
-{\p {S_0}E}+2CS_0\geq 0,
\ee 
and hence
by Corollary \ref{d-cor}, the operator $-{\tilde P}_{C,0}:L^2(G\times S\times I)\to L^2(G\times S\times I)$ is $m$-dissipative.
On the other hand, due to Lemma \ref{csdale1a} the bounded operator
$-(\Sigma-K_C)+c I:L^2(G\times S\times I)\to L^2(G\times S\times I)$ is dissipative.
These facts allow us to conclude that
$-{\tilde P}_{C,0}-(\Sigma-K_C)+cI :L^2(G\times S\times I)\to L^2(G\times S\times I)$ is $m$-dissipative
(\cite[Theorem 4.3 and Corollary 3.3]{pazy83}, or \cite[Theorem 4.4]{tervo17-up}).
This implies, as $c>0$, that $R\big(cI-(-{\tilde P}_{C,0}-(\Sigma-K_C)+cI)\big)=R({\tilde P}_{C,0}+\Sigma-K_C)=L^2(G\times S\times I)$, and so the existence of solutions follows.

Because $c>0$ and because $-{\tilde P}_{C,0}-(\Sigma-K_C)+cI$ is dissipative,
we have 
\be\label{eq:uniquesol-d}
\n{({\tilde P}_{C,0}+\Sigma-K_C)\phi}_{L^2(G\times S\times I)}\geq c\n{\phi}_{L^2(G\times S\times I)},\quad \forall\phi\in D(\tilde P_{C,0}),
\ee
which implies the uniqueness of the solution.
This completes the proof.
\end{proof}

\begin{remark}
We note that the inequality \eqref{eq:uniquesol-d} implies that for all ${ f_C}\in L^2(G\times S\times I)$
\bea\label{bestim-d}
\n{(-{\tilde P}_{C,0}+\Sigma-K_C)^{-1}{ f_C}}_{ L^2(G\times S\times I)}\leq {1\over c}\n{{ f_C}}_{ L^2(G\times S\times I)},
\eea
or in other words, the solution of the problem (\ref{co3-d}) satisfies
\be\label{bestim1-d}
\n{\phi}_{L^2(G\times S\times I)}
\leq {1\over c}\n{{ f_C}}_{L^2(G\times S\times I)}.
\ee
\end{remark}

The next results addresses the case with inhomogeneous inflow boundary data.
We begin with a lemma (see Lemmas 5.8 and 5.11 in \cite{tervo17-up}).

\begin{lemma}\label{le:H1_lift}
Let $d:=\mathrm{diam}(G)<\infty$ (diameter of $G$). Then for any $g\in H^1(I,T^2(\Gamma_-'))$
(recall $\Gamma_-'$ from section \ref{pre})
we have $Lg\in H^1(I, L^2(G\times S))$,
where
\[
(Lg)(x,\omega,E):=g(x-t(x,\omega)\omega,\omega,E),
\]
is the lift of $g$ (see Lemma \ref{le:lift} and Eq. \eqref{liftb}),
and
\[
\n{Lg}_{H^1(I, L^2(G\times S))}\leq \sqrt{d}\n{g}_{H^1(I,T^2(\Gamma_-'))}.
\]
\end{lemma}

\begin{proof}
Taking into account that $|\tau_-|\leq d$ on $\Gamma_-$,
we have by Lemma \ref{le:lift} (with $\Sigma=0$),
\[
\n{Lg}_{L^2(G\times S\times I)}\leq 
\sqrt{d}\n{g}_{T^2(\Gamma_-)} = \sqrt{d}\n{g}_{L^2(I,T^2(\Gamma'_-))}.
\]
On the other hand,
$\p{(Lg)}{E}=L\p{g}{E}$,
and therefore one has, again by Lemma \ref{le:lift} (with $\Sigma=0$),
\[
\n{\p{(Lg)}{E}}_{L^2(G\times S\times I)}\leq \sqrt{d}\n{\p{g}{E}}_{L^2(I,T^2(\Gamma'_-))}.
\]
This completes the proof.
\end{proof}

\begin{theorem}\label{coth3-dd}
Suppose that the assumptions
(\ref{ass1}), (\ref{ass5-a}), (\ref{ass7}), (\ref{ass-8}), (\ref{ass8-aa-a}), (\ref{ass9-a-b}), (\ref{evo16}), (\ref{evo8-a}) and (\ref{evo9-a}) are valid with $C={{\max\{q,0\}}\over{\kappa}}$ and $c>0$.
Furthermore, suppose that ${ f_C}\in L^2(G\times S\times I)$, and ${g_C}\in H^1(I,T^2(\Gamma_-'))$
is such that the \emph{compatibility condition}
\be\label{comp-dd}
{ g_C}(\cdot,\cdot,E_{\rm m})=0
\ee
holds.
Then the problem
\begin{gather}
-{\p {(S_0\phi)}E}+\omega\cdot\nabla_x\phi+CS_0\phi+\Sigma\phi -K_C\phi={ f_C},\ \nonumber\\
{\phi}_{|\Gamma_-}={ g_C},\quad \phi(\cdot,\cdot,E_{\rm m})=0, \label{co3aa-dd}
\end{gather}
has a unique solution 
\be\label{solin}
\phi\in {\s H}_{P_{C,0}}(G\times S\times I^\circ)+
L\big(H^1(I,T^2(\Gamma_-'))\big)\subset
{\s H}_{P_C}(G\times S\times I^\circ),
\ee
\end{theorem}

\begin{proof}
Substitute $\phi$ in the problem (\ref{co3aa-dd}) by $u:=\phi-L{g_C}$
to obtain
\[
&
-{\p {(S_0u)}E}+ \omega\cdot\nabla_x u+CS_0u+\Sigma u-K_C u,
\nonumber\\
={}&
{ f_C}+{\p {(S_0L{ g_C})}E}
-CS_0(L{ g_C})-\Sigma(L{ g_C}) +K_C(L { g_C})=:\tilde{ f_C}(x,\omega,E),\nonumber
\]
where we have used the fact that $\omega\cdot\nabla_x(L{ g_C})=0$ (see Lemma \ref{trathle1} with $\Sigma=0$).
On the other hand,
\[
u_{|\Gamma_-}={\phi}_{|\Gamma_-}-(L{g_C})_{|\Gamma_-}={g_C}-{ g_C}=0,
\]
and, using the compatibility condition (\ref{comp-dd}),
\[
u(\cdot,\cdot,E_{\rm m})=\phi(\cdot,\cdot,E_{\rm m})-{ g_C}(x-t(x,\omega)\omega,\omega,E_{\rm m})=0.
\]

We find that the assumptions guarantee that $\tilde { f_C}\in L^2(G\times S\times I)$ (for more details, see the proof of Corollary \ref{cdd} below),
and hence by Theorem \ref{coth2-d} the problem (for $u$)
\begin{gather*}
-{\p {(S_0u)}E}+ \omega\cdot\nabla_x u+CS_0u+\Sigma u -K_C u=\tilde{f_C}, \\
u_{|\Gamma_-}=0,\quad u(\cdot,\cdot,E_{\rm m})=0,
\end{gather*}
has a unique solution $u\in {\s H}_{P_{C,0}}(G\times S\times I^\circ)$.
It follows that then
\[
\phi:=u+L{ g_C}\in {\s H}_{P_{C,0}}(G\times S\times I^\circ) + L\big(H^1(I,T^2(\Gamma_-'))\big),
\]
is the wanted unique solution of (\ref{co3aa-dd}).

Finally, the last inclusion in \eqref{solin} is justified
by the inclusion (see \eqref{eq:H_P_is_D-P-dot-star}, \eqref{eq:P_0_in_P-dot-star}, \eqref{eq:H_tilde-P_0_is_D-tilde-P_0})
\[
{\s H}_{P_{C,0}}(G\times S\times I^\circ)\subset {\s H}_{P_{C}}(G\times S\times I^\circ),
\]
and by the fact that (see the proof of Corollary \ref{cdd} below)
\[
P_C(x,\omega,E,D)L\tilde{{g_C}}
=
-{\p {(S_0L\tilde{{ g_C}})}E}+CS_0(L\tilde{{ g_C}})\in L^2(G\times S\times I),
\]
where the equality $\omega\cdot\nabla_x(L \tilde{{ g_C}})=0$ has been used again.
\end{proof}

We additionally obtain the following a priori estimate.

\begin{corollary}\label{cdd}
Under the assumptions of Theorem \ref{coth3-dd} the solution $\phi$ of the problem (\ref{co3aa-dd}) satisfies, with a constant $C_1\geq 0$, the estimate
\be\label{ess1-d}
\n{\phi}_{L^2(G\times S\times I)}\leq C_1\big(\n{{f_C}}_{L^2(G\times S\times I)}
+\n{{ g_C}}_{H^1(I,T^2(\Gamma_-'))}\big).
\ee
\end{corollary}

\begin{proof} By estimate (\ref{bestim-d}), we have
\bea\label{ess1-da}
\n{\phi}_{L^2(G\times S\times I)}
=\n{u+L{ g_C}}_{L^2(G\times S\times I)}
\leq
{1\over c}\n{{\tilde{ f_C}}}_{L^2(G\times S\times I)}
+\n{L{ g_C}}_{L^2(G\times S\times I)}.
\eea
By Lemma \ref{le:H1_lift},
\[
\n{L{ g_C}}_{L^2(G\times S\times I)}=\n{L{ g_C}}_{L^2(I,L^2(G\times S))}
\leq \n{L{ g_C}}_{H^1(I, L^2(G\times S))}\leq \sqrt{d}\n{{ g_C}}_{H^1(I,T^2(\Gamma'_-))},
\]
where $d=\mathrm{diam}(G)<\infty$,
and similarly,
\[
\n{\p{(L{g_C})}{E}}_{L^2(G\times S\times I)}
=\n{\p{(L{ g_C})}{E}}_{L^2(I,L^2(G\times S))}
\leq \n{L{g_C}}_{H^1(I,L^2(G\times S))}
\leq \sqrt{d}\n{{ g_C}}_{H^1(I,T^2(\Gamma'_-))}.
\]
Finally, due to these estimates and the fact that $\omega\cdot\nabla_x(L{ g_C})=0$,
one has
\[
\n{\tilde{{f_C}}}_{L^2}
\leq {}&
\n{{ f_C}}_{L^2}+\n{S_0\p{(L{ g_C})}{E}+\Big(\p {S_0}{E}
-CS_0\Big)L{ g_C}-(\Sigma-K_C)L{ g_C}}_{L^2} \\
\leq {}&
\n{{ f_C}}_{L^2}
+\sqrt{d}\Big(\n{\p {S_0}{E}}_{L^\infty}+(C+1)\n{S_0}_{L^\infty}+\n{\Sigma-K_C}\Big)\n{{ g_C}}_{H^1(I,T^2(\Gamma'_-))},
\]
where we wrote, unambiguously, $L^2=L^2(G\times S\times I)$ and $L^\infty=L^\infty(G\times I)$ in order to compress the formulas.
This proves the estimate \eqref{ess1-d} as claimed.
\end{proof}

For the original problem we get 

\begin{corollary}\label{m-d-co1}
Suppose that the assumptions  (\ref{ass1}), (\ref{ass5-a}), (\ref{ass7}), (\ref{ass-8}), (\ref{ass8-aa-a}), (\ref{ass9-a-b}), (\ref{evo16}), (\ref{evo8-a}) and (\ref{evo9-a})  are valid with $C=\frac{\max\{q,0\}}{\kappa}$ and $c>0$.
Furthermore, suppose that ${f}\in L^2(G\times S\times I)$, and $g\in H^1(I,T^2(\Gamma_-'))$ is such that
\be\label{comp-d-d}
g(\cdot,\cdot,E_{\rm m})=0.
\ee
Then the problem
\begin{gather}
-{\p {(S_0\psi)}E}+\omega\cdot\nabla_x\psi+\Sigma\psi -K\psi=f,\ \nonumber\\
{\psi}_{|\Gamma_-}=g,\quad \psi(\cdot,\cdot,E_{\rm m})=0, \label{co3aa-ddd}
\end{gather}
has a unique solution $\psi\in {\s H}_P(G\times S\times I^\circ)$.
In addition, there exists a constant $C_1'>0$ such that \emph{a priori} estimate
\be\label{csda40aaa-dd}
\n{\psi}_{L^2(G\times S\times I)}\leq C_1'\big(
\n{{ f}}_{L^2(G\times S\times I)}+\n{{ g}}_{H^1(I,T^2(\Gamma_-'))}\big),
\ee
holds.
\end{corollary} 

\begin{proof}
Recalling (see \eqref{eq:exp_trick}, \eqref{eq:b_fg}) that
$\psi$ is a solution of the problem \eqref{co3aa-ddd}
with data $(f,g)$
if and only if 
$\phi=e^{CE}\psi$ is a solution of the problem \eqref{co3aa-dd}
with data $({ f_C}=(e^{CE}f,\ g_C=e^{CE}g)$,
we deduce the claims from Theorem \ref{coth3-dd} and Corollary \ref{cdd}.
\end{proof}

\subsection{An Existence Result Based on the Theory of Evolution Equations}\label{evcsd}

Due to the existence result (part (i)) of Corollary \ref{csdaco1} one has,
\[
-{\p {(S_0\psi)}E}+\omega\cdot\nabla_x\psi\in L^2(G\times S\times I),
\]
but we do not know \emph{a priori} if ${\p {\psi}E}\in L^2(G\times S\times I)$,
or if $\omega\cdot\nabla_x\psi\in L^2(G\times S\times I)$.
Hence it is not known if $\psi\in W^2_1(G\times S\times I)$,
and so the regularity assumption in parts (ii)-(iii) of Corollary \ref{csdaco1}
is not in general guaranteed. 
Here we give an alternative approach for a less general problem (\ref{se1}), 
(\ref{se2}), (\ref{se3}), which gives more regularity for its solution $\psi$,
and in particular provides the regularity criterion $\psi\in W_1^2(G\times S\times I)$ used above.

We assume that the collision operator is of the form given in section \ref{el-k}
\be\label{ecsd1}
(K\psi)(x,\omega,E)=\int_S\tilde\sigma(x,\omega',\omega,E)\psi(x,\omega',E)d\omega'
\ee
where $\tilde\sigma(x,\omega',\omega,E)$ obeys (\ref{bound-b}) and 
(\ref{k3-n3-disas-a}).
Then $K$ can be written as $(K\psi)(E)=K(E)\psi(E)$,
where $\psi(E)(x,\omega):=\psi(x,\omega,E)$ and where for any fixed $E\in I$,
the linear operator $K(E)$ is defined by 
\[
(K(E)\phi)(x,\omega)=\int_S \tilde{\sigma}(x,\omega',\omega,E)\phi(x,\omega') d\omega',
\quad \phi\in L^2(G\times S).
\]
For a condition related to the property $(K\psi)(E)=K(E)\psi(E)$, see \cite[Chapter VI, Prop. 9.13]{engelnagel}.
By Theorem \ref{el-k-b}
we have for all $E\in I$,
\begin{align}\label{co(b)}
\n{K(E)}\leq \n{\int_S \tilde{\sigma}(\cdot,\omega',\cdot,E) d\omega'}_{L^\infty(G\times S)}^{1/2}
\n{\int_S \tilde{\sigma}(\cdot,\cdot,\omega',E) d\omega'}_{L^\infty(G\times S)}^{1/2}\leq M_1^{1/2} M_2^{1/2},
\end{align}
uniformly for $E\in I$, where $\n{K(E)}$ is the norm of $K(E)$ as an operator in $L^2(G\times S)$.

Consider the problem \eqref{se1}, \eqref{se2}, \eqref{se3},
where $K$ is of the particular form given in \eqref{ecsd1}.
We shall assume that the restricted cross-sections $\Sigma$, $\tilde\sigma$ and the stopping power $S_0$ satisfy somewhat different assumptions than in the previous section.
We make the following change of variables and of the unknown function
\be\label{ecsd2} 
\tilde\psi(x,\omega,E):={}&\psi(x,\omega,E_m-E), \nonumber\\
\phi:={}&e^{-CE}\tilde\psi,
\ee
and denote 
\bea
\tilde S(x,E)={}&S_0(x,E_{\rm m}-E),\ \nonumber\\
\tilde \Sigma(x,\omega,E)={}&\Sigma(x,\omega,E_{\rm m}-E)\nonumber\\
\ol \sigma(x,\omega,\omega',E)={}&\tilde\sigma(x,\omega,\omega',E_{\rm m}-E)\nonumber\\
\tilde f(x,\omega,E)={}&f(x,\omega,E_{\rm m}-E)\nonumber\\
\tilde g(y,\omega,E)={}&g(y,\omega,E_{\rm m}-E)\nonumber\\
(\tilde K\phi)(x,\omega,E)={}&\int_S\ol\sigma(x,\omega',\omega,E)\phi(x,\omega',E)d\omega',
\eea
with $\phi\in L^2(G\times S\times I)$ in the definition of $\tilde{K}$.
Making the changes in \eqref{ecsd2}, we find that the problem 
\eqref{se1}, \eqref{se2}, \eqref{se3} is equivalent to
$\phi$ satisfying the equation
\bea\label{se1a}
&{\p {\phi}E}+{1\over{\tilde S}}\omega\cdot\nabla_x\phi+C\phi
+{1\over{\tilde S}}{\p {\tilde S}E}\phi
+{1\over{\tilde S}}\tilde\Sigma\phi
-{1\over{\tilde S}}\tilde K\phi
={1\over{\tilde S}}e^{-CE}\tilde f,
\eea
on $G\times S\times I$,
along with satisfying the following inflow boundary and initial value conditions,
\begin{align}
\phi_{|\Gamma_-}={}&e^{-CE}\tilde g, \label{se2a} \\
\phi(\cdot,\cdot,0)={}&0. \hspace{1.5cm} \textrm{(on $G\times S$)}. \label{se3a}
\end{align}

Furthermore, define \emph{for any fixed} $E\in I$ and $C\geq 0$ the linear operator
$A_C(E):L^2(G\times S)\to L^2(G\times S)$ with domain $D(A_C(E))$
by (here $\tilde S(E)=\tilde S(\cdot,E)$ and $\tilde \Sigma(E)=\tilde \Sigma(\cdot,\cdot,E)$),
\bea\label{ecsd4}
&D(A_C(E))=\tilde W^2_{-,0}(G\times S):=\{\phi\in\tilde W^2(G\times S)\ |\ \gamma_-'(\phi)=0\}, \nonumber\\
&A_C(E)\phi=-\Big({1\over {\tilde S(E)}}\omega\cdot\nabla_x\phi+C \phi+{1\over {\tilde S(E)}}\tilde\Sigma(E)\phi
+{1\over{\tilde S(E)}}{\p {\tilde S}E}(E)\phi
-{1\over{\tilde S(E)}} \tilde K(E)\phi\Big),
\eea
and a function ${ f_C}(E):G\times S\to\R$ such that
\[
{f_C}(E)(x,\omega)={1\over{\tilde S(x,E)}}e^{-CE}\tilde f(x,\omega,E),
\]
where
\[
(\tilde K(E)\phi)(x,\omega)=\int_S\ol\sigma(x,\omega',\omega,E)\phi(x,\omega')d\omega',\quad  \phi\in L^2(G\times S),
\]
and where $\Gamma'_{-}=\{(y,\omega)\in \partial G\times S\ |\ \omega\cdot\nu(y)<0\}$,
while $\gamma'_-:\tilde{W}^2(G\times S)\to \Gamma'_{-}$; $\gamma'_-(\psi)=\psi|_{\Gamma'_{-}}$
is the trace mapping (see section \ref{fs}).

We interpret $\phi$  as a mapping $I\to L^2(G\times S)$ by defining $\phi(E)(x,\omega):=\phi(x,\omega,E)$.
Assuming that $\phi(E)\in D(A_C(E))$ for any $E\in I$
(which takes care of the inflow boundary condition)
the problem (\ref{se1a}), (\ref{se2a}), (\ref{se3a}) {\it for $g=0$} can be put into the abstract form
\be\label{ecsd6}
{\p {\phi}E}-A_C(E)\phi={ f_C}(E),\quad \phi(0)=0.
\ee

We recall the following result from the theory of evolution equations.

\begin{theorem}\label{evoth}
Suppose that $X$ is a Banach space and that for any fixed $t\in [0,T]$
the operator $A(t):X\to X$ is linear and closed,
with domain $D(A(t))\subset X$. In addition, we assume that the following conditions hold:

(i) The domain $D:=D(A(t))$ is independent of $t$ and is a dense subspace of $X$.

(ii) The operator $A(t)$ is $m$-dissipative for any fixed $t\in [0,T]$ 

(iii) For every $u\in D$, the mapping $f_u:[0,T]\to X$ defined by $f_u(t):=A(t)u$
is in $C^1([0,T],X)$.

(iv) $f\in C^1([0,T],X)$  and $u_0\in D$.

Then the (evolution) equation
\be\label{ecsd5}
{\p {u}t}-A(t)u=f,\quad u(0)=u_0,
\ee
has a unique solution $u\in C([0,T],D)\cap C^1([0,T],X)$.  
In addition, the solution is given by
\be\label{solev}
u(t)=U(t,0)u_0+\int_0^t U(t,s)f(s) ds
\ee
where $U(t,s):X\to X$, $0\leq t\leq s\leq T$, is a family of bounded operators,
strongly continuous in $(t,s)$,
called the \emph{(two-parameter) evolution system of operators} of $A(t)$, $t\in [0,T]$,
and
$U(\cdot,s)u_0$ solves (for a fixed $s$) for every $u_0\in D$ the Cauchy problem
\be\label{solevb}
\pa{t}\big(U(t,s)u_0\big)-A(t)U(t,s)u_0=0,\quad U(s,s)u_0=u_0.
\ee
\end{theorem}

\begin{proof}
See \cite[Theorem 4.5.3, pp. 89-106]{tanabe}, \cite[pp. 126-182]{pazy83}, \cite[pp. 477-496]{engelnagel}.
\end{proof}

\begin{remark}\label{re:evoth_norm}
We make a brief remark concerning the meaning of the claim in Theorem \ref{evoth} that $u\in C([0,T],D)$,
since here the topology of $D$ needs to be specified in order to speak of continuity of maps into $D$.
In fact, in \cite[Corollary to Theorem 4.4.2, pp. 102-103]{tanabe} one equips $D$ with
the graph norm of $A(0)$.
Recall that for $t\in [0,T]$, the graph norm of $A(t)$ on $D(A(t))$ ($=D$) is defined as
$\n{v}_{A(t)}:=\n{v}_X+\n{A(t)v}_X$, for $v\in D$.
\end{remark}

In Theorem \ref{evoth} the definition of the concept
of a solution of \eqref{ecsd5} comprises that $u(t)\in D=D(A(t))$ for any $t\in [0,T]$.
Define a closed (see \cite{tervo17-up}) densely defined linear operator $A_0(E):L^2(G\times S)\to L^2(G\times S)$ by
\bea
&D(A_0(E))=\tilde W^2_{-,0}(G\times S)=:D\quad ({\rm independent\ of}\ E)\nonumber\\
&A_0(E)\phi=-{1\over{\tilde S(E)}}\omega\cdot\nabla_x\phi.
\eea

We begin with a lemma on dissipativity properties of $A_0(E)$.

\begin{lemma}\label{csdale1}
Suppose  that the assumptions  (\ref{csda9}), (\ref{csda9a}) and (\ref{csda9b}) are valid,
and let
\begin{align}\label{eq:C0}
C_0:={1\over 2}\kappa^{-2}\n{\nabla_x \tilde S}_{L^\infty(G\times I)}.
\end{align}
Then for any $C\geq C_0$
the operator $A_0(E)-CI:L^2(G\times S)\to L^2(G\times S)$ 
is $m$-dissipative for all $E\in I$. 
\end{lemma}

\begin{proof}

A. Dissipativity.  We have for $\phi\in D$,
\be\label{csda12a}
\omega\cdot \nabla_x\Big({\phi\over{\tilde{S}(E)^{1/2}}}\Big)={1\over{\tilde{S}(E)^{1/2}}}\omega\cdot\nabla_x\phi+\omega\cdot\nabla_x\Big({1\over{\tilde{S}(E)^{1/2}}}\Big)\phi
\ee
and $\nabla_x\Big({1\over{\tilde{S}(E)^{1/2}}}\Big)=-{1\over 2}\tilde{S}(E)^{-3/2}\nabla_x \tilde{S}(E)$.
By the assumptions (\ref{csda9a}) and (\ref{csda9b}) we get that ${\phi\over{\tilde{S}(E)^{1/2}}}\in \tilde W^2_{-,0}(G\times S)=D$. From the Green's formula (\ref{green}) we obtain (since $\omega\cdot\nu >0$ on $\Gamma_+'$ and since $\phi_{|\Gamma_-'}=0$) that for $\phi\in D$,
\bea\label{csda13a}
\la\omega\cdot\nabla_x\Big({\phi\over{\tilde{S}(E)^{1/2}}}\Big),{\phi\over{\tilde{S}(E)^{1/2}}}\ra_{L^2(G\times S)}
=&{1\over 2}\int_{\partial G\times S}(\omega\cdot\nu){{\phi^2}\over{\tilde{S}(E)}}
d\sigma d\omega \nonumber\\
=&
{1\over 2}\int_{\Gamma_+'}(\omega\cdot\nu){{\phi^2}\over{\tilde{S}(E)}}
d\sigma d\omega 
\geq 0.
\eea
That is why by (\ref{csda12a}) we have
\bea\label{csda17}
&\la -A_0(E)\phi,\phi\ra_{L^2(G\times S)}
=\int_{G\times S}{1\over{\tilde S(E)}}(\omega\cdot\nabla_x\phi)\phi d\omega dx \nonumber \\
=&\int_{G\times S}{1\over{\tilde{S}(E)^{1/2}}}(\omega\cdot\nabla_x\phi) {1\over{\tilde{S}(E)^{1/2}}}\phi d\omega dx\nonumber\\
=&
\int_{G\times S}\omega\cdot\nabla_x\Big( {\phi\over{\tilde{S}(E)^{1/2}}}\Big)
\Big( {\phi\over{\tilde{S}(E)^{1/2}}}\Big) d\omega dx
-\int_{G\times S}\omega\cdot\nabla_x\Big( {1\over{\tilde{S}(E)^{1/2}}}\Big){{\phi^2}\over{\tilde{S}(E)^{1/2}}}
d\omega dx\nonumber\\
\geq &
-\int_{G\times S}\omega\cdot\nabla_x\Big( {1\over{\tilde{S}(E)^{1/2}}}\Big){{\phi^2}\over{\tilde{S}(E)^{1/2}}} d\omega dx \nonumber\\
\geq & -\n{\nabla_x\Big( {1\over{\tilde{S}(\cdot)^{1/2}}}\Big){1\over{\tilde{S}(\cdot)^{1/2}}}}_{L^\infty(G\times I)}
\n{\phi}_{L^2(G\times S)}^2\nonumber\\
\geq &
-{1\over 2}\kappa^{-2}\n{\nabla_x\tilde S}_{L^\infty(G\times I)}
\n{\phi}_{L^2(G\times S)}^2 \nonumber \\
=&-C_0\n{\phi}_{L^2(G\times S)}^2,
\eea
and hence
\[
\la (A_0(E)-CI)\phi,\phi\ra_{L^2(G\times S)}
={}&\la A_0(E)\phi,\phi\ra_{L^2(G\times S)}-C\n{\phi}^2_{L^2(G\times S)} \\
\leq{}&(C_0-C)\n{\phi}^2_{L^2(G\times S)}.
\]
Choosing $C\geq C_0$, one finds that
$A_0(E)-CI$ is dissipative for any $E\in  I$.

B. We still have to show that $R(\lambda I-(A_0(E)-CI))=L^2(G\times S)$ for (any) $\lambda>0$.  The equation
\be\label{csda18op}
(\lambda I-(A_0(E)-CI))\phi=f,
\ee
means that $\phi\in \tilde W_{-,0}^2(G\times S)$ and that
\be\label{csda18}
{1\over{\tilde S(E)}}\omega\cdot\nabla_x\phi+(\lambda+C)\phi=f,
\ee
which is equivalent to
\be\label{csda19}
\omega\cdot\nabla_x\phi+(\lambda+C)\tilde S(E)\phi=\tilde S(E)f,
\ee
since by (\ref{csda9}) and (\ref{csda9a}) $f\in L^2(G\times S)$ if and only if $\tilde{S}(E)f\in L^2(G\times S)$.
Let
$B_0:L^2(G\times S)\to L^2(G\times S)$ be a linear operator with domain $D(B_0)$ defined by
\[
& D(B_0)=\tilde W^2_{-,0}(G\times S),\\
& B_0\phi=-\omega\cdot\nabla_x\phi.
\]
Then $B_0$ is  $m$-dissipative (\cite{tervo17-up}, \cite{dautraylionsv6}).  
Let $\lambda':=\kappa (\lambda+C)$.
The equation (\ref{csda19}) is equivalent to
\be\label{csda20}
(\lambda'I -(B_0+B_1))\phi=\tilde S(E)f
\ee
where $B_1$ is defined by
\[
B_1\phi=-((\lambda+C)\tilde S(E)-\lambda')\phi.
\]
It is clear that the operator $B_1:L^2(G\times S)\to L^2(G\times S)$ is bounded,
and since $(\lambda+C)\tilde{S}(E)-\lambda'\geq 0$
by the assumption \eqref{csda9a} and the definition of $\lambda'$,
it follows that $B_1$ is dissipative.
These observations imply that $B_0+B_1$ is $m$-dissipative (cf. \cite[Chapter III]{engelnagel}, or \cite[Theorem 4.2]{tervo17-up}).
Since \eqref{csda18op} is equivalent to \eqref{csda20} as explained above,
this shows that $R(\lambda I-(A_0(E)-CI))=L^2(G\times S)$,
and thus completes the proof.
\end{proof}

For $E\in I$, let $A_1(E):L^2(G\times S)\to L^2(G\times S)$ be the linear operator
\[
A_1(E)\phi:=-{1\over {\tilde S(E)}}\tilde\Sigma(E)\phi-{1\over {\tilde S(E)}}{\p {\tilde S}E}(E)\phi+{1\over {\tilde S(E)}}\tilde K(E)\phi.
\]
We have the following uniform bound for the family of operators $\{A_1(E)\ |\ E\in I\}$.

\begin{lemma}\label{csdale2}
Under the assumptions \eqref{ass1}, \eqref{csda9aa} and \eqref{bound-b}
the operator $A_1(E)$ is bounded for any fixed $E\in I$,
and collectively they obey a uniform bound,
\begin{align}\label{eq:A1E_unif_bound}
\sup_{E\in I} \n{A_1(E)}\leq \kappa^{-1}\Big(\n{\tilde \Sigma}_{L^\infty(G\times S\times  I)}+
\n{{\p {\tilde S}E}}_{L^\infty(G\times I)}+{M_1}^{1/2}{M_2}^{1/2}\Big)=:C_0'<\infty,
\end{align}
where $M_j\geq 0$, $j=1,2$ are as in \eqref{bound-b}.
\end{lemma}

\begin{proof}
The (uniform) estimate follows immediately from the assumptions and the estimate \eqref{co(b)}.
\end{proof}

Using the above notations we have for $C:=C_0+C_0'$ the following decomposition
\be\label{eq:ACE_decomp}
A_C(E)=A_0(E)-C_0I-C_0'I+A_1(E).
\ee
Recall that $C_0$ and $C_0'$ were defined in \eqref{eq:C0} and \eqref{eq:A1E_unif_bound}.
Since by Lemma \ref{csdale2},
\[
&\la (-C_0'I+A_1(E))\phi,\phi\ra_{L^2(G\times S\times I)}=-C_0'\n{\phi}^2_{L^2(G\times S\times I)}+\la A_1(E))\phi,\phi\ra_{L^2(G\times S\times I)}\nonumber\\
\leq &
-C_0'\n{\phi}^2_{L^2(G\times S\times I)}+C_0'\n{\phi}^2_{L^2(G\times S\times I)}=0,
\]
we see that $-C_0'I+A_1(E)$ is bounded and dissipative.
On the other hand, according to Lemma \ref{csdale1}, $A_0(E)-C_0I$ is $m$-dissipative.
Hence $A_C(E)$ is $m$-dissipative for any $E\in I$ (cf. \cite[Chapter III, Theorem 2.7]{engelnagel},
or \cite[Theorem 4.2]{tervo17-up}).
We record this observation into the next lemma.

\begin{lemma}\label{le:ACE_m_diss}
For $C=C_0+C_0'$ and for every fixed $E\in I$, the operator $A_C(E)$ is $m$-dissipative.
\end{lemma}

We shall assume that
\[
\tilde{\sigma}\in C(I,L^\infty(G\times S,L^1(S')))\cap
C(I,L^\infty(G\times S',L^1(S))),
\]
where $\tilde{\sigma}$ interpreted as an element of $C(I,L^\infty(G\times S,L^1(S')))$ is
\[
\tilde\sigma(E)(x,\omega)(\omega')=\tilde\sigma(x,\omega',\omega,E),
\]
and when interpreted as an element of $C(I,L^\infty(G\times S',L^1(S)))$ it is
\[
\tilde{\sigma}(E)(x,\omega')(\omega) = 
\tilde\sigma(x,\omega',\omega,E).
\]
Furthermore, in order to avoid ambiguity, we have denoted by $S'$ (resp. $S$) the unit sphere in $\R^3$ for the variable $\omega'$ (resp. $\omega$).
We find that 
the conditions \eqref{bound-b} are then satisfied valid for $\ol\sigma$ with
\[
M_1:=&\max_{E\in I}\n{\ol\sigma(E)}_{L^\infty(G\times S),L^1(S'))} = \n{\ol{\sigma}}_{C(I, L^\infty(G\times S),L^1(S')))}, \\
M_2:=&\max_{E\in I}\n{\ol\sigma(E)}_{L^\infty(G\times S'),L^1(S))} = \n{\ol{\sigma}}_{C(I, L^\infty(G\times S'),L^1(S)))}.
\]
Moreover, for all $E_1,\ E_2\in I$,
\bea\label{ccc-a} 
&
\esssup_{(x,\omega)\in G\times S}\int_{S'}|
\ol\sigma(x,\omega',\omega,E_1)-\ol\sigma(x,\omega',\omega,E_2)| d\omega'
=\n{\ol\sigma(E_1)-\ol\sigma(E_2)}_{L^\infty(G\times S,L^1(S'))}\nonumber\\
&
\esssup_{(x,\omega')\in G\times S'}\int_S|
\ol\sigma(x,\omega',\omega,E_1)-\ol\sigma(x,\omega',\omega,E_2)| d\omega
=\n{\ol\sigma(E_1)-\ol\sigma(E_2)}_{L^\infty(G\times S',L^1(S))}.
\eea

Supposing additionally that 
\begin{align}\label{eq:ass_sigma_C1}
\tilde\sigma\in C^1(I,L^\infty(G\times S,L^1(S')))\cap
C^1(I,L^\infty(G\times S',L^1(S))),
\end{align}
then for any fixed $E$ the operator ${\p K{E}}(E):L^2(G\times S)\to L^2(G\times S)$
defined by
\begin{align}\label{eq:dKE}
\Big({\p {\tilde K}{E}}(E)\phi\Big)(x,\omega):=\int_S{\p {\ol\sigma}E}(x,\omega',\omega,E)\phi(x,\omega') d\omega'
\end{align}
is a bounded operator and
\bea\label{ccc-b} 
&
\esssup_{(x,\omega)\in G\times S}\int_{S'}\big|
{\p {\ol\sigma}E}(x,\omega',\omega,E_1)-{\p {\ol\sigma}E}(x,\omega',\omega,E_2)\big| d\omega'
=\n{{\p {\ol\sigma}E}(E_1)-{\p {\ol\sigma}E}(E_2)}_{L^\infty(G\times S,L^1(S'))},\nonumber\\
&
\esssup_{(x,\omega')\in G\times S'}\int_S\big|
{\p {\ol\sigma}E}(x,\omega',\omega,E_1)-{\p {\ol\sigma}E}(x,\omega',\omega,E_2)\big| d\omega
=\n{{\p {\ol\sigma}E}(E_1)-{\p {\ol\sigma}E}(E_2)}_{L^\infty(G\times S',L^1(S))}.
\eea

The following lemma justifies the notation given in \eqref{eq:dKE}.

\begin{lemma}\label{le:K_C1}
Under the assumption \eqref{eq:ass_sigma_C1}, for any fixed $\phi\in L^2(G\times S)$
the map
\[
k_\phi:I\to L^2(G\times S);\quad k_\phi(E)=\tilde{K}(E)\phi
\]
is in $C^1(I, L^2(G\times S))$ and
$\p{k_\phi}{E}(E)=\p{\tilde{K}}{E}(E)\phi$.
\end{lemma}

\begin{proof}
If $\beta:G\times S'\times S\times I\to \R$ is measurable,
and one defines for any $\phi\in L^2(G\times S)$ such that the integral converges,
\[
L_{\beta}(E)\phi:=\int_{G\times S} \beta(x,\omega',\omega,E)\phi(x,\omega')d\omega dx,
\]
then for every $E\in I$ one has (compare to \eqref{co(b)}),
\begin{multline}\label{eq:LE_estim}
\n{L_{\beta}(E)\phi}^2_{L^2(G\times S)}
\leq \n{\beta(E)}_{L^\infty(G\times S, L^1(S'))}
\n{\beta(E)}_{L^\infty(G\times S', L^1(S))}
\n{\phi}_{L^2(G\times S)}^2,
\end{multline}
with the conventions $\beta(E)(x,\omega)(\omega')=\beta(E)(x,\omega')(\omega)=\beta(x,\omega',\omega,E)$ as above.

For a fixed $E_0\in I$,
taking $\beta(x,\omega',\omega,E)=\ol{\sigma}(x,\omega',\omega,E)-\ol{\sigma}(x,\omega',\omega,E_0)$,
we have $k_\phi(E)-k_\phi(E_0)=L_{\beta}(E)\phi$,
and the above estimate shows that $k_\phi$ is continuous at $E_0$, by assumption \eqref{eq:ass_sigma_C1}.
Similarly, one sees that $E\mapsto \p{\tilde{K}}{E}(E)\phi$,
is continuous at $E_0$, by choosing
$\beta(x,\omega',\omega,E)=\p{\ol{\sigma}}{E}(x,\omega',\omega,E)-\p{\ol{\sigma}}{E}(x,\omega',\omega,E_0)$.

Finally, the differentiability of $k_\phi$ at $E_0$
is obtained by taking in the above estimate,
\[
\beta(x,\omega',\omega,E)=\frac{\ol{\sigma}(x,\omega',\omega,E)-\ol{\sigma}(x,\omega',\omega,E_0)}{E-E_0}-\p{\ol{\sigma}}{E}(x,\omega',\omega,E_0),
\]
whence $\frac{k_\phi(E)-k_\phi(E_0)}{E-E_0}-\p{\tilde{K}}{E}(E_0)\phi=L_{\beta}(E)\phi$,
and we have $k_\phi'(E_0)=\p{\tilde{K}}{E}(E_0)\phi$.
This completes the proof of the lemma.
\end{proof}

\begin{remark}
In fact, the proof of the previous lemma shows that
$\tilde{K}$ as a map $I\to \mc{L}(L^2(G\times S))$ is $C^1$,
where for a Banach space $Z$ the set $\mc{L}(Z)$ is the space of bounded linear operators $Z\to Z$
equipped with the uniform operator norm topology.
\end{remark}

With the above notations and results at hand, we are ready to state the central result of this section.

\begin{theorem}\label{evoth1}
Suppose that the assumptions 
 \eqref{csda9a}, \eqref{csda9b} and \eqref{bound-b} are valid,
that $\tilde{\sigma}\geq 0$, and that
\bea
&
\Sigma\in C^1(I,L^\infty(G\times S)),\label{ecsd6a} \\
& 
S_0\in C^2(I,L^\infty(G)),\label{ecsd6a-a}\\
&
\tilde{\sigma}\in C^1(I,L^\infty(G\times S,L^1(S')))\cap
C^1(I,L^\infty(G\times S',L^1(S))).\label{ecsd6a-b}
\eea

Let $f\in C^1(I,L^2(G\times S))$ and let $g\in C^2(I,T^2(\Gamma_-))$ which satisfies the 
\emph{compatibility condition}
\be\label{cc}
g(E_m)=0.
\ee
Then the problem (\ref{se1}), (\ref{se2}), (\ref{se3}) has a unique solution
$\psi\in C(I,\tilde W^2(G\times S))\cap C^1(I,L^2(G\times S))$.

If in addition the assumptions (\ref{k3-n3-disas-a}) (with $c>0$) are also valid,
the estimate 
\be\label{evoest}
\n{\psi}_{{H_1}}\leq {{e^{{{qE_{\rm m}}\over\kappa}}}\over {c'}}
\big(\n{{ f}}_{L^2(G\times S\times I)}+\n{{ g}}_{T^2(\Gamma_-)}\big).
\ee 
holds.
(The constants $\kappa$, $q$ and $c'$ were defined in
\eqref{csda9a}, \eqref{q} and \eqref{cprime}, respectively.)
\end{theorem}

\begin{proof}
A.  Assume at first that $g=0$. We make the change of variables and the change of unknown function as above by setting $\tilde\psi(x,\omega,E)=\psi(x,\omega,E_m-E)$ and $\phi=e^{-CE}\tilde\psi$.
Choose $C=C_0+C_0'$ (see \eqref{eq:C0}, \eqref{eq:A1E_unif_bound}).
Then, as observed above, the problem (\ref{se1}), (\ref{se2}), (\ref{se3}) can be cast
into an equivalent form (see \eqref{ecsd6})
\be\label{ecsd7}
{\p {\phi}E}-A_C(E)\phi={ f_C}(E),\quad \phi(0)=0,
\ee
where the domain $D(A_C(E))=\tilde{W}^2_{-,0}(G\times S)=:D$
of definition of $A_C(E)$ is independent of $E$.
We have, moreover, demonstrated (see Lemma \ref{le:ACE_m_diss}) that the (densely defined) operator
$A_C(E):L^2(G\times S)\to L^2(G\times S)$ is $m$-dissipative for any fixed $E\in I$. 

The assumptions \eqref{ecsd6a}, \eqref{ecsd6a-a}, \eqref{ecsd6a-b} imply that for any fixed $\phi\in D$
the mapping
\[
h_{\phi}:I\to L^2(G\times S);\quad h_\phi(E):=A_C(E)\phi,
\]
is differentiable and
\[
h_\phi'(E)=&-{\partial\over{{\partial E}}}\Big({1\over{\tilde S(E)}}\Big)\omega\cdot \nabla_x\phi
-{\partial\over{{\partial E}}}\Big({1\over{\tilde S(E)}}\tilde\Sigma(E)\Big)\phi
-{\partial\over{{\partial E}}}\Big({1\over{\tilde S(E)}}{\p {\tilde S}E}\Big)\phi \\
&+{\partial\over{{\partial E}}}\Big({1\over{\tilde S(E)}}\Big)\tilde K(E)\phi
+{1\over{\tilde S(E)}}{\p {\tilde K}{E}}(E)\phi,
\]
where ${\p {\tilde K}{E}}(E)\phi$ is defined in \eqref{eq:dKE},
and the derivative $\pa{E}(\tilde{K}(E)\phi)={\p {\tilde K}{E}}(E)\phi$ 
is provided by Lemma \ref{le:K_C1}.
By assumptions \eqref{ecsd6a}, \eqref{ecsd6a-a}, \eqref{ecsd6a-b} we thus see that $h_\phi$ is in $C^1(I,L^2(G\times S))$.

By Theorem \ref{evoth}
there exists a unique solution $\phi\in C(I,W^2_{-,0}(G\times S))\cap C^1(I,L^2(G\times S))$ of (\ref{ecsd7}). Then  $\psi(x,\omega,E):=
e^{C(E_m-E)}\phi(x,\omega,E_m-E)$ is the required solution of 
the problem \eqref{se1}, \eqref{se2}, \eqref{se3} for $g=0$.

B. Suppose that more generally $g\in C^2(I,T^2(\Gamma_-))$ and that the compatibility condition (\ref{cc}) holds.
By \cite[Lemma 5.10]{tervo17-up} there exists a lift $Lg\in C^2(I,\tilde W^2(G\times S))$ for which $\gamma_-(Lg)=g$, and $(Lg)(\cdot,\cdot,E_m)=0$ (follows from (\ref{cc})),
and furthermore $\omega\cdot\nabla_x(Lg)=0$. Substituting into the problem (\ref{se1}), (\ref{se2}), (\ref{se3}) the function $u:=\psi-Lg$ for $\psi$ we obtain the following problem for $u$,
\begin{gather}
-{\p {(S_0u)}E}+\omega\cdot\nabla_x u+\Sigma u
- Ku = \tilde{f},\nonumber\\
u_{|\Gamma_-}=0,\nonumber\\
u(\cdot,\cdot,E_m)=0,\label{ecsd8}
\end{gather}
where
\[
\tilde{f}:=f-\Big(-{\p {(S_0(Lg))}E}+\Sigma (Lg)-K(Lg)\Big),
\]
and we have $\tilde{F}\in C^1(I,L^2(G\times S))$ under the assumption \eqref{ecsd6a}, \eqref{ecsd6a-a}, \eqref{ecsd6a-b}.

By Part A of the proof, the problem \eqref{se1}, \eqref{se2}, \eqref{se3} (for $g=0$) has a unique solution $u\in C(I,\tilde{W}^2_{-,0}(G\times S))\cap C^1(I,L^2(G\times S))$.
As argued above, $\psi:=u+Lg$ is then then desired unique solution
for the problem \eqref{se1}, \eqref{se2}, \eqref{se3}, for the given, arbitrary $g\in C^2(I,T^2(\Gamma_-))$.

We claim that the solution $\psi$ belongs to $W^2_1(G\times S\times I)$.
Indeed, since $\psi\in C^1(I,L^2(G\times S))$,
we have $\psi\in L^2(G\times S\times I)$ and ${\p{\psi}E}\in L^2(G\times S\times I)$.
On the other hand, since $\psi$ solves \eqref{se1},
these imply, together with the assumptions made,
that $\omega\cdot\nabla_x\psi\in L^2(G\times S\times I)$,
which confirms the claim.

Hence, under the additional assumptions (\ref{k3-n3-disas-a}), the estimate (\ref{evoest}) follows from Corollary \ref{csdaco1} (part (iii)). This completes the proof.
\end{proof}

The term compatibility condition used above for the assumption \eqref{cc} for $g$
comes from the observation that since the solution $\psi$
is to satisfy \eqref{se3} and we have $\psi\in C(I,\tilde{W}^2(G\times S))$,
it follows that $\psi(E_m)=0$, and therefore $0=\psi(E_m)|_{\Gamma_-}=g(E_m)$.

\begin{remark}
Strictly speaking, in the proof of Theorem \ref{evoth1} above
we did not fully address the claim that $u\in C(I,\tilde{W}^2_{-,0}(G\times S))$,
since, as pointed out in Remark \ref{re:evoth_norm},
the common domain $D=D(A_C(E))=\tilde{W}^2_{-,0}(G\times S)$ of the operators $A_C(E)$, $E\in I$,
is to be equipped with the graph norm of $A_C(0)$ when applying Theorem \ref{evoth}.
However, the norm $\n{\cdot}_{\tilde{W}^2_{-,0}(G\times S)}$ of $\tilde{W}^2_{-,0}(G\times S)$
is equivalent to the graph norm $\n{\cdot}_{A_C(E)}$ in $D=\tilde{W}^2_{-,0}(G\times S)$,
for every $E\in I$, and in particular for $E=0$.
This can be readily seen from the estimates, for $\phi\in D$,
\[
\n{A_C(E) \phi}_{L^2(G\times S)}\leq{} & \frac{1}{\kappa} \big(\n{\omega\cdot\nabla_x\phi}_{L^2(G\times S)} + C_2\n{\phi}_{L^2(G\times S)}\big), \\
\n{\omega\cdot\nabla_x\phi}_{L^2(G\times S)} \leq{} & C_1\n{A_C(E)\phi}_{L^2(G\times S)} + C_2\n{\phi}_{L^2(G\times S)},
\]
where
\[
C_1:=\n{\tilde{S}(E)}_{L^\infty},\quad
C_2:=C C_1 + \n{\tilde{\Sigma}(E)}_{L^\infty} + \n{\p{\tilde{S}}{E}(E)}_{L^\infty} + \n{\tilde{K}(E)},
\]
and we have used a short hand notation $L^\infty$ for $L^\infty(G\times S)$.
Notice also that $\n{\gamma_-(\phi)}_{T^2(\Gamma_-)}=0$ when $\phi\in \tilde{W}^2_{-,0}(G\times S)$.
\end{remark}

Let $H^m(I,X),\ m\in\N_0$ be the Sobolev space for Hilbert space valued functions $I\to X$.
We have the following corollary.

\begin{corollary}\label{evothco}
Suppose that the assumptions 
 \eqref{csda9a}, \eqref{csda9b}, \eqref{bound-b} and \eqref{ecsd6a}, \eqref{ecsd6a-a} and \eqref{ecsd6a-b}
of Theorem \ref{evoth1}
are valid, and $\tilde{\sigma}\geq 0$.
Let $f\in H^2(I,L^2(G\times S))$ and $g\in H^3(I,T^2(\Gamma_-))$ which satisfies  the compatibility condition
\be
g(E_m)=0.
\ee
Then 
the problem (\ref{se1}), (\ref{se2}), (\ref{se3}) has a unique solution
$\psi\in C(I,\tilde W^2(G\times S))\cap C^1(I,L^2(G\times S))$.

If in addition the assumptions \eqref{k3-n3-disas-a} (with $c>0$) are also valid, the estimate (\ref{evoest}) holds.

\end{corollary}

\begin{proof}
By the Sobolev Embedding Theorem
\[
H^m(I,X)\subset C^{j}(I,X)\ {\rm for}\ m>j+{1\over 2}
\]
and then the assertion follows from Theorem \ref{coupthev}.
\end{proof}

\begin{remark}
The  evolution equation based approach given above can be generalized for $L^p$-theory
when $1\leq p<\infty$.
The approach based on the Lions-Lax-Milgram Theorem (section \ref{esols}) is limited to the Hilbert space 
structure, and can therefore be only applied for $p=2$.
However, some (recent) generalizations for reflexive Banach 
spaces of Lions-Lax-Milgram theory might allow
methods of section \ref{esols} to be generalized also for $1<p<\infty$.
\end{remark}

\subsection{On the Existence of Solutions for Volterra Type Collision Operators}

In the previous section, we assumed that the collision operator $K$ is of the form
$(K\psi)(x,\omega,E)=\int_S \sigma(x,\omega',\omega,E)\psi(x,\omega',E) d\omega'$ in order to avoid integration over $I$ with respect to $E'$. Considerations were founded on the fact that $K\psi$ had a representation $(K\psi)(E)=K(E)\psi(E)$.
However, for some collision operators of special type,
also integration with respect to $E'$ is possible in evolution operator based approaches.
In this section, we give a short and \emph{formal} description of such a technique
for Volterra type collision operators.

Consider the problem (\ref{se1}), (\ref{se2}), (\ref{se3}) with $g=0$,
\begin{gather*}
 -{\p {\psi}E}+{1\over{S_0(E)}}\omega\cdot\nabla_x\psi+{1\over{S_0(E)}}\Sigma(E)\psi-{1\over{S_0(E)}}{\p {S_0}E}(E)\psi
-{1\over{S_0(E)}} K\psi= {1\over{S_0(E)}}f,\nonumber\\
\psi_{\Gamma_-}=0,\quad
\psi(\cdot,\cdot,E_{\rm m})=0.
\end{gather*}
For simplicity we denote ${1\over{S_0(E)}}f$ and ${1\over{S_0(E)}} K$ again by $f$ and $K$.
Assume that $I=[0,\infty[{}=:\R_+$
and that $K$  is of the \emph{Volterra type operator} (cf. \cite[pp. 447-452]{engelnagel})
\[
(K\psi)(x,\omega,E)=\int_0^E\int_S
\sigma(x,\omega',\omega,E-E',E)\psi(x,\omega',E') d\omega' dE'.
\]
In other words, $K\psi=\int_{S\times I} \tilde{\sigma}(\cdot,\omega',\cdot,E',\cdot)\psi(\cdot,\omega',E') d\omega' dE'$ for a differential cross section $\tilde{\sigma}$ of the form
$\tilde{\sigma}(x,\omega',\omega,E',E)=\chi_I(E-E')\sigma(x,\omega',\omega,E-E',E)$,
where $\chi_I$ is the characteristic function of $I$.
Assume additionally that $\sigma$ has a decomposition
\[
\sigma(x,\omega',\omega,E-E',E)=\sigma_1(x,\omega',\omega,E)\sigma_2(E-E').
\]
Let $K_1(E)$ (for a fixed $E$) be a linear operator $L^2(G\times S)\to L^2(G\times S)$ defined by
\[
(K_1(E)\phi)(x,\omega)=\int_S\sigma_1(x,\omega',\omega,E)\phi(x,\omega') d\omega'.
\]
Then we find that 
\[
(K\psi)(x,\omega,E)=\int_0^E
\sigma_2(E-E')K_1(E)\psi(E') dE'.
\]

Define an {\it extended  space} by ${\s X}:=L^2(G\times S)\times L^1(\R_+,L^2(G\times S))$ and a linear operator ${\s A}(E):{\s X}\to {\s X}$ for a fixed $E$ by (here the argument of $\R_+$ is denoted by $s$)
\bea
&
D({\s A}(E))=D:=\tilde W_{-,0}^2(G\times S)\times H^{1,1}(\R_+,L^2(G\times S)),\nonumber\\
&
{\s A}(E):=\qmatrix{A(E)&\delta_0\cr B(E)&{d\over{ds}}\cr},
\nonumber
\eea
where
\[
H^{1,1}(\R_+,L^2(G\times S)):= \{F\in L^1(\R_+,L^2(G\times S))\ |\ F'\in L^1(\R_+,L^2(G\times S))\},
\]
is the domain of the derivative operator $\frac{d}{ds}:L^1(\R_+,L^2(G\times S))\to L^1(\R_+,L^2(G\times S))$,
the linear operator $A(E):L^2(G\times S)\to L^2(G\times S)$ with domain $\tilde{W}^2_{-,0}(G\times S)$ (independent of $E$) is given by,
\[
A(E)\phi=-\Big({1\over { S_0(E)}}\omega\cdot\nabla_x\phi+{1\over { S_0(E)}}\Sigma(E)\phi
-{1\over{ S_0(E)}}{\p { S_0}E}(E)\phi\Big),
\]
the linear operator $\delta_0$ is bounded and defined by
\[
\delta_0:H^{1,1}(\R_+,L^2(G\times S))\to L^2(G\times S);
\quad \delta_0 F=F(0),
\]
and finally $B(E)$ is the bounded linear operator defined by
\[
B(E):L^2(G\times S)\to L^1(\R_+,L^2(G\times S));\quad
(B(E)\phi)(s)=\sigma_2(s)K_1(E)\phi.
\]

Let $\qmatrix{\phi\cr \eta\cr}\in {\s X}$ and let $\qmatrix{\psi\cr F\cr}\in H^1(\R_+,{\s X})$ (note that $\psi:\R_+\to L^2(G\times S),\ F:\R_+\to L^1(\R_+,L^2(G\times S))$; here the argument of $\psi$ and $F$ is denoted by $E$).
In this extended setting, 
for a given $\qmatrix{F_1\cr F_2}\in C(\R_+,{\s X})$,
the validity of the evolution equation 
\be\label{R1} 
\frac{\partial}{\partial E}\qmatrix{\psi\cr F\cr}-{\s A}(E)\qmatrix{\psi\cr F\cr}=\qmatrix{F_1\cr F_2\cr},
\quad \qmatrix{\psi\cr F\cr}(0)=\qmatrix{\phi\cr \eta\cr},
\ee
implies (from the first row of the matrix equation) that
\be \label{R4}
{\p {\psi}E}-A(E)\psi=F_1(E)+\delta_0(F(E))=F_1(E)+F(E)(0).
\ee

In some cases the evolution system of the operators ${\s A}(E)$ has properties which
enable one to deduce the existence of solutions for the transport problem
\bea\label{R0}
&
{\p {\psi}E}-A(E)\psi - K\psi =f,\nonumber\\
&
\psi_{\Gamma_-}=0,\quad \psi(0)=\phi.
\eea
Note that carrying out the change of variables as above,
one is able to replace the initial condition $\psi(0)=\phi$ with $\psi(E_{\rm m})=\phi$. 

The basis of the idea is as follows.
Assume, for example, that $S_0$, $\Sigma$ and $\sigma_1$ are independent of $E$. Then the operators $A(E)$, $K_1(E)$, $B(E)$ and ${\s A}(E)$ are independent of $E$ as well.
Let $T(E)$ be the $C^0$-semigroups generated by  $A$. Define a linear operator $R(E):L^1(\R_+,L^2(G\times S))\to L^2(G\times S)$ by
\[
R(E)\eta:
=\int_0^ET(E-s)\eta(s)ds.
\]
Denote by $S(E)$ the (left) translation semigroup on $L^1(\R_+,L^2(G\times S))$,
\be\label{tr}
(S(E)\eta)(s):=\eta(E+s).
\ee
Finally, let
\[
{\s T}(E):=\qmatrix{T(E)&R(E)\cr 0&S(E)\cr}.
\]
Then ${\s T}(E)$ is the $C^0$-semigroup generated by $\qmatrix{A&\delta_0\cr 0&{d\over{ds}}\cr}:{\s X}\to {\s X}$ (\cite[pp. 437-438]{engelnagel}).
Furthermore, letting ${\s S}(E):{\s X}\to {\s X}$ be the $C^0$-semigroup generated by the perturbed operator
\[
{\s A}:=\qmatrix{A&\delta_0\cr 0&{d\over{ds}}\cr}+
\qmatrix{0&0\cr B&0},
\]
it follows that $\qmatrix{\psi\cr F}(E):={\s S}(E)\qmatrix{\phi\cr \eta}$
is the solution of (\ref{R1}) for $\qmatrix{F_1\cr F_2}=\qmatrix{0\cr 0}$.

The above semigroups obey the following formula,
for every $\qmatrix{\phi\cr \eta\cr}\in {\s X}$:
(\cite[Corollary III.1.7]{engelnagel})
\bea
{\s S}(E)\qmatrix{\phi\cr \eta\cr}
=&
{\s T}(E)\qmatrix{\phi\cr \eta\cr}
+\int_0^E{\s T}(E-s){\s B}{\s S}(s)\qmatrix{\phi\cr \eta\cr} ds
\nonumber\\
=&
{\s T}(E)\qmatrix{\phi\cr \eta\cr}
+\int_0^E{\s T}(E-s){\s B}\qmatrix{\psi\cr F\cr}(s) ds,.
\eea 
where ${\s B}:=\qmatrix{0&0\cr B&0\cr}$.

Choose $\eta:=f$ and $\qmatrix{F_1\cr F_2}=\qmatrix{0\cr 0}$.
Applying the above facts we have
\bea\label{R7}
&
\qmatrix{\psi\cr F\cr}(E)={\s S}(E)\qmatrix{\phi\cr f\cr}
\nonumber\\
=&
\qmatrix{T(E)&R(E)\cr 0&S(E)\cr}\qmatrix{\phi\cr f\cr}
+
\int_0^E\qmatrix{T(E-s)&R(E-s)\cr 0&S(E-s)\cr}
\qmatrix{0&0\cr B&0\cr}\qmatrix{\psi\cr F\cr}(s)ds\nonumber\\
=&
\qmatrix{T(E)&R(E)\cr 0&S(E)\cr}\qmatrix{\phi\cr f\cr}
+
\int_0^E\qmatrix{T(E-s)&R(E-s)\cr 0&S(E-s)\cr}
\qmatrix{0\cr B\psi(s)\cr}ds,
\eea
from which the last row gives us
\be 
F(E)=S(E)f+\int_0^ES(E-s)B\psi(s) ds,
\ee
and then (recall (\ref{tr}) and that $(B\psi(s))(t)=\sigma_2(t)K_1\psi(s)$)
\be \label{R3}
F(E)(0)=f(E)+\int_0^E \big(S(E-s)(B\psi(s))\big)(0) ds
=
f(E)+\int_0^E\sigma_2(E-s)K_1\psi(s) ds.
\ee
Combining (\ref{R3}) and (\ref{R4}) we finally see that
$\psi$ is a solution of the transport problem (\ref{R0}).

Finally, we note that in the case where $\qmatrix{\psi\cr F}\in D({\s A})$, we have
$\psi(E)\in D(A)=\tilde W_{-,0}^2(G\times S)$, for any $E\geq 0$,
and thus in particular, $\psi_{|\Gamma_-}=0$.
We find by the first row of matrix equation (\ref{R7}) that $\psi(0)=T(0)\phi+R(0)f=\phi$.

It might be worth attempting to generalize this method under less restrictive assumptions,
especially for the case where $S_0$, $\Sigma$, $\sigma_1$ are allowed to be $E$-dependent.

\sectionspace
\section{Existence of Solutions for the Coupled System}\label{cosyst}

In this section, we consider the coupled transport problem.
For simplicity denote $\Sigma_j:=\Sigma_{j,r},\ S_j:=S_{j,r},\ \sigma_{jj,r}:=\sigma_{jj}$ for $j=2,3$.
Let $f=(f_1,f_2,f_3)\in L^2(G\times S\times I)^3$ and $g=(g_1,g_2,g_3)\in T^2(\Gamma_-)^3$.

We deal with the following coupled system of integro-partial differential equations
for $\psi=(\psi_1,\psi_2,\psi_3)$ on $G\times S\times I$,
\begin{gather}
\omega\cdot\nabla_x\psi_1+\Sigma_1\psi_1-K_{1}\psi=f_1, \label{csda1a}\\
-{\p {(S_{j}\psi_j)}E}+\omega\cdot\nabla_x\psi_j+\Sigma_{j}\psi_j-K_{j}\psi=f_j,\quad j=2,3.\label{csda1b}
\end{gather}
In order to guarantee uniqueness of solutions, we moreover impose the inflow boundary condition on $\Gamma_-$,
\be\label{csda2}
{\psi_j}_{|\Gamma_-}=g_j,\quad j=1,2,3,
\ee
and initial value (or energy boundary) condition on $G\times S$,
\be\label{csda3}
\psi_j(\cdot,\cdot,E_{\rm m})=0,\quad j=2,3,
\ee
where $E_{\rm m}$ is the cut-off energy. 
As mentioned in the introduction the problem (\ref{csda1a})-(\ref{csda3}) is an approximation of the problem (\ref{intro1}), (\ref{intro2}).

We assume that 
the {\it total (restricted) cross sections}
$\Sigma_j:G\times S\times I\to\R$, for $j=1,2,3$, are functions such that
\be\label{scateh}
\Sigma_j\in L^\infty(G\times S\times I),\quad \Sigma_j\geq 0,\quad j=1,2,3.
\ee
Furthermore, we assume that the
{\it differential (restricted) cross sections}
 $\sigma_{kj}:G\times S^2\times I^2\to\R$, $k,j=1,2,3$, are measurable functions such that
\begin{alignat}{3}\label{colleh}
&\sum_{k=1}^3\int_{I'}\int_{S'} \sigma_{kj}(x,\omega',\omega,E',E)d\rho_S^{kj}(\omega') d\rho_I^{kj}(E')\leq M_1,\quad
&& {\rm a.e.}\ G\times S\times I,\quad j=1,2,3,\nonumber\\[2mm]
&\sum_{k=1}^3\int_{I'}\int_{S'} \sigma_{jk}(x,\omega,\omega',E,E')d\rho_S^{jk}(\omega') d\rho_I^{jk}(E')\leq M_2,\quad
&& {\rm a.e.}\ G\times S\times I,\quad j=1,2,3,\\[2mm]
&\sigma_{kj}\geq 0,\quad && {\rm a.e.}\ G\times S^2\times I^2,\quad k,j=1,2,3\nonumber
\end{alignat}
where the measures $\rho_S^{jk},\ \rho_I^{jk}$ are as in section \ref{co-cs}.

Define the \emph{(restricted) scattering operator} $\Sigma_j$ and the \emph{(restricted) collision operator} $K_j$ corresponding to the particle type $j$, for $j=1,2,3$ and $\psi_j\in
 L^2(G\times S\times I)$, as follows
\be\label{scat}
(\Sigma_j\psi_j)(x,\omega,E)=\Sigma_j(x,\omega,E)\psi_j(x,\omega,E),
\ee
and for $\psi=(\psi_1,\psi_2,\psi_3)\in L^2(G\times S\times I)^3$,
\be\label{coll3}
(K_j\psi)(x,\omega,E)=\sum_{k=1}^3\int_{ I'}\int_{S'}\sigma_{kj}(x,\omega',\omega,E',E)\psi_k(x,\omega',E')d\rho_S^{kj}(\omega') d\rho_I^{kj}(E').
\ee
Furthermore, we define for $\psi=(\psi_1,\psi_2,\psi_3)\in L^2(G\times S\times I)^3$,
\be\label{sda1}
\Sigma\psi=(\Sigma_1\psi_1,\Sigma_2\psi_2,\Sigma_3\psi_3)
\ee
and
\be\label{sda2}
K\psi=(K_1\psi,K_2\psi,K_3\psi).
\ee
We see that $\Sigma:L^2(G\times S\times I)^3\to L^2(G\times S\times I)^3$ is a bounded linear operator. In addition,  
by Theorem \ref{coup-bound}
the  operator $K:L^2(G\times S\times I)^3\to L^2(G\times S\times I)^3$ is bounded and
\be\label{k-norm}
\n{K}\leq 
2\pi{M_1}^{1/2}{M_2}^{1/2}.
\ee

We assume that functions $S_j:\ol G\times I\to\R,\ j=2,3$,
the so-called \emph{restricted stopping powers}, satisfy the following assumptions:
\begin{align}
& S_j\in L^\infty(G\times  I), \label{sda2a} \\
& {\p {S_j }{E}}\in L^\infty(G\times  I), \label{sda2a-2} \\
& \kappa_j:=\inf_{(x,E)\in \ol{G}\times I} S_j(x,E) > 0, \label{sda2b} \\
& \nabla_x S_j\in L^\infty(G\times I), \label{sda2c}
\end{align}
Note that (\ref{sda2b}) implies that in $\ol{G}\times I$,
\be\label{sda2d}
{1\over{S_j}}\leq {1\over{\kappa_j}}.
\ee
We point out that the assumption \eqref{sda2c}
will, in fact, be needed only in section \ref{evo-op} when considering (a special case of) the
problem \eqref{csda1a}-\eqref{csda3} within the context of the theory of evolution operators
(see Theorem \ref{coupthev}, Eq. \eqref{ass5n}).

In order to prove some accretivity properties for the \emph{scattering-collision operator}
$\Sigma-K:L^2(G\times S\times I)^3\to L^2(G\times S\times I)^3$ (or, equivalently, dissipativity properties 
for the operator $-\Sigma+K$) we assume that the cross-sections $\Sigma_j$, $\sigma_{jk}$, satisfy the  condition given in section \ref{co-cs}: There exists $c\geq 0$ such that
for a.e. $(x,\omega,E)\in G\times S\times I$, and for every $j=1,2,3$,
\be\label{sda3}
\Sigma_j(x,\omega,E)-\sum_{k=1}^3\int_{S\times I}\sigma_{jk}(x,\omega,\omega',E,E') d\rho_S^{jk}(\omega') d\rho_I^{jk}(E')
\geq c,
\ee
and
\be\label{sda4}
\Sigma_j(x,\omega,E)-\sum_{k=1}^3\int_{S\times I}\sigma_{kj}(x,\omega',\omega,E',E) d\rho_S^{kj}(\omega') d\rho_I^{kj}(E')
\geq c.
\ee

We recall (Theorem \ref{diss-for-coupled}) the following {\it accretivity} result.

\begin{theorem}\label{sdath1}
Let the assumptions (\ref{scateh}), (\ref{colleh}), (\ref{sda3}) and (\ref{sda4}) be valid. Then
for every $\psi\in L^2(G\times S\times I)^3$,
\be\label{sda5}
\la(\Sigma-K)\psi,\psi\ra_{L^2(G\times S\times I)^3}\geq c \n{\psi}^2_{L^2(G\times S\times I)^3}.
\ee
\end{theorem}

Notice that the estimate \eqref{sda5} is equivalent to the property that
for every $\lambda >0$,
and $\psi\in L^2(G\times S\times I)^3$,
\be\label{sda5a}
\n{ (\lambda I-(-\Sigma+K+cI))\psi}_{L^2(G\times S\times I)^3}\geq \lambda \n{\psi}_{L^2(G\times S\times I)^3},
\ee
which means that the operator $-\Sigma+K+cI:L^2(G\times S\times I)^3\to L^2(G\times S\times I)^3$ is {\it dissipative} (\cite[Section II.3.b]{engelnagel}, or \cite[Section 1.4]{pazy83}).

For the  coupled BTE system (\ref{intro1}), (\ref{intro2}) we formulate the following result which is a slight modification of results given in \cite{tervo17-up}. Note that it is valid only for Schur collision operators and hence it does not govern completely the particle transport including charged particles, such as applications in radiation therapy.

\begin{theorem}\label{origbte}
Suppose that the assumptions (\ref{scateh}), (\ref{colleh}), (\ref{sda3}), (\ref{sda4}) are valid,
and that $c$ is strictly positive.
Then for every $f\in L^2(G\times S\times I)^3$ and $g\in T^2_{\tau_-}(\Gamma_-)^3$
the following assertions hold.
\begin{enumerate}
\item[(i)] The boundary value problem
\begin{gather}
\omega\cdot\nabla_x\psi_j+\Sigma_j\psi_j-K_j\psi=f_j, \nonumber\\
{\psi_j}_{|\Gamma_-}=g_j, \label{origbte1}
\end{gather}
for $j=1,2,3$, has a unique solution $\psi\in  W^2(G\times S\times I)^3$.

\item[(ii)] There exists a constant $C>0$ such that
\be\label{origbte2}
\n{\psi}_{W^2(G\times S\times I)^3}\leq C\big(\n{f}_{L^2(G\times S\times I)^3}+\n{g}_{T^2_{\tau_-}(\Gamma_-)^3}\big).
\ee

\item[(iii)] If $f\geq 0$ and $g\geq 0$, then $\psi\geq 0$,
i.e. the solution $\psi$ is non-negative for non-negative data $f,g$.
\end{enumerate}
\end{theorem}

\begin{proof}
The assertions follow from the considerations expressed in \cite{tervo17-up} noting that in Lemma 5.8 (see its proof) of \cite{tervo17-up} we actually have
\[
\n{Lg}_{L^2(G\times S\times I)}\leq \n{g}_{T^2_{\tau_-}(\Gamma_-)},\ {\rm for}\ g\in T^2_{\tau_-}(\Gamma_-).
\]
We omit details here.
\end{proof}

The corresponding result  for time-dependent coupled system of BTEs has been proven in \cite{tervo17-up} as well.

\subsection{Existence of Solutions Based on Variational Formulation} \label{LLM}

As before, we perform a change of unknown functions, by setting
\be 
\phi_j:=e^{CE}\psi_j,\quad j=1,2,3,
\ee
where the constant $C$ will be fixed below.
This transforms the problem (\ref{csda1a})-(\ref{csda3}) into an equivalent form,
with transport equation on $G\times S\times I$, 
\begin{gather}
\omega\cdot\nabla_x\phi_1+\Sigma_1\phi_1- K_{1,C}\phi = {\bf f}_1,\label{cosyst1}\\
-{\p {(S_j\phi_j)}E}+\omega\cdot\nabla_x\phi_j+CS_j\phi_j+\Sigma_{j}\phi_j
-{K}_{j,C}\phi={\bf f}_j,\quad j=2,3,\label{cosyst2}
\end{gather}
boundary condition on $\Gamma_-$,
\be\label{cosyst3}
{\phi}_{|\Gamma_-}={\bf g},
\ee
and initial condition on $G\times S$,
\be\label{cosyst4}
\phi_j(\cdot,\cdot,E_m)=0,\quad j=2,3,
\ee
where $\phi=(\phi_1,\phi_2,\phi_3)$ and
\[
{\bf f}_j={}e^{CE} f_j,\quad
{\bf g}_j={}e^{CE} g_j,\quad j=1,2,3.
\]
The operator $K_C=(K_{1,C},K_{2,C},K_{3,C})$ in (\ref{cosyst1}), (\ref{cosyst2}) above is given by
\be 
K_{j,C}\phi:=\sum_{k=1}^3\int_{S\times I}\sigma_{kj,C}(x,\omega',\omega,E',E)\phi_k(x,\omega',E')d\rho_S^{kj}(\omega') d\rho_I^{kj}(E'),
\ee
where the corresponding differential cross-sections are
\be 
\sigma_{kj,C}(x,\omega',\omega,E',E):=\sigma_{kj}(x,\omega',\omega,E',E)e^{C(E-E')},\quad j,k=1,2,3.
\ee

\begin{remark}
We could perform a more refined change of the unknown functions by setting
$\phi_1:=\psi_1$, $\phi_j:=e^{C_jE}\psi_j$.
In this case, the modified differential cross-sections $\sigma_{kj,C}$ would be
\begin{gather*}
\sigma_{11,C}=\sigma_{11},\quad 
\sigma_{1k,C}=\sigma_{1k}e^{C_k E},\quad k=2,3,
\nonumber\\
\sigma_{k1,C}=\sigma_{k1}e^{-C_k E'},\quad
\sigma_{kj,C}=\sigma_{kj}e^{C_2 E-C_k E'},\quad j,k=2,3.
\end{gather*}
We omit further considerations of such refinement in this work.
\end{remark}

In the rest of the section, work under the assumptions (\ref{scateh}), (\ref{colleh}), (\ref{sda2a}), \eqref{sda2a-2}, (\ref{sda2b}) (as already mentioned, the assumption \eqref{sda2c} will be needed only in section \ref{evo-op})
and suppose furthermore that (for $C$ given below in (\ref{cosyst6a}) and) for some $c\geq 0$ the estimates
\begin{align}
& \Sigma_j(x,\omega,E)-\sum_{k=1}^3\int_{ I'}\int_{S'}\sigma_{jk,C}(x,\omega,\omega',E,E') d\rho_S^{jk}(\omega') d\rho_I^{jk}(E')
\geq c,\label{csda3aa} \\
& \Sigma_j(x,\omega,E)-\sum_{k=1}^3\int_{I'}\int_{S'}\sigma_{kj,C}(x,\omega',\omega,E',E)  d\rho_S^{kj}(\omega') d\rho_I^{kj}(E')
\geq c, \label{csda4aa}
\end{align}
hold for a.e. $(x,\omega,E)\in G\times S\times I$.

At first,
we apply the variational formulations to deduce existence of solutions. 
Recall that  the inner product in $L^2(G\times S\times I)^3$ is given by
\[
\la \phi,v\ra_{L^2(G\times S\times I)^3}
=\sum_{j=1}^3 \la \phi_j,v_j\ra_{L^2(G\times S\times I)},
\]
and analogously in other products of inner product spaces.
Integrating by parts and applying the Green's formula (\ref{green}) we find (similarly as in section \ref{esols}) that the bilinear form ${\bf B}(\cdot,\cdot):C^1(\ol G\times S\times I)^3\times C^1(\ol G\times S\times I)^3\to\R $ and the linear form ${\bf F}:
C^1(\ol G\times S\times I)^3 \to\R$ corresponding to the problem 
(\ref{cosyst1})- (\ref{cosyst4}) are
\bea\label{cosyst5}
{\bf B}(\phi,v)=&
\sum_{j=2,3}\la\phi_j,S_j{\p {v_j}E}\ra_{L^2(G\times S\times I)}
-\la\phi,\omega\cdot\nabla_x v\ra_{L^2(G\times S\times I)^3}
\nonumber\\
&
+\sum_{j=2,3} C\la \phi_j,S_jv_j\ra_{L^2(G\times S\times I)}
+\la\phi,(\Sigma^*-K_C^*) v\ra_{L^2(G\times S\times I)^3}
\nonumber\\
&
+\la \gamma_+(\phi), \gamma_+(v) \ra_{T^2(\Gamma_+)^3}
+\sum_{j=2,3} \la \phi_j(\cdot,\cdot,0),S_j(\cdot,0) v_j(\cdot,\cdot,0)\ra_{L^2(G\times S)} ,
\eea
and
\be\label{cosyst6}
{\bf F}(v)
=
\la {\bf f},v\ra_{L^2(G\times S\times I)^3}
+\la {\bf g},\gamma_-(v)\ra_{T^2(\Gamma_-)^3}.
\ee

For $j=2,3$, let
\[
q_j:={1\over 2}\esssup_{(x,E)\in G\times I}{\p {S_j}E}(x,E).
\]
Moreover, define
\begin{align}\label{cosyst6a}
&C_j:={{\max\{q_j,0\}}\over\kappa_j},\quad j=2,3, \nonumber \\
&C:=\max\{C_1,C_2\}. 
\end{align}

The appropriate Hilbert spaces are defined as
\be\label{eq:tilde_H}
&{\s H}:=H\times H_1\times H_1, \nonumber \\
&\tilde{\s H}:=\tilde {\bf W}^2(G\times S\times I)\times H_2\times  H_2,
\ee
where the space $\tilde {\bf W}^2(G\times S\times I)$ was defined just before Corollary \ref{convg}.
Recall from section \ref{fs} that the elements of ${\s H}$  are of the form
\[
\tilde\phi=(\tilde\phi_1,\tilde\phi_2,\tilde\phi_3)
=\big((\phi_1,q_1), (\phi_2,q_2,p_{02},p_{{\rm m}2}), (\phi_3,q_3,p_{03},p_{{\rm m}3})\big),
\]
with $\phi_i\in L^2(G\times S\times I)$, $q_i\in T^2(\Gamma)$ for $i=1,2,3$,
and $p_{0j}, p_{{\rm m}j}\in L^2(G\times S)$ for $j=2,3$.
Moreover, $\tilde {\s H}\subset {\s H}$ through the continuous embedding
\begin{align}\label{eq:embed_sH}
&
v=(v_1,v_2,v_2)
\mapsto \big( (v_1,\gamma(v_1)),\varrho(v_2), \varrho(v_3) \big), \nonumber \\
&
\varrho(v_j):=\big(v_j,\gamma(v_j),v_j(\cdot,\cdot,0),v_j(\cdot,\cdot,E_{\rm m})\big),\hspace{1cm} j=2,3.
\end{align}
In spaces ${\s H}$ and $\tilde {\s H}$ we use respectively the inner products
\[
\la \phi,v\ra_{{\s H}}=\la\phi_1,v_1\ra_{H}+\sum_{j=2,3}\la\phi_j,v_j\ra_{H_1}
\]
and
\[
\la\phi,v\ra_{\tilde {\s H}}
=\la \phi_1,v_1\ra_{{\tilde {\bf W}^2(G\times S\times I)}}+\sum_{j=2,3}\la\phi_j,v_j\ra_{H_2}.
\]
The bilinear form ${\bf B}:C^1(\ol G\times S\times I)^3\times C^1(\ol G\times S\times I)^3\to\R$
has the following boundedness and coercivity properties.

\begin{theorem}\label{cosystth1}
Suppose that the assumptions 
(\ref{scateh}), (\ref{colleh}), (\ref{sda2a}), \eqref{sda2a-2}, (\ref{sda2b}) are valid and that
(\ref{csda3aa}), (\ref{csda4aa}) hold for $c>0$ and for $C$ given in (\ref{cosyst6a}). 
Then there exists $M>0$ such that
\be\label{cosyst12}
|{\bf B}(\phi,v)|\leq M\n{\phi}_{{\s H}}\n{v}_{\tilde {\s H}}\quad \forall\phi, v\in C^1(\ol G\times S\times I)^3,
\ee
and 
\be\label{cosyst13}
{\bf B}(\phi,\phi)\geq c'\n{\phi}_{{\s H}}^2\quad \forall\phi\in C^1(\ol G\times S\times I)^3,
\ee
where
\begin{align}\label{cocprime}
c':=\min\{{1\over 2},{{\kappa_2}\over 2},{{\kappa_3}\over 2},c\}.
\end{align}
In addition, for all $v\in C^1(\ol G\times S\times I)^3$,
\be\label{cosyst14}
|{\bf F}(v)|\leq 
\big(\n{\bf f}_{L^2(G\times S\times I)^3}+\n{\bf g}_{T^2(\Gamma_-)^3}\big)\n{v}_{{\s H}}.
\ee
\end{theorem}

\begin{proof}
The boundedness (\ref{cosyst12}) can be seen as in the proof of Theorem \ref{csdath1}. The assumptions (\ref{csda3aa}), (\ref{csda4aa}) imply 
by Theorem \ref{sdath1} that
\be
\la (\Sigma-K_C)\phi,\phi\ra_{L^2(G\times S\times I)^3}\geq c\n{\phi}^2_{L^2(G\times S\times I)^3}.
\ee
Hence we see as in the proof of Theorem \ref{csdath1}
that for $C=\max\{C_1,C_2\}$
the coercitivity (\ref{cosyst13}) holds with the stated $c'>0$. 
The estimate (\ref{cosyst14}) is immediate (see \eqref{csda39}) and so the proof is finished.
\end{proof}

Due to the above theorem, the bilinear form ${\bf B}(\cdot,\cdot)$ has a unique extension $\tilde {\bf B}(\cdot,\cdot):{\s H}\times \tilde {\s H}\to \R$ which satisfies
\be\label{cosyst15}
|\tilde {\bf B}(\tilde\phi,v)|\leq M\n{\tilde\phi}_{{\s H}}\n{v}_{\tilde {\s H}}\quad \forall\tilde\phi\in {\s H},\ v\in \tilde {\s H},
\ee
and 
\be\label{cosyst16}
\tilde {\bf B}(v,v)\geq c'\n{v}_{{\s H}}^2\quad \forall v\in \tilde {\s H}.
\ee
Likewise, \eqref{cosyst14} implies that the linear form ${\bf F}$ has a unique bounded extension ${\s H}\to\R$, which we still denote by ${\bf F}$.
The variational equation corresponding to the problem 
(\ref{cosyst1})-(\ref{cosyst4}) is
\be\label{cosyst7}
\tilde {\bf B}(\tilde\phi,v)={\bf F}(v)\quad \forall v\in \tilde {\s H}.
\ee

For the coupled BTE we have the following
variational existence theorem, to be compared with Theorem \ref{csdath3} for single particle transport.

\begin{theorem}\label{cosystth2}
Suppose that the assumptions
(\ref{scateh}), (\ref{colleh}), (\ref{sda2a}), \eqref{sda2a-2}, (\ref{sda2b}) are valid,
and that (\ref{csda3aa}), (\ref{csda4aa}) hold for $c>0$ and for $C$ given in (\ref{cosyst6a}). 
Let ${\bf f}\in L^2(G\times S\times I)^3$ and ${\bf g}\in T^2(\Gamma_-)^3$.
Then the following assertions hold.

(i) The variational equation
\be\label{cosyst17}
\tilde {\bf B}(\tilde\phi,v)={\bf F}(v)\quad \forall v\in \tilde {\s H},
\ee
has a solution $\tilde{\phi}=(\tilde{\phi}_1, \tilde{\phi}_2, \tilde{\phi}_3)\in {\s H}$.
Writing $\tilde{\phi}_1=(\phi_1,q_1)$, $\tilde \phi_j=(\phi_j,q_j,p_{0j},p_{mj})$, $j=2,3$,
and $\phi=(\phi_1,\phi_2,\phi_3)\in L^2(G\times S\times I)^3$,
then $\phi\in\mc{H}_{{\bf P}}(G\times S\times I^\circ)$
(see \eqref{eq:H_bfP})
and it is a weak (distributional) solution of the system of equations (\ref{cosyst1}), (\ref{cosyst2}),
and $\phi_1\in W^2(G\times S\times I)$.

(ii) Suppose that additionally the assumption ${\bf TC}$ holds (p. \pageref{as:TC}).
Then a solution $\phi$ of the equations (\ref{cosyst1}), (\ref{cosyst2}) obtained in part (i) is a solution of the problem 
\eqref{cosyst1}-\eqref{cosyst4}. 

(iii) Under the assumptions imposed in part (ii),
any solution $\phi$ of the problem
\eqref{cosyst1}-\eqref{cosyst4}
is unique and obeys the estimate
\be\label{cosyst17a}
\n{\phi}_{{\s H}}\leq {1\over{ c'}}
\big(\n{{\bf f}}_{L^2(G\times S\times I)^3}+\n{{\bf g}}_{T^2(\Gamma_-)^3}\big),
\ee
where $c'$ is given in \eqref{cocprime}.
\end{theorem}

\begin{proof}
The proofs of items (i)-(iii) are analogous to the  proofs of the corresponding items in Theorem \ref{csdath3}.
Note that the claim $\phi_1\in W^2(G\times S\times I)$ in part (i)
follows from noticing that by \eqref{cosyst1} we have
$\omega\cdot\nabla_x\phi_1=-\Sigma_1\phi_1+K_{1,C}\phi+{\bf f}_1$, and 
the right hand side belongs to $L^2(G\times S\times I)$.
\end{proof}

The problem (\ref{csda1a})-(\ref{csda3}) also admits a variational formulation.
Define bilinear $\tilde{\bf B}_0(\cdot,\cdot):{\s H}\times \tilde {\s H}\to\R$ by
\bea\label{cosyst5a}
\tilde{\bf B}_0(\tilde{\psi},v)
=&
\sum_{j=2,3} \la\psi_j,S_j{\p {v_j}E}\ra_{L^2(G\times S\times I)}
-\la \psi, \omega\cdot\nabla_x v\ra_{L^2(G\times S\times I)^3}
+\la\psi,(\Sigma^*-K^*) v\ra_{L^2(G\times S\times I)^3} \nonumber\\
&
+\la q, \gamma_+(v)\ra_{T^2(\Gamma_+)^3}
+\sum_{j=2,3}\la p_{0j},S_j(\cdot,0) v_j(\cdot,\cdot,0)\ra_{L^2(G\times S)},
\eea
where $\tilde{\psi}=(\tilde{\psi}_1,\tilde{\psi}_2,\tilde{\psi}_3)$,
$\tilde{\psi}_1=(\psi_1,q_1)$,
$\tilde{\psi}_j=(\psi_j,q_j,p_{0j},p_{{\rm m}j})$, $j=2,3$,
and $\psi=(\psi_1,\psi_2,\psi_3)$, $q=(q_1,q_2,q_3)$.
In other words, $\tilde{\bf B}_0$ is the unique extension of (\ref{cosyst5}) with $C=0$.
Then by part (ii) of Theorem \ref{cosystth2},
a given $\psi\in W^2(G\times S\times I)\times W_1^2(G\times S\times I)^2 $
is a solution of the problem  
(\ref{csda1a})-(\ref{csda3})
if (and only if) the variational equation 
\[
\tilde{\bf B}_0(\tilde{\psi},v)={\bf F}_0(v),\quad \forall v\in \tilde {\s H},
\]
holds. Here $\psi$ is identified with $\tilde{\psi}\in {\s H}$ through the mapping given in \eqref{eq:embed_sH}, and
\be\label{linf}
{\bf F}_0(v)
=\la {f},v\ra_{L^2(G\times S\times I)^3}+
\la {g},\gamma_-(v)\ra_{T^2(\Gamma_-)^3}.
\ee

We have the following corollary for the original problem.

\begin{corollary}\label{cosystco1} 
Suppose that the assumptions of Theorem \ref{cosystth2} are valid.
Let ${ f}\in L^2(G\times S\times I)^3$ and  ${ g}\in T^2(\Gamma_-)^3$.
Then the following assertions hold.

(i) The variational equation
\be\label{vareqco}
\tilde {\bf B}_0(\tilde\psi,v)={\bf F}_0(v)\quad \forall v\in \tilde{\s H}
\ee
has a solution $\tilde{\psi}=(\tilde{\psi}_1, \tilde{\psi}_2, \tilde{\psi}_3)\in {\s H}$. 
Writing $\tilde{\psi}_1=(\psi_1,q_1)$,
$\tilde{\psi}_j=(\psi_j,q_j,p_{0j},p_{mj}),\ j=2,3$, and $\psi=(\psi_1,\psi_2,\psi_3)\in L^2(G\times S\times I)^3$,
then $\psi\in\mc{H}_{{\bf P}}(G\times S\times I^\circ)$
(see \eqref{eq:H_bfP})
and it is a weak (distributional) solution of the system of equations (\ref{csda1a}), (\ref{csda1b}),
and $\psi_1\in W^2(G\times S\times I)$.

(ii) Suppose that additionally the assumption ${\bf TC}$ holds (p. \pageref{as:TC}).
Then a solution $\psi$ of the equations (\ref{csda1a}), (\ref{csda1b})
obtained in part (i) is a solution of the problem 
(\ref{csda1a})-(\ref{csda3}). 

(iii) Under the assumptions imposed in part (ii),
any solution $\psi$ of the problem (\ref{csda1a})-(\ref{csda3}) 
is unique and obeys the estimate
\be\label{csda40aaa-co}
\n{\psi}_{{\s H}}\leq {{e^{CE_{\rm m}}}\over {c'}}
\big(\n{{ f}}_{L^2(G\times S\times I)^3}+\n{{ g}}_{T^2(\Gamma_-)^3}\big).
\ee
(Recall that $C$ is defined in \eqref{cosyst6a}, $c'$ in \eqref{cocprime} and that $E_m$ is the cutoff energy.)
\end{corollary}

\begin{proof}  
A solution $\psi$ of the problem (\ref{csda1a})-(\ref{csda3}) is obtained
from a solution $\phi$ of the problem \eqref{cosyst1}-\eqref{cosyst4} by
taking $\psi=e^{-CE}\phi$.
Note that if $\phi_1\in W^2(G\times S\times I)$, then $\psi_1\in W^2(G\times S\times I)$ as well.
The rest of the proof proceeds in exactly the same way as that for Corollary \ref{csdaco1}
(of course, one uses Theorem \ref{cosystth2} instead of Theorem \ref{csdath3}).
\end{proof}

\subsection{Existence of Solutions Based on $m$-dissipativity}\label{mdiss-op}
 
The method of section \ref{m-d} can be extended to the case of 
coupled system in a straightforward manner.
Let us state what the problem to be solved for $\phi$ is in this context, in accordance with section \ref{m-d}.

Given ${\bf f}=({\bf f}_1, {\bf f}_2, {\bf f}_3)\in L^2(G\times S\times I)^3$,
and ${\bf g}\in T^2(\Gamma_-)\times H^1(I,T^2(\Gamma_-'))^2$,
find $\phi=(\phi_1,\phi_2,\phi_3)\in L^2(G\times S\times I)^3$
which satisfies the system of equations
on $G\times S\times I$,
\begin{align}
\omega\cdot\nabla_x\phi_1+\Sigma_1\phi_1- K_{1,C}\phi = {}& {\bf f}_1,\label{codiss1}\\
-{\p {(S_j\phi_j)}E}+\omega\cdot\nabla_x\phi_j+CS_j\phi_j+\Sigma_{j}\phi_j
-{K}_{j,C}\phi = {}& {\bf f}_j,\quad j=2,3,\label{codiss2}
\end{align}
the boundary condition on $\Gamma_-$,
\be\label{codiss3}
{\phi}_{|\Gamma_-}={\bf g},
\ee
and the initial condition on $G\times S$,
\be\label{codiss4}
\phi_j(\cdot,\cdot,E_m)=0,\quad j=2,3.
\ee
 
Let
\[
P_{1}(x,\omega,E,D)\phi_1:={}&\omega\cdot\nabla_x\phi_1, \\[2mm]
P_{C,j}(x,\omega,E,D)\phi_j:={}&-{\p {(S_j\phi_j)}E}+\omega\cdot\nabla_x\phi_j
+CS_j\phi_j,\quad j=2,3, \\[2mm]
{\bf P}_{C}(x,\omega,E,D)\phi:={}&\big(P_{1}(x,\omega,E,D)\phi_1,
P_{C,2}(x,\omega,E,D)\phi_2,P_{C,2}(x,\omega,E,D)\phi_3\big).
\]
When $C=0$ we write ${\bf P}(x,\omega,E,D):={\bf P}_0(x,\omega,E,D)$. Define
\begin{multline}
{\s H}_{\bf P}(G\times S\times I^\circ)
:=\{\psi\in L^2(G\times S\times I)^3\ | \\
{\bf P}(x,\omega,E,D)\psi\in 
L^2(G\times S\times I)^3\ \textrm{in the weak sense}\}. \label{eq:H_bfP}
\end{multline}

The operator $\tilde {\bf P}_{C,0}$ is defined
in the same way as $\tilde {P}_{C,0}$ in section \ref{m-d},
namely it is the smallest closed extension (closure) 
of ${\bf P}_{C,0}$, where
\[
D({\bf P}_{C,0}):={}&\big\{
\phi\in \tilde{W}^2(G\times S\times I)\times \big(\tilde{W}^2(G\times S\times I)\cap H^1(I,L^2(G\times S)\big)^2
\ \big| \\[2mm]
&\hspace{2.5mm} \phi_{|\Gamma_-}=0,\ \phi_j(\cdot,\cdot,E_m)=0,\ j=2,3\big\} \\[2mm]
{\bf P}_{C,0}\phi:={}&{\bf P}_{C}(x,\omega,E,D)\phi.
\]

Using these notations,
the problem \eqref{codiss1}-\eqref{codiss4}
with ${\bf g}=0$, in the strong sense, is equivalent to
\[
(\tilde {\bf P}_{C,0}+\Sigma-K_C)\phi={\bf f},
\]
where $\phi\in D(\tilde {\bf P}_{C,0})$. 

We  assume that the stopping powers for $j=2,3$ satisfy
\begin{align}
{}& S_j\in C^2(I, L^\infty(G)), \label{Sj-ass:1} \\[2mm]
{}& \kappa_j:=\inf_{(x,E)\in \ol{G}\times I}S_j(x,E)>0, \label{Sj-ass:2} \\[2mm]
{}& \nabla_x S_j\in L^\infty(G\times I). \label{Sj-ass:3}
\end{align}

We give here only the following concluding result,
which can be proven using the methods of section \ref{m-d}
and those of \cite[Section 5.3]{tervo17-up}.
Recall that by the conventions adopted above,
${\bf f}_j=e^{CE}f_j$ and ${\bf g}_j=e^{CE}g_j$,
where $C$ is as defined in \eqref{cosyst6a}.

\begin{theorem}\label{m-d-j-co1}
Suppose that the assumptions
(\ref{scateh}), (\ref{colleh}), (\ref{csda3aa}), (\ref{csda4aa})  (with $c>0$)
and \eqref{Sj-ass:1}, \eqref{Sj-ass:2}, \eqref{Sj-ass:3} are valid.
Furthermore, suppose that ${f}\in L^2(G\times S\times I)^3$,
and $g\in T^2(\Gamma_-)\times H^1(I,T^2(\Gamma_-'))^2$
is such that the compatibility condition
\be\label{comp-d-j}
g_j(\cdot,\cdot,E_{\rm m})=0,\quad j=2,3,
\ee
holds.
Then the problem \eqref{codiss1}-\eqref{codiss4}
has a unique solution 
$\psi\in {\s H}_{\bf P}(G\times S\times I^\circ)$.
In addition, there exists a constant $C_1>0$ such that
\emph{a priori} estimate
\be\label{diss-co-es}
\n{\psi}_{L^2(G\times S\times I)^3}
\leq
C_1\big(\n{{ f}}_{L^2(G\times S\times I)^3}+
\n{{ g}}_{T^2(\Gamma_-)\times H^1(I,T^2(\Gamma_-'))^2}\big),
\ee
holds.
\end{theorem}

\subsection{Existence of Solutions Based on the Theory of Evolution Equations}\label{evo-op}

For collision operator of special type,
the existence result based on the theory of evolution operators is valid also for the coupled system.
One of the important features of this approach is that it yields more regularity for the solution.

We assume that the collision operator $K=(K_1,K_2,K_3)$ is of the form 
\be\label{ec1moda}
(K_j\psi)(x,\omega,E)=\sum_{k=1}^3\int_S\tilde\sigma_{kj}(x,\omega',\omega,E)\psi_k(x,\omega',E)d\omega',\quad j=1,2,3.
\ee
 
For a fixed $E\in I$ we define  bounded linear operators
$L^2(G\times S)\to L^2(G\times S)$  by
\begin{gather*}
(\Sigma_1(E)v)(x,\omega):=\Sigma_1(x,\omega,E)v(x,\omega), \\
(\ol K_1(E)v)(x,\omega):=\int_S\tilde\sigma_{11}(x,\omega',\omega,E)v(x,\omega') d\omega'.
\end{gather*}

In order to avoid ambiguity,
below we often denote by $S'$ the sphere $S$ for the variable $\omega'$, while $S$ is reserved for $\omega$. For example, saying that $h$ is a function $S\to (S'\to\R)$ (resp. $S'\to (S\to \R)$)
means that we would write $h(\omega)(\omega')$ (resp. $h(\omega')(\omega)$).

\begin{lemma}\label{coupevle2}
Suppose that 
\begin{gather}
\Sigma_1\in C^1(I,L^\infty(G\times S)), \label{ass1-11} \\[2mm]
\tilde{\sigma}_{11}\in C^1(I,L^\infty(G\times S,L^1(S')))\cap
C^1(I,L^\infty(G\times S',L^1(S))), \label{ass3-11} \\[2mm]
\Sigma_1\geq 0,\quad \tilde\sigma_{11}\geq 0, \label{ass4-a}
\end{gather}
and that for some $c>0$ the following hold a.e. $(x,\omega,E)\in G\times S\times I$,
\be
&\Sigma_1(x,\omega,E)-\int_S\tilde\sigma_{11}(x,\omega,\omega',E) d\omega' \geq c, \label{ec8-a} \\
&\Sigma_1(x,\omega,E)-\int_S\tilde\sigma_{11}(x,\omega',\omega,E) d\omega' \geq c. \label{ec9-a}
\ee

Then for every $q\in C^1(I,L^2(G\times S))$ and every $E\in I$ the problem
\begin{gather}
\omega\cdot\nabla_x v+\Sigma_1(E)v-\ol K_1(E)v=q(E), \nonumber\\
v_{|\Gamma_-'}=0, \label{cle1}
\end{gather}
has a unique solution $v=v(E)\in \tilde W^2_{-,0}(G\times S)$,
and
\be\label{cle2}
\n{v(E)}_{L^2(G\times S)}\leq {1\over c} \n{q(E)}_{L^2(G\times S)},
\quad \forall E\in I.
\ee
In addition, $v(\cdot):I\to L^2(G\times S)$ belongs to $C^1(I,L^2(G\times S))$ and
for all $E\in I$,
\bea\label{cle3}
\n{{\p {v} E}(E)}_{L^2(G\times S)}
\leq{} &
{1\over c}\Big( \n{{\p {q} E}(E)}_{L^2(G\times S)}+{1\over c}
\n{{\p {\Sigma_1}E}(E)}_{L^\infty(G\times S)}\n{q(E)}_{L^2(G\times S)}\nonumber\\
{}& +{1\over c}
\n{{\p {\ol K_1}E}(E)}\n{q(E)}_{L^2(G\times S)}
\Big),
\eea
where
\[
\Big({\p {\ol K_1}E}(E)v\Big)(x,\omega)
:=\int_S{\p {\tilde\sigma_{11}}E}(x,\omega',\omega,E)v(x,\omega') d\omega' .
\]
\end{lemma}

\begin{proof}
That the problem \eqref{cle1} has, for every fixed $E\in I$, a unique solution $v=v(E)$,
and that the estimate \eqref{cle2} holds can be proven similarly to \cite[Corollary 5.15]{tervo17-up}
(for the existence of solutions, see also \cite[Lemma 4, p. 241]{dautraylionsv6}).

A. At first we show that $v\in C(I,L^2(G\times S))$. Let $E_1,\ E_2\in I$.
Since $v(E)\in \tilde W^2_{-,0}(G\times S)$ for all $E\in I$, it follows that $(v(E_1)-v(E_2))_{|\Gamma_-'}=0$.
Moreover, $v(E_1)-v(E_2)$ satisfies the equation
\bea\label{cle4}
&
\omega\cdot\nabla_x\big(v(E_1)-v(E_2)\big)+\Sigma_1(E_1)\big(v(E_1)-v(E_2)\big)
-\ol K_1(E_1)\big(v(E_1)-v(E_2)\big)\nonumber\\
={}&q(E_1)-q(E_2)-\big(\Sigma_1(E_1)-\Sigma_1(E_2)\big)v(E_2)
+\big(\ol K_1(E_1)-\ol K_1(E_2)\big)v(E_2).
\eea
Therefore, one can apply the estimate \eqref{cle2}, obtaining
\bea\label{cle5}
&\n{v(E_1)-v(E_2)}_{L^2(G\times S)} \nonumber\\
\leq{} &
{1\over c} \Big(\n{q(E_1)-q(E_2)}_{L^2(G\times S)}
+\n{(\Sigma_1(E_1)-\Sigma_1(E_2))v(E_2)}_{L^2(G\times S)}\nonumber\\
&
+\n{(\ol K_1(E_1)-\ol K_1(E_2))v(E_2)}_{L^2(G\times S)}\Big)\nonumber\\
\leq{} &
{1\over c} \Big(\n{q(E_1)-q(E_2)}_{L^2(G\times S)}
+\n{\Sigma_1(E_1)-\Sigma_1(E_2)}_{L^\infty(G\times S)}\n{v(E_2)}_{L^2(G\times S)}\nonumber\\
&
+\n{\ol K_1(E_1)-\ol K_1(E_2)}\n{v(E_2)}_{L^2(G\times S)}\Big),
\eea
where (see \eqref{co(b)} or \eqref{eq:LE_estim})
\bea\label{cle5a}
\n{\ol K_1(E_1)-\ol K_1(E_2)}\leq
\n{\tilde\sigma_{11}(E_1)-\tilde\sigma_{11}(E_2)}_{L^\infty(G\times S,L^1(S'))}^{1/2}
\n{\tilde\sigma_{11}(E_1)-\tilde\sigma_{11}(E_2)}_{L^\infty(G\times S',L^1(S))}^{1/2}.
\eea
The continuity of $v:I\to L^2(G\times S)$ follows immediately
from these estimates and assumptions \eqref{ass1-11}, \eqref{ass3-11}.

B. Next we verify that $v\in C^1(I,L^2(G\times S))$ and that the estimate (\ref{cle3}) holds. 
For a map $k:G\times S\times I\to\R$, which we identify as a map $I\to (G\times S\to\R)$; $k(E)(x,\omega)=k(x,\omega,E)$ whenever appropriate,
define for $h\not=0$ small enough
\[
(\delta_h k)(E):={{k(E+h)-k(E)}\over h}
\]
where at $E$-boundaries $E=0$ or $E=E_{\rm m}$ we take $h>0$ and $h<0$, respectively.
Since $v(E')\in \tilde W^2_{-,0}(G\times S)$ for all $E'\in I$,
we have $(\delta_h v)(E)_{|\Gamma_-'}=0$, and from the equation (\ref{cle1}) we obtain
(from \eqref{cle4}, with $E_1=E$, $E_2=E+h$)
\bea\label{cle6}
&
\omega\cdot\nabla_x ((\delta_hv)(E))+\Sigma_1(E)(\delta_hv)(E)-\ol K_1(E)(\delta_hv)(E)\nonumber\\
={}&(\delta_h q)(E)-(\delta_h\Sigma_1)(E)v(E+h)+(\delta_h \ol{K}_1)(E)v(E+h),
\eea
where for any $w\in L^2(G\times S)$,
\[
(\delta_h \ol{K}_1)(E)w
:=&\int_S{{\tilde\sigma_{11}(x,\omega',\omega,E+h)-
\tilde\sigma_{11}(x,\omega',\omega,E)}\over h}
w(x,\omega') d\omega' \\
=&\int_S (\delta_{h}\tilde\sigma_{11})(E)(x,\omega',\omega)w(x,\omega') d\omega'.
\]

To show that the limit $\lim_{h\to 0}(\delta_h v)(E)$ exists in $L^2(G\times S)$ it suffices to verify that $\lim_{n\to\infty}(\delta_{h_n} v)(E)$ exists in $L^2(G\times S)$
for any sequence $\{h_n\}$ tending to zero as $n\to\infty$,
or equivalently that $\{(\delta_{h_n}v)(E)\}$ is a Cauchy sequence in $L^2(G\times S)$
for any such sequence $\{h_n\}$.

We see  by (\ref{cle6}) that for $n,m\in\N$,
\bea\label{cle7}
&
\omega\cdot\nabla_x \big((\delta_{h_n}v)(E)-(\delta_{h_m}v)(E)\big)
+\Sigma_1(E)\big((\delta_{h_n}v)(E)-(\delta_{h_m}v)(E)\big) \nonumber\\
&
-\ol{K}_1(E)\big((\delta_{h_n}v)(E)-(\delta_{h_m}v)(E)\big) \nonumber\\
={}&
(\delta_{h_n}q)(E)-(\delta_{h_m}q)(E)
-\big((\delta_{h_n}\Sigma_1)(E)v(E+h_n)-(\delta_{h_m}\Sigma_1)(E)v(E+h_m)\big) \nonumber\\
&
+\big((\delta_{h_n}\ol{K}_1)(E)v(E+h_n)-(\delta_{h_m}\ol{K}_1)(E)v(E+h_m)\big) \nonumber\\
=:{}&f_{h_n,h_m}(E).
\eea
The following estimates hold (see \eqref{cle5}, \eqref{cle5a})
\[
&\n{(\delta_{h_n}\Sigma_1)(E)v(E+h_n)-(\delta_{h_m}\Sigma_1)(E)v(E+h_m)}_{L^2(G\times S)} \\
\leq &\n{(\delta_{h_n}\Sigma_1)(E)-(\delta_{h_m}\Sigma_1)(E)}_{L^\infty(G\times S)}\n{v(E+h_n)}_{L^2(G\times S)} \\
&+\n{(\delta_{h_m}\Sigma_1)(E)}_{L^\infty(G\times S)}\n{v(E+h_n)-v(E+h_m)}_{L^2(G\times S)},
\]
and
\[
& \n{(\delta_{h_n}\ol{K}_1)(E)v(E+h_n)-(\delta_{h_m}\ol{K}_1)(E)v(E+h_m)}_{L^2(G\times S)} \\
\leq & \n{(\delta_{h_n}\ol{K}_1)(E)-(\delta_{h_m}\ol{K}_1)(E)}\n{v(E+h_n)}_{L^2(G\times S)} \\
&+\n{(\delta_{h_m}\ol{K}_1)(E)}\n{v(E+h_n)-v(E+h_m)}_{L^2(G\times S)},
\]
where
\[
\n{(\delta_{h_m}\ol{K}_1)(E)}
\leq \n{(\delta_{h_m}\tilde\sigma_{11})(E)}_{L^\infty(G\times S,L^1(S'))}^{1/2}
\n{(\delta_{h_m}\tilde\sigma_{11})(E)}_{L^\infty(G\times S',L^1(S))}^{1/2},
\]
and
\[
& \n{(\delta_{h_n}\ol{K}_1)(E)-(\delta_{h_m}\ol{K}_1)(E)} \\
\leq &\n{(\delta_{h_n}\tilde\sigma_{11})(E_1)-(\delta_{h_m}\tilde\sigma_{11})(E_2)}_{L^\infty(G\times S,L^1(S'))}^{1/2}
\n{(\delta_{h_n}\tilde\sigma_{11})(E_1)-(\delta_{h_m}\tilde\sigma_{11})(E_2)}_{L^\infty(G\times S',L^1(S))}^{1/2}.
\]
Therefore, since we assume \eqref{ass1-11}, \eqref{ass3-11} and $q\in C^1(I,L^2(G\times S))$,
and since $v\in C(I,L^2(G\times S))$ by part A of the proof,
we find that $f_{h_n,h_m}(E)$ in (\ref{cle7}) convergences to zero in $L^2(G\times S)$
when $n,m\to\infty$.
This fact combined with the estimate obtained by applying \eqref{cle2},
\be
\n{(\delta_{h_n}v)(E)-(\delta_{h_m}v)(E)}_{L^2(G\times S)}
\leq
{1\over c}
\n{f_{h_n,h_m}(E)}_{L^2(G\times S)},
\ee
shows that $\{(\delta_{h_n} v)(E)\}$ is a Cauchy sequence, and so ${\p {v} E}(E)$ exists (in $L^2(G\times S)$) for every $E\in I$.

Applying the estimate (\ref{cle2}) to $(\delta_{h} v)(E)$ which satisfies \eqref{cle6}, we get
\bea\label{cle6a}
\n{(\delta_{h}v)(E)}_{L^2(G\times S)}
\leq{} &
{1\over c}
\Big(\n{(\delta_h q)(E)}_{L^2(G\times S)}+\n{(\delta_h\Sigma_1)(E)}_{L^\infty(G\times S)}\n{v(E+h)}_{L^2(G\times S)}\nonumber\\
&+\n{(\delta_h \ol{K}_1)(E)}\n{v(E+h)}_{L^2(G\times S)}\Big).
\eea
Under the standing assumptions (and the fact that $v$ is continuous),
the inequality \eqref{cle6a} gives in the limit $h\to 0$,
\bea\label{cle8}
\n{{\p {v}E}(E)}_{L^2(G\times S)}
\leq{}&
{1\over c}\Big(\n{{\p {q}E}(E)}_{L^2(G\times S)}+\n{{\p {\Sigma_1}E}(E)}_{L^\infty(G\times S)}\n{v(E)}_{L^2(G\times S)}\nonumber\\
&
+\n{{\p {\ol{K}_1}E}(E)}\n{v(E)}_{L^2(G\times S)}\Big),
\eea
and thus, using (\ref{cle2}) once more, we obtain (\ref{cle3}) as claimed.

It remains to be shown that $\p{v}{E}\in C(I,L^2(G\times S))$.
By letting $h\to 0$ in equation \eqref{cle6} we get that
$\omega\cdot\nabla_x \Big({\p {v}E}(E)\Big)\in L^2(G\times S)$ for all $E\in I$,
and
\begin{multline}\label{cle9}
\omega\cdot\nabla_x \Big({\p {v}E}(E)\Big)+\Sigma_1(E){\p {v}E}(E)-\ol{K}_1(E)\big({\p {v}E}(E)\big) \\
={\p {q}E}(E)-{\p {\Sigma_1}E}(E)v(E)+{\p {\ol{K}_1}E}(E)v(E).
\end{multline}
In addition, $(\delta_h v)(E)\to {\p v{E}}(E)$ in $W^2(G\times S\times I)$ when $h\to 0$,
and therefore, since $((\delta_h v)(E))_{|\Gamma_-'}=0$,
we conclude that $({\p {v}E})(E)_{|\Gamma_-'}=0$ (by applying inflow trace results;
see e.g. Remark \ref{re:W2_0_trace}).
Finally, an application of \eqref{cle2} yields, as in \eqref{cle4}, \eqref{cle5}
(we write $L^2=L^2(G\times S)$, $L^\infty=L^\infty(G\times S)$ in order to slightly compress the formulas),
\bea\label{cle10}
&\n{{\p {v}E}(E_1)-{\p {v}E}(E_2)}_{L^2} \\[2mm]
\leq{} &
{1\over c} \Big(
\n{{\p {q}E}(E_1)-{\p {q}E}(E_2)}_{L^2}
+\n{{\p {\Sigma_1}E}(E_1)-{\p {\Sigma_1}E}(E_2)}_{L^\infty}\n{v(E_2)}_{L^2}\nonumber\\[2mm]
&+\n{{\p {\Sigma_1}E}(E_2)}_{L^\infty}\n{v(E_2)-v(E_1)}_{L^2}
+\n{\Sigma_1(E_1) - \Sigma_1(E_2)}_{L^\infty}\n{{\p {v}E}(E_2)}_{L^2}\nonumber\\[2mm]
&
+\n{{\p {\ol{K}_1}E}(E_1)-{\p {\ol{K}_1}E}(E_2)}\n{v(E_2)}_{L^2}
+\n{{\p {\ol{K}_1}E}(E_2)}\n{v(E_2)-v(E_1)}_{L^2} \nonumber\\[2mm]
&
+\n{\ol{K}_1(E_1)-\ol{K}_1(E_2)}\n{{\p {v}E}(E_2)}_{L^2}
\Big),
\eea
and one can further estimate (see \eqref{cle5a}),
\[
\n{{\p {\ol{K}_1}E}(E_1)-{\p {\ol{K}_1}E}(E_2)}
\leq
\n{{\p {\tilde\sigma_{11}}E}(E_1)-{\p {\tilde\sigma_{11}}E}(E_2)}_{L(S,S')}^{1/2}
\n{{\p {\tilde\sigma_{11}}E}(E_1)-{\p {\tilde\sigma_{11}}E}(E_2)}_{L(S',S)}^{1/2},
\]
where we have written $L(S,S'):=L^\infty(G\times S,L^1(S'))$ and $L(S',S):=L^\infty(G\times S',L^1(S))$.
This implies the continuity of $\p{v}{E}$ that we were to demonstrate, and completes the proof.
\end{proof}

We are ready to formulate, in the context of the theory of evolution operators,
the key existence result for the coupled transport problem \eqref{csda1a}-\eqref{csda3},
assuming the collision operator $K=(K_1,K_2,K_3)$ is of the (special) form \eqref{ec1moda}.

\begin{theorem}\label{coupthev}
Assume that the cross-sections $\Sigma_j$, $\tilde\sigma_{jk}$, where $j,k=1,2,3$, satisfy
\begin{align}
&\Sigma_j\in C^1(I,L^\infty(G\times S)), \label{ass1-aa} \\[2mm]
&\tilde\sigma_{jk}\in C^1(I,L^\infty(G\times S,L^1(S')))\cap
C^1(I,L^\infty(G\times S',L^1(S))), \label{ass3-aa} \\[2mm]
&\Sigma_j\geq 0,\quad \tilde\sigma_{jk}\geq 0, \label{ass4}
\end{align}
and that for some $c>0$,
\begin{align}
&\Sigma_1(x,\omega,E)-\int_S\tilde\sigma_{11}(x,\omega,\omega',E) d\omega' 
\geq c, \label{ec8-aa} \\[2mm]
&\Sigma_1(x,\omega,E)-\int_S\tilde\sigma_{11}(x,\omega',\omega,E) d\omega' 
\geq c, \label{ec9-aa}
\end{align}
for a.e. $(x,\omega,E)\in G\times S\times I$.
Furthermore, assume that the stopping powers $S_j$, where $j=2,3$, satisfy
\begin{align}
& S_j\in C^2(I,L^\infty(G)), \label{ass5} \\[2mm]
& \nabla_x S_j\in L^\infty(G\times I), \label{ass5n} \\[2mm]
& \kappa_j:=\inf_{(x,E)\in \ol{G}\times I} S_j(x,E) > 0. \label{ec4}
\end{align}

Let $f\in C^1(I,L^2(G\times S)^3)$ and let $g\in C^2(I,T^2(\Gamma_-')^3)$ which satisfies the \emph{compatibility condition}
\[
g_j(E_m)=0,\quad j=2,3.
\]
Then the problem (\ref{csda1a})-(\ref{csda3}) has a unique solution $\psi\in \tilde W^2(G\times S\times I)\times \big(C(I,\tilde{W}^2(G\times S)^2)\cap C^1(I,L^2(G\times S)^2)\big)$.
In particular, $\psi\in \tilde{W}^2(G\times S\times I)\times (\tilde{W}^2(G\times S\times I)\cap
W_1^2(G\times S\times I))^2$.

If in addition, for some $c>0$, the inequalities
\begin{align}
\Sigma_j(x,\omega,E)-\sum_{k=1}^3\int_S\tilde\sigma_{jk}(x,\omega,\omega',E) d\omega' 
\geq c, \label{ec8} \\
\Sigma_j(x,\omega,E)-\sum_{k=1}^3\int_S\tilde\sigma_{kj}(x,\omega',\omega,E) d\omega' 
\geq c, \label{ec9}
\end{align}
hold for $j=1,2,3$ and for a.e. $(x,\omega,E)\in G\times S\times I$,
then the solution $\psi$ satisfies the estimate \eqref{csda40aaa-co}.
\end{theorem}

\begin{proof}
At first we notice that by the assumption (\ref{ass3-aa}),
for a.e. $(x,\omega,E)\in G\times S\times I$,
\begin{align}
&\sum_{k=1}^3\int_S\tilde\sigma_{kj}(x,\omega',\omega,E)d\omega'
\leq 
\sum_{k=1}^3\sup_{E\in I}\n{\tilde\sigma_{kj}(E)}_{L^\infty(G\times S,L^1(S'))}=:M_1<\infty, \label{m'} \\
&\sum_{k=1}^3\int_S\tilde\sigma_{jk}(x,\omega',\omega,E)d\omega
\leq 
\sum_{k=1}^3
\sup_{E\in I}\n{\tilde\sigma_{jk}(E)}_{L^\infty(G\times S',L^1(S))}=:M_1'<\infty. \label{m''}
\end{align}

We begin by treating the special case where $g=0$.
Recall that the system of equations of interest on $G\times S\times I$ for $\psi=(\psi_1,\psi_2,\psi_3)$ is
\bea
&\omega\cdot\nabla_x\psi_1+\Sigma_1\psi_1-K_{1}\psi=f_1
\label{proof1}\\
&-{\p {(S_{j}\psi_j)}E}+\omega\cdot\nabla_x\psi_j+\Sigma_{j}\psi_j-K_{j}\psi=f_j,\quad j=2,3.\label{proof2}
\eea
Equation \eqref{proof1} can be written as
\bea\label{pr5.5.1}
\omega\cdot\nabla_x\psi_1+\Sigma_1\psi_1-\ol K_1\psi_1-\ol K\hat\psi=f_1,
\eea
where
\[
\hat\psi:=(\psi_2,\psi_3),
\]
and
\[
& \ol{K}_1\psi_1
:=\int_S\tilde\sigma_{11}(x,\omega',\omega,E)\psi_1(x,\omega',E) d\omega'
=\ol{K}_1(E)\psi_1(E) \\
& \ol{K}\hat\psi
:=\sum_{k=2}^3\int_S\tilde\sigma_{k1}(x,\omega',\omega,E)\psi_k(x,\omega',E) d\omega'
=\ol K(E)\hat\psi(E),
\]
when we define
\begin{alignat*}{3}
&\ol{K}_1(E)v
:=\int_S\tilde\sigma_{11}(x,\omega',\omega,E)v(x,\omega') d\omega',\quad
 && v\in L^2(G\times S), \\
&\ol{K}(E)u
:=\sum_{k=2}^3\int_S\tilde\sigma_{k1}(x,\omega',\omega,E)u_k(x,\omega') d\omega',\quad && u=(u_2,u_3)\in L^2(G\times S)^2.
\end{alignat*}
For any $q\in L^2(G\times S\times I)$ the problem
\begin{gather}
\omega\cdot\nabla_x\psi_1+\Sigma_1\psi_1-\ol K_1\psi_1=q\nonumber
\\
{\psi_1}_{\Gamma_-}=0 \label{pr5.5.3}
\end{gather}
has a unique solution and 
\be\label{pr5.5.4}
\n{\psi_1}_{L^2(G\times S\times I)}\leq {1\over c}\n{q}_{L^2(G\times S\times I)}.
\ee
This can also be proven similarly to \cite[Corollary 5.15]{tervo17-up}
(for the existence of solutions cf. also \cite[Lemma 4, p. 241]{dautraylionsv6}).

Define a linear operator $T_{1,0}:L^2(G\times S\times I)\to L^2(G\times S\times I)$
with domain $D(T_{1,0})$ by 
\bea\label{pr5.5.5}
&
D(T_{1,0}):=\tilde{W}^2_{-,0}(G\times S\times I)\nonumber\\
&
T_{1,0}\psi_1:=\omega\cdot\nabla_x\psi_1+\Sigma_1\psi_1-\ol K_1\psi_1.
\eea
Then $\psi_1=T_{1,0}^{-1}q$ (exists and) is the solution of \eqref{pr5.5.3}, and it follows from \eqref{pr5.5.4} that
\be\label{pr5.5.6}
\n{T_{1,0}^{-1}q}_{L^2(G\times S\times I)}\leq {1\over c}\n{q}_{L^2(G\times S\times I)},\quad \forall q\in L^2(G\times S\times I).
\ee

Equations (\ref{proof2}) can be written as
\be\label{proof4}
-{\p {(S_{j}\psi_j)}E}+\omega\cdot\nabla_x\psi_j+\Sigma_{j}\psi_j-
\hat K_{j}\hat\psi-\hat K_{1,j}\psi_1=f_j,\quad j=2,3
\ee
where
\[
&\hat{K}_j\hat\psi
:=\sum_{k=2}^3
\int_S\tilde\sigma_{kj}(x,\omega',\omega,E)\psi_k(x,\omega',E) d\omega'=\hat K_j(E)\hat\psi(E) \\
&\hat{K}_{1,j}\psi_1
:=\int_S\tilde\sigma_{1j}(x,\omega',\omega,E)\psi_1(x,\omega',E) d\omega'=\hat K_{1,j}(E)\psi_1(E)
\]
when one defines
\begin{alignat*}{3}
&\hat{K}_j(E)u
:=\sum_{k=2}^3\int_S\tilde\sigma_{kj}(x,\omega',\omega,E)u_k(x,\omega') d\omega',\quad && u=(u_2,u_3)\in L^2(G\times S)^2, \\
&\hat{K}_{1,j}(E)v
:=\int_S\tilde\sigma_{1j}(x,\omega',\omega,E)v(x,\omega') d\omega',\quad && v\in L^2(G\times S).
\end{alignat*}
Therefore, if $\psi=(\psi_1,\psi_2,\psi_3)=(\psi_1,\hat{\psi})$ is a solution of \eqref{proof1}, \eqref{proof2},
we have by equation \eqref{pr5.5.1},
\be\label{psi1}
\psi_1=T_{1,0}^{-1}(f_1+\ol K\hat\psi)
\ee
and hence by \eqref{proof4} the function $\hat\psi$ satisfies the equation,
for $j=2,3$,
\bea\label{pr5.5.7}
-{\p {(S_j\psi_j)}E}+\omega\cdot\nabla_x\psi_j+\Sigma_j\psi_j-\hat K_j\hat\psi-\hat K_{1,j}(T_{1,0}^{-1}(\ol K\hat\psi))
=f_j+\hat f_j,
\eea
where we wrote
\[
\hat{f}_j:=\hat K_{1,j}(T_{1,0}^{-1}f_1),\quad j=2,3.
\]

Consider the term $\hat K_{1,j}(T^{-1}_{1,0}(\ol K\hat\psi))$.  
We find that for a fixed $E\in I$
\[
\hat K_{1,j}(T^{-1}_{1,0}(\ol K\hat\psi))=\hat K_{1,j}(E)T_{1,0}(E)^{-1}(\ol K(E)\hat\psi(E))
\]
where for every $E\in I$ the linear operator $T_{1,0}(E):L^2(G\times S)\to L^2(G\times S)$
with domain $D(T_{1,0}(E))$ is defined to be
\[
&
D(T_{1,0}(E)):=\tilde{W}^2_{-,0}(G\times S)\nonumber\\
&
T_{1,0}(E)v:=\omega\cdot\nabla_x v+\Sigma_1(E)v-\ol K_1(E)v.
\]
By (the proof of) Lemma \eqref{coupevle2}, the operator $T_{1,0}(E)$ is invertible and
\begin{align}\label{eq:T10Einv_bound}
\n{T_{1,0}(E)^{-1}\tilde{q}}_{L^2(G\times S)}\leq {1\over c}\n{\tilde{q}}_{L^2(G\times S)},\quad \forall\tilde{q}\in L^2(G\times S).
\end{align}

Define linear operators $Q_j(E):L^2(G\times S)^2\to L^2(G\times S)$, $j=2,3$, $E\in I$, by setting
\[
Q_j(E)u:=\hat{K}_{1,j}(E)T_{1,0}(E)^{-1}(\ol{K}(E)u),
\]
and let $Q(E):=(Q_1(E),Q_2(E))$. Then we have by \eqref{eq:T10Einv_bound} and \eqref{m'}, \eqref{m''}
(see e.g. \eqref{co(b)})
that for all $E\in I$,
\be
\n{Q_j(E)}\leq \n{\hat{K}_{1,j}(E)}\n{T_{1,0}(E)^{-1}}\n{\ol K(E)}
\leq {{M_1M'_1}\over c},
\ee
where the operator norms used are taken, in an obvious way, with respect to the space $L^2(G\times S)$ and
its product $L^2(G\times S)^2$.

Using assumption \eqref{ass3-aa} and Lemma \ref{coupevle2},
we find that for any fixed $u\in L^2(G\times S)^2$ the mapping $h_u:I\to L^2(G\times S)^2$ given by 
$h_u(E):=Q(E)u$ belongs to $C^1(I,L^2(G\times S)^2)$.
Similarly, we find that $\hat{f}:=(\hat{f}_1,\hat{f}_2)\in C^1(I,L^2(G\times S)^2)$,
since we have
\[
\hat{f}_j(E)=\hat{K}_{1,j}(E)T_{1,0}(E)^{-1}(f_1(E)),\quad E\in I,\ j=2,3,
\]
where $\hat{f}_j(E)(x,\omega)=\hat{f}_j(x,\omega,E)$.
 
Let $C:=\max\{C_2,C_3\}$ where for $j=2,3$,
\begin{multline*}
C_j:={1\over 2}\kappa_j^{-2}\n{\nabla_x S_j}_{L^\infty(G\times I)}
\\
+\kappa_j^{-1}\Big(\n{\Sigma_j}_{L^\infty(G\times S\times I)}
+\n{{\p {S_j}E}}_{L^\infty(G\times I)}+\sqrt{M_1M_1'}
+{{M_1M_1'}\over c}\Big).
\end{multline*}
Replacing $\hat\psi$ with $\hat\phi(x,\omega,E):=e^{-CE}\hat\psi(x,\omega,E_{\rm m}-E)$ (as in section \ref{evcsd}) we find that the system (\ref{pr5.5.7}) is equivalent to
\be\label{proof7}
{\p {\hat\phi}E}-{\bf A}_C(E)\hat\phi=F(E),\quad \hat\phi(0)=0,
\ee
where
\[
\hat{\phi}&=(\phi_2,\phi_3), \\[2mm]
{\bf A}_C(E)\hat\phi&=({A}_{C,2}(E)\hat\phi,{A}_{C,3}(E)\hat\phi), \\[2mm]
F&=(F_2,F_3),
\]
and for $j=2,3$,
\[
F_j(E):={1\over{\tilde S_j(E)}}e^{-CE}\big(\tilde f_j(x,\omega,E)+\tilde{\hat f}_j(x,\omega,E)\big),
\]
and
\begin{multline}\label{proof9}
A_{C,j}(E)\hat\phi
:=-\Big({1\over {\tilde S_j(E)}}\omega\cdot\nabla_x\phi_j+C \phi_j+{1\over {\tilde S_j(E)}}\tilde\Sigma_j(E)\phi_j
+{1\over{\tilde S_j(E)}}{\p {\tilde S_j}E}(E)\phi_j
\\
-{1\over{\tilde S_j(E)}} \tilde{\hat{K}}_j(E)\hat\phi
-{1\over{\tilde S_j(E)}} \tilde Q_{j}(E)\hat\phi\Big).
\end{multline}
Here $\tilde S_j(x,E):=S_j(x,E_m-E)$ and similarly for other  expressions equipped with "tilde".

Considering ${\bf A}_C(E)$ as an (unbounded) operator $L^2(G\times S)^2\to L^2(G\times S)^2$ with domain
\[
D({\bf A}_C(E))=\tilde W^2_{-,0}(G\times S)\times \tilde W^2_{-,0}(G\times S),
\] 
we get by applying Theorem \ref{evoth} along with computations analogous
to the ones done in section \ref{evcsd}, that the system (\ref{proof7}) has a unique solution 
$\hat{\phi}\in C(I,\tilde W^2_{-,0}(G\times S)^2)\cap  C^1(I,L^2(G\times S)^2)$,
which satisfies the homogeneous boundary and initial conditions, $\hat{\phi}_{|\Gamma_-}=0$,
$\hat{\phi}(\cdot,\cdot,0)=0$.

Then $\hat{\psi}(x,\omega,E)=e^{C(E_{\rm m}-E)}\hat\phi(x,\omega,E_{\rm m}-E)$,
and $\psi_1$ is obtained from (\ref{psi1}) that is, $\psi_1=T_{1,0}^{-1}(f_1+\ol K\hat\psi)$,
giving us the solution $\psi=(\psi_1,\psi_2,\psi_3)=(\psi_1,\hat{\psi})$ for homogeneous (inflow) boundary,
and initial condition data, that is $\psi|_{\Gamma_-}=0$, $\psi_j(\cdot,\cdot,E_m)=0$ for $j=2,3$,
which is what we were looking for.

By applying the lifts, the existence of a unique solution of the problem 
(\ref{csda1a})-(\ref{csda3}) satisfying the inhomogeneous boundary condition
\be\label{proof11}
{\psi_j}_{|\Gamma_-}=g_j,\quad j=1,2,3,
\ee
is obtained as in section \ref{evcsd} as well (see the proof of Theorem \ref{evoth1}).
The estimate \eqref{csda40aaa-co} under the assumptions \eqref{ec8}, \eqref{ec9}
follows from Corollary \ref{cosystco1}.
This completes the proof.
\end{proof}

\begin{corollary}\label{coupcoev}
Suppose that the assumptions (\ref{ass1-aa})-(\ref{ec9-aa})
of Theorem \ref{coupthev} are valid.
Furthermore, suppose that $f\in  H^2(I,L^2(G\times S)^3)$ and $g\in  H^3(I,T^2(\Gamma_-')^3)$ which satisfies  the compatibility condition
\[
g_j(E_m)=0,\quad j=2,3.
\]
Then the problem (\ref{csda1a})-(\ref{csda3}) has a unique solution $\psi\in \tilde{W}^2(G\times S\times I)\times \big(C^1(I,L^2(G\times S)^2)\cap C(I,\tilde W(G\times S)^2)\big)$.
If, in addition, (\ref{ec8}) and (\ref{ec9}) are valid, the estimate \eqref{csda40aaa-co} holds.
\end{corollary}

\begin{proof}
Follows from the Sobolev Embedding Theorem as in Corollary \ref{evothco}.
\end{proof}

\begin{remark}
The evolution operator based approach given above can be generalized for $L^p$-theory when $1\leq p<\infty$.  
\end{remark}

\subsection{Note on Reflective Boundary Conditions}

Let $R:=(R_1,R_2,R_3):T^2(\Gamma_+)^3\to T^2(\Gamma_-)^3$ be a linear, possibly unbounded operator
with domain $D(R)$. In some problems the inflow boundary condition
\be\label{inflowbc}
{\psi}_{|\Gamma_-}=g,
\ee
of the problem (\ref{csda1a})-(\ref{csda3}) is substituted with a more general,
so-called \emph{reflective boundary condition}
(cf. \cite[Chapter XXI, Appendix of \S 2, p. 249-262]{dautraylionsv6})
\be\label{reflexbc}
{\psi}_{|\Gamma_-}=R(\psi_{|\Gamma_+})+g,
\ee
where $g\in T^2(\Gamma_-)^3$.
We call $R$ a \emph{reflection operator}. Note that when $R=0$, we obtain (\ref{inflowbc}).

We consider here one specific type of reflection operator $R=R_{\rm b}$
which is of interest in applications where backscattering from outside the region $G$
is of importance.
Let $G'\subset\R^3$ be an open bounded set such that $\ol{G'}$ is a $C^1$-manifold with boundary
(i.e. $G'$ has the same regularity properties as $G$)
and such that $\ol{G}\subset G'$. Let $G_{\rm e}:=G'\setminus \ol G$. Then
\[
\partial G_{\rm e}=\partial^1 G_{\rm e}\cup \partial^2G_{\rm e},
\]
where $\partial^1 G_{\rm e}:=\partial G$ and $\partial^2 G_{\rm e}:=\partial G_{\rm e}\backslash \partial^1 G_{\rm e}$.
Let
\[
\Gamma_{\rm e}:= (\partial G_{\rm e})\times S\times I
=
\Gamma_{\rm e,+}\cup \Gamma_{\rm e,-}\cup \Gamma_{\rm e,0},
\]
where on the right hand side, the decomposition of $\Gamma_{\rm e}$ into the three disjoint subsets
corresponds to $\Gamma=\Gamma_+\cup \Gamma_-\cup\Gamma_0$ when considering $G_{\rm e}\times S\times I$
instead of $G\times S\times I$.
Finally, we can decompose $\Gamma_{\rm e,+}$ and $\Gamma_{\rm e,-}$ respectively as
\[
\Gamma_{\rm e,\pm} = \Gamma^1_{\rm e,\pm}\cup \Gamma^2_{\rm e,\pm},
\]
where for $j=1,2$,
\[
\Gamma^j_{\rm e,\pm}:=\big((\partial^j G_{\rm e})\times S\times I\big)\cap \Gamma_{\rm e,\pm}.
\]
Notice that $\Gamma^1_{\rm e,-}=\Gamma_+$, while $\Gamma^1_{\rm e,+}=\Gamma_-$.

We assume (for simplicity) that $f=0$. Consider the problem (\ref{csda1a})-(\ref{csda3}) that is,
find $\psi=(\psi_1,\psi_2,\psi_3)$ that satisfies on $G\times S\times I$ the system of transport equations,
\begin{gather}
\omega\cdot\nabla_x\psi_1+\Sigma_1\psi_1-K_{1}\psi=0,\label{ref2}\\
-{\p {(S_{j}\psi_j)}E}+\omega\cdot\nabla_x\psi_j+\Sigma_{j}\psi_j-K_{j}\psi=0,\quad j=2,3\label{ref3}
\end{gather}
under the (inflow) boundary condition on $\Gamma_-$,
\be\label{ref4}
{\psi_j}_{|\Gamma_-}=g_j,\quad j=1,2,3,
\ee
and the initial condition on $G\times S$,
\be\label{ref5}
\psi_j(\cdot,\cdot,E_{\rm m})=0,\quad j=2,3.
\ee

Since $\Gamma^1_{\rm e,-}=\Gamma_+$,
we find that the flux $\psi_{|\Gamma_+}$ is an inflow boundary source for the domain $G_{\rm e}$,
on the part $\partial^1 G_{\rm e}=\partial G$ of its boundary.
Suppose that $G_{\rm e}$ does not contain any extra internal or boundary sources. 
Then the transport of particles in $G_{\rm e}$ is governed by
the system of equations on $G_{\rm e}\times S\times I$,
\begin{gather}
\omega\cdot\nabla_x\Psi_1+\Sigma_{{\rm e},1}\Psi_1-K_{{\rm e},1}\Psi=0,\label{ref2a}\\
-{\p {(S_{{\rm e},j}\Psi_j)}E}+\omega\cdot\nabla_x\Psi_j+\Sigma_{{\rm e},j}\Psi_j-K_{{\rm e},j}\Psi=0,\quad j=2,3,\label{ref3a}
\end{gather}
for $\Psi=(\Psi_1,\Psi_2,\Psi_3)$,
along with the boundary conditions
\begin{alignat}{5}\label{ref4a}
& {\Psi_j}_{|\Gamma^1_{\rm e,-}}={\psi_j}_{|\Gamma_+}\quad && {\rm on}\ \Gamma^1_{\rm e,-}=\Gamma_+,\quad && \\
& {\Psi_j}_{|\Gamma^2_{\rm e,-}}=0\quad && {\rm on}\ \Gamma^2_{\rm e,-},\quad && j=1,2,3, 
\end{alignat}
and the initial condition 
\be\label{ref5a}
\Psi_j(x,\omega,E_{\rm m})=0\quad {\rm on}\ G_{\rm e}\times S\times I,\quad j=2,3.
\ee
Above $\Sigma_{{\rm e},j}$, $\sigma_{{\rm e},kj}$ and $S_{{\rm e},j}$ are
the (restricted) cross-sections and the (restricted) stopping powers for the medium inside $G_{\rm e}$,
and
\[
(K_{{\rm e},j}\Psi)(x,\omega,E)
:=
\sum_{k=1}^3\int_{S\times I}\sigma_{{\rm e},kj}(x,\omega',\omega,E',E)\Psi_k(x,\omega',E')d\omega' dE',
\]
for $\Psi\in L^2(G_{\rm e}\times S\times I)^3$.

In this setup, we define a reflection operator $R=R_{\rm b}$ by setting
(recall that $\Gamma^1_{{\rm e},+}=\Gamma_-$)
\begin{align}
& R_{\rm b}(\psi_{|\Gamma_+}):=\Psi_{|\Gamma^1_{{\rm e},+}}=\Psi_{|\Gamma_-}, \nonumber\\
& D(R_{\rm b}):=\{\psi_{|\Gamma_+}\ |\ \psi\ {\rm is\ a\ solution\ of\ (\ref{ref2})-(\ref{ref5})\ for \ some}\ g\in T^2(\Gamma_-)^3\}. \label{ref6}
\end{align}
If the assumptions of Theorem \ref{cosystth2} hold for $G'$ in place of $G$
(and for the respective cross-sections, stopping powers, $\Sigma'_j,\sigma'_{kj},S'_j$),
the operator $R_{\rm b}$ is a linear operator $T^2(\Gamma_+)^3\to T^2(\Gamma_-)^3$ with domain of definition $D(R_{\rm b})\subset T^2(\Gamma_+)^3$.
In general, $R_{\rm b}$ is not bounded.

The meaning of this definition is that $R_{\rm b}(\psi_{|\Gamma_+})$ models an extra source on $\Gamma_-$
(i.e. an inflow boundary source for $G$)
due to backscattering of particles from the given external region $G_{\rm e}$.
The flux $u=(u_1,u_2,u_3)$ contributed by this inflow source is governed by
the system of equations
\begin{gather}
\omega\cdot\nabla_x u_1+\Sigma_1 u_1-K_{1}u=0,\label{ref9}\\
-{\p {(S_{j}u_j)}E}+\omega\cdot\nabla_x u_j+\Sigma_{j} u_j-K_{j}u=0,\quad j=2,3,\label{ref10}
\end{gather}
on $G\times S\times I$, such that on $\Gamma_-$,
\be\label{ref11}
u_{|\Gamma_-}=R_{\rm b}(\psi_{|\Gamma_+}),
\ee
and for almost every $(x,\omega)\in G\times S$,
\be\label{ref12}
u_j(x,\omega,E_{\rm m})=0,\quad j=2,3.
\ee

We point out that if $\psi\in D(R_{\rm b})$ and if $u$ solves the problem \eqref{ref9}-\eqref{ref12},
and is such that
\be \label{nrep}
R_{\rm b}(u_{|\Gamma_+})=0,
\ee
holds, i.e. no particles are backscattered repeatedly from $G_{\rm e}$ into $G$,
then for $\varphi:=\psi+u$ we have
\[
\varphi_{|\Gamma_-}
=\psi_{|\Gamma_-}+u_{|\Gamma_-}
=g+R_{\rm b}(\psi_{|\Gamma_+})
=g+R_{\rm b}((\psi+u)_{|\Gamma_+})
=g+R_{\rm b}(\varphi_{|\Gamma_+}),
\]
and hence $\varphi=(\varphi_1,\varphi_2,\varphi_3)=\psi+u$ is a solution
of the following problem on $G\times S\times I$,
with boundary and initial conditions holding on $\Gamma_-$ and $G\times S$, respectively,
\begin{gather}
\omega\cdot\nabla_x \varphi_1+\Sigma_1 \varphi_1-K_{1}\varphi=0, \label{ref13}\\[2mm]
-{\p {(S_{j}\varphi_j)}E}+\omega\cdot\nabla_x \varphi_j+\Sigma_{j}\varphi_j-K_{j}\varphi=0,\label{ref14} \\[2mm]
\varphi_{|\Gamma_-}=R_{\rm b}(\varphi_{|\Gamma_+})+g, \label{ref15} \\[2mm]
\varphi_j(\cdot,\cdot,E_{\rm m})=0, \label{ref16}
\end{gather}
where $j=2,3$.
This shows that solving for $\varphi$ (in place of $\psi$) the problem \eqref{ref2}, \eqref{ref3}, \eqref{ref5} under boundary condition \eqref{reflexbc},
is equivalent to solving first for $\psi$ the problem \eqref{ref2}-\eqref{ref5},
and then for $u$ the problem \eqref{ref9}-\eqref{ref12} under the additional condition \eqref{nrep}.

We will not explore the question of the existence of solutions for the problem \eqref{ref13}-\eqref{ref16}
(or for \eqref{ref9}-\eqref{nrep}) in this paper.

\sectionspace
\section{Coupled Adjoint Transport Problem}\label{adjoint}

We will discuss briefly the adjoint version of the transport problem (\ref{csda1a})-(\ref{csda3}),
and related operators. For simplicity we assume that $K=(K_1,K_2,K_3)$ where $K_j, \ j=1,2,3$ are of the form (for this operator the adjoint $K^*$ can easily be computed)
\[
(K_j\psi)(x,\omega,E)
=\sum_{k=1}^3\int_{S\times I} \sigma_{kj}(x,\omega',\omega,E',E)\psi_k(x,\omega',E')d\omega' dE'.
\]
Write
\bea 
&T_1\psi:=\omega\cdot\nabla_x\psi_1+\Sigma_1\psi_1-K_1\psi,\nonumber\\
&T_j\psi:=-{\p {(S_j\psi_j)}E}+\omega\cdot\nabla_x\psi_j+
\Sigma_j\psi_j
- K_j\psi,\quad j=2,3,\nonumber
\eea
and define a (densely defined) linear operator $T:L^2(G\times S\times I)^3\to L^2(G\times S\times I)^3$ by
\bea 
D(T):={}&\{\psi\in L^2(G\times S\times I)^3\ |\ T_j\psi\in L^2(G\times S\times I),\ j=1,2,3\},
\nonumber\\[2mm]
T\psi:={}&(T_1\psi,T_2\psi,T_3\psi).
\eea
Let $f\in L^2(G\times S\times I)^3$ and $g\in T^2(\Gamma_-)^3$. The problem (\ref{csda1a})-(\ref{csda3}) can be expressed equivalently as the problem
\be\label{adjoint1}
T\psi=f,\quad \psi_{|\Gamma_-}=g,\quad \psi_j(\cdot,\cdot,E_m)=0,\quad j=2,3.
\ee

As in section \ref{esols}, an application of integration by parts and the Green's formula (\ref{green}) implies
\be\label{adjoint2}
\la T\psi,v\ra_{L^2(G\times S\times I)^3}=\la \psi,T^*v\ra_{L^2(G\times S\times I)^3}\quad \forall v\in C_0^1(G\times S\times I^{\circ}),
\ee
where $T^*v=(T_1^*v,T_2^*v,T_3^*v)$, and
\bea \label{adjoint3a}
T_1^*v:={}&-\omega\cdot\nabla_x v_1+\Sigma_1^*v_1-K_1^*v,\nonumber\\
T_j^*v:={}&S_j{\p {v_j}E}-\omega\cdot\nabla_x v_j
+\Sigma_j^*v_j-K_j^*v,\quad j=2,3.
\eea
Moreover, we have $\Sigma_j^*=\Sigma_j$ and for $v\in L^2(G\times S\times I)^3$, $j=1,2,3$,
\[
(K_j^*v)(x,\omega,E)
=\sum_{k=1}^3\int_{S\times I} \sigma_{jk}(x,\omega,\omega',E,E')\psi_k(x,\omega',E')d\omega' dE'.
\]

Let $f^*\in L^2(G\times S\times I)^3$ and $g^*\in T^2(\Gamma_+)^3$. The \emph{adjoint problem}  of
(\ref{adjoint1})  (or equivalently (\ref{desol10})-(\ref{desol12})) is defined by (cf. \cite[pp. 24-28]{agoshkov})
\be\label{adjoint4}
T^*\psi^*=f^*,\quad {\psi^*}_{|\Gamma_+}=g^*,\quad \psi_j^*(\cdot,\cdot,0)=0, \quad j=2,3,
\ee
or more explicitly
\bea
-\omega\cdot\nabla_x\psi_1^*+\Sigma_1^*\psi_1^*-K_1^*\psi^*={}& f_1^*, \label{adesol10:1} \\[2mm]
S_j{\p {\psi_j^*}E}-\omega\cdot\nabla_x\psi_j^*
+\Sigma_j^*\psi_j^* - K_j^*\psi^*={}& f_j^*,\quad j=2,3,\label{adesol10:2}
\eea
holding a.e. on $G\times S\times I$,
together with the \emph{outflow} boundary and initial values
\begin{alignat}{3}
\psi^*_{|\Gamma_+}={}&g^* && \quad {\rm a.e.\ on}\ \Gamma_+, \label{adesol11} \\[2mm]
\psi_j^*(\cdot,\cdot,0)={}&0\ && \quad {\rm a.e.\ on}\ G\times S,\quad j=2,3. \label{adesol12}
\end{alignat}
The adjoint problem has various kind of applications both in the existence theory of solutions and in computations.
At the end of this section we give  an example which is related to the dose calculation in radiation therapy.
We refer also to computations concerning a related optimal control problem considered in \cite[Section 7]{tervo17-up}, where the adjoint problem plays a significant role.
In addition, recall that the concept of Green distribution (or Green function)  is usually founded on the theory of adjoint problem.
We also point out that the adjoint field $\psi^*$ obeying Eqs. \eqref{adesol10:1}, \eqref{adesol10:2}
is sometimes called an \emph{importance function} (cf. \cite[Ch. 5, Sec. V]{duderstadt76}).

Consider the variational formulation of the adjoint problem.
As in sections \ref{single-eq} and \ref{cosyst} (for $C=0$) we find that the bilinear form
${\bf B}_0^*(\cdot,\cdot):C^1(\ol G\times S\times I)^3\times C^1(\ol G\times S\times I)^3$ and the linear form ${\bf F}^*$ corresponding to the adjoint problem
are 
\bea\label{adjoint5}
{\bf B}_0^*(\psi^*,v)
={}&
-\sum_{j=2,3}\la\psi_j^*,{\p {(S_jv_j)}E}\ra_{L^2(G\times S\times I)}
+\la \psi^*, \omega\cdot\nabla_x v\ra_{L^2(G\times S\times I)^3} \nonumber\\
&
+\la\psi^*,(\Sigma-K) v\ra_{L^2(G\times S\times I)^3}
+\la \gamma_-(\psi^*), \gamma_-(v)\ra_{T^2(\Gamma_-)^3} \nonumber\\
&
+\sum_{j=2,3} \la \psi_j^*(\cdot,\cdot,E_m),S_j(\cdot,E_m) v_j(\cdot,\cdot,E_m)\ra_{L^2(G\times S)} ,
\eea
where $\psi^*=(\psi_1^*,\psi_2^*,\psi_3^*)$ and $v=(v_1,v_2,v_2)$, and 
\be \label{adjoint6}
{\bf F}_0^*(v)=\la f^*,v\ra_{L^2(G\times S\times I)^3}+\la g^*,\gamma_+(v)\ra_{T^2(\Gamma_+)^3}.
\ee

Similarly as in Theorem \ref{cosystth1} we find that
\be\label{ajoint6}
|{\bf B}_0^*(\psi^*,v)|\leq M\n{\psi^*}_{{\s H}}\n{v}_{\tilde {\s H}}\quad \forall \psi^*, v\in C^1(\ol G\times S\times I)^3
\ee
and
\be 
|{\bf F}_0^*(v)|\leq \big(\n{f^*}_{L^2(G\times S\times I)^3}+\n{g^*}_{T^2(\Gamma_+)^3}\big)\n{v}_{{\s H}}
\quad \forall v\in C^1(\ol G\times S\times I)^3.
\ee
The bilinear form ${\bf B}_0^*(\cdot,\cdot)$ has a unique extension $\tilde {\bf B}_0^*(\cdot,\cdot):{\s H}\times \tilde {\s H}\to \R$
(and $\tilde{\bf B}_0^*(\cdot,\cdot)$ has an explicit expression in the same way as in (\ref{cosyst5a}))
which satisfies
\be\label{adjoint8}
|\tilde {\bf B}_0^*(\tilde\psi^*,v)|\leq M\n{\tilde\psi^*}_{{\s H}}\n{v}_{\tilde {\s H}}\quad \forall \tilde\psi^*=(\psi^*,q^*,p_0^*,p_{\rm m}^*)\in {\s H},\ v\in \tilde {\s H}.
\ee
Moreover, the linear form ${\bf F}_0^*$ has a unique extension to a bounded linear form ${\s H}\to\R$ which we still denote by ${\bf F}_0^*$.
The variational equation corresponding to the adjoint problem 
\eqref{adesol10:1}, \eqref{adesol10:2}, \eqref{adesol11}, \eqref{adesol12} is then
\be\label{adjoint10}
\tilde {\bf B}^*_0(\tilde\psi^*,v)={\bf F}_0^*(v),\quad \forall v\in \tilde {\s H}.
\ee

Suppose that the assumptions (\ref{scateh}), (\ref{colleh}), (\ref{sda2a}), \eqref{sda2a-2} and (\ref{sda2b}) are valid.
Using integration by parts and Green's formula \eqref{green}
we have the following.

\begin{proposition}
For all $\psi,\psi^*\in \tilde {\s H}$
one has
\be\label{adjoint7}
\tilde {\bf B}_0^*(\psi^*,\psi)=\tilde {\bf B}_0(\psi,\psi^*),
\ee
where $\tilde{\bf B}_0(\cdot,\cdot)$ is the bilinear form (\ref{cosyst5a}).
\end{proposition}

The relation \eqref{adjoint7} is the justification for the term \emph{adjoint problem}.
For the existence of solutions of the adjoint problem, we formulate the following result.

\begin{theorem}\label{adjointco1}
Suppose that the assumptions
(\ref{scateh}), (\ref{colleh}), (\ref{sda2a}), \eqref{sda2a-2}, (\ref{sda2b}) are valid, and that
(\ref{csda3aa}), (\ref{csda4aa}) hold for $c>0$ and for $C$ given in (\ref{cosyst6a}).
Let $f^*\in L^2(G\times S\times I)^3$ and that $g^*\in T^2(\Gamma_+)^3$.
Then the following assertions hold.

(i) The variational equation
\be\label{vareqcoad}
\tilde {\bf B}_0^*(\tilde\psi^*,v)={\bf F}_0^*(v)\quad \forall v\in \tilde{\s H},
\ee
has a solution $\tilde{\psi}^*=(\tilde{\psi}^*_1, \tilde{\psi}^*_2, \tilde{\psi}^*_3)\in {\s H}$. 
Writing $\tilde{\psi}_1^*:=(\psi_1^*,q_1^*)$, $\tilde{\psi}_j^*:=(\psi_j^*,q_j^*,p_{0j}^*,p_{{\rm m}j}^*)$, $j=2,3$,
and $\psi^*=(\psi_1^*, \psi_2^*, \psi_3^*)\in L^2(G\times S\times I)^3$,
then $\psi^*\in{\s H}_{{\bf P}^*}(G\times S\times I^\circ)$ (see \eqref{eq:H_bfP_star}) is a weak (distributional) solution of the system of equations \eqref{adesol10:1}, \eqref{adesol10:2},
and $\psi^*_1\in W^2(G\times S\times I)$.

(ii) Suppose that additionally the assumption ${\bf TC}$ holds (p. \pageref{as:TC}). Then a
solution $\psi^*$ of the equations \eqref{adesol10:1}, \eqref{adesol10:2}
obtained in part (i) is a solution of the problem \eqref{adesol10:1}-\eqref{adesol12}.

(iii) Under the assumptions imposed in part (ii), any solution $\psi^*$ of the problem \eqref{adesol10:1}-\eqref{adesol12}
is unique and obeys the estimate
\be\label{adjoint11b}
\n{\psi^*}_{{\s H}}\leq 
{{e^{CE_{\rm m}}}\over{c'}}\big(
\n{{f^*}}_{L^2(G\times S\times I)^3}+\n{{ g^*}}_{T^2(\Gamma_+)^3}\big).
\ee
(Recall that $C$ is defined in \eqref{cosyst6a}, $c'$ in \eqref{cocprime} and that $E_m$ is the cutoff energy.)
\end{theorem}

Note that if $\psi^*$ is a solution of the problem \eqref{adesol10:1}-\eqref{adesol12},
then it is a solution of the variational problem (\ref{vareqcoad}) and vice versa.

The method founded on the $m$-dissipativity (see sections \ref{m-d} and \ref{mdiss-op})
can be applied also to the adjoint problem. Define
\[
P_{1}^*(x,\omega,E,D)\psi_1^*:={}&-\omega\cdot\nabla_x\psi_1^*, \\[2mm]
P_{j}^*(x,\omega,E,D)\psi_j^*:={}&S_j{\p {\psi_j^*}E}-\omega\cdot\nabla_x\psi_j^*,\quad j=2,3, \\[2mm]
{\bf P}^*(x,\omega,E,D)\psi^*:={}&\big(P_{1}^*(x,\omega,E,D)\psi_1^*, P_{2}^*(x,\omega,E,D)\psi_2^*,P_{3}^*(x,\omega,E,D)\psi_3^*\big),
\]
and the space
\begin{multline}
{\s H}_{{\bf P}^*}(G\times S\times I^\circ)
:=\{\psi^*\in L^2(G\times S\times I)^3\ | \\
{\bf P}^*(x,\omega,E,D)\psi^*\in 
L^2(G\times S\times I)^3\ \textrm{in the weak sense}\}. \label{eq:H_bfP_star}
\end{multline}

The relevant operator here is the
smallest closed extension (closure) $\tilde{\bf P}^*_0$
of the operator ${\bf P}^*_0$ defined by
\[
D({\bf P}_0^*):={}&\big\{
\psi^*\in \tilde{W}^2(G\times S\times I)\times \big(\tilde{W}^2(G\times S\times I)\cap H^1(I,L^2(G\times S)\big)^2
\ \big| \\[2mm]
&\hspace{2.5mm} \psi^*_{|\Gamma_+}=0,\ \psi^*_j(\cdot,\cdot,0)=0,\ j=2,3\big\} \\[2mm]
{\bf P}_0^*\phi:={}&{\bf P}^*(x,\omega,E,D)\phi.
\]
When $g^*=0$, the problem \eqref{adesol10:1}-\eqref{adesol12} is equivalent, in the strong sense, to
\[
(\tilde{\bf P}_0^*+\Sigma^*-K^*)\psi^*=f^*,
\]
where $\psi^*\in D(\tilde {\bf P}_0^*)$,
and $\Sigma^*\psi^*=(\Sigma_1^*\psi_1^*, \Sigma_2^*\psi_2^*, \Sigma_3^*\psi_3^*)$,
$K^*\psi^*=(K_1^*\psi^*, K_2^*\psi^*, K_3^*\psi^*)$.
 
The result analogous to Theorem \ref{m-d-j-co1} is the following.

\begin{theorem}\label{m-d-ad}
Suppose that the assumptions (\ref{scateh}), (\ref{colleh}), (\ref{csda3aa}), (\ref{csda4aa}) (with $c>0$)
and \eqref{Sj-ass:1}, \eqref{Sj-ass:2}, \eqref{Sj-ass:3} are valid.
Furthermore, suppose that ${f^*}\in L^2(G\times S\times I)^3$
and  $g^*\in T^2(\Gamma_+)\times H^1(I,T^2(\Gamma_+'))^2$ is such that the compatibility condition
\be\label{comp-d-j-ad}
g_j^*(\cdot,\cdot,E_0)=0,\ j=2,3,
\ee
holds.
Then the problem \eqref{adesol10:1}-\eqref{adesol12}
has a unique solution 
$\psi^*\in {\s H}_{\bf P^*}(G\times S\times I^\circ)$.
In addition, there exists a constant $C_1>0$ such that \emph{a priori} estimate
\be\label{diss-co-es-ad}
\n{\psi^*}_{L^2(G\times S\times I)^3}
\leq
C_1\big(\n{{ f^*}}_{L^2(G\times S\times I)^3}+
\n{{ g^*}}_{T^2(\Gamma_+)\times H^1(I,T^2(\Gamma_+'))^2}\big),
\ee
holds.
\end{theorem}

The existence result analogous to Theorem \ref{coupthev}
(based on the theory of evolution operators) holds also for the adjoint problem, and it guarantees that $\psi^*
\in\tilde W^2(G\times S\times I)\times (\tilde W^2(G\times S\times I)\cap W_1^2(G\times S\times I))^2$.
In this case, one assumes that $K$ takes the form \eqref{ec1moda},
and that consequently its adjoint version $K^*=(K_1^*, K_2^*, K_3^*)$ is,
\begin{align}\label{ec1moda:ad}
(K_j^*\psi)(x,\omega,E)=\sum_{k=1}^3\int_S\tilde\sigma_{kj}(x,\omega,\omega',E)\psi_k(x,\omega',E)d\omega',\quad j=1,2,3.
\end{align}

The adjoint version of Theorem \ref{coupthev} can be formulated as follows.

\begin{theorem}\label{coupthevad}
Suppose that the adjoint collision operator is of the form \eqref{ec1moda:ad},
and that the assumptions (\ref{ass1-aa})-(\ref{ec9-aa}) of Theorem \ref{coupthev} 
are valid for $\Sigma_j,\ \sigma_{jk}$ and $S_j$. Furthermore,
suppose that $f^*\in  C^1(I,L^2(G\times S)^3)$ and $g^*\in C^2(I,T^2(\Gamma_+')^3)$ which satisfies  the compatibility condition
\[
g_j^*(0)=0,\quad j=2,3.
\]
Then the problem \eqref{adesol10:1}-\eqref{adesol12} has a unique solution $\psi^*\in \tilde W^2(G\times S\times I)\times \big( C(I,\tilde W^2(G\times S)^2)\cap C^1(I,L^2(G\times S)^2)\big)$.
In particular, $\psi^*\in\tilde{W}^2(G\times S\times I)\times \big(\tilde{W}^2(G\times S\times I)\cap W_1^2(G\times S\times I)\big)^2$.

If, in addition, the conditions \eqref{ec8}, \eqref{ec9} are valid,
then the solution $\psi^*$ satisfies the estimate (\ref{adjoint11b}).
\end{theorem}

It is clear that a Sobolev space version of the above theorem analogous to Corollary \ref{coupcoev},
holds for the adjoint problem as well.

\begin{example}\label{exgreenf1}

In radiation therapy the absorbed \emph{dose} from the particle field $\psi=(\psi_1,\psi_2,\psi_3)$ is defined by the functional
\be 
D(x)=(D\psi)(x):=\sum_{j=1}^3\int_{S\times I}\varsigma_j(x,E)\psi_j(x,\omega,E) d\omega dE,
\ee
where $\psi$ is the solution of (\ref{desol10})-(\ref{desol12}).
We see that 
\be 
(D\psi)(x)=\sum_{j=1}^3\la \varsigma_j(x,\cdot),\psi_j(x,\cdot,\cdot)\ra_{L^2(S\times I)}
=\la \varsigma(x,\cdot),\psi(x,\cdot,\cdot)\ra_{L^2(S\times I)^3}.
\ee
Define for any fixed $x\in G$ distribution $T_{\varsigma_{j,x}}$ on $S\times I^\circ$ by
\be 
T_{\varsigma_{j,x}}\varphi:= \la \varsigma_j(x,\cdot),\varphi\ra_{L^2(S\times I)},\quad \varphi\in C_0^\infty(S\times I^\circ),\quad j=1,2,3.
\ee
Then for all $\psi_j=\phi_j\otimes\varphi_j$, where $\phi_j\in C_0^\infty(G)$, $\varphi_j\in C_0^\infty(S\times I^\circ)$,
and where $(\phi_j\otimes\varphi_j)(x,\omega,E):=\phi_j(x)\varphi_j(\omega,E)$,
we find that
\bea\label{gf1} 
\sum_{j=1}^3(\delta_x\otimes T_{\varsigma_{j,x}})(\psi_j)
={}&
\sum_{j=1}^3\delta_x(\phi_j)T_{\varsigma_{j,x}}(\varphi_j)
=\sum_{j=1}^3 \phi_j(x)\int_{S\times I}\varsigma_j(x,E)\varphi_j(\omega,E) d\omega dE
\nonumber\\
={}&\sum_{j=1}^3 \int_{S\times I}\varsigma_j(x,E)(\phi_j\otimes\varphi_j)(x,\omega,E) d\omega dE
=(D\psi)(x).
\eea
Hence also for a general element $\psi=(\psi_1,\psi_2,\psi_3)\in C_0^\infty(G\times S\times I^\circ)$
one has
\be\label{gf2}
(D\psi)(x)=\sum_{j=1}^3(\delta_x\otimes T_{\varsigma_{j,x}})(\psi_j)=:T_x(\psi),
\ee
where $T_x$ is a distribution on $G\times S\times I^\circ$ with values in $\R^3$.

The following discussion will be formal.
Let $x\in G$ be fixed. Assume that there exists a sufficient regular solution $\Psi^*_x$
to the variational problem
\be\label{gf3}
\tilde {\bf B}_0^*(\Psi^*_x,v)=T_x(v),\quad \forall v\in\tilde {\s H}.
\ee
Furthermore, assume that $\psi$ (which depends on $f,g$) is a sufficient regular solution of (\ref{csda1a})-(\ref{csda3})
and that
$\Psi^*_x=\Psi_x^*(x',\omega,E)$ is a solution of (\ref{gf3}) such that (\ref{adjoint7}) holds (for $\psi$ and $\Psi_x^*$). Then we have
\bea\label{gf4}
(D\psi)(x)={}&T_x(\psi)=\tilde{\bf B}_0^*(\Psi^*_x,\psi)
=\tilde{\bf B}_0(\psi,\Psi^*_x)
={\bf F}_0(\Psi^*_x)
\nonumber\\
={}&
\la f,\Psi^*_x\ra_{L^2(G\times S\times I)^3}
+
\la g, \gamma_-(\Psi_x^*) \ra_{T^2(\Gamma_-)^3},
\eea
implying that the dose can be obtained with the help of $\Psi^*_x$. That is why, when the equation (\ref{gf3}) is solved once,
one can obtain the dose $(D\psi)(x)$ at $x\in G$ for any $f$ and $g$ from (\ref{gf4}).
\end{example}

\sectionspace
\section{On the Non-negativity of Solutions}\label{possol}

\subsection{Non-negativity of Solutions for a Single Boltzmann Transport Equation}
\label{possol-s-b}

We start the handling of non-negativity of solutions for a single Boltzmann transport equation.

\begin{theorem}\label{p-1}
Suppose that
\[
\Sigma\in L^\infty(G\times S\times I),\quad \Sigma\geq 0\quad {\rm a.e.\ in}\ G\times S\times I.
\]
Furthermore, suppose that the collision operator is of the form $K=K^1+K^2+K^3$ given in section \ref{col-sum} where $\sigma^1,\ \sigma^2,\ \sigma^3$ obey the assumptions (\ref{ass5-a}), (\ref{ass7}), (\ref{ass-8}), (\ref{ass8-aa}) and (\ref{ass9-a}).
Then for every $f\in L^2(G\times S\times I)$ and $g\in T_{\tau_-}^2(\Gamma_-)$ the problem
\bea\label{co3aa}
\omega\cdot\nabla\psi+\Sigma\psi -K\psi&=f\ \nonumber\\
{\psi}_{|\Gamma_-}&=g,
\eea
has a unique solution $\psi\in \widetilde W^2(G\times S\times I)$.

Additionally, suppose that $f\geq 0,\ g\geq 0$.
Then the solution is non-negative that is, $\psi\geq 0$.
\end{theorem}

\begin{proof}
The  existence of the unique solution can be shown as in \cite{tervo17-up}, Theorem 6.14. Slight modifications are needed, the most notable of them being the fact that here we only assume   that the boundary data is in 
$T_{\tau_-}^2(\Gamma_-)$. It follows straightforwardly from the inflow trace Theorem \ref{tth}  that this generalization is allowed.

The non-negativity of solutions can be shown as in \cite{tervo17-up}, Theorem 6.16. The proof leans on the Trotter's formula (\cite{engelnagel}, Theorem III.5.2, p. 220 or \cite{goldstein}, p.53, cf. also \cite{dautraylionsv6}, pp. 226-227). We omit details.

\end{proof}

\subsection{Non-negativity of Solutions for a Single Boltzmann CSDA Transport Equation}
\label{possol-s-csda}

Consider the problem
\begin{gather}
-{\p {(S_0\psi)}E}+\omega\cdot\nabla\psi+\Sigma\psi -K\psi=f,\ \nonumber\\
{\psi}_{|\Gamma_-}=g,\quad \psi(\cdot,\cdot,E_{\rm m})=0. \label{p-2}
\end{gather}
The existence of the unique solution 
$\psi\in {\s H}_P(G\times S\times I^\circ)$ is guaranteed by Corollary \ref{m-d-co1}.

Using a change of unknown $\phi=e^{CE}\psi$,
with $C$ given in \eqref{eq:def_C},
we return the problem to the one considered in Theorem \ref{coth3-dd}, that is
\bea
&
-{\p {(S_0\phi)}E}+\omega\cdot\nabla_x\phi+CS_0\phi+\Sigma\phi-K_C\phi={}{f}_C \quad \textrm{on}\ G\times S\times I, \label{p-s1-b} \\
&
\phi_{|\Gamma_-}={}{ g}_C \quad \textrm{on}\ \Gamma_-, \quad
\phi(\cdot,\cdot,E_{\rm m})={}0 \quad \textrm{on}\ G\times S. \label{p-s3-b}
\eea
Clearly, the positivity condition on data, $f\geq 0$ and $g\geq 0$, is equivalent to having ${ f}_C\geq 0$ and ${ g}_C\geq 0$.
Recall that
\[
P_C(x,\omega,E,D)\phi:=-{\p {(S_0\phi)}E}+\omega\cdot\nabla_x\phi+CS_0\phi.
\]
We state our main theorem of this section.

\begin{theorem}\label{pos-th2}
Suppose that the assumptions (\ref{ass1}), (\ref{ass5-a}), (\ref{ass7}), (\ref{ass-8}),  (\ref{ass8-aa}), (\ref{ass9-a}) hold with $C$ given by (\ref{eq:def_C}).
Furthermore, suppose that (\ref{evo16}), (\ref{evo8-a}), (\ref{evo9-a}) are valid.
Let ${f}\in L^2(G\times S\times I)$, and let $g\in H^1(I,T^2(\Gamma_-'))$ be such that
\be\label{comp-d-d-p}
g(\cdot,\cdot,E_{\rm m})=0
\ee
and that $g\geq 0,\ f\geq 0$.
Then the solution of the problem (\ref{p-2}) is non-negative that is, $\psi\geq 0$.
\end{theorem}

\begin{proof}

A. For the first instance, we assume that $g=0$.
Let $\tilde P_{C,0}$ be the smallest closed extension defined above in section \ref{m-d}.
Using the these notations, the problem (\ref{p-s1-b}), (\ref{p-s3-b}) is equivalent to
\begin{align}\label{eq:problem_x}
(\tilde P_{C,0}+\Sigma-K_C)\phi={ f}_C 
\end{align}
where $\phi\in D(\tilde P_{C,0})$. 
Let 
\[
T_{C,0}:=\tilde P_{C,0}+\Sigma-K_C.
\]
Since $-\tilde P_{C,0}$ is $m$-dissipative (recall the proof of Theorem \ref{coth3-dd}) and $-(\Sigma-K_C)+cI$ is dissipative (recall (\ref{K-coer}))
one sees that the operator $-T_{C,0}+cI$ is $m$-dissipative. Recall also that $c>0$ is a part of the assumptions \eqref{ass8-aa-a}, (\ref{ass9-a-b}).
So $-T_{C,0}+cI$ generates a contraction $C^0$-semigroup $G_c(t)$. 
Hence the operator $-T_{C,0}=(-T_{C,0}+cI)-cI$ generates a (contraction) $C^0$-semigroup $G(t)=G_c(t)e^{-ct}$, $t\geq 0$,
such that $\n{G(t)}\leq e^{-ct}$ for all $t\geq 0$.
In addition, the solution $\phi$ of \eqref{eq:problem_x} is given by
(see \cite[Chapter II, Theorem 1.10]{engelnagel})
\be\label{p-s8}
\phi=T_{C,0}^{-1}{ f}_C
=(0-(-T_{C,0}))^{-1}{f}_C
=\int_0^\infty  G(t){f}_C dt.
\ee

We decompose the operator $-T_{C,0}$ (with the domain $\widetilde W^2_{-,0}(G\times S\times I)\cap H^1(I,L^2(G\times S))$) as follows 
\be\label{p-s9}
-T_{C,0}= B_0+A_0-(\Sigma+CS I)+K_C,
\ee
where linear operators $B_0$ and $A_0$ are defined by
\begin{gather*}
D(A_0):=\widetilde W^2_{-,0}(G\times S\times I),
\quad A_0\phi:=-\omega\cdot\nabla_x\phi\nonumber\\
D(B_0):=\{\phi\in H^1(I,L^2(G\times S))\ |\ \phi(\cdot,\cdot,E_{\rm m})=0\},
\quad B_0\phi:={\p {(S_0\phi)}E}.
\end{gather*}
The semi-groups generated by the last three components are given by (below $H$ is the Heaviside function) 
\bea\label{p-s10}
&
T_{K_C}(t)=e^{tK_C},\nonumber\\
&
T_{-(\Sigma+CS I)}(t)=e^{-t(\Sigma+CSI)},\nonumber\\
&
(T_{A_0}(t)f)(x,\omega,E)=H(t(x,\omega)-t)f(x-t\omega,\omega,E),
\eea
and they are clearly all of positive type (that is, $T(t)f\geq 0$ for $f\geq 0$). 

Also, the semi-group $T_{B_0}(t)$ generated by $B_0$ is of positive type. Indeed, letting
$(U(t))(x,\omega,E):=(T_{B_0}(t)f)(x,\omega,E)$ for a given $f\in D(B_0)$,
then $U$ satisfies the Cauchy problem
\[
{\p U{t}}-B_0U=0,\quad U(0)=f,
\]
or equivalently,
\be\label{p-s11}
{\p U{t}}-{\p {(S_0U)}E}=0,\quad U(0)=f.
\ee
The solution of \eqref{p-s11} can be written in the form (see Example \ref{desolex1})
\begin{multline*}
U(x,\omega,E,t)
=
(U(t))(x,\omega,E) \\
=
H(R_x(E_m)-R_x(E)-t)
{{S_0(x,R_x^{-1}(R_x(E)+t))}\over{S_0(x,E)}} f(x,\omega,R_x^{-1}(R_x(E)+t)),
\end{multline*}
where again $H$ is the Heaviside function and
\[
R_x(E):=\int_0^E{1\over{S_0(x,\tau)}}d\tau,
\quad E\in [0,E_m].
\]
In other words, 
\begin{multline*}
(T_{B_0}(t)f)(x,\omega,E) \\
=
H(R_x(E_m)-R_x(E)-t)
{{S_0(x,R_x^{-1}(R_x(E)+t))}\over{S_0(x,E)}} f(x,\omega,R_x^{-1}(R_x(E)+t)),
\end{multline*}
and therefore $T_{B_0}(t)$ is evidently of positive type.

Due to \eqref{evo16} and \eqref{evo8-a} there is $M>0$ such that
$0<\kappa\leq S_0\leq M$ a.e. on $\ol{G}\times I$.
For fixed $E\in I$, letting $s_E(t):=R_x^{-1}(R_x(E)+t)-E$,
we have $s_E(0)=0$, and $s_E'(t)=S_0(x,s_E(t)+E)$,
and hence the estimates
\[
\kappa t\leq s_E(t)\leq Mt,\quad \forall t\geq 0
\]
hold, and thus in particular,
\[
E\leq \kappa t+E\leq R_x^{-1}(R_x(E)+t).
\]
Since by the definition of $C$
\[
\frac{1}{S_0}\p{S_0}{E}\leq {1\over {S_0}}q\leq {q\over\kappa}\leq 2C
\]
we get after integration from $E$ to $E'$, where $E\leq E'$,
\[
\frac{S_0(x,E')}{S(_0x,E)}\leq e^{2C(E'-E)},
\]
and hence by the above,
\[
\frac{S_0(x,R_x^{-1}(R_x(E)+t))}{S_0(x,E)}\leq e^{2Cs_E(t)}\leq e^{2CMt}.
\]
Furthermore, since
\[
J_{x,t}(E):=\pa{E} R_x^{-1}(R_x(E)+t)=\frac{R_x'(E)}{R_x'\big(R_x^{-1}(R_x(E)+t)\big)}=\frac{S_0(x,R_x^{-1}(R_x(E)+t))}{S_0(x,E)},
\]
we obtain the following estimate for $T_{B_0}(t)$, $t\geq 0$,
\[
\n{T_{B_0}(t)f}_{L^2(G\times S\times I)}^2
\leq{}&
e^{2CMt}\int_{G\times S\times I} |\ol f(x,\omega,R_x^{-1}(R_x(E)+t))|^2 J_{x,t}(E) dxd\omega dE \\
\leq {}&e^{2CMt}\int_{G\times S\times I'} |f(x,\omega,E')|^2 dxd\omega dE'
=e^{2CMt}\n{f}_{L^2(G\times S\times I)}^2
\]
where $\ol f$ is the extension by zero of $f$ onto $G\times S\times [0,\infty[$ and where we performed the change of variables $E'=R_x^{-1}(R_x(E)+t))$.

The above computations show that for any $n\in\N$, we have
\[
\n{\Big[T_{B_0}(t/n)T_{A_0}(t/n)T_{-(\Sigma+CS_0 I)}(t/n)T_{K_C}(t/n)\Big]^n}
\leq e^{t(CM+\n{\Sigma-CS_0I}+\n{K_C})}.
\]

Let $D:=D(A_0)\cap D(B_0)=\{\phi\in\widetilde W^2_{-,0}(G\times S\times I)\cap H^1(I,L^2(G\times S))\ |\ \phi_{|\Gamma_-}=0,\ \phi(\cdot,\cdot,E_m)=0\}$.
Furthermore, let $\lambda_0\geq 0$. Applying Theorem \ref{coth3-dd} with $g=0$
and with $\Sigma$ replaced by $\Sigma+\lambda_0$ we find that for any $f\in L^2(G\times S\times I)$ the problem
\[
(\lambda_0I-(-T_{C,0}))\phi=f
\]
has a unique solution. In addition the solution  $\phi$ belongs to $D(\widetilde P_{C,0})$. That is why the set $(\lambda_0I-(-T_{C,0}))D$ is dense in $L^2(G\times S\times I)$.

The above verified facts imply that we are able to apply the Trotter's product formula (\cite[Corollary 5.8, p. 227]{engelnagel}) which gives  for ${ f}_C\geq 0$,
and for all $t\geq 0$ 
\be 
G(t){f}_C=\lim_{n\to\infty}\Big[T_{B_0}(t/n)T_{A_0}(t/n)T_{-(\Sigma+CS_0 I)}(t/n)T_{K_C}(t/n)\Big]^n{ f}_C\geq 0,
\ee
and thus by \eqref{p-s8} $\phi\geq 0$ . This implies that $\psi=e^{-CE}\phi\geq 0$.
We remark that a similar technique has been applied e.g. in  \cite[Section XXI-\S 2, Proposition 2, pp. 226-227]{dautraylionsv6}, and \cite[Theorem 5.16]{tervo17-up}.

B. Suppose that $g\geq 0$ is more general. Let $g\in H^1(I, T^2(\Gamma'_-))$ for which $g(\cdot,\cdot,E_{\rm m})=0$.
We decompose the solution $\phi$ of the problem (\ref{p-s1-b}), (\ref{p-s3-b}) as follows. 
Let $u$
be the solution of the problem
\begin{gather}
-{\p {(S_0u)}E}+\omega\cdot\nabla_x u+CS_0 u+\Sigma u=0,\nonumber\\
u_{|\Gamma_-}={g}_C,
\quad
u(\cdot,\cdot,E_{\rm m})=0, \label{eq:problem_simple_u}
\end{gather}
and let $w$ be the solution of the problem
\begin{gather}
-{\p {(S_0w)}E}+\omega\cdot\nabla_x w+CS_0 w+\Sigma w-K_Cw={ f}_C+K_Cu, \nonumber \\
w_{|\Gamma_-}=0,
\quad
w(\cdot,\cdot,E_{\rm m})=0. \label{eq:problem_w_u}
\end{gather}
Then $\phi:=w+u$ is the solution of \eqref{p-s1-b}, \eqref{p-s3-b}.

Since ${ g}_C\geq 0$, it follows from Corollary \ref{pos-le} given below that the solution $u$ is non-negative i.e. $u\geq 0$ .
Finally, because ${ f}_C\geq 0$, and $K_Cu\geq 0$ as $u\geq 0$,
it follows from the part A above that $w\geq 0$.
This allows us to conclude that $\phi=w+u\geq 0$, and so $\psi\geq 0$ as desired. This completes the proof.

\end{proof}

\subsection{On Maximum/Minimum Principles for a Transport Equation}\label{sec-max-min}

Let
\[
\mc P_Cu:=\mc P_C(x,\omega,E,D)u:= -{\p {(S_0u)}E}+\omega\cdot\nabla_xu+CS_0u+\Sigma u.
\]
In the sequel we for simplicity  denote  $\mc P_C(x,\omega,E,D)=\mc P_C$.
Consider the initial inflow boundary value problem
\be \label{max-min0}
\mc P_Cu={ f},\quad u_{|\Gamma_-}=g,\quad u(\cdot,\cdot,E_m)=u_0.
\ee
Here $C$ is any constant such that
\be\label{corr-C}
C\geq \max\{0,{{q}\over{2\kappa}}\}.
\ee

\begin{lemma}\label{max-min-le1}
Suppose that
the assumptions  (\ref{ass1}), (\ref{evo16}), (\ref{evo8-a}), (\ref{evo9-a})   are valid. 
Let $C$ be given by (\ref{eq:def_C}). 
Furthermore, 
let ${ g}\in H^1(I,T^2(\Gamma'))$, $f\in L^2(G\times S\times I)$ and $u_0\in \widetilde W^2(G\times S)$ be such that the compatibility condition 
\be\label{comp1}
{ g}(\cdot,\cdot,E_m)={u_0}_{|\Gamma_-'}
\ee
holds.
Then
the solution $u\in {\s H}_{\mc P_{C,0}}(G\times S\times I^\circ)+
L\big(H^1(I,T^2(\Gamma_-'))\big)$ of the problem (\ref{max-min0}) exists and it is a strong solution in the sense that there exists a sequence
$\{u_n\}\subset \widetilde W^2(G\times S\times I)\cap H^1(I,L^2(G\times S))$  such that for $n\to\infty$
\bea\label{max-min20-a}
&
\n{u_n- u}_{L^2(G\times S\times I)}+
\n{\mc P_Cu_n-f}_{L^2(G\times S\times I)}\to 0,
\nonumber\\
&
{u_n}_{|\Gamma_-}\to u_{|\Gamma_-}= { g}\quad {\rm in}\ T^2(\Gamma_-),
\nonumber\\
&
u_n(\cdot,\cdot,E_m)\to u(\cdot,\cdot,E_m)=u_0
\ {\rm in}\ L^2(G\times S).
\eea 
\end{lemma}

\begin{proof}
A. At first we prove the assertion when $u_0=0$. 
The existence of the solution $u\in {\s H}_{\mc P_{C,0}}(G\times S\times I^\circ)+
L\big(H^1(I,T^2(\Gamma_-'))\big)\subset
{\s H}_{\mc P_C}(G\times S\times I^\circ)$ for the problem (\ref{max-min0}) is guaranteed by Theorem \ref{coth3-dd} (apply the theorem with $K_C=0$). In addition, we see by the proof of Theorem \ref{coth3-dd}
that $u$ is of the form $u=\tilde u+L{ g}$  where
$\tilde u\in D(\widetilde{\mc P_{C,0}})={\s H}_{\mc P_{C,0}}(G\times S\times I^\circ)$.

Since $\tilde u\in D(\widetilde{\mc P}_{C,0})$ there exists a sequence $\{\phi_n\}
\subset \widetilde W^2(G\times S\times I)\cap H^1(I,L^2(G\times S))$ such that
\be\label{max-min18}
\n{\phi_n-\tilde u}_{L^2(G\times S\times I)}+
\n{\mc P_{C,0}\phi_n-\widetilde{\mc P}_{C,0}\tilde u}_{L^2(G\times S\times I)}\to 0,\ n\to\infty
\ee
and ${\phi_n}_{|\Gamma_-}=0,\ \phi_n(\cdot,\cdot,E_m)=0$. 
Since ${ g}\in H^1(I,T^2(\Gamma'))\subset T^2(\Gamma_-)$ we know that $Lg\in \widetilde W^2(G\times S\times I)\cap H^1(I,L^2(G\times S))$ (\cite{tervo17-up}, Lemmas 6.8 and 6.11). Furthermore,
 by (\ref{comp1})  we have $(Lg)(x,\omega,E_m)=g(x-t(x,\omega)\omega,\omega,E_m)=0$ (since $u_0=0$). 
Noting that
\[
\widetilde{\mc P}_{C,0}\tilde u+\mc P_C(Lg)=\mc P_C(\tilde u+Lg)=\mc P_Cu=f
\] 
we see that $\{u_n\}:=\{\phi_n+Lg\}$ is the required sequence.

B. More generally, let $u_0\in \widetilde W^2(G\times S)$ be as in the
statement of this lemma.
Perform the change of unknown $v:=u-u_0$ in the problem (\ref{max-min0}). Then $v$ satisfies
\be \label{max-min0-1}
\mc P_Cv={ f}-\mc P_Cu_0,\ v_{|\Gamma_-}= g-{u_0}_{|\Gamma_-'},\ v(\cdot,\cdot,E_m)=0.
\ee
By Part A the solution $v$ of the problem (\ref{max-min0-1}) exists and there exists a sequence $\{v_n\}
\subset \widetilde W^2(G\times S\times I)\cap H^1(I,L^2(G\times S))$  such that for $n\to\infty$
\be\label{max-min18-1}
\n{v_n-v}_{L^2(G\times S\times I)}+
\n{\mc P_{C}v_n-(f-\mc P_Cu_0)}_{L^2(G\times S\times I)}\to 0
\ee
and ${v_n}_{|\Gamma_-}\to g-{u_0}_{|\Gamma_-'},\ v_n(\cdot,\cdot,E_m)\to 0$. 
We find that $\{u_n\}:=\{v_n+u_0\}$ is the required sequence. This completes the proof.

\end{proof}

The following Poincare-type inequality is valid

\begin{lemma}\label{max-min-le2} 
Suppose that the assumptions of Lemma \ref{max-min-le1} are valid. 
Let $C$ be given by (\ref{eq:def_C}). Furthermore,
let $u\in {\s H}_{\mc P_{C,0}}(G\times S\times I^\circ)+
L\big(H^1(I,T^2(\Gamma_-'))\big)$ be the solution of the problem (\ref{max-min0}) (guaranteed by Lemma \ref{max-min-le1}).
Then
\begin{multline}
\label{apriori}
\la \mc P_Cu,u\ra_{L^2(G\times S\times I)}+\int_{\Gamma_-}(u_{|\Gamma_-})^2(\omega\cdot \nu)_-d\sigma d\omega dE
\\
+\int_{G\times S}S(\cdot,E_m)u(\cdot,\cdot,E_m)^2dx d\omega
\geq 
c'\n{u}_{L^2(G\times S\times I)}^2
\end{multline}
where $c':=\min\{{1\over 2},{\kappa\over 2},c\}$.
\end{lemma}

\begin{proof}

Applying the Green's formula (\ref{green})  and integrating by parts, we have as in section \ref{esols}
for $\phi,\ v\in \widetilde W^2(G\times S\times I)\cap H^1(I,L^2(G\times S))$ 
\[
&
\la \mc P_C\phi,v\ra_{L^2(G\times S\times I)}
=
-\la{\p {(S_0\phi)}E},v\ra_{L^2(G\times S\times I)}
\nonumber\\
&
+\la\omega\cdot\nabla_x\phi,v\ra_{L^2(G\times S\times I)}+\la CS_0\phi,v\ra_{L^2(G\times S\times I)}+\la\Sigma\phi,v\ra_{L^2(G\times S\times I)}
\nonumber\\
={}&
\la\phi,S_0{\p {v}E}\ra_{L^2(G\times S\times I)}-
\int_{G\times S}\big( (S_0\phi v)(\cdot,\cdot,E_m)-(S_0\phi v)(\cdot,\cdot,0)\big)dx d\omega  \nonumber\\ 
&-\la\phi,\omega\cdot\nabla_x v\ra_{L^2(G\times S\times I)}
+\int_{\Gamma}(\omega\cdot\nu)\phi v d\sigma d\omega dE  
\nonumber\\
&+\la \phi,CS_0v\ra_{L^2(G\times S\times I)}
+\la\phi,\Sigma^* v\ra_{L^2(G\times S\times I)}.
\]
Hence noting that 
\[
\int_{\Gamma}(\omega\cdot\nu)\phi v d\sigma d\omega dE=\int_{\Gamma_+}(\omega\cdot\nu)_+\phi v d\sigma d\omega dE -
\int_{\Gamma_-}(\omega\cdot\nu)_-\phi v d\sigma d\omega dE 
\]
we obtain
\bea\label{csda25-a}
&
\la \mc P_C\phi,v\ra_{L^2(G\times S\times I)}
+\int_{\Gamma_-}(\omega\cdot\nu)_-\phi v d\sigma d\omega dE 
+\int_{G\times S} S_0(\cdot,E_m)\phi(\cdot,\cdot,E_m) v(\cdot,\cdot,E_m) dx d\omega 
\nonumber\\
={}&
\la\phi,S_0{\p {v}E}\ra_{L^2(G\times S\times I)}
-\la\phi,\omega\cdot\nabla_x v\ra_{L^2(G\times S\times I)}\nonumber\\
&
+\int_{G\times S} S_0(\cdot,0)\phi(\cdot,\cdot,0) v(\cdot,\cdot,0) dx d\omega 
+\int_{\Gamma_+}(\omega\cdot\nu)_+\phi v d\sigma d\omega dE  \nonumber\\
&+\la \phi,CS_0v\ra_{L^2(G\times S\times I)}
+\la\phi,\Sigma^* v\ra_{L^2(G\times S\times I)} \\
={}&
B(\phi,v)
\eea
where $B(\cdot,\cdot)$ is the bilinear form \eqref{csda27} with $K_C=0$.
In Theorem \ref{csdath1} we have shown that
\be\label{max-min22}
B(\phi,\phi)\geq c'\n{\phi}_{H^1}^2\geq c'\n{\phi}_{L^2(G\times S\times I)}^2
,\quad \forall \phi\in C^1(\ol G\times S\times I)
\ee
which can be  deduced similarly for all $\phi\in \widetilde W^2(G\times S\times I)\cap H^1(I,L^2(G\times S))$.
Combining (\ref{csda25-a}) (with $v=\phi$) and (\ref{max-min22}) we obtain
\begin{multline}
\label{csda25-b}
\la \mc P_C\phi,\phi\ra_{L^2(G\times S\times I)}
+\int_{\Gamma_-}(\omega\cdot\nu)_-(\phi_{|\Gamma_-})^2  d\sigma d\omega dE 
+\int_{G\times S} S_0(\cdot,E_m)\phi(\cdot,\cdot,E_m)^2 dx d\omega
\\
\geq c'\n{\phi}_{L^2(G\times S\times I)}^2
,\quad \forall \phi\in \widetilde W^2(G\times S\times I)\cap H^1(I,L^2(G\times S)).
\end{multline}
By Lemma \ref{max-min-le1} $u$ is the strong solution of the problem (\ref{max-min0}). Choosing a sequence $\{u_n\}$ given in Lemma \ref{max-min-le1}, choosing $\phi=u_n$ in the estimate (\ref{csda25-b}) and letting $n\to\infty$ we conclude the desired estimate (\ref{apriori}) which finishes the proof.

\end{proof}

We show the following minimum/maximum principle. For simplicity we consider the case $g \geq 0$ only.
In addition, to shorten the considerations, we assume 
that the constant $C$ is so large (for example, $C\geq {q\over\kappa}$) that 
\be\label{corr-C-a}
-{\p { S_0}E}+C S_0\geq 0.
\ee
This is sufficient to our needs. Generalizations of these results are obvious .

\begin{theorem}\label{max-min-th}
Suppose that 
the assumptions  (\ref{ass1}), (\ref{evo16}), (\ref{evo8-a}), (\ref{evo9-a})  
of Theorem \ref{pos-th2} are valid. 
Let $C$ be given by (\ref{eq:def_C}). 
Furthermore, let $u$ be the solution of the problem (\ref{max-min0}) with ${ f}=0$, $u_0=0$ and $g\in H^1(I,T^2(\Gamma_-'))$ such that
\[
g\geq 0\quad {\rm and}\quad g(\cdot,\cdot,E_m)=0.
\]
Then for a.e. $(x,\omega,E)\in G\times S\times I$,
\be\label{max-min3}
\essinf_{(y',\omega',E')\in \Gamma_-}{g}(y',\omega',E')\leq u(x,\omega,E)\leq 
\esssup_{(y',\omega',E')\in \Gamma_-}{g}(y',\omega',E').
\ee
\end{theorem}

\begin{proof}
We proceed in the spirit of \cite{besson}.  

A. First, we verify that
\be \label{max-min8}
u(x,\omega,E)\leq 
\esssup_{(y',\omega',E')\in \Gamma_-}{ g}(y',\omega',E').
\ee
Denote $M:=\esssup_{(y',\omega',E')\in \Gamma_-}{g}(y',\omega',E')\geq 0$ and
\[
U:=(u-M)_+={1\over 2}\big(|u-M|+(u-M)\big).
\]
Furthermore, let
\begin{gather*}
A:=\{(x,\omega,E)\in \ol G\times S\times I\ |\ u(x,\omega,E)-M>0\}, \\
(\partial A)_-:=A\cap\Gamma_-, \\
A_1:=\{(x,\omega)\in G\times S\ |\ (x,\omega,E_m)\in A\}.
\end{gather*}

We know that $|u|\in {\s H}_{\mc P_C}(G\times S\times I^\circ)$ when $u\in 
{\s H}_{\mc P_C}(G\times S\times I^\circ)$. 
This can be seen as follows (cf. \cite{grigoryan}, Lemma 5.2; see also \cite{besson}, p. 8):

We have
\[
\mc P_Cu=-S_0{\p u{E}}+\omega\cdot\nabla_xu+a_Cu=Qu+a_Cu
\]
where 
\[
Qu:=-S_0{\p u{E}}+\omega\cdot\nabla_xu,\quad a_C:=-{\p {S_0}{E}}+CS_0+\Sigma.
\]
Since $u, \ \mc P_Cu\in L^2(G\times S\times I)$ we see that $Qu\in L^2(G\times S\times I)$.
It suffices to show that $Q|u|\in L^2(G\times S\times I)$. Then 
$\mc P_C|u|=Q|u|+a_C|u|\in L^2(G\times S\times I)$, as desired.
Let
\[
U_\epsilon:=\sqrt{u^2+\epsilon}.
\]
Then we find that 
\be\label{le4.1-1}
QU_\epsilon=-S_0{u\over{\sqrt{u^2+\epsilon}}}{\p u{E}}+\omega\cdot\big(
{u\over{\sqrt{u^2+\epsilon}}}\big)\nabla_xu={u\over{\sqrt{u^2+\epsilon}}}Qu.
\ee
Since for $0<\epsilon<1$ (here we choose the branch $\sqrt{u^2}= |u|$)
\[
|U_\epsilon| \leq \sqrt{u^2+1}\ {\rm and}\ U_\epsilon\to \sqrt{u^2}= |u|\ {\rm a.e. \ in} \ G\times S\times I \ {\rm for}\ \epsilon \to 0
\]
we obtain by the Lebesgue Dominated Convergence Theorem that when $\epsilon\to 0$,
\be\label{le4.1-2}
U_\epsilon \to |u| \quad {\rm in}\ L^2(G\times S\times I). 
\ee
Similarly, by 
the Lebesgue Dominated Convergence Theorem when $\epsilon\to 0$,
\be\label{le4.1-3}
QU_\epsilon ={u\over{\sqrt{u^2+\epsilon}}}Qu\to {u\over{\sqrt{u^2}}}Qu
={u\over{|u|}}Qu={\rm sign}(u)Qu \quad {\rm in}\ L^2(G\times S\times I). 
\ee 
That is why $Q|u|\in L^2(G\times S\times I)$ and $Q|u|={\rm sign}(u)Qu$.

Hence we find that $U\in {\s H}_{\mc P_C}(G\times S\times I^\circ)$ and
$U$ is the solution of the problem (here $A^\circ$ denotes the interior of $A$)
\bea\label{prob-for-U}
&
\mc P_CU=\hat{ f}:=\begin{cases}
\mc P_Cu-\mc P_CM=-(-{\p {S_0}E}+CS_0+\Sigma)M& {\rm in}\ A^\circ\\
0& {\rm in }\ (G\times S\times I^\circ)\setminus A\end{cases}\nonumber\\
&
U_{|\Gamma_-}=\hat{ g}:=
\begin{cases}
{ g}-M& {\rm in}\ (\partial A)_-\\
0& {\rm in }\ \Gamma_-\setminus (\partial A)_-
\end{cases}\nonumber\\
&
U(\cdot,\cdot,E_m)= U_0:=
\begin{cases}
-M& {\rm in}\ A_1\\
0& {\rm in }\ (G\times S)\setminus A_1\end{cases}.
\eea
We find that $\hat g\in H^1(I,L^2(G\times S))$ and by $g(\cdot,\cdot,E_m)=0$ the compatibility condition $\hat g(\cdot,\cdot,E_m)={ U_0}_{|\Gamma_-'}$ holds.

In virtue of Lemma \ref{max-min-le2}
\begin{multline}
\label{apriori-U}
\la \mc P_CU,U\ra_{L^2(G\times S\times I)}+\int_{\Gamma_-}(U_{|\Gamma_-})^2(\omega\cdot \nu)_-d\sigma d\omega dE
\\
+\int_{G\times S}S_0(\cdot,E_m)U(\cdot,\cdot,E_m)^2dx d\omega
\geq 
c'\n{U}_{L^2(G\times S\times I)}^2.
\end{multline}
Since
\bea
\int_{\Gamma_-}(U_{|\Gamma_-})^2(\omega\cdot \nu)_-d\sigma d\omega dE
=&{}
\int_{(\partial A)_-}U_{|(\partial A)_-}U_{|(\partial A)_-}(\omega\cdot \nu)_-d\sigma d\omega dE \nonumber\\
=&{}
\int_{(\partial A)_-}({ g}-M)(u-M)_{|(\partial A)_-}(\omega\cdot \nu)_-d\sigma d\omega dE\nonumber
\eea
and
\[
\int_{G\times S}S_0(\cdot,E_m)U(\cdot,\cdot,E_m)^2dx d\omega=\int_{A_1}S_0(\cdot,E_m)(-M)(u-M)(\cdot,\cdot,E_m)dx d\omega 
\]
we see  by (\ref{prob-for-U}) that 
\begin{multline}
\label{apriori-U-a}
\la -(-{\p {S_0}E}+CS_0+\Sigma)M,u-M\ra_{L^2(A)}+\int_{(\partial A)_-}({g}-M)(u-M)_{|(\partial A)_-}(\omega\cdot \nu)_-d\sigma d\omega dE
\\
+\int_{A_1}S_0(\cdot,E_m)(-M)(u-M)(\cdot,\cdot,E_m)dx d\omega
\geq 
c'\n{u-M}_{L^2(A)}^2
\end{multline}

Since 
$(-{\p {S_0}E}+CS_0+\Sigma)M\geq 0$, ${ g}-M\leq 0$, $-M\leq 0$ and by the definition of $A$,
$u-M>0$ in $A,\ (\partial A)_-, \ A_1$. That is why we conclude from (\ref{apriori-U-a})
that $\n{u-M}_{L^2(A)}=0$ and so $A$ has zero measure (since $u-M>0$ in $A$). This proves (\ref{max-min8}).

B. 
The lower inequality follows by similar treatments as applied in Part A by replacing $U$ with $(u-m)_-$, and replacing the set $A$ by
$
A=\{(x,\omega,E)\in \ol G\times S\times I\ |\ u(x,\omega,E)-m<0\}.
$ 
This finishes the proof.

\end{proof}

\begin{corollary}\label{pos-le}
Make the assumptions (\ref{ass1}), (\ref{evo16}), (\ref{evo8-a}), (\ref{evo9-a})   
of Theorem \ref{pos-th2} with $C$ given by (\ref{eq:def_C}). 
Furthermore, let $u$ be the solution of the problem (\ref{max-min0}) with ${ f}=0$, $u_0=0$ and $g\in H^1(I,T^2(\Gamma_-'))$ such that
\[
g\geq 0\quad {\rm and}\quad g(\cdot,\cdot,E_m)=0.
\]
Then  $u\geq 0$.
\end{corollary}

\begin{proof}
Perform (as in section \ref{single-eq}) the change of unknown by $v=e^{aE}u$ 
where $a\in\R$ is large enough. Then $v$ satisfies the problem (\ref{max-min0}) 
for which (\ref{corr-C-a}) holds. By the previous theorem  $v\geq 0$ and so $u\geq 0$.
\end{proof}

\begin{remark}
A. The assumption $g\geq 0$ in Theorem \ref{max-min-th} is not necessary but it simplifies somewhat the considerations.

B.
The technique applied above has obvious variations/generalizations. For example, we may assume that $u_0\not =0$ in the problem (\ref{max-min0}) and show that
\be 
\inf_{(x,\omega)\in G\times S}u_0(x,\omega)\leq u\leq \sup_{(x,\omega)\in G\times S}u_0(x,\omega).
\ee
We omit here a systematic study of max-min-principles for the transport problems.
\end{remark}

\begin{example}

For example, in the special case of Example \ref{desolex1},
where $S_0(x,E)=S_0(E)$ does not depend on $x$, and $\Sigma(x,\omega,E)=\Sigma(x,\omega)$ does not depend on $E$,
the solution $u$ can be expressed explicitly in the form
\be
u={1\over{S_0(E)}}H\big(R(E_{\rm m})-Q(x,\omega,E)\big)e^{-\int_0^{t(x,\omega)}\Sigma(x-s\omega,\omega)ds}\tilde{ g}(x-t(x,\omega)\omega,\omega,Q(x,\omega,E)), \label{eq:explicit_u_in_simplified_case}
\ee
where
\[
R(E):={}& \int_0^E{1\over{S_0(\tau)}}d\tau, \\[2mm]
Q(x,\omega,E):={}&R(E)+t(x,\omega), \\[2mm]
\tilde{ g}(y,\omega,\eta):={}& S_0(R^{-1}(\eta)){ g}(y,\omega,R^{-1}(\eta)),
\]
From \eqref{eq:explicit_u_in_simplified_case} it is clear that $u\geq 0$ whenever $g\geq 0$ such as the  Corollary \ref{pos-le} predicts.

\end{example}

\subsection{A Geometric Approach to Show the Non-negativity of Solutions for the Problem (\ref{max-min0})}

Consider the problem (\ref{max-min0}) that is,
\be\label{alter-1}
 -S_0{\p {u}E}+\omega\cdot\nabla_xu-{\p {S_0}{E}}u+CS_0u+\Sigma u={ 0},\ u_{|\Gamma_-}= g,\ u(\cdot,\cdot,E_m)=0.
\ee
We sketch here an alternative proof that, under the standing assumptions of this section,
the solution $u$ of the problem \eqref{alter-1} is non-negative, when ${ g}\geq 0$. 

The {\it characteristics} for \eqref{alter-1}
are curves $\zeta=(\zeta_1,\zeta_2,\zeta_3,\zeta_4)$ in $\R^3\times S\times\R\times\R $ for which
\bea\label{alter-2}
&
\zeta_1'(s)=\omega,\ \zeta_2'(t)=0, \zeta_3'(s)=-S(\zeta(s)),\ \zeta_4'(s)=W(\zeta(s))\\
&
\zeta(0)=(y,\omega',E',\zeta_0(y,\omega',E')),
\eea
where
\[
W(x,\omega,E):=-\Big(-\p{S_0}{E}(x,E)+CS_0(x,E)+\Sigma(x,\omega,E)\Big),
\]
and 
\[
\zeta_0(y,\omega',E')
:=\begin{cases}
0, & {\rm if}\ (y,\omega',E')\in G\times S\times \{E_m\}, \\
{g}(y,\omega',E'), & {\rm if}\ (y,\omega',E')\in \Gamma_-. \\
\end{cases}
\]
Denote $z:=(y,\omega',E')$ and denote the solution of (\ref{alter-2}) by \linebreak
$\zeta(z,s)=(\zeta_1(z,s),\zeta_2(z,s),\zeta_3(z,s),\zeta_4(z,s))$.
Moreover, $u$ is obtained as follows: Solve $(z,s)$ from 
\be\label{alter-3}
\zeta_1(z,s)=x,\ \zeta_2(z,s)=\omega,\ \zeta_3(z,s)=E
\ee
and then substitute $(z,s)$ into $\zeta_4(z,s)$. The solution is $u=\zeta_4(F(x,\omega,E))$ where $F(x,\omega,E)$ is the solution of (\ref{alter-3}).
One can see that the positivity of $\zeta_4$ implies the positivity of $u$.

We assume that
\begin{align}\label{eq:strong_S_0_ass}
S_0\in C^2(I,C(\ol G)).
\end{align}
Clearly, assumption \eqref{eq:strong_S_0_ass} implies \eqref{evo16}
(recall that $G$ is bounded).
Since $\p{S_0}{E}\in C(I,C(\ol G))$,
there is a constant $M>0$ such that
$\n{\p{S_0}{E}(E)}_{L^\infty(G)}\leq M$, and hence
for every $E,E'\in I$, it holds
\[
\n{S_0(E')-S_0(E)}_{L^\infty(G)}\leq M|E-E'|.
\]
For a given $(x,\omega)\in G\times S$,
letting
\[
s_{x,\omega}:[0,t(x,\omega)]\times I\to\R;
\quad s_{x,\omega}(t,E):=S_0(x-t\omega,E),
\]
we see that $s_{x,\omega}$ is continuous, and
\[
|s_{x,\omega}(t,E')-s_{x,\omega}(t,E)|\leq M|E'-E|
\]
i.e. the map $(t,E)\mapsto s_{x,\omega}(t,E)$ satisfies the Lipschitz condition on $I$ uniformly with respect to $E$.
Therefore, Cauchy-Lipschitz (or Picard-Lindel\"of) theorem (cf. \cite[Chapter IV, Proposition 1.1]{lang95})
implies that for every $(x,\omega,E)\in G\times S\times I$ the problem
\[
\dot{\xi}(t)={}&S_0(x-t\omega,\xi(t)),
\quad t\in [0,t(x,\omega)],\\
\xi(0)={}&E,
\]
has a unique solution $\xi\in C^1([0,\ol{\tau}])$ defined on the maximal interval $[0,\ol{\tau}]$,
such that $(x-t\omega,\xi(t))\in \ol{G}\times I$ for all $t\in [0,\ol{\tau}]$.

We write this solution as $\xi(x,\omega,E,t)$, and write $\ol{\tau}(x,\omega,E)$ for the end-point of its maximal interval of existence.
Because $S_0\geq \kappa>0$ on $G\times I$ (by \eqref{evo8-a}), we have
\begin{align}\label{eq:inequality_gamma}
\xi(x,\omega,E,t')\geq \xi(x,\omega,E,t)+\kappa(t'-t),
\end{align}
whenever $t\leq t'$,
and therefore one can see that (i) or (ii) (or both) below holds:
\begin{align}\label{eq:alpha_endpoint}
\textrm{(i)}\ \xi\big(x,\omega,E,\ol{\tau}(x,\omega,E)\big)=E_m,
\quad\textrm{or}\quad \textrm{(ii)}\ \ol{\tau}(x,\omega,E)=t(x,\omega).
\end{align}

Denote for $(x,\omega,E,t)\in G\times S\times I\times [0,\ol{\tau}(x,\omega,E)]$,
\[
\eta(x,\omega,E,t):=(x-t\omega, \omega, \xi(x,\omega,E,t)),
\]
and notice that $\eta(x,\omega,E,0)=(x,\omega,E)$.
It follows from \eqref{eq:alpha_endpoint} that
\begin{align}\label{eq:T_boundary}
\alpha(x,\omega,E):=\eta(x,\omega,E,\ol{\tau}(x,\omega,E))\in \Gamma_-\cup (G\times S\times \{E_m\}),
\end{align}
for all $(x,\omega,E)\in G\times S\times I$.

Below we shall understand that $S_0(\eta(x,\omega,E,t))$ means $S_0(x-t\omega, \xi(x,\omega,E,t))$,
and similarly for $\p{S_0}{E}$, since $S_0(x,E)$ is assumed not to depend on $\omega$.

If $\varphi:G\times S\times I\to\R$ be smooth enough (say $C^1$),
then
\begin{align}\label{eq:alter-x}
\pa{t} \Big(\varphi\big(\eta(x,\omega,E,t)\big)\Big)
={}&
\la (\nabla_{(x,\omega,E)}\varphi)(\eta(x,\omega,E,t)),{\p {\eta}t}(x,\omega,E,t)\ra \\
={}&
S_0\big(\eta(x,\omega,E,t)\big)\p{\varphi}{E}\big(\eta(x,\omega,E,t)\big)-(\omega\cdot\nabla_x\varphi)\big(\eta(x,\omega,E,t)\big). \nonumber
\end{align}

Define
\begin{align}\label{eq:general_explicit_u}
u(x,\omega,E):=h(x,\omega,E)e^{-\int_0^{\ol{\tau}(x,\omega,E)}W(\eta(x,\omega,E,t))dt}.
\end{align}
where (recall \eqref{eq:T_boundary})
\begin{align}\label{eq:h_for_explicit_u}
h(x,\omega,E)
:=\begin{cases}
0, & {\rm if}\ \alpha(x,\omega,E)\in G\times S\times \{E_m\}, \\
{ g}(\alpha(x,\omega,E)), & {\rm if}\ \alpha(x,\omega,E)\in \Gamma_-. \\
\end{cases}
\end{align}
We claim that $u$ is a solution to the problem \eqref{alter-1}.

Indeed, it is not difficult to see that the following hold
(we omit the details),
for $(x,\omega,E)\in G\times S\times I^\circ$ and for $t,s\in\R$ small enough
($I^\circ=]E_0,E_m[$),
\[
& \eta(x,\omega,E,t+s)=\eta(\eta(x,\omega,E,s),t), \nonumber \\[2mm]
& \ol{\tau}(\eta(x,\omega,E,t))=\ol{\tau}(x,\omega,E)-t, \nonumber \\
& \pa{t} h(\eta(x,\omega,E,t))=0.
\]
Briefly, the first one of these relations is simply the group property of solutions to an ODE, the second one is analogous to $t(x+s\omega,\omega)=t(x,\omega)-s$,
while the last one is analogous to to $\omega\cdot\nabla_x (L_-g)=0$ with $\Sigma=0$ in Lemma \ref{trathle1}.
It follows from these relations that
\[
u(\eta(x,\omega,E,s))
=h(\eta(x,\omega,E,s))e^{-\int_s^{\ol{\tau}(x,\omega,E)} W(\eta(x,\omega,E,t))dt}
\]
and further
\begin{align}\label{eq:alter-y}
\pa{s} u(\eta(x,\omega,E,s))
={}&
h(\eta(x,\omega,E,s))\pa{s}e^{-\int_s^{\ol{\tau}(x,\omega,E)} W(\eta(x,\omega,E,t))dt} \\
={}& u(\eta(x,\omega,E,s))W(\eta(x,\omega,E,s)).
\end{align}

We claim, without a proof, that
arguments similar to those used in the proof of Lemma \ref{trathle1}
combined with \eqref{eq:alter-x} and \eqref{eq:alter-y}
then show that $u$ is a (distributional) solution on $G\times S\times I^\circ$
in the weak sense to the PDE in \eqref{alter-1}.
That $u$ also satisfies the boundary conditions (on $\Gamma_-$ and on $G\times S\times \{E_m\}$)
is evident from \eqref{eq:general_explicit_u}, \eqref{eq:h_for_explicit_u},
and hence one concludes that $u$ indeed solves \eqref{alter-1}.

To conclude our argument that
the solution $u$ of \eqref{alter-1} is non-negative,
it only remains note that $h\geq 0$ in \eqref{eq:h_for_explicit_u} because of the assumption ${g}\geq 0$,
and hence \eqref{eq:general_explicit_u} directly implies that $u\geq 0$.

\sectionspace
\section{Some Notes on Computational Methods}\label{comp}

\subsection{A Decomposition of Solutions Corresponding to Primary and Secondary Particles}\label{desol}

Consider the system of transport equations as above,
\begin{align}
\omega\cdot\nabla_x\psi_1+\Sigma_1\psi_1-K_1\psi={}&f_1,\nonumber\\
-{\p {(S_j\psi_j)}E}+\omega\cdot\nabla_x\psi_j+
\Sigma_j\psi_j
- K_j\psi={}&f_j, \quad j=2,3, \label{desol10}
\end{align}
holding a.e. on $G\times S\times  I$,
together with the inflow boundary and initial values
\begin{alignat}{3}
\psi_{|\Gamma_-}&=g && \quad {\rm a.e.\ on}\ \Gamma_-, \label{desol11} \\[2mm]
\psi_j(\cdot,\cdot,E_m)&=0\quad && \quad {\rm a.e.\ on}\ G\times S,\ j=2,3. \label{desol12}
\end{alignat}

The solution $\psi=(\psi_1,\psi_2,\psi_3)$ for this problem can be decomposed as follows.
Let $u=(u_1,u_2,u_3)$ be the solution of the problem without collisions
\begin{align}
\omega\cdot\nabla_x u_1+\Sigma_1 u_1={}&f_1,\nonumber\\
-{\p {(S_ju_j)}E}+\omega\cdot\nabla_x u_j+\Sigma_j u_j
={}& f_j,\quad j=2,3, \label{desol13}
\end{align}
together with the inflow boundary and initial values
\be
u_{|\Gamma_-}&=g, \label{desol14} \\[2mm]
u_j(\cdot,\cdot,E_m)&=0,\quad j=2,3.\label{desol15}
\ee
Furthermore, let $w=(w_1,w_2,w_3)$ be the solution of the problem
\begin{align}
\omega\cdot\nabla_x w_1+\Sigma_1 w_1-K_1w={}& K_1u\nonumber\\
-{\p {(S_jw_j)}E}+\omega\cdot\nabla_x w_j
+\Sigma_j w_j-K_jw
={}& K_ju,\quad j=2,3,\label{desol16}
\end{align}
together with \emph{homogeneous} inflow boundary and initial values
\be
w_{|\Gamma_-}&=0,\label{desol17} \\[2mm]
w_j(\cdot,\cdot,E_m)&=0,\quad j=2,3. \label{desol18}
\ee

Then we find that $\psi=u+w$ is the solution of \eqref{desol10}-\eqref{desol12}.
This corresponds to decomposing the evolution of
the particle field $\psi$ obeying the full CSDA Boltzmann transport problem \eqref{desol10}-\eqref{desol12}
in terms of the evolution of the
\emph{primary (uncollided) particles}, represented by $u$,
and of \emph{secondary (collided) particles}, represented by $w$.
The method of decomposing $\psi=u+w$ in this way is useful e.g. in constructing numerical solutions, 
and is known, for example in neutron transport theory, under the name collided-uncollided split (cf.\ the recent work \cite{hauck2013coll} and references therein).

To explain a bit this terminology, notice that the primary, uncollided field $u$ obeys \eqref{desol13}
which does not involve the collision operator $K=(K_1,K_2,K_3)$,
and $u$ contains a direct contribution from the external (boundary) sources $g$ (through \eqref{desol14}).
On the other hand, the field $u$ \emph{right after collision} as modelled by the term $Ku$,
acts as an internal source in the equation \eqref{desol16} for the secondary field $w$,
while external sources do not contribute to $w$ directly (a fact captured by \eqref{desol17}).

Note especially that the system (\ref{desol13})-(\ref{desol15}) is uncoupled,
in that the different particle species (photon, electron, positron that we consider here)
evolve independently of each other.
In some cases the primary component $u$ can be calculated exactly such as the following example shows.

\begin{example}\label{desolex1}
Suppose that $\Sigma_1\in L^2(G\times S\times I)$, $\Sigma_1\geq c>0$ for some constant $c$, and that,
$\Sigma_j(x,\omega,E)=\Sigma_j(x,\omega)$ (i.e. $\Sigma_j$ does not depend on $E$) and $\Sigma_j\in L^2(G\times S),\ \Sigma_j\geq c>0$, for $j=2,3$.
Furthermore, suppose that $S_j(x,E)=S_j(E)$, $j=2,3$ (i.e. $S_j$ is independent of $x$) and that $S_j:I\to \R_+$ are continuous, strictly positive functions. Finally, let 
$f_1\in L^2(G\times S\times I)$, $g_1\in T^2(\Gamma_-)$, and
for $j=2,3$ let $f_j\in H^1(I,L^2(G\times S))$, $g_j\in H^1(I,T^2(\Gamma'_-))$,
such that $g_j(E_{\rm m})=0$ (compatibility condition).
Define $R_j:I\to\R$ by 
\be 
R_j(E):=\int_0^E{1\over{S_j(\tau)}}d\tau,\quad j=2,3.
\ee
Let $r_{m,j}:=R_j(E_m)$. Then $R_j:I\to [0,r_{m,j}]$ are continuously differentiable and strictly increasing bijections. Let $R_j^{-1}: [0,r_{m,j}]\to I$ be their inverses.
We denote the argument of $R_j^{-1}$ on $[0,r_{m,j}]$ by $\eta$,
i.e. $E=R_j^{-1}(\eta)$ (or equivalently $\eta=R_j(E)$).

Consider first the (primary) uncoupled problem,
\begin{gather}
-{\p {(S_ju_j)}E}+\omega\cdot\nabla_x u_j+
\Sigma_j u_j
= f_j,
\nonumber\\
{u_j}_{|\Gamma_-}=g_j,\quad 
u_j(\cdot,\cdot,E_m)=0,  \label{desol19}
\end{gather}
where $j=2,3$.
We perform a well-known change of variables (see e.g. \cite{frank10}, \cite{rockell}) in the problem \eqref{desol19} by defining a new unknowns $v_j$, for $j=2,3$, by setting
\be 
v_j(x,\omega,\eta):=S_j(R_j^{-1}(\eta))u_j(x,\omega,R_j^{-1}(\eta)),
\ee
i.e.
\be 
v_j(x,\omega,R_j(E))=S_j(E)u_j(x,\omega,E).
\ee
Then we find that
\be 
{\p {(S_ju_j)}E}={\p {v_j}\eta}R_j'(E)={\p {v_j}\eta}{1\over{S_j(E)}}
={\p {v_j}\eta}{1\over{S_j(R_j^{-1}(\eta))}},
\ee
and so, after writing
\begin{alignat*}{2}
\tilde{f}_j(x,\omega,\eta):=S_j(R_j^{-1}(\eta))f_j(x,\omega,R_j^{-1}(\eta)),
\quad & (x,\omega,\eta)\in G\times S\times [0,r_{m,j}], \\
\tilde{g}_j(y,\omega,\eta):=S_j(R_j^{-1}(\eta))g_j(y,\omega,R_j^{-1}(\eta)),
\quad & (y,\omega,\eta)\in \tilde\Gamma_{-,j},
\end{alignat*}
where
\[
\tilde\Gamma_{-,j}:=\{(y,\omega,\eta)\in \partial G\times S\times [0,r_{m,j}]\ |\ \omega\cdot\nu(y)<0\},
\]
we see that the problem \eqref{desol19} is equivalent to
\bea \label{desol20}
&-{\p {v_j}\eta}+\omega\cdot\nabla_x v_j+
\Sigma_j v_j
=\tilde{f}_j\quad {\rm a.e.\ on}\ G\times S\times [0,r_{m,j}],
\eea
subject to inflow boundary and initial value conditions,
\begin{alignat}{3}
{v_j}_{|\tilde\Gamma_{-,j}}&=\tilde{g}_j\quad && {\rm a.e.\ on}\ \tilde\Gamma_{-,j},\label{desol19a} \\
v_j(\cdot,\cdot,r_{m,j})&=0\quad && {\rm a.e.\ on}\ G\times S,\ j=2,3. \label{desol19aa}
\end{alignat}
Notice that $\tilde{g}_j(r_{m,j})=0$ since $g_j(E_{\rm m})=0$.
The original unknowns $u_j$, $j=2,3$, are given in terms of $v_j$ by
\be\label{desol21}
u_j(x,\omega,E)={1\over{S_j(E)}}v_j(x,\omega, R_j(E)).
\ee

The problem \eqref{desol20}-\eqref{desol19aa} can be solved explicitly, at least formally.
The solution $v_j$ of \eqref{desol20} is the sum $v_{1,j}+v_{2,j}$
of solutions $v_{1,j}$ and $v_{2,j}$  of the following problems 
\begin{gather}
-{\p {v_{1,j}}\eta}+\omega\cdot\nabla_x {v_{1,j}}+
\Sigma_j v_{1,j}
=\tilde f_j,\nonumber\\
{v_{1,j}}_{|\tilde\Gamma_{-,j}}=0,
\quad
v_{1,j}(\cdot,\cdot,r_{m,j})=0,\label{desol20a}
\end{gather}
and
\begin{gather}
-{\p {v_{2,j}}\eta}+\omega\cdot\nabla_x {v_{2,j}}+
\Sigma_j v_{2,j}
=0,\nonumber\\
{v_{2,j}}_{|\tilde\Gamma_{-,j}}=\tilde g_j,
\quad
v_{2,j}(\cdot,\cdot,r_{m,j})=0,\label{desol20b}
\end{gather}
where the (partial differential) equations are to be satisfies on $G\times S\times [0,r_{m,j}]$,
the (inflow) boundary conditions on $\tilde\Gamma_{-,j}$
and the initial (energy) conditions on $G\times S$.

The  solution of (\ref{desol20b}) is (cf. \cite[Ch. XXI, Sec. 3.2, pp. 233-235]{dautraylionsv6},
or \cite[proof of Theorem 6.3]{tervo17-up}; replace first $\eta$ by $r_{m,j}-\eta$)
\begin{multline}\label{desol20ab}
v_{2,j}(x,\omega,\eta)
\\
=H(r_{m,j}-\eta-t(x,\omega))e^{\int_0^{t(x,\omega)}-\Sigma_j(x-s\omega,\omega)ds}
\tilde{g}_j(x-t(x,\omega)\omega,\omega,\eta+t(x,\omega)),
\end{multline}
where $H$ is the Heaviside function.
By performing similar computations as in the proof of Lemma \ref{trathle1},
one sees that $v_{2,j}$ defined by \eqref{desol20ab} is in fact a weak (distributional) solution of \eqref{desol20b}.
Moreover, $v_{2,j}$ satisfies (in generalized sense) the inflow boundary condition,
since $t(y,\omega)=0$ on $\tilde{\Gamma}_{-,j}$ (see Lemma \ref{le:esccont:1}) and $r_{m,j}-\eta>0$
(therefore $H(r_{m,j}-\eta-t(y,\omega))=1$ on $\tilde{\Gamma}_{-,j}$),
as well as the initial (energy) condition,
since $t(x,\omega)>0$ on $G\times S$ (hence $H(r_{m,j}-\eta-t(x,\omega))=0$ for all $\eta$ close to $r_{m,j}$).

The solution of \eqref{desol20a}, on the other hand, is obtained as follows. 
Let $V_{1,j}(x,\omega,\eta):=v_{1,j}(x,\omega,r_{m,j}-\eta)$. Then the problem (\ref{desol20a})
is equivalent to
\begin{gather}
{\p {V_{1,j}}\eta}+\omega\cdot\nabla_x {V_{1,j}}+
\Sigma_j V_{1,j}
=F_j,\nonumber\\
{V_{1,j}}_{|\tilde\Gamma_{-,j}}=0,
\quad V_{1,j}(\cdot,\cdot,0)=0, \label{desol20aa}
\end{gather}
where $F_j(x,\omega,\eta):=\tilde f(x,\omega,r_{m,j}-\eta)$.
Let $B_0:L^2(G\times S)\to L^2(G\times S)$ be a densily defined operator (as in section \ref{evcsd}) such that
\[
&D(B_0)=\tilde{W}_{-,0}^2(G\times S),
\quad
B_0\psi=-\omega\cdot\nabla_x\psi.
\]
Then $B_0$ generates a contraction $C^0$-semigroup $T(\eta)$, and in fact for $h\in L^2(G\times S)$
we have
(cf. \cite[Ch. XXI, Sec. 2.2, p. 222]{dautraylionsv6}, or \cite[proof of Theorem 5.15]{tervo17-up})
\[
(T(\eta)h)(x,\omega)=H(t(x,\omega)-\eta)h(x-\eta\omega,\omega),
\]
where $H$ is the Heaviside function.
The problem (\ref{desol20aa}) can be put into the abstract form
\be \label{desol20aaa}
{\p {V_{1,j}}\eta}-(B_0-\Sigma_j) {V_{1,j}}
=F_j,
\quad V_{1,j}(0)=0,
\ee
where $(V_{1,j}(\eta))(x,\omega)=V_{1,j}(x,\omega,\eta)$ and $(F_j(\eta))(x,\omega)=F_j(x,\omega,\eta)$.
The $C^0$-semigroup $G(\eta)$ generated by $B_0-\Sigma_j$ is (by the Trotter's formula)
given by
\be\label{infgen}
(G(\eta)h)(x,\omega)
=
e^{-\int_0^\eta \Sigma_j(x-\tau\omega,\omega) d\tau}H(t(x,\omega)-\eta)h(x-\eta\omega,\omega).
\ee
Hence the  solution $V_{1,j}$ is (cf. \cite[p. 439]{engelnagel}, \cite[pp. 105-108]{pazy83})
\be\label{solV}
V_{1,j}(\eta)=\int_0^{\eta}G(\eta-s)F_j(s) ds,
\ee
and thus
\bea\label{solv1}
&v_{1,j}(x,\omega,\eta)=V_{1,j}(x,\omega,r_{m,j}-\eta)
=\int_0^{r_{m,j}-\eta} \big(G(r_{m,j}-\eta-s)F_j(s)\big)(x,\omega) ds
\nonumber\\
=&
\int_0^{r_{m,j}-\eta}
e^{-\int_0^{r_{m,j}-\eta-s} \Sigma_j(x-\tau\omega,\omega) d\tau}\nonumber\\
&
\cdot
H(t(x,\omega)-(r_{m,j}-\eta-s))\tilde f_j(x-(r_{m,j}-\eta-s)\omega,\omega,r_{m,j}-s) ds.
\eea
This can be shown to be a weak (distributional) solution of \eqref{desol20a}
by a similar argument as in the proof of Lemma \ref{trathle2}.

Moreover, the weak solution of the (primary) problem
\[
\omega\cdot\nabla_x u_1+\Sigma_1 u_1={}&f_1, \\[2mm]
{u_1}_{|\Gamma_-}={}&g_1,
\]
is given by (see Lemmas \ref{trathle1} and
\ref{trathle2})
\bea\label{desol23}
u_1(x,\omega,E)
=&
\int_0^{t(x,\omega)}e^{-\int_0^t\Sigma_1(x-s\omega,\omega,E) ds}f_1(x-t\omega,\omega,E)dt
\nonumber\\
&+e^{-\int_0^{t(x,\omega)}\Sigma_1(x-s\omega,\omega,E) ds}g_1(x-t(x,\omega)\omega,\omega,E).
\eea

Hence the explicit solution of  the total primary problem
\begin{align}
\omega\cdot\nabla_x u_1+\Sigma_1(x,\omega,E) u_1={}&f_1,\nonumber\\
-{\p {(S_ju_j)}E}+\omega\cdot\nabla_x u_j
+\Sigma_j u_j
={}& f_j,\quad j=2,3
\nonumber\\
u_{|\Gamma_-}={}&g
\nonumber\\[2mm]
u_j(\cdot,\cdot,E_m)={}&0,\quad j=2,3, \label{desol24}
\end{align}
is given by $u=(u_1,u_2,u_3)$, where $u_1$ is obtained from \eqref{desol23} and $u_j$, $j=2,3$ are obtained from formulas \eqref{desol21}, \eqref{desol20ab} and \eqref{solv1},
recalling that $v_j=v_{1,j}+v_{2,j}$, $j=2,3$.
\end{example}

\subsection{A Solution Based on Neumann Series}\label{meta}

The application of Neumann series in theoretical and numerical treatises of transport problems has long traditions (cf. e.g. \cite{case63}). 
This method is known in the neutron transport community under the name \emph{source iteration} \cite{Lewis-Miller}.
It can also be used to prove existence, uniqueness and positivity of solutions for the standard transport equation
(see e.g. \cite{egger14}, \cite{frank-goudon}).
Below we formally consider these techniques for BTE-CSDA-type coupled problems. Firstly, we deal with the case where the Lebesgue index $p=2$ and after that the cases $p=\infty$ and $p=1$. We mention that the needed norm condition $\n{Q}<1$ can be more generally met in the cases $p=\infty,\ p=1$ (cf.  \cite{egger14}) than in the case $p=2$. The Neumann series solutions give (at least in some cases) a method to compute the solution without inversion of large matrices since the inverse ${\mathbb P}_0^{-1}$ (given below) can be (in some cases) computed explicitly.

\subsubsection{Solution in $L^2$-based Spaces}\label{sl2s}

Consider the transport problem (\ref{desol10}), (\ref{desol11}), (\ref{desol12}).
Setting $\phi=e^{CE}\psi$ as in Section \ref{cosyst}, we recall that the problem
then takes the equivalent form (\ref{cosyst1})-(\ref{cosyst4}).
Denote 
\bea\label{cmp4} 
T_{1,C}\phi:={}&\omega\cdot\nabla_x\phi_1+\Sigma_1\phi_1-K_{1,C}\phi\nonumber\\
T_{j,C}\phi:={}&-{\p {(S_j\phi_j)}E}+\omega\cdot\nabla_x\phi_j+
CS_j\phi_j+\Sigma_j\phi_j
- K_{j,C}\phi,\quad j=2,3,
\eea
and define a (densily defined) closed linear  operator  $T_C:L^2(G\times S\times I)^3\to L^2(G\times S\times I)^3$ by setting
\bea 
D(T_C):={}&\{\phi\in L^2(G\times S\times I)^3\ |\ T_{j,C}\phi\in L^2(G\times S\times I),\ j=1,2,3\}
\nonumber\\[2mm]
T_C\phi:={}&\big(T_{1,C}\phi,T_{2,C}\phi,T_{3,C}\phi\big).
\eea
Using these notations 
the problem (\ref{desol10}), (\ref{desol11}), (\ref{desol12}) can be expressed equivalently as
\be\label{comp5}
T_C\phi={\bf f},\quad \phi_{|\Gamma_-}={\bf g},\quad \phi_j(\cdot,\cdot,E_m)=0,\quad j=2,3,
\ee
where ${\bf f}=e^{CE}f$, ${\bf g}=e^{CE}g$ as in Section \ref{cosyst}.

Let
\[
V_0^{2,1}(I,T^2(\Gamma_-')):=\{g\in  H^1(I,T^2(\Gamma'_-))\ |\ { g}(\cdot,\cdot,E_{\rm m})=0\},
\]
and assume that $g\in T^2(\Gamma_-)\times V_0^{2,1}(I,T^2(\Gamma'_-))^2$. 
Then ${\bf g}\in T^2(\Gamma_-)\times V_0^{2,1}(I,T^2(\Gamma'_-))^2$.
Applying on ${\bf g}$ the lift operator $L$ given by
\[
\big((L{\bf g})(E)\big)(x,\omega):={\bf g}(E)(x-t(x,\omega)\omega,\omega)={\bf g}(x-t(x,\omega)\omega,\omega,E),
\]
we have $L{\bf g}\in \widetilde W^2(G\times S\times I)\times H^1(I,\tilde W^2(G\times S))^2$,
and it satisfies (cf. \cite[Lemma 5.11]{tervo17-up})
\[
\omega\cdot\nabla_x (L{\bf g})=0,\quad
(L{\bf g})_{|\Gamma_-}={\bf g}.
\]
Furthermore, the condition $g_j(\cdot,\cdot,E_{\rm m})=0$ implies that 
\be\label{comp7}
(L{\bf g}_j)(\cdot,\cdot,E_{\rm m})=0,\quad j=2,3.
\ee

Denoting
\bea\label{comp8} 
P_1(x,\omega,E,D)\phi_1:={}&\omega\cdot\nabla_x\phi_1+\Sigma_1\phi_1\nonumber\\[2mm]
P_{j,C}(x,\omega,E,D)\phi_j:={}&-{\p {(S_j\phi_j)}E}+\omega\cdot\nabla_x\phi_j+
CS_j\phi_j+\Sigma_j\phi_j,\quad j=2,3,
\eea
and
\[
{\mathbb P}_C(x,\omega,E,D)\phi:={}& \big(P_1(x,\omega,E,D)\phi_1,P_{2,C}(x,\omega,E,D)\phi_2,P_{3,C}(x,\omega,E,D)\phi_3\big), \\[2mm]
K_C\phi:={}&(K_{1,C}\phi,K_{2,C}\phi,K_{3,C}\phi),
\]
we find that $T_C={\mathbb P}_C-K_C$. 
To simplify the notation, we shall write below $T=T_C$, $K=K_C$, and ${\mathbb P}={\mathbb P}_C$.

Let ${\mathbb P}_0$ be the densely defined linear operator
acting in $L^2(G\times S\times I)^3$ such that 
\bea
D({\mathbb P}_0):={}&\{u\in \tilde W^2(G\times S\times I)\times (\tilde W^2(G\times S\times I)\cap W_1^2(G\times S\times I))^2\ |\nonumber\\
{}&\hspace{2mm}
u_{|\Gamma_-}=0,\ u(\cdot,\cdot,E_{\rm m})=0\},\nonumber\\
{\mathbb P}_0u:={}&{\mathbb P}u.
\eea
Furthermore, let $\widetilde{\mathbb P}_0:L^2(G\times S\times I)^3\to L^2(G\times S\times I)^3$ be the smallest closed extension of ${\mathbb P}_0$. 
Writing $u:=\phi-L{\bf g}$,
we see (by (\ref{comp7})) that $\phi=u+L{\bf g}$ is a solution of (\ref{comp5}) if and only if
\be\label{comp10}
({\mathbb P}-K)(u+L{\bf g})={\bf f},\quad u_{|\Gamma_-}=0,\quad u_j(\cdot,\cdot,E_m)=0,\quad j=2,3.
\ee
When $u\in D(\widetilde{\mathbb P}_0)$ the problem (\ref{comp10}) is equivalent to
\be\label{comp11}
\widetilde{\mathbb P}_0u=Ku+{\bf f}-(\mathbb P-K)(L{\bf g}).
\ee

Now suppose that for some $k=0,1,2,\dots$ the following assumptions hold:
\begin{itemize}
\item[(A1)] $\ol{\bf f}:={\bf f}-({\mathbb P}-K)(L{\bf g})\in H^k(I,L^2(G\times S)^3)$,
\item[(A2)] $Ku\in H^k(I,L^2(G\times S)^3)$ for all $u\in L^2(G\times S\times I)^3$, and
\item[(A3)] $\widetilde{\mathbb P}_0^{-1}$ exists as an operator $H^k(I,L^2(G\times S)^3)\to L^2(G\times S\times I)^3$.
\end{itemize}
These assumptions can be met if the data (and the geometry) are \emph{regular} enough;
see the example below.

Then Eq. \eqref{comp11} gives
\be\label{comp12}
u=\widetilde{\mathbb P}_0^{-1}Ku+\widetilde{\mathbb P}_0^{-1}\ol{\bf f},
\ee
which is equivalent to
\be\label{comp13}
(I-Q)u=\widetilde{\mathbb P}_0^{-1}\ol{\bf f},
\ee
where
\[
Q:=\widetilde{\mathbb P}_0^{-1}K.
\]
If $1$ belongs to the resolvent set $\rho(Q)$ of $Q$ we thus have
\be 
u=(I-Q)^{-1}\widetilde{\mathbb P}_0^{-1}\ol{\bf f},
\ee
and therefore
\be 
\phi=(I-Q)^{-1}\widetilde{\mathbb P}_0^{-1}\ol{\bf f}+L{\bf g}
=:\mc L_2({\bf f},{\bf g}).
\ee
We find that the solution operator 
\[
\mc L_2:
L^2(G\times S\times I)^3\times 
\big(T^2(\Gamma_-)\times V_0^{2,1}(I,T^2(\Gamma'_-))^2\big)\to L^2(G\times S\times  I)^3
\]
is bounded.

Assuming that 
$Q:L^2(G\times S\times I)^3\to L^2(G\times S\times  I)^3$ is bounded and that
\be\label{comp14}
\n{Q}<1,
\ee
which implies in particular that $1\in \rho(Q)$,
the solution $u$ of \eqref{comp11} can be computed through \emph{Neumann series}
\be\label{15}
u=
\sum_{k=0}^\infty Q^k(\widetilde{\mathbb P}_0^{-1}\ol{\bf f})
=
\sum_{k=0}^\infty
(\widetilde{\mathbb P}_0^{-1}K)^k(\widetilde{\mathbb P}_0^{-1}\ol{\bf f}).
\ee
Finally, the solution $\phi$ of \eqref{comp5} is then
\be\label{comp16}
\phi=\sum_{k=0}^\infty
(\widetilde{\mathbb P}_0^{-1}K)^k(\widetilde{\mathbb P}_0^{-1}\ol{\bf f})+L{\bf g},
\ee
from which the solution $\psi$ of the original problem \eqref{desol10}-\eqref{desol12} is obtained by
\[
\psi=e^{-CE}\phi.
\]

\begin{example}
 
In the case treated in Example \ref{desolex1}, choose $C=0$ and $k=0$.
Then a bounded inverse $\widehat{\mathbb P}_0^{-1}:L^2(G\times S\times I)^3\to L^2(G\times S\times I)^3$ exists
and, in fact, can be explicitly computed using formulas given in Example \ref{desolex1}
(choose $\tilde{g}=0$ in Eqs. \eqref{desol21}, \eqref{desol23}, \eqref{solv1}). The inverse is
\bea\label{comp17}
& \widetilde{\mathbb P}_0^{-1}h=\big((\widetilde{\mathbb P}_0^{-1}h)_1,(\widetilde{\mathbb P}_0^{-1}h)_2,(\widetilde{\mathbb P}_0^{-1}h)_3\big), \nonumber\\
& (\widetilde{\mathbb P}_0^{-1}h)_1(x,\omega,E)
=
\int_0^{t(x,\omega)}e^{-\int_0^t\Sigma_1(x-s\omega,\omega,E) ds}h_1(x-t\omega,\omega,E)dt,
\nonumber\\
&(\widetilde{\mathbb P}_0^{-1}h)_j(x,\omega,E)
=
{1\over{S_2(E)}}\Big(
\int_0^{r_{m,j}-R_j(E)}
e^{-\int_0^{r_{m,j}-R_j(E)-s}\Sigma_2(x-\tau\omega,\omega) d\tau}\nonumber\\
&
\quad \cdot H(t(x,\omega)-(r_{m,j}-R_j(E)-s))\tilde{h}_j(x-(r_{m,j}-R_j(E)-s)\omega,\omega,r_{m,j}-s) ds\Big),
\eea
for $j=2,3$, and where $H$ is the Heaviside function, and
\[
\tilde{h}_i(x,\omega,\eta)=S_i(R_i^{-1}(\eta))h_i(x,\omega,R_i^{-1}(\eta)),\quad i=1,2,3.
\]
Hence under the stated assumptions the solution $\psi$ $(=\phi)$ is obtained from (\ref{comp16}) (note that ${\bf f}=f$, ${\bf g}=g$ for $C=0$) where $\widetilde{\mathbb P}_0$ is computed by (\ref{comp17}).
This shows that the assumption (A1)-(A3) can indeed be met (here for $k=0$).

It is also worth noticing that the above expression for $(\widetilde{\mathbb P}_0^{-1}h)_j$, $j=2,3$, can be simplified into
\begin{multline}\label{comp17:1}
(\widetilde{\mathbb P}_0^{-1}h)_j(x,\omega,E) \\
=
{1\over{S_2(E)}}\Big(
\int_0^{\min\{r_{m,j}-R_j(E),t(x,\omega)\}}
e^{-\int_0^{s}\Sigma_2(x-\tau\omega,\omega) d\tau}\tilde{h}_j(x-s\omega,\omega,R_j(E)+s) ds\Big).
\end{multline}
\end{example}

The Neumann series based method enables one to compute the solution, or an approximation of it by considering in the series \eqref{comp16} only finitely many terms, without any explicit inversion of matrices (coming from a chosen discretization of the problem), which tend to be large due to the dimensionality of the problem, which is 6: there are 3 spatial ($x$), 2 angular ($\omega$) and one energy ($E$) dimensions.
Sufficient criteria for the condition ${\n Q}<1$ must be retrieved.

\begin{remark}
The applicability of the above method for general cases remains open.
One possibility to apply formulae like (\ref{comp17}) for spatially inhomogeneous substance (that is, $S_j$ and $\Sigma_j$ are dependent on $x\in G$) is to apply {\it domain decomposition method} in such a way that $S_j$ are assumed to be constant in subdomains.
\end{remark}

\begin{example}
In this example, we write out Eq. \eqref{comp17:1} in special case of constant $S_0\geq 0$ and $\Sigma\geq 0$.
Moreover, we consider a single particle CSDA transport equation only.
Let
\[
{}&P(x,\omega,E,D)u:=-\p{(S_0u)}{E}+\omega\cdot\nabla_x u+\Sigma u,
\]
In this case, $R(E)=\int_0^E \frac{1}{S_0}d\tau=\frac{1}{S_0}E$,
and Eq. \eqref{comp17:1} gives, when $S_0>0$,
using the notation $\eta(E):=(E_m-E)/S_0$
and noticing that $r_m:=R(E_m)=\frac{E_m}{S_0}$,
\[
(\widetilde {P}_0^{-1}h)(x,\omega,E)=\int_0^{\min\{\eta(E),t(x,\omega)\}} e^{-\Sigma s} h\big(x-s\omega,\omega,E+S_0s\big)ds,
\]
for $h\in L^2(G\times S\times I)$.
It is clear that this last formula gives the correct (explicit) expression for $\widetilde{P}_0^{-1}$
also in the case where $S_0=0$,
if we make the convention that $\eta(E)=+\infty$ for all $E\in I$ when $S_0=0$.
\end{example}

\subsubsection{Solution in $L^\infty$- and $L^1$-based Spaces}\label{sinfty1s}

The spaces $W^p(G\times S\times I)$, $T^p(\Gamma_{-})$, $V_0^{p,1}(I,T^p(\Gamma_-'))$ (and related spaces) can be analogously defined
for a general Lebesgue index $p\in [1,\infty]$ as for $p=2$ above.

A. At first, we consider the use of Neumann series in $L^\infty(G\times S\times I)$-based spaces. We assume that $f\in L^\infty(G\times S\times I)^3,\ g\in T^\infty(\Gamma_-)\times  V_0^{\infty,1}(I,T^\infty(\Gamma_-'))^2$.
Define a  linear  operator  $T_{C,\infty}:L^\infty(G\times S\times I)^3\to L^\infty(G\times S\times I)^3$ by setting
\bea 
D(T_{C,\infty}):={}&\{\phi\in L^\infty(G\times S\times I)^3\ |\ T_{j,C}\phi\in L^\infty(G\times S\times I),\ j=1,2,3\}
\nonumber\\[2mm]
T_{C,\infty}\phi:={}&\big(T_{1,C}\phi,T_{2,C}\phi,T_{3,C}\phi\big)
\eea
where $T_{j,C}$ are as in the previous section \ref{sl2s}.
Then
the problem (\ref{desol10}), (\ref{desol11}), (\ref{desol12}) can be expressed equivalently as
\be\label{comp5-a}
T_{C,\infty}\phi={\bf f},\quad \phi_{|\Gamma_-}={\bf g},\quad \phi_j(\cdot,\cdot,E_m)=0,\quad j=2,3,
\ee
where ${\bf f}=e^{CE}f$, ${\bf g}=e^{CE}g$ are as above.
The operators ${\mathbb P}_{C,\infty}$ and $K_{C,\infty}$ are also defined  as in section \ref{sl2s}  and then  $T_{C,\infty}={\mathbb P}_{C,\infty}-K_{C,\infty}$. Again, for simplicity, we denote 
$T:=T_{C,\infty}$, ${\mathbb P}:={\mathbb P}_{C,\infty}$, $K:=K_{C,\infty}$.

Let ${\mathbb P}_0$ be the  linear operator
acting in $L^\infty(G\times S\times I)^3$ such that 
\bea
D({\mathbb P}_0):={}&\{u\in  L^\infty(G\times S\times I)\ |\nonumber\\
{}&\hspace{2mm} {\rm the\ traces}\ u_{|\Gamma_-},\ u(\cdot,\cdot,E_{\rm m})\ \textrm{are\ well-defined} \nonumber\\
{}&\hspace{2mm} {\rm and }\ u_{|\Gamma_-}=0,\ u(\cdot,\cdot,E_{\rm m})=0\},\nonumber\\
{\mathbb P}_0u:={}&{\mathbb P}u.
\eea
Write $u:=\phi-L{\bf g}$. In the case where $u\in D({\mathbb P}_0)$
we  see that $\phi=u+L{\bf g}$ is a solution of (\ref{comp5-a}) if and only if
\be\label{comp11-a}
{\mathbb P}_0u=Ku+{\bf f}-(\mathbb P-K)(L{\bf g}).
\ee
Denote $\ol{\bf f}:={\bf f}-(\mathbb P-K)(L{\bf g})$.

Now suppose that  
\[
{\mathbb P}_0^{-1}:L^2(G\times S\times I)^3
\times T^\infty(\Gamma_-)\times  V_0^{\infty,1}(I,T^\infty(\Gamma_-'))^2
\to L^2(G\times S\times I)^3
\]
exists.
Then Eq. \eqref{comp11-a} gives
\be\label{comp12-a}
u={\mathbb P}_0^{-1}Ku+{\mathbb P}_0^{-1}\ol{\bf f},
\ee
which is equivalent to
\be\label{comp13-a}
(I-Q_\infty)u={\mathbb P}_0^{-1}\ol{\bf f},
\ee
when
\[
Q_\infty:={\mathbb P}_0^{-1}K.
\]
If $1$ belongs to the resolvent set $\rho(Q_\infty)$ of $Q_\infty$ we thus have
\be 
u=(I-Q_\infty)^{-1}{\mathbb P}_0^{-1}\ol{\bf f},
\ee
and therefore
\be 
\phi=u+ L{\bf g}=(I-Q_\infty)^{-1}{\mathbb P}_0^{-1}\ol{\bf f}+L{\bf g}.
\ee

In the case where
\be\label{comp14-a}
\n{Q_\infty}<1
\ee
the solution $u$ of \eqref{comp11-a} can again be computed through \emph{Neumann series}
\be\label{comp15-a}
u=
\sum_{k=0}^\infty Q_\infty^k({\mathbb P}_0^{-1}\ol{\bf f}).
\ee

By the above, the solution operator is
\be\label{comp16-a}
\mc L_\infty({\bf f},{\bf g}):=
(I-Q_\infty)^{-1}{\mathbb P}_0^{-1}\ol{\bf f}+L{\bf g},
\ee
and we find that
\[
\mc{L}_\infty:L^\infty(G\times S\times I)^3\times \big(T^\infty(\Gamma_-)\times V_0^{\infty,1}(I,T^\infty(\Gamma_-'))^2\big)\to L^\infty(G\times S\times I)^3
\] is bounded and
\bea\label{comp17-a}
&
\n{\mc L_\infty({\bf f},{\bf g})}\leq 
\n{(I-Q_\infty)^{-1}{\mathbb P}_0^{-1}}\n{\ol{\bf f}}_{L^\infty}+\n{L{\bf g}}_{L^\infty}
\eea
from which an explicit bound for $\n{\mc L_\infty}$ can be calculated.

B. Secondly, we consider the  solution in $L^1(G\times S\times I)$-based spaces.
 We assume that $f\in L^1(G\times S\times I)^3,\ g\in T^1(\Gamma_-)\times  V_0^{1,1}(I,T^1(\Gamma_-'))^2$.
Define a dense linear  operator  $T_{C,1}:L^1(G\times S\times I)^3\to L^1(G\times S\times I)^3$ by setting
\bea 
D(T_{C,1}):={}&\{\phi\in L^1(G\times S\times I)^3\ |\ T_{j,C}\phi\in L^1(G\times S\times I),\ j=1,2,3\}
\nonumber\\[2mm]
T_{C,1}\phi:={}&\big(T_{1,C}\phi,T_{2,C}\phi,T_{3,C}\phi\big)
\eea
where $T_{j,C}$ are as in  section \ref{sl2s}.
Then
the problem (\ref{desol10}), (\ref{desol11}), (\ref{desol12}) can be expressed equivalently as
\be\label{comp5-b}
T_{C,1}\phi={\bf f},\quad \phi_{|\Gamma_-}={\bf g},\quad \phi_j(\cdot,\cdot,E_m)=0,\quad j=2,3,
\ee
where ${\bf f}=e^{CE}f$, ${\bf g}=e^{CE}g$.
The operators ${\mathbb P}_{C,1}$ and $K_{C,1}$ are defined  as in section \ref{sl2s}  and then  $T_{C,1}={\mathbb P}_{C,1}-K_{C,1}$. We, for simplicity denote 
$T:=T_{C,1},\ {\mathbb P}:={\mathbb P}_{C,1},\ K:=K_{C,1}$.

Let ${\mathbb P}_0$ be the (densely defined) linear operator
acting in $L^1(G\times S\times I)^3$ such that 
\bea
D({\mathbb P}_0):={}&\{u\in  L^1(G\times S\times I)\ |\nonumber\\
{}&\hspace{2mm} {\rm the\ traces}\  u_{|\Gamma_-},\ u(\cdot,\cdot,E_{\rm m})\ \textrm{are\ well-defined}\nonumber\\
{}&\hspace{2mm}
{\rm and}\ u_{|\Gamma_-}=0,\ u(\cdot,\cdot,E_{\rm m})=0\},\nonumber\\
{\mathbb P}_0\phi:={}&{\mathbb P}\phi.
\eea
Writing $u:=\phi-L{\bf g}$,
we again see that $\phi=u+L{\bf g}$ is a solution of (\ref{comp5-b}) if and only if
\be\label{comp11-b}
{\mathbb P}_0u=Ku+\ol{\bf f}
\ee
where $\ol{\bf f}:={\bf f}-(\mathbb P-K)(L{\bf g})$.

In this $L^1$-case we proceed as follows. We assume that 
\[
{\mathbb P}_0:
L^1(G\times S\times I)^3\times
\big(T^1(\Gamma_-)\times  V_0^{1,1}(I,T^1(\Gamma_-'))^2\big)
\to L^1(G\times S\times I)^3
\]
is invertible. Let
\be\label{comp12-b}
v:={\mathbb P}_0u
\ee
or equivalently $u={\mathbb  P}_0^{-1}v$. this implies by (\ref{comp11-b}) that
\be\label{comp13-b}
v=K{\mathbb P}_0^{-1}v+\ol{\bf f}.
\ee
Denote
\be\label{comp14-b}
Q_1:=K{\mathbb  P}_0^{-1}:L^1(G\times S\times I)^3
\times
\big(T^1(\Gamma_-)\times  V_0^{1,1}(I,T^1(\Gamma_-'))^2\big)
\to L^1(G\times S\times I)^3.
\ee
Then by (\ref{comp13-b})
\[
(I-Q_1)v=\ol {\bf f}.
\]
If $1$ belongs to the resolvent set $\rho(Q_1)$ we thus have
\be 
v=(I-Q_1)^{-1}\ol{\bf f},
\ee
and therefore by (\ref{comp12-b})
\be 
u={\mathbb P}_0^{-1}(I-Q_1)^{-1}\ol{\bf f}.
\ee
Finally
\be\label{comp15-b} 
\phi={\mathbb P}_0^{-1}(I-Q_1)^{-1}\ol{\bf f}+L{\bf g}.
\ee

Again in the case where
\be\label{comp16-b}
\n{Q_1}<1
\ee
the solution $u$  can  be computed through \emph{Neumann series}
\be\label{comp17-b}
u=
\sum_{k=0}^\infty {\mathbb P}_0^{-1}Q_1^k\ol{\bf f}.
\ee
The solution $\phi$  is then
$\phi=u+L{\bf g}$.

Denote the above solution operator by
\be\label{comp18-b}
\mc L_1({\bf f},{\bf g}):=
{\mathbb P}_0^{-1}(I-Q_1)^{-1}\ol{\bf f}+L{\bf g}.
\ee
We find that
\[
\mc L_1:L^1(G\times S\times I)^3\times 
\big(T^1(\Gamma_-)\times  V_0^{1,1}(I,T^1(\Gamma_-'))^2\big)
\to L^1(G\times S\times I)^3
\]
is bounded and
\bea\label{comp19-b}
&
\n{\mc L_1({\bf f},{\bf g})}\leq 
\n{{\mathbb P}_0^{-1}(I-Q_1)^{-1}}\n{\ol{\bf f}}_{L^1}+\n{L{\bf g}}_{L^1}
\eea
from which an explicit bound for the norm $\n{\mc L_1}$ can be calculated.

\begin{remark}\label{interp}
By the above paragraphs A and B the solution operators
\[
\mc L_\infty:L^\infty(G\times S\times I)^3\times \big(T^\infty(\Gamma_-)\times V_0^{\infty,1}(I,T^\infty(\Gamma_-'))^2\big)\to L^\infty(G\times S\times I)^3,
\]
\[
\mc L_1:L^1(G\times S\times I)^3\times \big(T^1(\Gamma_-)\times V_0^{1,1}(I,T^1(\Gamma_-'))^2\big)\to L^1(G\times S\times I)^3
\]
are bounded and one is able to calculate explicit bounds for the norms
$\n{\mc L_\infty},\ \n{\mc L_1}$.
By Riez-Thorin interpolation theorem (\cite{berg}) the solution operator 
\[
\mc{L}_2:L^2(G\times S\times I)^3\times \big(T^2(\Gamma_-)\times V_0^{2,1}(I,T^2(\Gamma_-'))^2\big)\to L^2(G\times S\times I)^3
\]
is bounded as well and
\[
\n{\mc L_2}\leq {1\over 2}(\n{\mc L_\infty}+\n{\mc L_1}).
\]
\end{remark}

\subsection{An Approximative Solution Based on the Theory of Evolution Equations}\label{appevo-o}

In this section we for simplicity restrict ourselves to a single particle CSDA-equation
\begin{gather}
-{\p {(S_0\psi)}E}+\omega\cdot\nabla_x\psi+
\Sigma\psi
- K\psi=f, \label{mb1} \\
\psi_{|\Gamma_-}=g,\quad 
\psi(\cdot,\cdot,E_m)=0. \label{mb2}
\end{gather}
Suppose that the assumptions  of Theorem \ref{coth3-dd} are valid.

Another method to compute approximately the solution of the problem
(\ref{mb1}), (\ref{mb2})
which avoids the explicit inversions of matrices, can be (formally) described as follows.
Note that we will be using throughout this section the notations of Section \ref{possol}. 
After the change of unknown $\phi=e^{CE}\psi$ the problem is
\be\label{ps1} 
T_{C}\phi={\bf f},\quad \phi_{\Gamma_-}={\bf g},\quad \phi(\cdot,\cdot,E_{\rm m})=0,
\ee
where
\[
T_C\phi:=-{\p {(S_0\phi)}E}+\omega\cdot\nabla_x\phi+CS_0\phi+\Sigma\phi-K_C\phi.
\]
Assume that $g\in H^1(I,T^2(\Gamma'_-))$ and that $g(\cdot,\cdot,E_{\rm m})=0$ on $G\times S$.
Let $u:=\phi-L({\bf g})$. Then $u$ satisfies
\be\label{ps0} 
T_{C,0}u=\tilde {\bf f},\quad u_{\Gamma_-}=0,\quad u(\cdot,\cdot,E_{\rm m})=0.
\ee 
where
\[
\tilde{\bf f} := {\bf f}-T_C(L({\bf g})).
\]

By the Trotter's formula,
the semigroup $G(t)$ generated by $T_{C,0}$ is given by
\be\label{calceq2}
 G(t)\tilde {\bf f}
=
\lim_{n\to\infty}\Big(T_{B_0}(t/n)T_{A_0}(t/n)T_{-(\Sigma+C I)}(t/n)T_{K_C}(t/n)\Big)^n\tilde{\bf f},
\ee
where the convergence is uniform on compact $t$-intervals $[0,T]$.
Note that the individual semi-groups $T_{B_0}(t)$, $T_{A_0}(t)$, $T_{-(\Sigma+C I)}(t)$
and $T_{K_C}(t)$ contributing to this expression can be computed explicitly (see section \ref{possol}).
Hence we get
\bea\label{calceq4}
\psi={}&\int_0^\infty G(t)\tilde {\bf f} dt\approx \int_0^T \tilde G(t)\tilde {\bf f} dt\nonumber\\
={}&\int_0^T\lim_{n\to\infty}
\Big(T_{B_0}(t/n)T_{A_0}(t/n)T_{-(\Sigma+C I)}(t/n)T_{K_C}(t/n)\Big)^n\tilde{\bf f}dt\nonumber\\
\approx {}&
\int_0^T\Big(T_{B_0}(t/n_0)T_{A_0}(t/n_0)T_{-(\Sigma+C I)}(t/n_0)T_{K_C}(t/n_0)\Big)^{n_0}\tilde{\bf f}dt,
\eea
for large enought $T$ and $n_0$.
On the other hand, the semi-group $T_{K_C}$ generated by the bounded operator $K_C$
can be approximately computed from
\[
T_{K_C}(t)
\approx
\sum_{k=0}^{N_0} \frac{1}{k!}(tK_C)^k,
\]
for large enough $N_0$.

We point out that this semi-group theory -based approach, unlike the one given
in the previous section, does not require extra assumptions on cross-sections,
like the ones imposed by the condition \eqref{comp14}.

\begin{remark}\label{evocomp} 

Assume that $K$ is of the form
\[
(K\psi)(x,\omega,E)=\int_S\sigma(x,\omega',\omega,E)\psi(x,\omega',E)d\omega'.
\]
Furthermore,
as in Example \ref{desolex1} we assume that $S_0=S_0(E)$ (independent of $x$) and we define
$R(E):=\int_0^E{1\over{S_0(\tau)}}d \tau$, $\eta:=R(E)$, $r_m:=R(E_m)$ and $\tilde{I}:=R(I)=[0,r_{\rm m}]$.
Let
\begin{gather*}
v(x,\omega,\eta):=S_0(R^{-1}(\eta))\psi(x,\omega,R^{-1}(\eta)), \\[2mm]
\tilde\Sigma(x,\omega,\eta):=\Sigma(x,\omega,R^{-1}(\eta)),\\[2mm]
\tilde\sigma(x,\omega',\omega,\eta):=\sigma(x,\omega',\omega,R^{-1}(\eta)), \\
(\tilde K v)(x,\omega,\eta):=\int_S\tilde\sigma(x,\omega',\omega,\eta)
v(x,\omega',\eta)d\omega', \\
\tilde f(x,\omega,\eta):=S_0(R^{-1}(\eta))f(x,\omega,R^{-1}(\eta)), \\[2mm]
\tilde g(x,\omega,\eta):=S_0(R^{-1}(\eta))g(y,\omega,R^{-1}(\eta)).
\end{gather*}
 
As a further simplification, suppose that $g=0$.
After some technical considerations, the problem (\ref{mb1}), (\ref{mb2}) can be cast
into the abstract form
\be\label{mb4aa}
{\p {V}\eta}-A(\eta) V= F(\eta),\quad V(0)=0,
\ee  
where $V(\eta):=v(r_m-\eta)$, $F(\eta):=\tilde f(\cdot,\cdot,r_m-\eta)$.
In the case where appropriate assumptions are valid (cf. Theorem \ref{evoth}), the solution of  (\ref{mb4aa}) is given by
\be\label{mb7}
V(\eta)=\int_0^{\eta}U(\eta,s) F(s) ds,
\ee
where $U(\eta,s):L^2(G\times S)\to L^2(G\times S)$, $0\leq \eta\leq s\leq r_m$, is the  evolution family of operators $A(\eta)$, $\eta\in\tilde I$.

By making use of the explicit expressions for the semi-groups 
$T_{B_0}(s)$ (here $B_0$ is as in the proof of Lemma \ref{csdale1}), $T_{\tilde \Sigma(\eta)}(s)$, and $T_{\tilde K(\eta)}(s)$,
there exist some approximative methods
(e.g. {\it Cauchy-Peano approximations}) for computing the family of evolution operators $U(\eta,s)$ (cf. \cite{goldstein}). This gives an approach for calculating $V$ approximately by 
\[
V(\eta)\approx \tilde V(\eta):=\int_0^\eta \tilde U(\eta,s)F(s) ds,
\]
where $\tilde U(\eta,s)$ is an approximation of $U(\eta,s)$.
The solution $\psi$ would then be approximated as
\be\label{mb11}
\psi(x,\omega,E)\approx {1\over{S_0(E)}}\tilde V(r_m-R(E)).
\ee 
For general $g$ the idea remains the same.
It is worth studying this approach under more general assumptions as well. 
\end{remark}

\sectionspace
\section{Outlook Towards Inverse Radiation Treatment Planning}\label{irtpre}

As mentioned in Introduction, in radiation therapy the dose absorbed from particle field $\psi=(\psi_1,\psi_2,\psi_3)$ is defined by
\[
D(x)=(D\psi)(x):=\sum_{j=1}^3\int_{S\times I}\varsigma_j (x,E)\psi_j(x,\omega,E) d\omega dE,
\]
where $\varsigma_j\in L^\infty(G\times I)$, $\varsigma_j\geq 0$ are the so-called \emph{(total) stopping powers}.

Also, it is worth recalling that the component fields of $\psi$, relevant to photon and electron radiation therapy, are $\psi_1={}$photons, $\psi_2={}$electrons and $\psi_3={}$positrons.
Clearly, nothing we have said above depends on the number, or designation (to certain particle species) of fields treated (or even the dimensionality of the spaces $G$, $S$ or $I$ in fact),
with the exception that typically only charged particle fields (are assumed to) obey CSDA version of 
the transport equation (cf. \eqref{intro10}), while non-charged particles obey the
standard linear BTE (cf. \eqref{intro10a}). Thus with very minor modifications, and in particular if 
one is interested in radiation therapy, what will and has been said works, in principle, equally well in proton (and ion) therapy framework as well.

We find that $D:L^2(G\times S\times I)^3\to L^2(G)$ is a bounded linear operator and  its adjoint operator $D^*:L^2(G)\to L^2(G\times S\times I)^3$ is simply a multiplication type operator,
\[
D^*d=(\varsigma_1,\varsigma_2,\varsigma_3)d,\quad {\rm for}\ d\in L^2(G).
\]
We describe shortly an optimization problem related 
to inverse radiation treatment planning. We restrict ourselves
to  {\it external radiation therapy} 
in which the particles are inflowing through the patch(es) of patient surface. This means that in the transport problem  $f=0$ (i.e. the internal particle source vanishes)
and $g$ (the inflow particle flux)
is the variable to be controlled.
Conversely, for the internal radiation therapy problems one sets $g=0$,
and $f$ would be the variable to be controlled.
We refer to \cite{tervo17-up} (section 8) and to the references therein for a more detail exposition of \emph{inverse problem} (optimization) in this setting.

Let $g\in T^2(\Gamma_-)^3$ and let $\psi=\psi(g)\in 
{\s H}_{{\bf P}}(G\times S\times I^\circ)$ be the solution of the problem
\eqref{csda1a}--\eqref{csda3} guaranteed by Theorem \ref{m-d-j-co1}.
The deposited dose is then
\[
D(x)=(D(\psi(g)))(x),\quad x\in G.
\]
We shall also denote ${\s D}(g):=D(\psi(g))$.
 
Denote the target region by ${\bf T}\subset G$, the critical organ region by ${\bf C}\subset G$ and 
the normal tissue region by ${\bf N}\subset G$.
Then $G={\bf T}\cup {\bf C}\cup {\bf N}$ where the union is mutually disjoint.
Suppose that $ d_{\bf T}\in L^2({\bf T})$, $d_{\bf C}\in L^2({\bf C})$, $d_{\bf N}\in L^2({\bf N})$ are given dose distributions in the respective regions (for example, they may be constants). We define a strictly convex \emph{object (cost) function} $J:X\to\R$ by (see \cite{tervo17-up})
\[
J(g)={}&\frac{c_{\bf T}}{2}\n{ d_{\bf T}-{\s D}(g)}_{L^2({\bf T})}^2+\frac{c_{\bf C}}{2}\n{ d_{\bf C}-{\s D}(g)}_{L^2({\bf C})}^2\nonumber\\
&
+\frac{c_{\bf N}}{2}\n{d_{\bf N}-{\s D}(g)}_{L^2({\bf N})}^2+\frac{\ol{c}}{2}\n{g}_X^2,
\]
for some constants (weights) $c_{\bf T},c_{\bf C},c_{\bf N},\ol{c}>0$,
where $X:=T^2(\Gamma_-)\times H^1(I,T^2(\Gamma_-'))^2$ and it is equipped
with the inner product
\[
\la g,h\ra_{X} :=
\la g_1,h_1\ra_{T^2(\Gamma_-)}+\sum_{j=2}^3
 \int_I\la {{\partial g_j}\over{\partial E}},
{{\partial h_j}\over{\partial E}}\ra_{T^2(\Gamma_-')} dE.
\]
Let $Y$ be a closed subspace of $X$ defined by (here we denote $g(y,\omega,E):=(g(E))(y,\omega)$)
\[
Y:=\{g\in X\ |\ g_j(\cdot,\cdot,E_{\rm m})=0,\ j=2,3\}.
\]
A relevant \emph{admissible set} (of controls) is 
\[
U_{\rm ad}=\{g\in Y\ |\ g\geq 0\ {\rm a.e.\ on}\ G\times S\times I\},
\]
which is a \emph{closed convex} subset of $X$. 

Suppose that  the assumptions of Theorem \ref{m-d-j-co1} are valid.
Furthermore, suppose that
$\varsigma_j\ j=1,2,3$ are regular enough.
Then the minimum $\min_{g\in U_{\rm ad}}{J}(g)$ exists and the optimal control $\ol g$ is obtained from the system of variational equations (cf. \cite[the proof of Theorem 8.7]{tervo17-up}) 
\bea
&
-\la \gamma_-(\psi^*),w\ra_{T^2(\Gamma_-)^3}+\ol{c}\la \ol g,w\ra_{X}\geq 0
\quad \forall w\in U_{\rm ad}, \label{irtp16}\\
&
-\la \gamma_-(\psi^*),\ol g\ra_{T^2(\Gamma_-)^3}
+\ol{c}\n{\ol g}_{X}^2=0, \label{irtp16a}\\
&
\tilde{\bf B}_0(\psi,v)=\la\ol g,v\ra_{T^2(\Gamma_-)^3}
\quad \forall  v\in \tilde{\s H},\label{irtp18}\\
&
\tilde{\bf B}_0^*(\psi^*,v)+
c_{\bf T}\la D\psi,Dv\ra_{ L^2({\bf T})}+c_{\bf C}\la D\psi,Dv\ra_{ L^2({\bf C})}+c_{\bf N}\la D\psi,Dv\ra_{ L^2({\bf N})}\nonumber\\
&
=c_{\bf T}\la d_{\bf T},Dv\ra_{ L^2({\bf T})}+c_{\bf C}\la d_{\bf C},Dv\ra_{ L^2({\bf C})}
+c_{\bf N}\la d_{\bf N},Dv\ra_{ L^2({\bf N})}
\quad \forall v\in \tilde{\s H}.\label{irtp19}
\eea
Here $\tilde{\bf B}_0(\cdot,\cdot)$ is the bilinear form  given in (\ref{cosyst5a})  and $\tilde{\bf B}_0^*(\cdot,\cdot)$ 
is the (extended) bilinear form  \eqref{adjoint5} corresponding to the adjoint problem.
Likewise, the Hilbert space $\tilde{\s H}$ is defined in \eqref{eq:tilde_H}.
We notice that
if we used $Z:=\{(g_1,0,0)|\ g_1\in T^2(\Gamma_-)\}$ as a control space instead of $Y$, the relations (\ref{irtp16}), (\ref{irtp16a})  imply (as in \cite{tervo17-up}, Theorem 8.7) that $g_1={1\over{\ol c}}(\gamma_-(\psi^*))_+$. Here $(h)_+$ denotes the positive part of a function $h$. 

Omitting details we mention that equivalently, this variational system (\ref{irtp18}), (\ref{irtp19}) can be written as a coupled system in the operator form
\[
({\bf P}^*(x,\omega,E,D)+\Sigma^*-K^*)\psi^*
+& c_{\bf T}D^*e_{\bf T}D\psi+c_{\bf C}D^*e_{\bf C}D\psi+c_{\bf N}D^*e_{\bf N}D\psi \nonumber\\
&=c_{\bf T}D^*e_{\bf T}d_{\bf T}+c_{\bf C}D^*e_{\bf C}d_{\bf C}+c_{\bf N}D^*e_{\bf N}d_{\bf N},\nonumber\\
({\bf P}(x,\omega,E,D)+\Sigma-K)\psi{}&=0 \\
\psi_{|\Gamma_-}{}&=\ol{g},\quad \psi_j(\cdot,\cdot,E_m)=0,\quad j=2,3\\
{\psi^*}_{|\Gamma_+}{}&=0,\quad \psi_j^*(\cdot,\cdot,0)=0,\quad j=2,3,
\]
where, moreover, $\ol{g}\in U_{\rm ad}$ satisfies \eqref{irtp16}, \eqref{irtp16a}.
The operator ${\bf P}(x,\omega,E,D)$ is ${\bf P}_C(x,\omega,E,D)$ when $C=0$ (see \eqref{cosyst6a},
and the operator ${\bf P}^*(x,\omega,E,D)$ is defined by (see  section \ref{mdiss-op})
\[
P_{1}^*(x,\omega,E,D)\psi_1^*:={}&-\omega\cdot\nabla_x\psi_1^*, \\[2mm]
P_{j}^*(x,\omega,E,D)\psi_j^*:={}&S_j{\p {\psi_j^*}E}-\omega\cdot\nabla_x\psi_j^*,\quad j=2,3, \\[2mm]
{\bf P}^*(x,\omega,E,D)\psi^*:={}&\big(P_{1}^*(x,\omega,E,D)\psi_1^*, P_{2}^*(x,\omega,E,D)\psi_2^*,P_{3}^*(x,\omega,E,D)\psi_3^*\big).
\]
Finally, $\Sigma^*=\Sigma$ and $K^*$ is the adjoint on $K$.
In addition, for any subset $A$ of $G$ and function $h$ defined on $A$,
we wrote $e_Ah:G\to\R$ for the extension by zero of $h$ onto $G$.

We emphasize that here the described solution $\ol{g}$
of the optimal control problem can be used only as an \emph{initial point} for the actual treatment 
planning.  The more realistic object function is given in \cite{tervo17-up}, section 8 and it has been found to be very multi-extremal. Hence the actual optimization requires \emph{global optimization} (see e.g. \cite{pinter}).
The determination of a carefully chosen initial point for a large dimensional global optimization 
scheme is very essential for achieving (time savings and) satisfactory results (\cite{pinter14}).

 We also notice that if we contented ourselves with so-called \emph{mild solutions} (\cite{pazy83}, p. 146),
then the existence of an optimal control $\ol{g}$ (the admissible set being a subset of $T^2(\Gamma_-)^3$),
together with the explicit formula $\ol{g}={1\over c}(\gamma_-(\psi^*))_+$ for it,
could be proven under quite weak assumptions.
In any case, the validity of estimates such as (\ref{diss-co-es})
is essential for guaranteeing that ${\s D}$ be a bounded linear operator in appropriate spaces.
The use of mild solutions, however, has the drawback that the inflow boundary
conditions are not necessarily satisfied by the solutions (and thus the solutions might be non-physical).

We remark that, for example with respect to $x$-variable the solution of the transport problem is generally at most in $H^{2,(s,0,0)}(G\times S\times I^\circ)$ with $s<\frac{3}{2}$ where $H^{2,(s,0,0)}(G\times S\times I^\circ)$ is the mixed-norm Sobolev-Slobodeckij space with the Lebesgue index $2$ and with the fractional index $s$  (with respect to $x$-variable), see \cite{tervo17-up}, Examples 7.4, 7.5. This must be kept in mind when solving the problem e.g. by the finite element methods (FEM). 
Above we derived the corresponding variational formulation of the problem together with the pertinent estimates for the due bilinear form. These estimates imply (by the {\it Cea's estimate}) that the FEM-scheme in principle convergences (in relevant spaces),
when one uses appropriate basis functions.
Nevertheless, the above mentioned limited regularity of transport problems implies that the standard local interpolation results are not necessarily applicable and more advanced analysis (e.g. in choosing basis functions) is needed.

We omit in this paper further discussion of the inverse radiation treatment problem which was outlined above and  refer to e.g. \cite{frank10} for related treatments.

\appendix

\sectionspace
\section{Proof of Lemma \ref{le-m:0}}\label{ap:le-m:0}

We would like to thank Prof. Y. Chitour for
providing valuable comments and suggestions concerning the proof below.

\begin{proof}
We shall lift the problem onto $\SO(3)$. Define first,
\[
& I:S\to\R;\quad I(\omega):=\int_{\Gamma_\mu^\omega} f(\omega')d\ell(\omega'), \\
& \tilde{I}:\SO(3)\to\R;\quad \tilde{I}(R):=I(Re_3).
\]

We choose a parametrization $\gamma_\mu(s)$, $s\in [0,2\pi]$,
of $\Gamma_\mu^{e_3}$
For example $\gamma_\mu(s)=(\sqrt{1-\mu^2}\cos(s),\sqrt{1-\mu^2}\sin(s),\mu)$, $s\in [0,2\pi]$.

Since for $R\in\SO(3)$ we have $R\Gamma_\mu^{e_3}=\Gamma_\mu^{Re_3}$,
we have that $R\gamma_\mu(s)$ is a parametrization of $\Gamma_\mu^{Re_3}$
and therefore,
\[
\tilde{I}(R)
=\int_{\Gamma_\mu^{Re_3}} f(\omega')d\ell(\omega')
={}&
\int_0^{2\pi} f(R\gamma_\mu(s))\n{(R\gamma_\mu)'(s)}ds \\
={}&
\int_0^{2\pi} f(R\gamma_\mu(s))\n{R\gamma_\mu'(s)}ds \\
={}&
\int_0^{2\pi} f(R\gamma_\mu(s))\n{\gamma_\mu'(s)}ds \\
={}&
\int_{\Gamma_\mu^{e_3}} f(R\omega')d\ell(\omega'),
\]
where at the second to last phase we used the fact that $R$ is an isometry.

Letting $\mc{H}_{\SO(n)}$ be the Haar-measure on $\SO(n)$, which is normalized to $1$, we have
\[
\int_{\SO(3)} \tilde{I}(R)d\mc{H}_{\SO(3)}(R)
=
\int_{\Gamma_\mu^{e_3}} \int_{\SO(3)} f(R\omega') d\mc{H}_{\SO(3)}(R) d\ell(\omega').
\]
Now we prove that the map
\[
J:S\to\R;\quad
J(\omega'):=\int_{\SO(3)} f(R\omega') d\mc{H}_{\SO(3)}(R)
\]
is independent of $\omega'\in S$.

Indeed, if $\omega'\in S$, then
there exists $A=A(\omega')\in\SO(3)$ such that $\omega'=Ae_3$,
and hence
\[
J(\omega')=\int_{\SO(3)} f(R\omega') d\mc{H}_{\SO(3)}(R)
=\int_{\SO(3)} f(RAe_3) d\mc{H}_{\SO(3)}(R),
\]
and writing $h:\SO(3)\to\R$, $h(R')=f(R'e_3)$, we have
\[
J(\omega')=\int_{\SO(3)} h(RA) d\mc{H}_{\SO(3)}(R).
\]
Using right-invariance of the Haar-measure $\mc{H}_{\SO(3)}$ (recall that the Haar-measure on a compact group is both left and right invariant), we have
\[
J(\omega')=\int_{\SO(3)} h(R) d\mc{H}_{\SO(3)}(R)
=\int_{\SO(3)} f(Re_3) d\mc{H}_{\SO(3)}(R)
=J(e_3),
\]
which proves the claim, namely that the map $\omega'\to J(\omega')$ is constant on $S$.
Let us simply denote by $J$ its constant value.
We have thus shown that 
\[
\int_{\SO(3)} \tilde{I}(R)d\mc{H}_{\SO(3)}(R)
=
\int_{\Gamma_\mu^{e_3}} J d\ell(\omega'),
\]
i.e.
\begin{align}\label{eq:proof:integral:1}
\int_{\SO(3)} I(Re_3)d\mc{H}_{\SO(3)}(R)
=\ell(\Gamma_{\mu}^{e_3})J,
\end{align}
where $\ell(\Gamma_{\mu}^{e_3})$ is the length of $\Gamma_{\mu}^{e_3}$,
whose value is well known to be,
\[
\ell(\Gamma_{\mu}^{e_3})=2\pi\sqrt{1-\mu^2}.
\]

The map $\pi:\SO(3)\to S$; $R\mapsto Re_3$ is surjective, with
$\pi^{-1}(e_3)=\{R\in\SO(3)\ |\ Re_3=e_3\}$,
which is a subgroup of $\SO(3)$ isomorphic to $\SO(2)$.
The sphere $S$ is identified with $\SO(3)/\SO(2)$ by $\pi$.
The groups $\SO(3)$ and $\SO(2)$ being unimodular, since they are compact 
(implying that $\Delta_{\SO(3)}|_{\SO(2)}=\Delta_{\SO(2)}$, since both sides 
are $=1$, where $\Delta_G$ is the modular function of the locally compact Hausdorff group $G$),
one has for any integrable function $\tilde{g}:\SO(3)\to\R$
(see \cite[Chap. I]{deitmar09})
\[
\int_{\SO(3)} \tilde{g}(R)d\mc{H}_{\SO(3)}(R)
=\int_S \mc{I}(\tilde{g})(\omega)d\xi(\omega),
\]
where
\[
\mc{I}(\tilde{g}):S\to\R;\quad
\mc{I}(\tilde{g})(\pi(R)):=\int_{\SO(2)} \tilde{g}(RA)d\mc{H}_{\SO(2)}(A),\quad R\in\SO(3).
\]
and $\xi$ is (the) left $G$-invariant (Haar) measure on $S$.

It is known that up to a factor $C'>0$,
one has
\[
\int_S k(\omega)d\omega=C'\int_{\SO(2)} k(\omega)d\xi(\omega),
\]
for any Borel $k:S\to\R$, for which the integrals exist.

Let $g:S\to\R$ be an integrable function on $S$, and
define $\tilde{g}:\SO(3)\to\R$ by $\tilde{g}(R)=g(Re_3)$.
Then the above formula gives
\[
\int_{\SO(3)} g(Re_3)d\mc{H}_{\SO(3)}(R)
=C'\int_S \mc{I}(\tilde{g})(\omega) d\omega.
\]

Since $\SO(2)$ here is identified with $\pi^{-1}(e_3)$,
one has $Ae_3=e_3$ for all $A\in \SO(2)$, and thus
\[
\mc{I}(\tilde{g})(\pi(R))=\int_{\SO(2)} g(RAe_3)d\mc{H}_{\SO(2)}(A)
=\int_{\SO(2)} g(Re_3)d\mc{H}_{\SO(2)}(A)
=g(Re_3)=g(\pi(R)).
\]

This allows us to write the above equality as
\[
\int_{\SO(3)} g(Re_3)d\mc{H}_{\SO(3)}(R)
=C'\int_S g(\omega) d\omega.
\]

As a consequence,
\eqref{eq:proof:integral:1} becomes
\[
C'\int_{S} I(\omega)d\omega
=\ell(\Gamma_{\mu}^{e_3})J,
\]
where
\[
J=J(e_3)=\int_{\SO(3)} f(Re_3)d\mc{H}_{\SO(3)}(R)
=C'\int_S f(\omega)d\omega,
\]
and therefore (cancelling $C'$s on both sides),
\[
\int_{S} I(\omega)d\omega
=\ell(\Gamma_{\mu}^{e_3})\int_S f(\omega)d\omega.
\]
Recalling that $I(\omega)=\int_{\Gamma_\mu^\omega} f(\omega')d\ell(\omega')$,
and that
$\ell(\Gamma_{\mu}^{e_3})=2\pi\sqrt{1-\mu^2}$,
this is precisely the formula \eqref{eq:le-m:0}.
This completes the proof of Lemma \ref{le-m:0}.
\end{proof}

 
\end{document}